\documentclass[twoside]{article}

\usepackage[utf8]{inputenc}
\usepackage{amsmath,amssymb,amsthm,mathtools}
\usepackage{microtype}
\usepackage{graphicx,booktabs,tabularx,array,multirow,longtable}
\usepackage[dvipsnames]{xcolor}
\definecolor{revisionpink}{RGB}{0,0,0}
\definecolor{linkred}{RGB}{215,32,45}
\definecolor{citationgreen}{RGB}{0,138,0}
\usepackage{enumitem}
\usepackage{algorithm,algpseudocode}
\algtext*{EndFor}
\usepackage[round,authoryear]{natbib}
\usepackage{xurl}
\usepackage[colorlinks=true,linkcolor=linkred,citecolor=citationgreen,urlcolor=blue,hypertexnames=false]{hyperref}
\usepackage[capitalise,noabbrev]{cleveref}
\floatname{algorithm}{Procedure}
\crefname{algorithm}{Procedure}{Procedures}
\Crefname{algorithm}{Procedure}{Procedures}
\usepackage{etoc,placeins}
\setlist[itemize]{label=$\bullet$}
\usepackage{pgfplots}
\usepgfplotslibrary{fillbetween,groupplots}
\pgfplotsset{compat=1.18}

\newcommand{\R}{\mathbb{R}}
\newcommand{\N}{\mathbb{N}}

\newcommand{\E}{\mathbb{E}}
\newcommand{\PP}{\mathbb{P}}
\newcommand{\Sph}{\mathbb{S}^{d-1}}
\newcommand{\norm}[1]{\left\lVert #1\right\rVert}
\newcommand{\abs}[1]{\left\lvert #1\right\rvert}
\newcommand{\inner}[2]{\left\langle #1,#2\right\rangle}
\newcommand{\eps}{\varepsilon}
\newcommand{\xstar}{x^{\star}}
\newcommand{\fstar}{f^{\star}}
\newcommand{\op}{\mathrm{op}}
\newcommand{\HS}{\mathrm{HS}}
\newcommand{\Risk}{\mathfrak{R}}

\DeclareMathOperator{\Unif}{Unif}
\DeclareMathOperator{\tr}{tr}
\DeclareMathOperator{\KL}{KL}
\DeclareMathOperator{\TV}{TV}

\DeclareMathOperator{\Proj}{Proj}
\DeclareMathOperator{\sgn}{sgn}
\DeclareMathOperator*{\argmin}{arg\,min}

\theoremstyle{plain}
\newtheorem{theorem}{Theorem}
\newtheorem{proposition}{Proposition}
\newtheorem{corollary}{Corollary}
\newtheorem{alemma}{Lemma}[section]
\newtheorem{atheorem}[alemma]{Theorem}
\newtheorem{aproposition}[alemma]{Proposition}
\newtheorem{acorollary}[alemma]{Corollary}
\theoremstyle{definition}
\newtheorem{definition}{Definition}
\newtheorem{adefinition}[alemma]{Definition}
\theoremstyle{remark}

\newtheorem{aremark}[alemma]{Remark}

\newcommand{\statementrefstyle}[3]{%
  \crefname{#1}{\textcolor{black}{#2}}{\textcolor{black}{#3}}%
  \Crefname{#1}{\textcolor{black}{#2}}{\textcolor{black}{#3}}%
  \creflabelformat{#1}{##2\textcolor{linkred}{##1}##3}%
}
\statementrefstyle{theorem}{Theorem}{Theorems}
\statementrefstyle{proposition}{Proposition}{Propositions}
\statementrefstyle{corollary}{Corollary}{Corollaries}
\statementrefstyle{definition}{Definition}{Definitions}
\statementrefstyle{remark}{Remark}{Remarks}
\statementrefstyle{alemma}{Lemma}{Lemmas}
\statementrefstyle{atheorem}{Theorem}{Theorems}
\statementrefstyle{aproposition}{Proposition}{Propositions}
\statementrefstyle{acorollary}{Corollary}{Corollaries}
\statementrefstyle{adefinition}{Definition}{Definitions}
\statementrefstyle{aremark}{Remark}{Remarks}
\statementrefstyle{algorithm}{Procedure}{Procedures}
\statementrefstyle{procedure}{Procedure}{Procedures}
\DeclareRobustCommand{\statementnumref}[1]{%
  \hyperref[#1]{\textcolor{linkred}{\ref*{#1}}}%
}
\newcommand{\yesmark}{\textcolor{green}{\ensuremath{\boldsymbol{\checkmark}}}}
\newcommand{\nomark}{\textcolor{red}{\ensuremath{\boldsymbol{\times}}}}

\renewcommand{\arraystretch}{1.15}

\newcommand{\Fsc}{\mathcal{F}_{\mathrm{sc}}}
\newcommand{\Fcvx}{\mathcal{F}_{\mathrm{cvx}}}

\newcommand{\bfC}{C_{\mathrm b}}
\newcommand{\Cf}{C_{\mathrm f}}
\newcommand{\gbar}{\bar g}

\makeatletter
\def\theHALG@line{\thealgorithm.\arabic{ALG@line}}
\makeatother

\usepackage[preprint]{aistats2027}
\begin{document}
\etocdepthtag.toc{mainpart}

\runningtitle{Minimax Oracle Call Complexity, Rounds and Admissible Noise in
Zeroth-Order Optimization}
\runningauthor{Timofei Loginov, Darina Dvinskikh,  Nazarii Tupitsa,  Osman Osmanov,  Yuriy Dorn,  Alexander Gasnikov}

\twocolumn[
\aistatstitle{Minimax Oracle Call Complexity, Rounds and Admissible Noise in
Zeroth-Order Optimization of Highly Smooth Functions}
\aistatsauthor{
  Timofei Loginov\\[-1pt] {\normalfont MIRIAI} \And
  Darina Dvinskikh\\[-1pt] {\normalfont HSE University} \And
  Nazarii Tupitsa\\[-1pt] {\normalfont MBZUAI} \AND
  Osman Osmanov\\[-1pt] {\normalfont Innopolis University} \And
  Yuriy Dorn\\[-1pt] {\normalfont AI Institute MSU} \And
  Alexander Gasnikov\\[-1pt] {\normalfont Innopolis University}
}
\aistatsaddress{}
]

\begin{abstract}
We study zeroth-order optimization of convex and strongly convex functions
with H\"older smoothness of order $\beta>2$ from values corrupted by random
noise with bias independent of the method's random directions, a Lipschitz
systematic error and a bounded adversarial error. We measure the number of
function evaluations (calls), of sequential rounds of parallel calls, and the
largest tolerable systematic and adversarial errors. For strongly convex
functions with a bounded Hessian, we prove a lower bound of order
$d^2\varepsilon^{-\beta/(\beta-1)}$ on the number of calls, where $d$ is the
dimension and $\varepsilon$ the accuracy, and find the sharp power of $d$
under the stronger Hilbert--Schmidt smoothness for all $\beta>2$. An
accelerated finite-difference method attains them in
$O(\sqrt{\kappa}\log(1/\varepsilon))$ rounds for condition number $\kappa$,
like Nesterov's method with exact gradients, and tolerates the optimal
systematic and, up to logarithms, adversarial errors; rate-optimal methods
need $\Omega(\log d)$ rounds. The same number of calls is attained when the
Hessian grows at most linearly with the gradient and, for centered noise,
without an upper bound on the Hessian, also in logarithmically many rounds.
For convex functions, our lower bound improves the known one polynomially in
$d$ and is tight at the curvature of the hard instances; the systematic level
is again optimal. These bounds extend to general norms and constraints; on
the $\ell_1$ ball, mirror descent needs calls nearly linear in $d$.
Experiments confirm the theory.
\end{abstract}

\section{Introduction}
\label{sec:intro}

{
Zeroth-order optimization seeks a minimizer using only function evaluations
\citep{Conn2009,Larson2019}. It applies when derivatives are unavailable, as
in black-box attacks \citep{Chen2017ZOO} and forward-only fine-tuning of
language models \citep{Malladi2023MeZO}. We consider the noisy-value problem
\begin{equation}
\label{eq:problem}
\min_{x\in\R^d} f(x).
\end{equation}
Each oracle call returns a corrupted scalar value, whereas the unavailable
gradient has $d$ coordinates. Dimension dependence is therefore a central
complexity issue. Higher-order smoothness improves the accuracy dependence,
but its effect on the sharp power of $d$ has remained unresolved.

For strongly convex objectives with a bounded Hessian and smoothness order
$\beta>2$, the best known upper bound at accuracy $\eps$ was
$O\!\left(d^2\eps^{-\frac{\beta}{\beta-1}}\right)$ oracle calls
\citep{Akhavan2024}, whereas separable lower bounds gave
$\Omega\!\left(d^{\frac{\beta}{\beta-1}}
\eps^{-\frac{\beta}{\beta-1}}\right)$
\citep{Akhavan2020,Akhavan2024}. Their ratio
$d^{\frac{\beta-2}{\beta-1}}$ approaches $d$ as $\beta$ grows. Throughout,
asymptotic constants are independent of $d,T,\eps$; other dependencies are explicit.

We close this gap for strongly convex and convex objectives with H\"older
smoothness of order $\beta$, measured in the operator norm or in the
Hilbert--Schmidt norm. Oracle values may contain centered or
direction-oblivious random noise, a fixed Lipschitz systematic error, and a
bounded adaptive error.
}

\paragraph{{Related work.}}
{
Higher-order finite-difference and kernel methods
\citep{Polyak1990,Bach2016,Akhavan2020,Novitskii2022,Akhavan2024} attain the
$d^2$ dependence. Matching bounds were previously known only for $\beta=3$
under Hilbert--Schmidt smoothness: \citet{Yu2024} obtained
$d^{3/2}\eps^{-3/2}$ for centered noise, unrestricted Hessian magnitude, and
sufficiently large $T$. Under operator smoothness,
\citet{Lobanov2024Improved} attains the $d^2$ rate up to logarithms in
$O(\sqrt\kappa\log(1/\eps))$ rounds, {where
$\kappa$ is the condition number}. Persistent errors and round lower bounds
were studied in \citet{Gasnikov2022Power,Dvinskikh2023,Pasechnyuk2026MALN,DRY2018}.
\Cref{tab:noise,tab:convex-comparison} give detailed comparisons.
}

\paragraph{Main result.}
{
Let $q=\lceil\beta\rceil-1$ and set $s=2$ in operator norm; in
Hilbert--Schmidt norm, let $s=\frac32$ for even $q$ and
$s=2-\frac{q-1}{2(\beta-1)}$ for odd $q$. For the class
$\Fsc(\beta,1;1,3,1)$ defined in \cref{sec:setup} under i.i.d.
$\mathcal N(0,1)$ noise, denote the Gaussian minimax expected suboptimality
after $T$ oracle calls by $\Risk_T^{\mathrm G}$. Then
\[
\Risk_T^{\mathrm G}=\Theta\!\left(
\min\left\{1,
\left(\frac{d^s}{T}\right)^{\frac{\beta-1}{\beta}}\right\}\right),
\qquad d\ge2,\ T\ge1.
\]
Consequently, operator smoothness requires
$\Theta(d^2\eps^{-\frac{\beta}{\beta-1}})$ oracle calls: higher-order
smoothness improves the accuracy exponent but not the sharp $d^2$ factor.
For Hilbert--Schmidt smoothness, $s=\frac32$ for even $q$ and
$\frac32<s\le\frac53$ for odd $q$.

}

\begin{table*}[t]
\centering
\caption{{Leading dependence on $d$, $\eps$, and $\kappa$ for strongly convex objectives.}}
\label{tab:results}
\footnotesize
\renewcommand{\arraystretch}{1.06}
\setlength{\tabcolsep}{2.5pt}
\begin{tabular*}{\textwidth}{@{\extracolsep{\fill}}llllll@{}}
\toprule
\textbf{Setting} & \textbf{Reference} &
\begin{tabular}[c]{@{}c@{}}\textbf{{Oracle calls}}\\\textbf{{lower}}\end{tabular} &
\begin{tabular}[c]{@{}c@{}}\textbf{{Oracle calls}}\\\textbf{{upper}}\end{tabular} &
\begin{tabular}[c]{@{}c@{}}\textbf{{Rounds}}\\\textbf{{lower}}\end{tabular} &
\begin{tabular}[c]{@{}c@{}}\textbf{{Rounds}}\\\textbf{{upper}}\end{tabular}\\
\midrule
\multirow{3}{*}{\begin{tabular}[c]{@{}l@{}}$\alpha I\preceq\nabla^2f\preceq LI$,\\ $[f]_{\beta,\op}\le L_\beta$\end{tabular}}
& \citet{Akhavan2024} & $d^{\frac{\beta}{\beta-1}}\eps^{-\frac{\beta}{\beta-1}}$
& $d^2\eps^{-\frac{\beta}{\beta-1}}$ & -- & --\\
& \citet{Lobanov2024Improved} & -- & $d^2\eps^{-\frac{\beta}{\beta-1}}$\,$^{*}$
& -- & $\sqrt\kappa\,\log\frac1\eps$\\
& \textbf{This paper} (Thms.~\ref{thm:main}, \ref{thm:rounds-lower}) & $d^2\eps^{-\frac{\beta}{\beta-1}}$ & $d^2\eps^{-\frac{\beta}{\beta-1}}$
& $\log d$\,$^{\S}$ & $\sqrt\kappa\,\log\frac1\eps$\\
\midrule
\multirow{2}{*}{\begin{tabular}[c]{@{}l@{}}$\alpha I\preceq\nabla^2f\preceq LI$,\\ $[f]_{\beta,\HS}\le L_\beta$\end{tabular}}
& \citet{Akhavan2024} & -- & $d^2\eps^{-\frac{\beta}{\beta-1}}$ & -- & --\\
& \textbf{This paper} (Thms.~\ref{thm:main}, \ref{thm:rounds-lower})$^{\P}$ & $d^{{s}}\eps^{-\frac{\beta}{\beta-1}}$ & $d^{{s}}\eps^{-\frac{\beta}{\beta-1}}$ & $\log d$\,$^{\S}$ & $\sqrt\kappa\,\log\frac1\eps$\\
\midrule
$\norm{\nabla^2f}_{\op}\le L_0+L_1\norm{\nabla f}$
& \textbf{This paper} (Sec.~\ref{sec:hessian})\,$^{\P}$ & $d^{{s}}\eps^{-\frac{\beta}{\beta-1}}$ & $d^{{s}}\eps^{-\frac{\beta}{\beta-1}}$
& -- & $\kappa_{\mathrm{eff}}\log\frac1\eps$\,$^{\ddagger}$\\
\midrule
\multirow{2}{*}{\begin{tabular}[c]{@{}l@{}}$\alpha I\preceq\nabla^2f$,\\ centered noise\end{tabular}}
& \citet{Yu2024}, $\beta=3$, HS\,$^{\dagger}$ & $d^{\frac32}\eps^{-\frac32}$ & $d^{\frac32}\eps^{-\frac32}$ & -- & --\\
& \textbf{This paper} (Thm.~\ref{thm:nh-main})\,$^{\P\|}$ & $d^{{s}}\eps^{-\frac{\beta}{\beta-1}}$
& $d^{{s}}\eps^{-\frac{\beta}{\beta-1}}$ & -- & $\log\frac d\eps$\\
\bottomrule
\end{tabular*}

\smallskip
\parbox{\textwidth}{\footnotesize $^{*}$Up to logarithmic factors.
$^{\dagger}$Hilbert--Schmidt seminorm; asymptotic in $T$.
$^{\ddagger}${$\kappa_{\mathrm{eff}}=L_{\mathrm{eff}}/\alpha$} (\cref{sec:hessian}).
$^{\S}$Rate-optimal methods, fixed $d$, $L\ge3\alpha$.
$^{\P}${$s=2$} for $[f]_{\beta,\op}$; for $[f]_{\beta,\HS}$,
{$s=\frac32$} if $q$ is even and
{$s=2-\frac{q-1}{2(\beta-1)}$} if $q$ is odd.
$^{\|}$For $T$ above a polynomial in $d$.}
\end{table*}

\paragraph{Contributions.}
{
\begin{itemize}[leftmargin=*,itemsep=1pt,topsep=2pt]
\item \emph{Oracle calls.} Random-frame higher-order differences attain the
factors $d^2$ under operator smoothness and $d^s$ under
Hilbert--Schmidt smoothness; dense hard instances give matching lower bounds.

\item \emph{Rounds and persistent errors.} A restarted Nesterov scheme attains
the rate in $O(\sqrt\kappa\log(1/\eps))$ rounds, and rate-optimal methods
require $\Theta(\log d)$ rounds. For $\alpha$-strong convexity, the sharp
systematic-error level is $\Theta(\sqrt{\alpha\eps})$; the adaptive
threshold is determined up to a factor
$\min\{\sqrt d,\sqrt{\log(4T)}\}$.

\item \emph{Broader curvature classes.} The same rate holds under
gradient-dependent Hessian growth. With centered noise and unrestricted
Hessian magnitude, it holds once $T$ exceeds a polynomial in $d$.

\item \emph{Convexity and geometry.} We improve the convex lower bounds and
match central differences in a curvature regime. An image-ball reduction
extends the lower bounds to general constraints; on the $\ell_1$ ball, the
dimension dependence is linear up to $\log^2(2d)$.
\end{itemize}
}

\paragraph{Organization.}
{
\Cref{sec:setup} defines the model; Sections~\ref{sec:results}--\ref{sec:lower}
develop the strongly convex theory; Sections~\ref{sec:convex} and~\ref{sec:other}
cover convexity and geometry; and \cref{sec:experiments} reports experiments.
The supplement contains proofs, extensions, experimental details,
computational scope, and detailed comparisons.
}

\begin{table*}[t]
\centering
\caption{{Leading oracle-call dependence on $d$, $\eps$, and $L$ for convex objectives.}}
\label{tab:convex-comparison}
\footnotesize
\renewcommand{\arraystretch}{1.06}
\setlength{\tabcolsep}{3.5pt}
\begin{tabular*}{\textwidth}{@{\extracolsep{\fill}}llll@{}}
\toprule
\textbf{Setting} & \textbf{Reference} &
\begin{tabular}[c]{@{}c@{}}\textbf{{Oracle calls}}\\\textbf{{lower bound}}\end{tabular} &
\begin{tabular}[c]{@{}c@{}}\textbf{{Oracle calls}}\\\textbf{{upper bound}}\end{tabular}\\
\midrule
\multirow{6}{*}{\begin{tabular}[c]{@{}l@{}}$\nabla^2f\preceq LI$,\\ $[f]_{\beta,\op}\le L_\beta$\end{tabular}}
& \citet{Akhavan2020,Akhavan2024} & $\min\bigl\{d^2\eps^{-2},\,d^{\frac{\beta}{\beta-1}}\eps^{-\frac{2\beta}{\beta-1}}\bigr\}$\,$^{\|}$ & $d^2\eps^{-\frac{2\beta}{\beta-1}}$\\
& \citet{Bach2016} & -- & $d^2\eps^{-\frac{2\beta}{\beta-1}}$\\
& \citet{Gasnikov2022Survey} & -- & $d^2L\eps^{-3}$\\
& \begin{tabular}[c]{@{}l@{}}\citet{Lattimore2023Second},\\ \citet{Fokkema2024}\end{tabular}
& -- & $d^{p}\eps^{-2}$ for $p\in[3,\frac72]$\,$^{*\dagger}$\\
& \begin{tabular}[c]{@{}l@{}}\textbf{This paper}\\(Thms.~\ref{thm:convex-upper}, \ref{thm:convex-lower-main})\end{tabular} & $\min\bigl\{d^{\frac52}\eps^{-2},\,d^2\eps^{-\frac{2(\beta+1)}{\beta}},\,d^2L\eps^{-3}\bigr\}$
& $d^2\min\bigl\{\eps^{-\frac{2\beta}{\beta-1}},L\eps^{-3}\bigr\}$\\
& \begin{tabular}[c]{@{}l@{}}\textbf{This paper}, $L\lesssim\eps^{1-\frac2\beta}$\,$^{\ddagger}$\\(Sec.~\ref{sec:convex})\end{tabular} & $d^2L\eps^{-3}$
& $d^2L\eps^{-3}$\\
\midrule
\multirow{3}{*}{\begin{tabular}[c]{@{}l@{}}$\nabla^2f\preceq LI$,\\ $[f]_{\beta,\HS}\le L_\beta$\end{tabular}}
& \citet{Bach2016} & -- & $d^2\eps^{-\frac{2\beta}{\beta-1}}$\\
& \begin{tabular}[c]{@{}l@{}}\textbf{This paper}, $q$ even\\(Sec.~\ref{sec:convex})\end{tabular} & $d^2\eps^{-2}$\,$^{\S}$ & $d^{\frac32}\eps^{-\frac{2\beta}{\beta-1}}$\\
& \begin{tabular}[c]{@{}l@{}}\textbf{This paper}, $q$ odd\\(Sec.~\ref{sec:convex})\end{tabular} & $d^2\eps^{-2}$\,$^{\S}$ & $d^{2-\frac{q-1}{2(\beta-1)}}\eps^{-\frac{2\beta}{\beta-1}}$\\
\midrule
\begin{tabular}[c]{@{}l@{}}Constraints with an\\$n$-dimensional image ball\end{tabular}
& \textbf{This paper} (Sec.~\ref{sec:other}) & $\min\bigl\{n^{\frac52}\eps^{-2},\,n^2\eps^{-\frac{2(\beta+1)}{\beta}}\bigr\}$ & --\\
\midrule
$\ell_1$ ball, $[f]_{\beta,1}\le L_\beta$ & \textbf{This paper} (Sec.~\ref{sec:other}) & $d\,\eps^{-\frac{2(\beta+1)}{\beta}}$\,$^{\P}$
& $d\log^2(2d)\,\eps^{-\frac{2\beta}{\beta-1}}$\,$^{*}$\\
\bottomrule
\end{tabular*}

\smallskip
\parbox{\textwidth}{\footnotesize $a\lesssim b$ and $a\gtrsim b$ mean $a\le Cb$ and
$a\ge cb$ with $c,C>0$ independent of $d$, $\eps$ and $L$.
$^{*}$Centered sub-Gaussian noise. $^{\dagger}$Up to
logarithmic factors. $^{\ddagger}$For $d^{-\frac12}L\lesssim\eps\lesssim\min\{1,L\}$.
$^{\S}$For $\eps\lesssim d^{-\frac12\lfloor\frac q2\rfloor}$, $L\gtrsim1$.
$^{\P}$For $d^{-\frac{3\beta}4}\lesssim\eps\lesssim d^{-\frac\beta2}$.
$^{\|}$For $L\gtrsim\eps^{1-\frac2\beta}$.}
\end{table*}

\begin{table*}[t]
\centering
\caption{{Admissible persistent-error scales.}}
\label{tab:noise}
\footnotesize
\renewcommand{\arraystretch}{1.06}
\setlength{\tabcolsep}{2.8pt}
\begin{tabular*}{\textwidth}{@{\extracolsep{\fill}}lllll@{}}
\toprule
\textbf{Setting} & \textbf{Reference} & \textbf{Oblivious} &
\textbf{$M_{\max}$} & \textbf{$\delta_{\max}$}\\
\midrule
\multirow{3}{*}{\begin{tabular}[c]{@{}l@{}}$\alpha I\preceq\nabla^2f\preceq LI$,\\ $[f]_{\beta,\op}\le L_\beta$\end{tabular}}
& \begin{tabular}[c]{@{}l@{}}\citet{Akhavan2020,Akhavan2024},\\\citet{Lobanov2024Improved}\end{tabular} & \yesmark & -- & --\\
& \citet{GasnikovLobanovStonyakin2024} & -- & -- & $\gtrsim\eps^{\frac{\beta}{2(\beta-1)}}d^{-\frac{\beta}{\beta-1}}$\,$^{*}$\\
& \textbf{This paper} (Thm.~\ref{thm:main}) & \yesmark & $\asymp\eps^{\frac12}$
& \begin{tabular}[c]{@{}l@{}}$\gtrsim\eps^{\frac{\beta}{2(\beta-1)}}d^{-\frac12}$,\\ $\lesssim\eps^{\frac{\beta}{2(\beta-1)}}\sqrt{\tfrac{\log(4T)}{d}}$\end{tabular}\\
\midrule
\begin{tabular}[c]{@{}l@{}}$\alpha I\preceq\nabla^2f\preceq LI$,\\ $[f]_{\beta,\HS}\le L_\beta$\end{tabular}
& \textbf{This paper} (Thm.~\ref{thm:main}) & \yesmark & $\asymp\eps^{\frac12}$
& \begin{tabular}[c]{@{}l@{}}$\gtrsim\eps^{\frac{\beta}{2(\beta-1)}}d^{-\frac14}$\,$^{\ddagger}$,\\ $\lesssim\eps^{\frac{\beta}{2(\beta-1)}}d^{-\frac14}\sqrt{\log(4T)}$\,$^{\ddagger}$\end{tabular}\\
\midrule
\begin{tabular}[c]{@{}l@{}}$\alpha I\preceq\nabla^2f$, $[f]_{\beta,\op}\le L_\beta$,\\ $\norm{\nabla^2f}_{\op}\le L_0+L_1\norm{\nabla f}$\end{tabular}
& \textbf{This paper} (Sec.~\ref{sec:hessian}) & \yesmark & $\asymp\eps^{\frac12}$
& \begin{tabular}[c]{@{}l@{}}$\gtrsim\eps^{\frac{\beta}{2(\beta-1)}}d^{-\frac12}$,\\ $\lesssim\eps^{\frac{\beta}{2(\beta-1)}}\sqrt{\tfrac{\log(4T)}{d}}$\end{tabular}\\
\midrule
\multirow{2}{*}{\begin{tabular}[c]{@{}l@{}}$\alpha I\preceq\nabla^2f$,\\ $[f]_{\beta,\op}\le L_\beta$\end{tabular}}
& \begin{tabular}[c]{@{}l@{}}\citet{Yu2024}, $\beta=3$,\\$[f]_{\beta,\HS}\le L_\beta$\end{tabular} & \nomark & -- & --\\
& \textbf{This paper} (Thms.~\ref{thm:main}, \ref{thm:nh-main}) & \nomark & $\lesssim\eps^{\frac12}$
& $\lesssim\eps^{\frac{\beta}{2(\beta-1)}}\sqrt{\tfrac{\log(4T)}{d}}$\\
\midrule
\multirow{4}{*}{\begin{tabular}[c]{@{}l@{}}$0\preceq\nabla^2f\preceq LI$,\\ $[f]_{\beta,\op}\le L_\beta$\end{tabular}}
& \citet{Akhavan2020,Akhavan2024} & \yesmark & -- & --\\
& \citet{Dvinskikh2023}, $\sigma=0$ & -- & $\gtrsim\eps d^{-\frac12}$ & --\\
& \citet{Pasechnyuk2026MALN}, $\sigma=0$ & -- & -- & $\gtrsim\eps^{\frac32}d^{-\frac12}$\,$^{\dagger}$\\
& \textbf{This paper} (Sec.~\ref{sec:convex}) & \yesmark & $\asymp\eps$
& \begin{tabular}[c]{@{}l@{}}$\gtrsim\eps^{\min\{\frac{\beta}{\beta-1},\frac32\}}d^{-\frac12}$,\\ $\lesssim\eps\sqrt{\tfrac{\log(4T)}{d}}$\end{tabular}\\
\midrule
$\ell_1$ ball, $[f]_{\beta,1}\le L_\beta$ & \textbf{This paper} (Sec.~\ref{sec:other}) & \nomark & $\gtrsim\eps$
& $\gtrsim\eps^{\frac{\beta}{\beta-1}}$\\
\bottomrule
\end{tabular*}

\smallskip
\parbox{\textwidth}{\footnotesize $\lesssim$, $\gtrsim$, $\asymp$: up to
constants independent of $d,\eps,T$, for small $\eps$ and large $T$.
\yesmark: oblivious noise allowed; \nomark: centered noise
assumed; --: not considered. $^{*}$Fixed errors; Polyak--{\L}ojasiewicz
functions. $^{\dagger}$Fixed errors and queries in a ball. $^{\ddagger}$For
even $q$; for odd $q$, $d^{-\frac12+\frac{q-1}{4(\beta-1)}}$ instead.}
\end{table*}

\section{Problem setup}
\label{sec:setup}

{This section defines the function classes, noisy value oracle, and
complexity and robustness criteria for problem~\eqref{eq:problem}.}

\paragraph{{Notation.}}
{
Fix $\beta>2$ and set $q=\lceil\beta\rceil-1\ge2$. For operator-norm
smoothness let $\tau=0$; for Hilbert--Schmidt smoothness let
$\tau=\tau_\beta$, where
\begin{equation}
\label{eq:tau}
\tau_\beta=\tfrac{\beta-1}2\ \text{ for even }q,\qquad
\tau_\beta=\tfrac{q-1}2\ \text{ for odd }q .
\end{equation}
In both cases, set
$s=2-\frac{\tau}{\beta-1}$. Constants $c_\beta,C_\beta>0$ depend only on
$\beta$ and are fixed within each formal statement. We use
$\log^+x=\max\{0,\log x\}$ and
$\log_2^+x=\log^+x/\log2$.
}

\paragraph{Smoothness and classes.}
{
Let $\gamma=\beta-q\in(0,1]$. For the symmetric $q$-linear form $D^qf(x)$,
set $\norm A_{\op}=\sup_{\norm{v_i}\le1}\abs{A[v_1,\dots,v_q]}$ and let
$\norm A_{\HS}$ be the Euclidean norm of its coefficient array. The seminorm
$[f]_{\beta,\op}$ is the least $L_\beta$ such that
$\norm{D^qf(x)-D^qf(y)}_{\op}\le L_\beta\norm{x-y}^\gamma$; define
$[f]_{\beta,\HS}$ analogously. Since $\norm A_{\op}\le\norm A_{\HS}$, the
Hilbert--Schmidt assumption is stronger; $[f]_\beta$ denotes either seminorm.
}

For $0<\alpha\le L\le\infty$ and $R>0$, the class
$\Fsc=\Fsc(\beta,L_\beta;\alpha,L,R)$ consists of the $q$-times continuously
differentiable $f$ with $\alpha I\preceq\nabla^2f\preceq LI$, $[f]_\beta\le L_\beta$
and {unique minimizer $\xstar$ in}
$B(R)=\{x:\norm x\le R\}$. {Here $L=\infty$ allows
arbitrary Hessian norms; $\kappa=L/\alpha$ is used when $L<\infty$.} The class
$\Fcvx=\Fcvx(\beta,L_\beta;L,R)$ is defined in the same way with
$0\preceq\nabla^2f\preceq LI$ and some minimizer in $B(R)$. Methods start at
{the origin, so $R$ bounds the initial distance to a
minimizer; it does not constrain the queries or the output. These may lie}
anywhere in $\R^d$, except in \cref{sec:other}.

\paragraph{{Oracle calls, rounds, and noisy values.}}
{
Before round $k$, the past and the randomness
$\Theta_1,\dots,\Theta_{k-1}$ generate $\mathcal F_{k-1}$. The method draws
fresh $\Theta_k$ and selects all queries of that round jointly from
$(\mathcal F_{k-1},\Theta_k)$; hence rounds measure adaptive depth. It makes
$T$ oracle calls in $N$ rounds (at most $n$ per round in
\cref{thm:rounds-lower}), and its output may use the full transcript.
}

\begin{definition}[Noisy values]
\label{def:oracle}
Let $\mathcal I_k$ index round $k$. A query at $X_t$, $t\in\mathcal I_k$, returns
$Y_t=f(X_t)+\eta(X_t)+\xi^{\mathrm c}_t+\xi^{\mathrm o}_t+\delta_t$.
{Let $\mathcal E_t$ be the $\sigma$-field generated by $\mathcal F_{k-1}$,
$\Theta_k$, the queries of round $k$, and the errors of its earlier queries.}
For levels $M,\sigma,\delta\ge0$,
\begin{align}
&\abs{\eta(x)-\eta(y)}\le M\norm{x-y}\ \text{ for all }x,y;
\label{eq:noise-lip}\tag{L}\displaybreak[0]\\
&\E[\xi^{\mathrm c}_t\mid\mathcal E_t]=0,\quad
\E[(\xi^{\mathrm c}_t)^2\mid\mathcal E_t]\le\sigma^2;
\label{eq:noise-centered}\tag{C}\displaybreak[0]\\
&\E[(\xi^{\mathrm o}_t)^2\mid\mathcal F_{k-1}]\le\sigma^2,\quad
(\xi^{\mathrm o}_t)_{t\in\mathcal I_k}
\mathrel{{\perp\!\!\!\perp}}
\Theta_k\mid\mathcal F_{k-1};
\label{eq:noise-oblivious}\tag{O}\\
&\abs{\delta_t}\le\delta .
\label{eq:noise-adv}\tag{A}
\end{align}
In (O), $\perp\!\!\!\perp$ denotes conditional independence given
$\mathcal F_{k-1}$. Here $\eta$ is fixed before the run, and $\delta_t$ may
depend on the past, on $\Theta_k$, and on the query. The noise is then
\emph{admissible at levels
$(M,\sigma,\delta)$}.
\end{definition}

{
Thus $\eta$ is a fixed systematic error; $\xi^{\mathrm c}$ is centered;
$\xi^{\mathrm o}$ may be biased but is direction-oblivious within each round
\citep{Akhavan2020}; and $\delta_t$ is adaptive but bounded. \emph{Gaussian
noise} means independent $\xi^{\mathrm c}_t\sim\mathcal N(0,\sigma^2)$ and no other
component. Methods know the class parameters and $\sigma$, but not $M,\delta$.
}

\paragraph{Criteria and rates.}
{A method is \emph{$\eps$-accurate} on a class at
levels $(M,\sigma,\delta)$ if
$\E[f(\hat x)-\fstar]\le\eps$ for every $f$ in the class and every noise
admissible at these levels, where $\fstar=\min_{x\in\R^d}f(x)$.} The
\emph{admissible levels} $M_{\max}(\eps,T)$ and $\delta_{\max}(\eps,T)$ are the
suprema of the $M$ (with $\delta=0$) and of the $\delta$ (with $M=0$) for which
some method with at most $T$ {oracle calls} is $\eps$-accurate. For strongly
convex functions and $\sigma>0$, the rate $\Phi_T$ and the radius $h_T$ that
balances the bias and the noise of the gradient estimate of
\cref{sec:method} are
\begin{equation}
\label{eq:rate}
\Phi_T=\frac{L_\beta^{\frac{2}{\beta}}}{\alpha}d^{2-\frac{2+\tau}{\beta}}
\Bigl(\frac{\sigma^2}{T}\Bigr)^{\frac{\beta-1}{\beta}},\quad
h_T=\Bigl(\frac{d^{2+\tau}\sigma^2}{L_\beta^2T}\Bigr)^{\frac1{2\beta}},
\end{equation}
and $\Phi_1:=\Phi_T|_{T=1}$. Solving $\Phi_T=\eps$ gives the
\emph{{required number of oracle calls}}
\begin{equation}
\label{eq:Tstar}
{
T_\star(\eps)=d^s\sigma^2L_\beta^{\frac{2}{\beta-1}}
(\alpha\eps)^{-\frac{\beta}{\beta-1}}.}
\end{equation}
A gradient error of norm $b$ causes function-value error at the scale
$\frac{b^2}{\alpha}$ under strong convexity, so accuracy $\eps$ requires gradients with
accuracy $\sqrt{\alpha\eps}$. The radius at which the bias of the estimate
equals $\sqrt{\alpha\eps}$, and the adversarial level with the same effect, are
\begin{equation}
\label{eq:heps}
{
h_\eps=\Bigl(\frac{\alpha\eps\,d^\tau}{L_\beta^2}\Bigr)^{\frac1{2(\beta-1)}},
\hspace{-0.1em}
\delta_\eps=\frac{\sqrt{\alpha\eps}\,h_\eps}{\sqrt d}
=\sigma d^{\frac12}T_\star(\eps)^{-\frac12}.}
\end{equation}
{
The last equality requires $\sigma>0$. The scale $\delta_\eps$ equals the
standard-error scale associated with $T_\star(\eps)/d$ centered observations;
the oblivious component of \cref{def:oracle} may remain dependent across
oracle calls.
}

\section{Main results}
\label{sec:results}

{This section gives oracle-call, round, and persistent-error bounds for
strongly convex objectives under bounded, gradient-dependent, and unrestricted curvature.}

\begin{theorem}[{Oracle calls, rounds and admissible noise}]
\label{thm:main}
Let $\beta>2$, $0<\alpha<L<\infty$, $\kappa=\frac L\alpha$ and $\sigma>0$, and let
$\Fsc$ be defined with either seminorm. There are constants
$0<c_\beta\le C_\beta$ such that the following holds.
\begin{enumerate}[label=(\alph*),leftmargin=*,itemsep=1pt,topsep=2pt]
\item \emph{(Upper bound.)} Let
$T\ge C_\beta d\sqrt\kappa\,\log_2\bigl(C_\beta d\kappa\,(1+\frac{LR^2}{\Phi_1})\bigr)$.
The accelerated method of \cref{alg:stencil} with radius $h_T$ and target
{$C_\beta\Phi_T$} uses at most $T$ {oracle calls} in
$N\le(1+5\sqrt\kappa)(1+\log_2^+\frac{LR^2}{\Phi_T})$ rounds, and for every
$f\in\Fsc$ and every noise admissible at levels $(M,\sigma,\delta)$
\begin{equation}
\label{eq:main-upper}
\E[f(\hat x)-\fstar]\le C_\beta\Bigl(\Phi_T+\frac{M^2}{\alpha}
+\frac{T\delta^2}{d\sigma^2}\,\Phi_T\Bigr).
\end{equation}
\item \emph{(Lower bound.)} Let $\rho=\min\{\alpha,\frac{L-\alpha}{2}\}$ and
$\Phi_0=\min\{\alpha R^2,\rho^{\frac{2(\beta-1)}{\beta-2}}\alpha^{-1}L_\beta^{-\frac{2}{\beta-2}}\}$.
If $M\le\alpha R$, then for every $T\ge1$ every method with $T$ {oracle calls} in any
number of rounds satisfies
$\sup_{f,\eta}\E[f(\hat x)-\fstar]\ge c_\beta(\min\{\Phi_T,\Phi_0\}+\frac{M^2}{\alpha})$
under Gaussian noise, where $f$ ranges over $\Fsc$ and $\eta$ over the
$M$-Lipschitz systematic errors.
\item \emph{(Admissible levels.)} For every $\eps\in(0,LR^2)$, this method
with target $\frac\eps2$ and radius $h=c_\beta h_\eps$ is $\eps$-accurate with
$C_\beta[T_\star(\eps)+d\sqrt\kappa\,(1+\log\frac{LR^2}{\eps})]$ {oracle calls} in
at most $C_\beta\sqrt\kappa\,(1+\log\frac{LR^2}{\eps})$ rounds whenever
$M\le c_\beta\sqrt{\alpha\eps}$ and $\delta\le c_\beta\delta_\eps$. Conversely,
$M_{\max}(\eps,T)\le\sqrt{2\alpha\eps}$ for all $T$ and $\eps<\frac{\alpha R^2}{2}$, and
$\delta_{\max}(\eps,T)\le C_\beta\min\{\sqrt d,\sqrt{\log(4T)}\}\,\delta_\eps$
for all $T$ and $\eps\le\eps_0$, where $\eps_0>0$ depends on the class
parameters and $d$, but not on $T$.
\end{enumerate}
\end{theorem}

\paragraph{When the bounds match.}
Let $\delta=0$ and $M\le\alpha R$; the latter is needed, since the output $0$
has error at most $\frac{LR^2}{2}$. As $\Phi_T\le\Phi_0$ if and only if
$T\ge T_\star(\Phi_0)$, (a) and (b) give the minimax error
$\Theta(\Phi_T+\frac{M^2}{\alpha})$ over $\Fsc$ and the noises admissible at levels
$(M,\sigma,0)$, jointly in all parameters, for
$T\ge\max\{T_\star(\Phi_0),C_\beta d\sqrt\kappa\log_2(C_\beta d\kappa(1+\frac{LR^2}{\Phi_1}))\}$,
where {$T_\star(\Phi_0)=\Theta(d^s)$} for fixed other parameters.
{Part (b) also applies below this threshold; for
$L\ge3\alpha$, \cref{thm:rounds-lower}(b) gives the corresponding small-budget
lower bound for exact values.} With
$\alpha=1$, $L=3$, $R=L_\beta=\sigma=1$ and $M=\delta=0$ the minimax error is
$\Theta(\min\{1,(\frac{d^{{s}}}{T})^{\frac{\beta-1}{\beta}}\})$, for all
$d\ge2$ and $T\ge1$ (\cref{cor:normalized}).

\paragraph{Admissible errors.}
{Part (c) compares errors at the statistically optimal
budget $T=\Theta(T_\star(\eps))$: the systematic level is sharp in every
dimension and, for $\eps\le\eps_0$, the adaptive level is characterized
within $\min\{\sqrt d,\sqrt{\log(4T)}\}$. For fixed $d$, once
$T\ge T_\star(\eps)$ and $\log(4T)\ge d$, the same bounds give the factor
$\sqrt d$.}
Both admissible levels also hold under two-point feedback. The adversarial
converse extends to the high-probability criterion of \citet{Pasechnyuk2026MALN}. The two-phase
method of \cref{sec:method} attains \eqref{eq:main-upper} in
$O(\kappa(1+\log^+\frac{\alpha R^2}{\Phi_T}))$ rounds if $d\kappa$ replaces
$d\sqrt\kappa$ in (a).

\paragraph{Rounds.}
Let $\Risk_{N,n}(\mathcal C;\sigma)$ be the minimax risk of methods with $N$
rounds of at most $n$ {oracle calls} each, under Gaussian noise of variance $\sigma^2$
or exact values ($\sigma=0$).

\begin{theorem}[Lower bounds on the number of rounds]
\label{thm:rounds-lower}
Let $\beta>2$, $L_\beta,\alpha,R>0$, $L\ge3\alpha$, let $\Fsc=\Fsc(\beta,L_\beta;\alpha,L,R)$
be defined with either seminorm, and let $N,n\ge1$.
\begin{enumerate}[label=(\alph*),leftmargin=*,itemsep=1pt,topsep=2pt]
\item If $\sigma>0$, then
$\Risk_{1,n}(\Fsc;\sigma)\ge c_1n^{-\frac{2(\beta-1)}{2\beta+d}}$ with $c_1>0$
independent of $n$.
\item For every $\sigma\ge0$,
$\Risk_{N,n}(\Fsc;\sigma)\ge c_Nn^{-\nu_N}$ with
$\nu_N=\frac{2(\beta-1)((\beta-1)^N-1)}{d(\beta-2)}$, where $c_N>0$ does not depend
on $n$.
\item Let $\sigma>0$ and fix $d$ and all parameters except $n$. A method with
$N$ rounds of at most $n$ {oracle calls} whose worst-case error on $\Fsc$ is
$O(n^{-\frac{\beta-1}{\beta}}\log^k(n+1))$ for some $k$ (\emph{rate-optimal}) has
$N\ge N_\beta(d):=\max\{2,\lceil\log_{\beta-1}(1+\frac{(\beta-2)d}{2\beta})\rceil\}$.
\end{enumerate}
\end{theorem}

The hard functions satisfy $\alpha I\preceq\nabla^2f\preceq3\alpha I$. Part (c)
concerns fixed $d$ and growing {oracle calls} per round; see \cref{app:discussion} on
its constants and on finite budgets.

\paragraph{{Gradient-dependent curvature.}}
\label{sec:hessian}
The gradient-dependent condition below controls the initial region while
allowing unrestricted global curvature:
$\norm{\nabla^2f(x)}_{\op}\le L_0+L_1\norm{\nabla f(x)}$
\citep{Zhang2020Clipping}. It implies $\nabla^2f\preceq L_{\mathrm{eff}}I$ on $B(R)$ for
$L_{\mathrm{eff}}=\min\{L_0e^{2L_1R},C_\beta(L_0(1+L_1R)^{q-2}+L_\beta R^{\beta-2})\}$,
and above an explicit budget the projected two-phase method with
$L_{\mathrm{eff}}$ in place of $L$ satisfies \eqref{eq:main-upper} in
$O(\kappa_{\mathrm{eff}}(1+\log^+\frac{\alpha R^2}{\Phi_T}))$ rounds,
$\kappa_{\mathrm{eff}}=\frac{L_{\mathrm{eff}}}{\alpha}$; the lower bounds hold with $L=L_0$
if $L_0>\alpha$. {The explicit budget and proof appear in \cref{thm:gs}.}

{For the full class
$\Fsc(\beta,L_\beta;\alpha,\infty,R)$, let}
$\Risk_T(\mathcal C;\sigma)$ be the minimax error over $\mathcal C$ of methods
with at most $T$ {oracle calls} in any number of rounds, under centered noise
\eqref{eq:noise-centered} with $\eta=\xi^{\mathrm o}=\delta_t=0$.

\begin{theorem}[{Minimax rates under unrestricted curvature}]
\label{thm:nh-main}\label{thm:nh-batched-main}
Let $\beta>2$, $\alpha,L_\beta,R,\sigma>0$, let
$\mathcal C=\Fsc(\beta,L_\beta;\alpha,\infty,R)$ be defined with either
seminorm, and set $K_q=\binom{d+q}{q}$.
\begin{enumerate}[label=(\alph*),leftmargin=*,itemsep=1pt,topsep=2pt]
\item There is $C_0=C_0(\beta,\alpha,L_\beta,R,\sigma)$ such that, if
$T\ge C_0d^2$ for $2<\beta\le4$, or $T\ge C_0d^q$ for $\beta>4$, then
$c_\beta\Phi_T\le\Risk_T(\mathcal C;\sigma)\le C_\beta\Phi_T$, and the upper
bound is attained in at most $C_\beta T/d$ rounds.
\item {There are $C_1=C_1(\beta,\alpha,L_\beta,R,\sigma)$ and
$C_2=C_2(\beta,\alpha,L_\beta,R)$ such that a batched method attains
$\Risk_T(\mathcal C;\sigma)\le C_\beta\Phi_T$ in at most
$C_2(1+\log T)$ rounds whenever
$T\ge C_1K_q^{\,2+q/\gamma}\log(2+d+T)$.}
\end{enumerate}
\end{theorem}

For integer $\beta$, the threshold in (b) is at most a $\beta$-dependent
constant times $d^{\beta^2-1}\log(2+d+T)$.

{Both methods assume centered noise and first fit the
unrestricted low-order polynomial component of $f$; every pilot oracle call
is included in $T$. Part (a) uses $1+d$ pilot rounds for $3<\beta\le4$.
For $\beta>4$, and for $\beta>3$ in part (b), global polynomial minimization
is treated as exact computation in the oracle model.} For $\beta=3$ and Hilbert--Schmidt
smoothness, (a) recovers the $d^{\frac32}\eps^{-\frac32}$ oracle-call
complexity of \citet{Yu2024} on their
class and noise for all $T\ge C_0d^2$. The converse bounds of \cref{thm:main}(c)
hold here with $L=3\alpha$, since the hard instances have
$\nabla^2f\preceq3\alpha I$.

\section{{Gradient estimation}}
\label{sec:method}

{This section isolates the random-frame gradient estimate used by a standard
strongly convex optimizer.}

\paragraph{{Higher-order finite differences.}}
Let $m=\lceil\frac q2\rceil$, $t_j=\frac jm$ for $j\le m$, and let $a_1,\dots,a_m$
solve $\sum_{j=1}^ma_jt_j^{2i+1}=\mathbf 1\{i=0\}$, $i=0,\dots,m-1$. Then
$\frac1{2h}\sum_ja_j(\varphi(ht_j)-\varphi(-ht_j))=\varphi'(0)$ for every
polynomial $\varphi$ of degree at most $2m\ge q$. For $m=1$ this is the
central difference, and for $m=2$ the weights are $(\frac83,-\frac13)$. The constants
$S_1=\sum_j\abs{a_j}$, $S_2=\sum_ja_j^2$, $A_1=\sum_j\abs{a_j}t_j$ and
$C_\sigma=(S_1+(\frac{S_2}{2})^{\frac12})^2$ {measure the amplification of value errors}. A
\emph{frame} is a uniformly random orthonormal basis of $\R^d$; differences
along random directions are standard
\citep{Polyak1990,NesterovSpokoiny2017,Gasnikov2022Survey}.

{The update in the second block below is Nesterov's
standard constant-momentum scheme \citep[Section~2.2]{Nesterov2004}. The
wrapper adds restarts, growing batches, and calls to the random-frame
estimate.}

\begin{algorithm}[h]
\caption{{Random-frame finite differences with restarted Nesterov descent}}
\label{alg:stencil}
\begin{algorithmic}[1]
\Require radius $h>0$, target $\eps>0$; $\alpha,L,R,\sigma$; $\kappa=\frac L\alpha$
\Statex \emph{{Random-frame estimate $\gbar_{h,B}(z)$, one round of $2mdB$ oracle calls:}}
\State Draw $B$ independent frames $(U^{(\ell)}_i)_{i=1}^d$
\State In one round query all $z\pm ht_jU^{(\ell)}_i$, receive $Y^\pm_{ij\ell}$
\State $\hat g^{(\ell)}=\sum_{i=1}^d\frac1{2h}\sum_{j=1}^m a_j(Y^+_{ij\ell}-Y^-_{ij\ell})U^{(\ell)}_i$
\State \Return $\gbar=\frac1B\sum_{\ell=1}^B\hat g^{(\ell)}$
\Statex \emph{{Restarted Nesterov scheme, $J$ stages of $N_0$ rounds:}}
\State $N_0=\lceil5\sqrt\kappa\,\rceil$, $\theta=\frac{\sqrt\kappa-1}{\sqrt\kappa+1}$, $x^{(0)}=0$
\State $\bar V=\frac{(L+\alpha)R^2}2$, $J=\lceil\log_2\frac{\bar V}{\eps}\rceil$
\For{$j=1,\dots,J$}
\State $B_j=\max\bigl\{1,\bigl\lceil\frac{12\sqrt\kappa\,C_\sigma d\sigma^2\,2^j}{Lh^2\bar V}\bigr\rceil\bigr\}$
\State $x_{-1}=x_0=x^{(j-1)}$
\For{$k=0,\dots,N_0-1$} \Comment{one round each}
\State $y_k=x_k+\theta(x_k-x_{k-1})$, $x_{k+1}=y_k-L^{-1}\gbar_{h,B_j}(y_k)$
\EndFor
\State $x^{(j)}=x_{N_0}$
\EndFor
\State \Return $\hat x=x^{(J)}$
\end{algorithmic}
\end{algorithm}

\begin{proposition}[Random-frame error decomposition]
\label{prop:estimator}
Let $[f]_\beta\le L_\beta$ and $d\ge2$, and let the center $z$ and the radius $h$ of
round $k$ be $\mathcal F_{k-1}$-measurable. For every noise admissible at
levels $(M,\sigma,\delta)$, the {random-frame estimate in} \cref{alg:stencil} is
$\gbar=\nabla f(z)+b+\zeta+\Delta$, where $b$ is $\mathcal F_{k-1}$-measurable,
$\E[\zeta\mid\mathcal F_{k-1}]=0$, and
\begin{gather*}
\norm b\le C_\beta L_\beta d^{-\frac\tau2}h^{\beta-1}+A_1M,\qquad
\norm{\Delta}\le\frac{\sqrt d\,S_1\delta}h,\\
\E\bigl[\norm\zeta^2\mid\mathcal F_{k-1}\bigr]\le
\frac{C_\beta L_\beta^2h^{2\beta-2}}{Bd^{\tau}}+\frac{4A_1^2M^2}B
+\frac{C_\sigma d\sigma^2}{Bh^2},
\end{gather*}
{The estimate relies only on higher-order smoothness, independently of the
Hessian magnitude.}
\end{proposition}

The oblivious noise multiplies each direction $U^{(\ell)}_i$ by a coefficient
independent of it, and a uniformly random direction has mean zero. A
systematic error is a frame sum with an odd $A_1M$-Lipschitz coefficient
function on the sphere, so by the Poincar\'e inequality its second moment is
at most $\frac d{d-1}A_1^2M^2$. If those coefficients depended on the sampled
frame, the shift could scale as $\sqrt d\,M$; a bounded adaptive error can
align all $d$ differences.

\paragraph{Descent.}
{The second block of \cref{alg:stencil} restarts
Nesterov's constant-momentum scheme}. Each stage contracts the error by a
constant factor, while the growing batch keeps the estimation variance at the
new target scale. Unlike general inexact oracles
\citep{Devolder2014,dAspremont2008,Cohen2018}, a fixed gradient bias of norm
$b$ contributes only an absolute constant times $\frac{b^2}{\alpha}$. At
radius $h_T$, smoothing bias and centered noise each contribute $\Phi_T$;
systematic and adaptive errors contribute $\frac{M^2}{\alpha}$ and
$\frac{d\delta^2}{\alpha h_T^2}$, respectively. At $h_\eps$, the bias is
$\sqrt{\alpha\eps}$, which gives \cref{thm:main}(c). The
\emph{two-phase} method (\cref{alg:descent}) takes projected steps $\frac1{4L}$
and then decreasing steps with averaging, and needs the Hessian bound only on
$B(R)$.

\paragraph{{Dimension dependence.}}
Under Hilbert--Schmidt smoothness, ball averaging contracts the tensors that
determine the finite-difference bias. A $k$-fold trace of
$D^qf$ has H\"older constant at most $d^{\frac k2}L_\beta$ under the
Hilbert--Schmidt bound but enters with the weight $\frac{h^{2k}}{d^k}$: for odd
$q=2k+1$ this gains $d^{-\frac k2}=d^{-\frac\tau2}$, and for even $q=2k$, where the
$k$-fold trace is only $\gamma$-H\"older, one more difference at the radius
$hd^{-\frac14}$ gains $d^{-\frac{\beta-1}{4}}=d^{-\frac\tau2}$. The hard functions mirror
this: a radial cutoff raises the Hilbert--Schmidt seminorm by
$d^{\frac12\lfloor\frac q2\rfloor}$, which matches odd $q$, and even $q$ needs a ridge of
width $hd^{-\frac14}$. The same parity dependence appears in the matching
hard instances.

\section{{Lower bounds}}
\label{sec:lower}

{This section gives the proof ideas behind the oracle-call,
persistent-error, and round lower bounds.}

Every hard instance is a quadratic plus a smooth perturbation supported on a
ball, with a cutoff $\chi=1$ on $[0,\frac14]$ and $\chi=0$ on
$[1,\infty)$. A perturbation of slope $a$ and radius $h$ has seminorm scale
$ah^{1-\beta}$ and Hessian scale $a/h$; hence
$a=c_\beta L_\beta h^{\beta-1}$, while $L-\alpha$ bounds $h$ from above.

\paragraph{{Oracle calls.}}
For $\alpha<\mu\le2\alpha$ and $\omega\in\{\pm1\}^d$, define the hard functions
$f_\omega(x)=\frac\mu2\norm x^2+\frac{a}{\sqrt d}\inner\omega x\,\chi(\frac{\norm x^2}{h^2})$
for \cref{thm:main}(b). Each coordinate of the
minimizer $-\frac{a\omega}{\mu\sqrt d}$ carries one sign, and a wrong sign costs
$\frac{a^2}{8\alpha d}$. Flipping $\omega_i$ changes a value by at most
$\frac{2a\abs{x_i}}{\sqrt d}$, and the squares of these changes sum over $i$ to at
most $\frac{4a^2h^2}{d}$ at every query. By Assouad's lemma a constant fraction of the
signs are wrong unless {$Ta^2h^2\ge c_\beta d^2\sigma^2$}; this spreading of the
signal over all coordinates produces the factor $d^2$. Balancing
$Ta^2h^2$ with $d^2\sigma^2$ gives the error $\Phi_T$. For small $T$ this
radius violates the curvature or the radius constraint, and the bound
saturates at $\Phi_0$. Under Hilbert--Schmidt smoothness the cutoff, or for
even $q$ a ridge, changes the power of $d$ (\cref{sec:method}).

\paragraph{{Persistent errors.}}
A Lipschitz error can be observationally indistinguishable from a change of
the objective: for any unit vector $v$, the
instances $f_\pm(x)=\frac\alpha2\norm{x\mp\frac M\alpha v}^2$ with
$\eta_\pm(x)=\pm M\inner vx$ give the same values, and one of them has error
at least $\frac{M^2}{2\alpha}$ whatever the number of {oracle calls}. For the adversarial
converse, a perturbation of slope of order $\sqrt{\alpha\eps}$ and radius of
order $h_\eps$ lies along a uniformly random unit vector $v$. With
$\theta=\Theta(\min\{1,\sqrt{\frac{\log(4T)}{d}}\})$, an adversary at a level $\delta$ of
order $\theta\sqrt d\,\delta_\eps$ returns the unperturbed value outside the cone
$\{\abs{\inner vx}>\theta\norm x\}$, which the $T$ queries of a run miss with
probability at least $\frac12$; then $v$ and $-v$ cannot be told apart, as for
convex Lipschitz functions in \citet{Pasechnyuk2026MALN}.

\paragraph{Rounds.}
For \cref{thm:rounds-lower}(a), the queries of one round hit a perturbation of
radius $r$ at a random center in $B(\frac R2)$ with probability at most
$(\frac{2r}{R})^d$ each. For (b), the minimizer is hidden behind $N+1$ perturbations
with nested random centers, as in \citet[Appendix~E]{Dorn2026GradientFree} for
exact values: with probability at least $\frac34$ the queries of round $k$ miss
the ball where $f$ differs from the sum of the first $k$ perturbations, and
the recurrence has scale $r_{k+1}=c_\beta n^{-\frac1d}r_k^{\beta-1}$, which gives
$\nu_N$, and (c) follows by comparing $\nu_N$ with $\frac{\beta-1}{\beta}$.

\section{Convex functions}
\label{sec:convex}

{This section gives convex upper and lower bounds and identifies a curvature
regime in which central differences are minimax.}

Without strong convexity a gradient error of norm $b$ costs $bR$ in value.
Set $V=L$, or $V=\min\{L,L_\beta(2R)^{\beta-2}\}$ for operator smoothness
when $2<\beta\le3$. At the \emph{gradient scale}
$h=\Theta((\frac{\eps d^{\frac\tau2}}{L_\beta R})^{\frac1{\beta-1}})$, the
finite-difference bias is at most $\frac\eps R$. At the \emph{value scale}
$h=\Theta((\frac\eps V)^{\frac12})$, central differences estimate the gradient of
$f_h(x)=\E f(x+hW)$, $W\sim\Unif(B(1))$. Here only the oscillation of $f_h-f$
matters, and $\abs{\tr\nabla^2f(x)-\tr\nabla^2f(y)}\le dV$ controls it.

\begin{theorem}[Convex upper bounds]
\label{thm:convex-upper}
Let $\beta>2$, $0<L<\infty$, $f\in\Fcvx(\beta,L_\beta;L,R)$, and let the noise be
admissible at levels $(M,\sigma,\delta)$. An accelerated projected method on
$B(R)$ with the estimate of \cref{alg:stencil} at radius $h$ outputs
$\hat x\in B(R)$ with $\E[f(\hat x)-\fstar]\le\eps$ in
$\lceil(\frac{32LR^2}{\eps})^{\frac12}\rceil$ rounds, as follows.
\begin{enumerate}[label=(\alph*),leftmargin=*,itemsep=1pt,topsep=2pt]
\item \emph{(Gradient scale.)} If
$h=c_\beta(\frac{\eps d^{\frac\tau2}}{L_\beta R})^{\frac1{\beta-1}}$, $M\le c_\beta\frac\eps R$
and $\delta\le c_\beta\frac{\eps h}{R\sqrt d}$, it uses
$T\le C_\beta\frac{\sigma^2}{\eps^2}d^{\,2-\frac{\tau}{\beta-1}}
(\frac{L_\beta R^\beta}{\eps})^{\frac2{\beta-1}}+C_\beta d(1+\frac{LR^2}{\eps})$
    {oracle calls}.
\item \emph{(Value scale.)} With central differences ($m=1$, $a_1=t_1=1$)
for every $\beta$, if $h=(\frac{\eps}{8V})^{\frac12}$, $M\le\frac{\eps}{16R}$,
$\delta\le\frac{\eps h}{16R\sqrt d}$, and $\eps\le8L_\beta(2R)^\beta$ when $V<L$, it
    uses $T\le C\frac{d^2\sigma^2VR^2}{\eps^3}+32d(1+\frac{LR^2}{\eps})$ {oracle calls},
with an absolute constant $C$.
\end{enumerate}
\end{theorem}

Part (a) has the leading term of \citet{Bach2016}; (b) with $V=L$ is
classical smoothing \citep{Flaxman2005,SahaTewari2011} and uses central
differences because ball averages preserve convexity. Both hold on a
compact convex set $Q$, with a minimizer possibly on its boundary and the
diameter of $Q$ in place of $2R$ where it enters $h$, $V$, $M$ and $\delta$.

\begin{theorem}[Convex lower bounds]
\label{thm:convex-lower-main}
\label{thm:smooth-convex-hidden}
Let $\beta>2$ and $0<L\le\infty$.
\begin{enumerate}[label=(\alph*),leftmargin=*,itemsep=1pt,topsep=2pt]
\item Let $d\ge256$, $0<\eps\le c'_\beta\min\{L_\beta R^\beta,LR^2\}$, and let the
noise be Gaussian. Every method that is $\eps$-accurate on
$\Fcvx(\beta,L_\beta;L,R)$ with the operator seminorm, in any number of
rounds, must use
\[
T\ge c_\beta\frac{\sigma^2}{\eps^2}\min\Bigl\{d^{\frac52},\ d^2
\Bigl(\frac{L_\beta R^\beta}{\eps}\Bigr)^{\frac2\beta},\ \frac{d^2LR^2}{\eps}\Bigr\}.
\]
\item If $\eps<\frac{LR^2}2$ and $M>\frac{2\eps}R$, no method is $\eps$-accurate on
$\Fcvx(\beta,L_\beta;L,R)$. For every $T$ covered by \cref{thm:convex-upper}(b),
$\frac{\eps}{16R}\le M_{\max}(\eps,T)\le\frac{2\eps}R$.
\end{enumerate}
\end{theorem}

{
Part (a) extends \citet{RajaramanHan2026} to H\"older derivatives of every
order. The hard instances have curvature
$L_\eps=C_\beta L_\beta^{2/\beta}\eps^{1-2/\beta}$; flattening them when
$L<L_\eps$ yields the last term. For $L\ge L_\eps$ and
$L_\beta R^\beta d^{-\beta/4}\le\eps\le c'_\beta L_\beta R^\beta$, the lower
bound improves the separable bound \citep{Akhavan2020,Akhavan2024} by at least
$d^{(\beta-2)/\beta}$. At higher accuracy its leading term is
$d^{5/2}\sigma^2\eps^{-2}$, a factor $d^{1/2}$ larger than the previous bound.
Under Hilbert--Schmidt smoothness a $d^2\sigma^2\eps^{-2}$ lower bound holds
in the regime of \cref{cor:convex}.

For $d$ above the threshold in \cref{app:discussion} and
$L\le C_\beta L_\eps$, central differences attain the minimax complexity
$d^2\sigma^2LR^2\eps^{-3}$ over the stated accuracy range without requiring
$\beta$ or $L_\beta$ (\cref{cor:natural-curvature}(a)).
}

\section{{Norms and constraints}}
\label{sec:other}

{This section extends the convex lower bounds to general norms,
$\ell_p$ and $\ell_1$ geometries, and functional constraints.}

Let $f$ be convex with a minimizer in $\mathcal K$, and let
$\norm\cdot_{\mathcal B}$ be a norm on $\R^d$ with unit ball $\mathcal B$.
Define $[f]_{\beta,\mathcal B}$ by replacing the Euclidean norm in
$[f]_{\beta,\op}$ for $x-y$ and the arguments of $D^qf$ by $\norm\cdot_{\mathcal B}$.

\begin{definition}[Image ball]
\label{def:image-ball}
A triple $(P,z,r)$ with a linear $P:\R^d\to\R^n$, $z\in\R^n$ and $r>0$ is an
\emph{image ball} of $\mathcal K$ if $\norm{Pu}\le\norm u_{\mathcal B}$ for all
$u$ and $z+rB_2^n\subseteq P(\mathcal K)$, with $B_2^n$ the Euclidean unit ball.
\end{definition}

For a Euclidean hard function $g$ on $\R^n$, $f=g\circ P$ is convex, has
$[f]_{\beta,\mathcal B}\le[g]_{\beta,\op}$ and a minimizer in $\mathcal K$, and a
method for $f$ induces a method for $g$ with identical {oracle calls}, rounds,
and noise. Thus the Euclidean lower bounds transfer to $\mathcal K$, as
\citet{Dorn2026GradientFree} showed for exact values.

\begin{corollary}[Any norm and known constraints]
\label{thm:geometry-main}
Let $(P,z,r)$ be an image ball of $\mathcal K$ in $\R^n$, and let the noise be
Gaussian. On convex functions with $[f]_{\beta,\mathcal B}\le L_\beta$ and a
minimizer in $\mathcal K$, \cref{thm:convex-lower-main}(a) with $L=\infty$
holds with $(n,r)$ in place of $(d,R)$, and every $\eps$-accurate method needs
$c_\beta n^2\sigma^2\eps^{-2}$ {oracle calls} whenever
$\eps\le c'_\beta L_\beta r^\beta$, for every $n\ge1$. This remains true for
restricted queries and outputs and for known functional constraints.
\end{corollary}

For the $\ell_p$ ball $\mathcal K=B_p(R)$, $1\le p<2$, and $[f]_{\beta,p}\le L_\beta$
($[f]_{\beta,\mathcal B}$ for the $\ell_p$ norm), coordinate
projections are image balls with $r=Rn^{\frac12-\frac1p}$, and the bound becomes
$T\ge c_\beta\frac{\sigma^2}{\eps^2}\min\{n^{\frac52},n^{3-\frac2p}(\frac{L_\beta R^\beta}{\eps})^{\frac2\beta}\}$ for
$256\le n\le d$ and $\eps\le c'_\beta L_\beta r^\beta$. Let $f$ be convex with
$[f]_{\beta,1}\le L_\beta$, a minimizer in $B_1(R)$ and
$\norm{\nabla f}_\infty\le G_\infty$ there. Under centered, conditionally
sub-Gaussian noise, entropic mirror descent
\citep{NemirovskiYudin1983,BeckTeboulle2003} with {coordinate finite differences}
and independently resampled signs is $\eps$-accurate with
$T\le C_\beta d\log(2d)\eps^{-2}(G_\infty^2R^2+\sigma^2\log(2d)
(\frac{L_\beta R^\beta}{\eps})^{\frac2{\beta-1}})$ {oracle calls} for $M\le c_\beta\frac\eps R$ and
$\delta\le c_\beta\eps(\frac{\eps}{L_\beta R^\beta})^{\frac1{\beta-1}}$. If $G_\infty\ge C_\beta L_\beta R^{\beta-1}d^{-\frac{\beta-2}2}$, the noise is
Gaussian and $L_\beta R^\beta d^{-\frac{3\beta}4}\le\eps\le c'_\beta L_\beta R^\beta d^{-\frac\beta2}$,
every method needs $c_\beta d(\frac{L_\beta R^\beta}{\eps})^{\frac2\beta}\sigma^2\eps^{-2}$ {oracle calls}.
{The bounds are linear in $d$ up to $\log^2(2d)$, with accuracy exponents
$2\beta/(\beta-1)$ and $2(\beta+1)/\beta$;} the interval is nonempty above
the threshold in \cref{app:discussion}.

{
For functional constraints, stochastic Mirror--Prox \citep{JNT2011} with
random-frame finite differences controls objective error and constraint
violation under calibrated values, bounded multipliers, and a Lipschitz
primal--dual operator. Each oracle call returns the objective and
$n_{\mathrm c}$ constraint values; \cref{app:functional} gives the details.
}

\section{Experiments}
\label{sec:experiments}

{
\paragraph{Setup.}
We use smooth strongly convex test functions with known minimizers.
Panels~(a)--(d) evaluate: admissible systematic and adaptive errors at $d=20$
and $T=1.024\cdot10^6$; round complexity on quadratic--log-cosh objectives
with $\kappa=10,100$ and $T=1{,}310{,}720$; oracle calls needed for
$\eps=2\cdot10^{-4}$ over $8\le d\le64$ (and $d=128$ for the separable case);
and central, fourth-order frame, and kernel estimators at equal budgets in
$d=16$. Noise is i.i.d. Gaussian with $\sigma=0.03$ ($0.1$ in panel~(b)). Each
point averages at least $20$ runs, and $40$ for $d\le16$ in panel~(c).

\paragraph{Results.}
Panel~(a) scales $M$ and $\delta$ by the thresholds in
\cref{thm:main}. The error doubles at normalized levels $1.07$ and $1.30$, close to the
predictions $1.11$ and $1.37$. At fixed $T$, panel~(b) comes within $3.5$
times the statistical level in $16$ and $32$ accelerated rounds, whereas
two-phase descent takes $256$ and $2048$ rounds to come within $1.7$ times
that level.
Panel~(c) tests \eqref{eq:Tstar}: the matched operator and Hilbert--Schmidt
designs give dimension exponents $2.03\pm0.03$ and $1.62\pm0.03$, close to
$2.00$ and $1.59$. At equal oracle budgets, panel~(d) gives slopes near
$T^{-4/5}$ for the higher-order estimates and $T^{-2/3}$ for central
differences. \Cref{app:experiments} provides
full protocols, five additional experiments, and code.
}

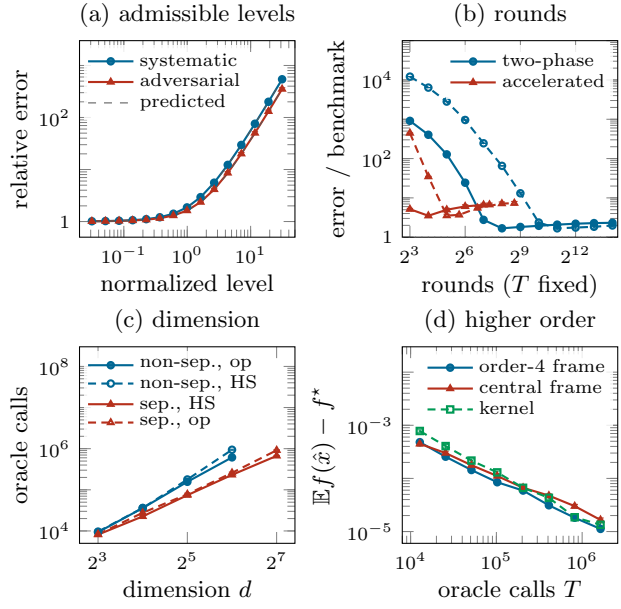
\begin{figure}[h]
\centering
\begin{tikzpicture}
\begin{groupplot}[
  group style={group size=2 by 2, horizontal sep=1.42cm, vertical sep=1.45cm},
  scale only axis, width=2.85cm, height=2.6cm,
  tick label style={font=\scriptsize}, label style={font=\footnotesize},
  title style={font=\footnotesize, yshift=-1.0ex},
  xlabel style={yshift=0.6ex}, ylabel style={yshift=-0.6ex},
  legend style={font=\scriptsize, draw=none, fill=white, fill opacity=0.85, text opacity=1,
    inner sep=1pt, row sep=-2.5pt, cells={anchor=west}},
  every axis plot/.append style={mark size=1.2pt},
]
\nextgroupplot[xmode=log, ymode=log, xlabel={normalized level}, ylabel={relative error},
  title={(a) admissible levels}, xmin=0.02, xmax=50, ymin=0.5, ymax=3000,
  ytick={1,10,100,1000}, yticklabels={$1$,,$10^2$,},
  legend style={at={(0.02,0.98)}, anchor=north west}]
\addplot[thick, MidnightBlue, mark=*, error bars/.cd, y dir=both, y explicit] table[x=x_M, y=err_lip, y error=se_lip] {experiments/exp3_levels.dat};
\addlegendentry{systematic}
\addplot[thick, BrickRed, mark=triangle*, error bars/.cd, y dir=both, y explicit] table[x=x_delta, y=err_adv, y error=se_adv] {experiments/exp3_levels.dat};
\addlegendentry{adversarial}
\addplot[gray!70!black, dashed] table[x=x_M, y=pred_lip] {experiments/exp3_levels.dat};
\addlegendentry{predicted}
\addplot[gray!70!black, dashed] table[x=x_delta, y=pred_adv] {experiments/exp3_levels.dat};
\nextgroupplot[xmode=log, ymode=log, log basis x=2, xlabel={rounds ($T$ fixed)},
  ylabel={error / benchmark}, title={(b) rounds},
  xmin=6, xmax=20000, ymin=1, ymax=1e5, ytick={1,10,100,1000,10000,100000},
  yticklabels={$1$,,$10^2$,,$10^4$,}, xtick={8,64,512,4096},
  legend style={at={(0.98,0.98)}, anchor=north east}]
\addplot[thick, MidnightBlue, mark=*, error bars/.cd, y dir=both, y explicit] table[x=N, y=rel_tp, y error=se_rel_tp] {experiments/exp4_rounds_kappa10.dat};
\addlegendentry{two-phase}
\addplot[thick, BrickRed, mark=triangle*, restrict expr to domain={\thisrow{new_acc}}{1:1}, error bars/.cd, y dir=both, y explicit] table[x=rounds_acc, y=rel_acc, y error=se_rel_acc] {experiments/exp4_rounds_kappa10.dat};
\addlegendentry{accelerated}
\addplot[thick, MidnightBlue, densely dashed, mark=o, mark options={solid}, error bars/.cd, y dir=both, y explicit, error bar style={solid}] table[x=N, y=rel_tp, y error=se_rel_tp] {experiments/exp4_rounds_kappa100.dat};
\addplot[thick, BrickRed, densely dashed, mark=triangle, mark options={solid}, restrict expr to domain={\thisrow{new_acc}}{1:1}, error bars/.cd, y dir=both, y explicit, error bar style={solid}] table[x=rounds_acc, y=rel_acc, y error=se_rel_acc] {experiments/exp4_rounds_kappa100.dat};
\addplot[gray!70!black, dotted, domain=6:20000, samples=2] {1};
\nextgroupplot[xmode=log, ymode=log, log basis x=2, xlabel={dimension $d$}, ylabel={{oracle calls}},
  title={(c) dimension}, xmin=6, xmax=170, ymin=5e3, ymax=3e8,
  ytick={1e4,1e5,1e6,1e7,1e8}, yticklabels={$10^4$,,$10^6$,,$10^8$},
  legend style={at={(0.02,0.98)}, anchor=north west}]
\addplot[thick, MidnightBlue, mark=*, error bars/.cd, y dir=both, y explicit] table[x=d, y=calls_ns, y error=se_ns] {experiments/exp5_dimension.dat};
\addlegendentry{non-sep., op}
\addplot[thick, MidnightBlue, densely dashed, mark=o, mark options={solid}, error bars/.cd, y dir=both, y explicit, error bar style={solid}] table[x=d, y=calls_ns_hs, y error=se_ns_hs] {experiments/exp5_dimension.dat};
\addlegendentry{non-sep., HS}
\addplot[thick, BrickRed, mark=triangle*, error bars/.cd, y dir=both, y explicit] table[x=d, y=calls_sep, y error=se_sep] {experiments/exp5_dimension.dat};
\addlegendentry{sep., HS}
\addplot[thick, BrickRed, densely dashed, mark=triangle, mark options={solid}, error bars/.cd, y dir=both, y explicit, error bar style={solid}] table[x=d, y=calls_sep_op, y error=se_sep_op] {experiments/exp5_dimension.dat};
\addlegendentry{sep., op}
\nextgroupplot[xmode=log, ymode=log, xlabel={{oracle calls $T$}}, ylabel={$\E f(\hat x)-\fstar$},
  title={(d) higher order}, xmin=8e3, xmax=2.6e6, ymin=6e-6, ymax=3e-2,
  ytick={1e-5,1e-4,1e-3,1e-2}, yticklabels={$10^{-5}$,,$10^{-3}$,},
  legend style={at={(0.98,0.98)}, anchor=north east}]
\addplot[thick, MidnightBlue, mark=*, error bars/.cd, y dir=both, y explicit] table[x=T, y=err_m2, y error=se_m2] {experiments/exp7_highorder.dat};
\addlegendentry{order-4 frame}
\addplot[thick, BrickRed, mark=triangle*, error bars/.cd, y dir=both, y explicit] table[x=T, y=err_m1, y error=se_m1] {experiments/exp7_highorder.dat};
\addlegendentry{central frame}
\addplot[thick, ForestGreen, densely dashed, mark=square, mark options={solid}, error bars/.cd, y dir=both, y explicit, error bar style={solid}] table[x=T, y=err_ker, y error=se_ker] {experiments/exp7_highorder.dat};
\addlegendentry{kernel}
\end{groupplot}
\end{tikzpicture}
\caption{Checks of (a) admissible errors, (b) rounds, (c) dimension scaling,
and (d) higher-order estimators. Bars: one standard error; op/HS denote the
smoothness norm.}
\label{fig:experiments}
\end{figure}

\FloatBarrier
\label{page:main-end}
\clearpage
\section*{AI Use Statement}
{\color{revisionpink}Generative AI tools were used to improve the clarity and
presentation of the manuscript and to help with the implementation and
checking of the numerical experiments. The authors reviewed all AI-assisted
changes, independently verified the reported results, and take full
responsibility for the final content.}

\bibliographystyle{plainnat}
\bibliography{refs}
\section*{Checklist}

\begin{enumerate}

  \item For all models and algorithms presented, check if you include:
  \begin{enumerate}
    \item A clear description of the mathematical setting, assumptions, algorithm, and/or model. [Yes] Sections~\ref{sec:setup}--\ref{sec:other} and the algorithm descriptions specify the model and assumptions.
    \item An analysis of the properties and complexity (time, space, sample size) of any algorithm. [Yes] Calls and sequential rounds are counted, including the calls of the polynomial fit; for $\beta>4$ (for $\beta>3$ in the method with logarithmically many rounds) the methods without a Hessian bound use global polynomial minimization, whose cost is not analyzed.
    \item (Optional) Anonymized source code, with specification of all dependencies, including external libraries. [Yes] The supplementary code lists its dependencies (NumPy; Matplotlib only for an optional preview), the seeds and the commands that reproduce the figure.
  \end{enumerate}

  \item For any theoretical claim, check if you include:
  \begin{enumerate}
    \item Statements of the full set of assumptions of all theoretical results. [Yes] Each result specifies its function class, noise model, and budget regime.
    \item Complete proofs of all theoretical results. [Yes] Proofs of the new results are in the supplementary material; standard results are cited.
    \item Clear explanations of any assumptions. [Yes] The two tensor seminorms, the components of the noise and the calibration of constraint values are defined and discussed.
  \end{enumerate}

  \item For all figures and tables that present empirical results, check if you include:
  \begin{enumerate}
    \item The code, data, and instructions needed to reproduce the main experimental results (either in the supplemental material or as a URL). [Yes] The supplementary material contains the code, the data tables behind the figure and the run log.
    \item All the training details (e.g., data splits, hyperparameters, how they were chosen). [Yes] Appendix~\ref{app:experiments} gives every objective, parameter, budget and tuning rule; there are no data splits.
    \item A clear definition of the specific measure or statistics and error bars (e.g., with respect to the random seed after running experiments multiple times). [Yes] Error bars are one standard error over independent runs; Appendix~\ref{app:experiments} gives the number of runs for each experiment.
    \item A description of the computing infrastructure used. (e.g., type of GPUs, internal cluster, or cloud provider). [Yes] All experiments run on a CPU; Appendix~\ref{app:experiments} reports the processor, the run time and the software versions.
  \end{enumerate}

  \item If you are using existing assets (e.g., code, data, models) or curating/releasing new assets, check if you include:
  \begin{enumerate}
    \item Citations of the creator If your work uses existing assets. [Yes] NumPy is cited.
    \item The license information of the assets, if applicable. [Yes] NumPy is distributed under a permissive, BSD-compatible license.
    \item New assets either in the supplemental material or as a URL, if applicable. [Yes] The code and the synthetic data are supplied.
    \item Information about consent from data providers/curators. [Not Applicable] All experimental data are generated synthetically.
    \item Discussion of sensible content if applicable, e.g., personally identifiable information or offensive content. [Not Applicable] The data contain only synthetic numerical values.
  \end{enumerate}

  \item If you used crowdsourcing or conducted research with human subjects, check if you include:
  \begin{enumerate}
    \item The full text of instructions given to participants and screenshots. [Not Applicable] No human subjects or crowdsourcing are involved.
    \item Descriptions of potential participant risks, with links to Institutional Review Board (IRB) approvals if applicable. [Not Applicable] No human subjects or crowdsourcing are involved.
    \item The estimated hourly wage paid to participants and the total amount spent on participant compensation. [Not Applicable] No human subjects or crowdsourcing are involved.
  \end{enumerate}

\end{enumerate}

\clearpage
\onecolumn
\aistatstitle{Supplementary Material}
\appendix
\numberwithin{equation}{section}
\etocdepthtag.toc{appendixpart}
\etocsettagdepth{mainpart}{none}
\etocsettagdepth{appendixpart}{section}
\renewcommand{\contentsname}{Contents of the supplementary material}
\tableofcontents
\clearpage
\section{Guide to the proofs}
\label{app:overview}

The notation and the noise model are those of \cref{sec:setup}. Constants
$c_\beta,C_\beta$ depend only on $\beta$ and on the fixed one-dimensional
stencil and cutoff functions; other dependences are stated where they occur.
Expectations are over the randomness of the method and of the noise. For
unconstrained problems
the loss is $f(\hat x)-\min_{x\in\R^d}f(x)$; over a known compact set $Q$ it
is $f(\hat x)-\min_{x\in Q}f(x)$ with $\hat x\in Q$; for functional constraints the objective
gap and the constraint violation are bounded separately.

\begin{center}
\small
\renewcommand{\arraystretch}{1.12}
\begin{tabular}{@{}p{0.43\textwidth}p{0.52\textwidth}@{}}
\toprule
Statement in the main text & Proof\\
\midrule
\Cref{prop:estimator} (stencil estimate) & \cref{app:estimator,app:estimator-unified}, \cref{prop:unified-estimator}\\
{\Cref{alg:stencil}, random-frame estimate and restarted Nesterov descent} & \cref{app:accelerated-parameters}: \cref{lem:multistage,thm:acc-explicit}\\
\Cref{thm:main}(a) and the achievable levels in (c) & \cref{app:rounds}: \cref{thm:sc-rounds,thm:acc-explicit}\\
\Cref{thm:main}(b) and the converse levels in (c) & \cref{app:lower}: \cref{thm:radial-lower,thm:ridge-lower,prop:lipschitz-floor}; \cref{thm:adv-unified}\\
Matching bounds and normalized minimax error (\cref{sec:results}) & \cref{thm:radial-lower,thm:ridge-lower,cor:normalized}\\
\Cref{thm:rounds-lower} (rounds) & \cref{app:rounds-lower}: \cref{thm:rl-one-round,thm:rl-nested,cor:rl-threshold}\\
Hessian growing with the gradient (\cref{sec:hessian}) & \cref{app:gs}: \cref{thm:gs}\\
\Cref{thm:nh-main}(a), no Hessian bound & \cref{app:no-hessian}: \cref{thm:nh-upper-general,thm:nh-upper-quadratic}; lower bound from \cref{thm:radial-lower,thm:ridge-lower}\\
\Cref{thm:nh-main}(b), logarithmic rounds & \cref{app:batched-no-hessian}: \cref{thm:nh-batched}\\
Converse levels on the classes of \cref{thm:gs} ($L=L_0$) and \cref{thm:nh-main} ($L=3\alpha$) & \cref{prop:lipschitz-floor}; \cref{thm:adv-unified}; proof of \cref{thm:gs} in \cref{app:gs}\\
Two-point feedback (\cref{sec:results}) & \cref{rem:adv-twopoint,rem:adv-probability}\\
\Cref{thm:convex-upper} (convex methods) & \cref{app:convex-upper}: \cref{cor:direct-convex,thm:value-analysis}\\
\Cref{thm:convex-lower-main} and the fixed class (\cref{sec:convex}) & \cref{app:convex-lower,app:convex-matching}: \cref{cor:natural-curvature,cor:mmax-cvx}; \cref{prop:lipschitz-floor}\\
Convex lower bound for $d<256$ and Hilbert--Schmidt smoothness & \cref{cor:convex}\\
Convex adversarial levels (\cref{tab:noise}) & \cref{app:adv-cvx,app:adv-cvx-upper}\\
\Cref{thm:geometry-main}, $\ell_p$ and $\ell_1$ balls & \cref{app:geometry}: \cref{thm:geometry-lower,cor:geometry-constrained,cor:lp-lower,thm:l1-upper,rem:l1-local}\\
Functional constraints (\cref{sec:other}) & \cref{app:functional}: \cref{prop:fc-obstruction,cor:fc-stencil}\\
Comparison with other works (\cref{tab:results,tab:convex-comparison,tab:noise}) & \cref{app:prior-convex}, \cref{app:related}\\
Experiments (\cref{sec:experiments}) & \cref{app:experiments}\\
Comparison with concurrent work & \cref{app:concurrent}\\
\bottomrule
\end{tabular}
\end{center}

\Cref{rem:centered-filtration,rem:adaptive-lipschitz} show that two features of
\cref{def:oracle} cannot be dropped. \Cref{app:batched-no-hessian} also gives a
convex method without a Hessian bound (\cref{thm:convex-nh-direct}). The sizes of the constants quoted in
\cref{app:discussion} are discussed in \cref{rem:rl-comparison}, after
\cref{thm:nh-batched} and in \cref{rem:lp-upper}.

\section{Preliminaries}
\label{app:prelim}

We write $B(z,r)=z+B(r)$ for the closed Euclidean ball of radius $r$
around $z$, and $e_1,\dots,e_d$ for the standard basis of $\R^d$.

\subsection{Tensors}
\label{app:tensors}

A $k$-tensor on $\R^d$ is an array $A=(A_{i_1\dots i_k})$; it is symmetric
if its entries do not change under permutations of the indices. We write
$\inner{A}{B}_{\HS}=\sum A_{i_1\dots i_k}B_{i_1\dots i_k}$, so that
$\norm{A}_{\HS}^2=\inner{A}{A}_{\HS}$, and $A[v_1,\dots,v_k]$ for the
multilinear form. For $j\le k$, $A[v_1,\dots,v_j]$ is the $(k-j)$-tensor
obtained by contracting the first $j$ indices. For a $k$-tensor $A$ and an
$l$-tensor $A'$, $A\otimes A'$ is the $(k+l)$-tensor with entries
$A_{i_1\dots i_k}A'_{i_{k+1}\dots i_{k+l}}$, and
$v^{\otimes j}=v\otimes\dots\otimes v$ ($j$ factors). The identity matrix
$I$ is a $2$-tensor, and $I^{\otimes m}$ has the entries
$\prod_{r=1}^m\mathbf 1\{i_{2r-1}=i_{2r}\}$.

\begin{alemma}[Elementary tensor facts]
\label{lem:tensor-facts}
Let $A$ be a $k$-tensor with $k\ge1$, and $v\in\R^d$.
\begin{enumerate}[label=(\alph*)]
\item $\norm{A}_{\op}\le\norm{A}_{\HS}\le d^{(k-1)/2}\norm{A}_{\op}$, and
$\norm{A[v]}_{\HS}\le\norm{v}\,\norm{A}_{\HS}$.
\item $\norm{I^{\otimes m}}_{\HS}=d^{m/2}$, and
$\abs{\inner{A}{I^{\otimes m}\otimes B}_{\HS}}\le d^{m/2}\norm{A}_{\HS}
\norm{B}_{\HS}$ for every $(k-2m)$-tensor $B$.
\end{enumerate}
\end{alemma}

\begin{proof}
Exactly $d^m$ entries of $I^{\otimes m}$ equal one and the others vanish,
and $\norm{I^{\otimes m}\otimes B}_{\HS}=\norm{I^{\otimes m}}_{\HS}\norm{B}_{\HS}$.
With this, the first inequality in (a), the contraction bound and (b)
follow from the Cauchy--Schwarz inequality. The second inequality in (a)
follows by induction on $k$ from
$\norm{A}_{\HS}^2=\sum_{i=1}^d\norm{A[e_i]}_{\HS}^2$ and
$\norm{A[e_i]}_{\op}\le\norm{A}_{\op}$.
\end{proof}

\subsection{Taylor remainders}
\label{app:taylor}

For $f\in C^q$ and $z,v\in\R^d$ let
$T_qf(z;v)=\sum_{j=0}^q D^jf(z)[v^{\otimes j}]/j!$ be the Taylor polynomial of
degree $q$, and $T_{q-j}D^jf(z;v)$ the Taylor polynomial of degree $q-j$ of
the tensor-valued map $v\mapsto D^jf(z+v)$.

\begin{alemma}[Taylor remainders]
\label{lem:taylor}
Let $\beta=q+\gamma$ with $q\ge1$ and $\gamma\in(0,1]$, let $f\in C^q(\R^d)$,
and let $z,v\in\R^d$.
\begin{enumerate}[label=(\alph*)]
\item $\abs{f(z+v)-T_qf(z;v)}\le[f]_{\beta,\op}\norm{v}^{\beta}/q!$.
\item For $1\le j\le q$,
$\norm{D^jf(z+v)-T_{q-j}D^jf(z;v)}_{\HS}\le[f]_{\beta,\HS}\norm{v}^{\beta-j}/(q-j)!$.
\item The bound (b) holds with $\op$ in place of $\HS$ on both sides. For
$j=1$ the left-hand side is then the Euclidean norm of
$\nabla f(z+v)-T_{q-1}\nabla f(z;v)$.
\end{enumerate}
\end{alemma}

\begin{proof}
For $j=q$, (b) and (c) are the definitions of the seminorms. For
$0\le j<q$, with $D^0f=f$, the integral form of the remainder is
\[
D^jf(z+v)-T_{q-j}D^jf(z;v)=\int_0^1\frac{(1-s)^{q-j-1}}{(q-j-1)!}
\bigl(D^qf(z+sv)-D^qf(z)\bigr)[v^{\otimes(q-j)}]\,ds .
\]
A contraction with $v$ multiplies the operator norm, and by
\cref{lem:tensor-facts}(a) the Hilbert--Schmidt norm, by at most $\norm v$.
Since $\int_0^1(1-s)^{q-j-1}s^\gamma\,ds\le1/(q-j)$, this gives (a)
($j=0$), (b) and (c).
\end{proof}

\subsection{Spherical and ball averages}
\label{app:pizzetti}

Throughout, $U$ is uniform on $\Sph$ and $W$ is uniform on the unit ball
$B(1)$. For a continuous scalar, vector or tensor valued function $g$, a
center $z\in\R^d$ and $s\ge0$ let
\[
S_g(s)=\E\,g(z+sU),\qquad B_g(s)=\E\,g(z+sW).
\]
As a function of the center we write $\mathcal A_sg(z)=B_g(s)$ for the ball
average. Since $\norm W$ has density $dr^{d-1}$ on $[0,1]$,
$B_g(s)=d\int_0^1r^{d-1}S_g(sr)\,dr$. The Laplacian $\Delta=\sum_i\partial_i^2$
acts entrywise. For continuous $a$ on $[0,s_0]$ define the operator
\begin{equation}
\label{eq:Jd}
(J_da)(s)=\int_0^st^{1-d}\int_0^tu^{d-1}a(u)\,du\,dt .
\end{equation}
It is positive ($a\ge0$ implies $J_da\ge0$), and
$J_d(s^{2j})=s^{2j+2}/((2j+2)(d+2j))$, so that
$(J_d^j1)(s)=s^{2j}/(2^jj!\prod_{i=0}^{j-1}(d+2i))$.

\begin{alemma}[Finite Pizzetti formula]
\label{lem:pizzetti}
Let $k\ge1$ and let $g$ be $2k$ times continuously differentiable near
$B(z,s_0)$. For $0\le s\le s_0$,
\begin{equation}
\label{eq:pizzetti-sphere}
S_g(s)=\sum_{j=0}^{k-1}\frac{s^{2j}\,\Delta^jg(z)}{2^jj!\prod_{i=0}^{j-1}(d+2i)}
+\bigl(J_d^kS_{\Delta^kg}\bigr)(s),
\end{equation}
where empty products equal one, and
\begin{equation}
\label{eq:pizzetti-ball}
B_g(s)=\sum_{j=0}^{k-1}\frac{s^{2j}\,\Delta^jg(z)}{2^jj!\prod_{i=1}^{j}(d+2i)}
+d\int_0^1r^{d-1}\bigl(J_d^kS_{\Delta^kg}\bigr)(sr)\,dr .
\end{equation}
The norm of the last term is at most
$s^{2k}\sup_{\norm{x-z}\le s}\norm{\Delta^kg(x)}\big/\bigl(2^kk!\prod_{i=1}^k(d+2i)\bigr)$.
\end{alemma}

\begin{proof}
For $g\in C^2$ the divergence theorem gives
$S_g'(s)=s^{1-d}\int_0^su^{d-1}S_{\Delta g}(u)\,du$ (for $d=1$ differentiate
$S_g(s)=(g(z+s)+g(z-s))/2$), that is, $S_g=g(z)+J_dS_{\Delta g}$. Iterating
$k$ times gives \eqref{eq:pizzetti-sphere}. Averaging over the radius
multiplies the coefficient of $s^{2j}$ by $d/(d+2j)$, which gives
\eqref{eq:pizzetti-ball}, and positivity of $J_d$ bounds the remainder.
\end{proof}

\Cref{lem:pizzetti} is a finite form, with an explicit remainder, of
Pizzetti's formula \citep{Ovall2016,Toma2008}.

\begin{alemma}[Gradients of ball averages]
\label{lem:ball-gradient}
Let $g:\R^d\to\R$ be continuous and $s>0$. Then $\mathcal A_sg$ is
continuously differentiable and
\[
\nabla\mathcal A_sg(z)=\frac ds\,\E\bigl[g(z+sU)U\bigr].
\]
If $g$ is Lipschitz, then $\mathcal A_sg$ is Lipschitz with the same
constant. If $g\in C^1$, then $\nabla\mathcal A_sg(z)=B_{\nabla g}(s)$.
\end{alemma}

\begin{proof}
The identity is \citet[Lemma~1]{Flaxman2005}, and its right-hand side is
continuous in $z$. The other two claims follow from
$\mathcal A_sg(z)=\E g(z+sW)$, by averaging and by differentiation under
the integral.
\end{proof}

We use the Poincar\'e inequality on the sphere: for $d\ge2$ and a
differentiable function $\phi$ on a neighborhood of $\Sph$ with
$\E\phi(U)=0$,
\begin{equation}
\label{eq:poincare}
\E\phi(U)^2\le\frac{1}{d-1}\,\E\norm{\nabla_{\mathbb S}\phi(U)}^2,
\end{equation}
where $\nabla_{\mathbb S}$ is the tangential part of the gradient. This is
the statement that the first nonzero eigenvalue of the Laplace--Beltrami
operator on $\Sph$ equals $d-1$ \citep[Chapter~1]{DaiXu2013Spheres}. Applied
to each component, it holds for vector and tensor valued $\phi$, with the
sum of the squared component gradients on the right. By mollification it
also gives $\E\phi(U)^2-(\E\phi(U))^2\le\ell^2/(d-1)$ for every
$\ell$-Lipschitz $\phi:\R^d\to\R$.
In particular, if $e:\R^d\to\R$ is odd and $\ell$-Lipschitz and
$(U_i)_{i=1}^d$ is a uniformly random orthonormal basis, then $\E e(U)=0$
and, by orthonormality,
\begin{equation}
\label{eq:odd-frame}
\E\Bigl\|\sum_{i=1}^de(U_i)U_i\Bigr\|^2=d\,\E e(U)^2\le\frac{d\,\ell^2}{d-1}.
\end{equation}

\subsection{Testing inequalities}
\label{app:testing}

For probability measures $P,Q$ let $\TV(P,Q)=\sup_A\abs{P(A)-Q(A)}$ and
$\KL(P\Vert Q)=\int\log(dP/dQ)\,dP$ ($+\infty$ without absolute
continuity).

\begin{alemma}[Testing and Gaussian observations]
\label{lem:testing}\leavevmode
\begin{enumerate}[label=(\alph*)]
\item Every test between $P$ and $Q$ has error probabilities summing to at
least $1-\TV(P,Q)$, and $\TV(P,Q)\le\sqrt{\KL(P\Vert Q)/2}$.
\item Let an algorithm interact with $f$ or with $g$ through independent
$N(0,\sigma^2)$ noise, and let $P_f,P_g$ be the laws of its seed and
transcript. If $f-g$ is bounded, then
\[
\KL(P_f\Vert P_g)=\frac1{2\sigma^2}\E_f\sum_{t=1}^T\bigl(f(X_t)-g(X_t)\bigr)^2 .
\]
\end{enumerate}
\end{alemma}

Part (a) is the two-point testing bound and Pinsker's inequality
\citep[Section~2.4 and Lemma~2.5]{Tsybakov2009}. Part (b) is the chain rule
for the Kullback--Leibler divergence \citep[Theorem~2.5.3]{CoverThomas2006}:
given the seed and the past, $X_t$ is the same under both hypotheses, and the
conditional divergence of $Y_t$ is $(f(X_t)-g(X_t))^2/(2\sigma^2)$.

We use Assouad's lemma in the following form \citep[Section~2.7.2]{Tsybakov2009}.
Let $\{P_\omega:\omega\in\{-1,1\}^d\}$ be laws of transcripts and let
$\omega^{(j)}$ denote $\omega$ with its $j$-th coordinate flipped. For any
estimator $\hat\omega$ of $\omega$,
\begin{equation}
\label{eq:assouad}
\frac1{2^d}\sum_\omega\E_\omega\sum_{j=1}^d\mathbf 1\{\hat\omega_j\ne\omega_j\}
\ge\frac d2\Bigl(1-\frac1{d2^d}\sum_{\omega,j}\TV\bigl(P_\omega,P_{\omega^{(j)}}\bigr)\Bigr),
\end{equation}
which follows from \cref{lem:testing}(a) applied to each pair
$(\omega,\omega^{(j)})$.

\subsection{Spherical caps}
\label{app:caps}

\begin{alemma}[Caps]
\label{lem:caps}
Let $d\ge2$, $0<\theta<1$, $w$ uniform on $\Sph$, and $x\ne0$. Then
\[
\PP\bigl(\abs{\inner{w}{x}}>\theta\norm x\bigr)\le2e^{-d\theta^2/2}\le2e^{-\theta^2(d-1)/2}.
\]
\end{alemma}

By \citet[Lemma~2.2]{Ball1997}, each of the caps
$\{w:\inner{w}{\pm x}\ge\theta\norm x\}$ has probability at most
$e^{-d\theta^2/2}$.

\section{{The finite-difference estimate with exact values}}
\label{app:estimator}

This appendix analyzes one frame of \cref{alg:stencil} with exact values;
\cref{app:estimator-unified} adds the noise of \cref{def:oracle} and
proves \cref{prop:estimator,prop:unified-estimator}.

Throughout, $\beta>2$, $q=\lceil\beta\rceil-1\ge2$, $\gamma=\beta-q$,
$m=\lceil q/2\rceil$ and $t_j=j/m$, and $L_\beta$ bounds $[f]_{\beta,\op}$ (with
$\tau=0$) or $[f]_{\beta,\HS}$ (with $\tau=\tau_\beta$ from \eqref{eq:tau}).
The weights $a_1,\dots,a_m$ solve
\begin{equation}
\label{eq:stencil-moments}
\sum_{j=1}^m a_jt_j^{2i+1}=\mathbf 1\{i=0\},\qquad i=0,1,\dots,m-1 .
\end{equation}
The matrix $(t_j^{2i+1})$ is a Vandermonde matrix in the distinct numbers
$t_j^2$ times the diagonal matrix $\mathrm{diag}(t_j)$, so the solution
exists, is unique and depends only on $\beta$. The constants
\begin{gather}
\label{eq:A1}
S_1=\sum_j\abs{a_j},\qquad S_2=\sum_ja_j^2,\qquad A_1=\sum_j\abs{a_j}t_j,\\
\label{eq:stencil-constants}
c_i=\frac{1}{(q-i)!}\sum_j\abs{a_j}t_j^{\beta}\qquad(0\le i\le q)
\end{gather}
also depend only on $\beta$, and $A_1\le S_1$ since $t_j\le1$. The values of
$S_1,S_2,A_1$ for $m\le4$ are in \cref{tab:unified-constants}. We use the notation $U$,
$W$, $S_g$, $B_g$, $\mathcal A_s$ and $J_d$ of \cref{app:pizzetti}. For a
function $g$, a center $z$, a radius $h>0$ and $u\in\R^d$ let
\begin{equation}
\label{eq:directional}
D_{h,u}g(z)=\frac{1}{2h}\sum_{j=1}^ma_j\bigl(g(z+ht_ju)-g(z-ht_ju)\bigr).
\end{equation}
With exact values, one frame $(U_1,\dots,U_d)$ of \cref{alg:stencil}
returns $\sum_{i=1}^dD_{h,U_i}f(z)U_i$.

\subsection{Exactness and the mean}

\begin{alemma}[Exactness]
\label{lem:exact}
For every polynomial $\varphi$ on $\R$ of degree at most $2m$ and every
$h>0$,
$\frac1{2h}\sum_ja_j\bigl(\varphi(ht_j)-\varphi(-ht_j)\bigr)=\varphi'(0)$.
If $[f]_{\beta,\op}\le L_\beta$ and $\norm u=1$, then
$\abs{D_{h,u}f(z)-\inner{\nabla f(z)}{u}}\le c_0L_\beta h^{\beta-1}$.
\end{alemma}

\begin{proof}
The symmetric difference annihilates even powers of $s$ and maps
$s^{2i+1}$, $0\le i\le m-1$, to $h^{2i}\sum_ja_jt_j^{2i+1}=\mathbf 1\{i=0\}$.
For the second claim apply the first to $\varphi(s)=T_qf(z;su)$, which has
degree $q\le2m$ and $\varphi'(0)=\inner{\nabla f(z)}{u}$. By
\cref{lem:taylor}(a) the remainders at $\pm ht_ju$ are at most
$L_\beta(ht_j)^\beta/q!$, and $\frac1{2h}\sum_j\abs{a_j}\,2L_\beta(ht_j)^\beta/q!=c_0L_\beta h^{\beta-1}$.
\end{proof}

\begin{alemma}[The mean]
\label{lem:mean}
For a continuous $g:\R^d\to\R$ and $h>0$ define the \emph{surrogate}
$\mathcal S_hg$, and for $f\in C^1$ the vector $G_h(z)$, by
\begin{equation}
\label{eq:mean}
\mathcal S_hg(z)=\sum_{j=1}^ma_jt_j\,\E g(z+ht_jW),\qquad
G_h(z)=\nabla\mathcal S_hf(z)=\sum_{j=1}^ma_jt_j\,B_{\nabla f}(ht_j).
\end{equation}
Then $\mathcal S_hg$ is continuously differentiable and
$d\,\E[D_{h,U}g(z)\,U]=\nabla\mathcal S_hg(z)$. Consequently, for a
uniformly random orthonormal basis $(U_i)$,
$\E\sum_iD_{h,U_i}g(z)U_i=\nabla\mathcal S_hg(z)$. For $g=f$ this says
that the noiseless part of one frame has mean $G_h(z)$.
\end{alemma}

\begin{proof}
By the symmetry $U\mapsto-U$ and \cref{lem:ball-gradient},
\[
d\,\E[D_{h,U}g(z)U]=\frac dh\sum_ja_j\E[g(z+ht_jU)U]
=\sum_ja_jt_j\nabla\mathcal A_{ht_j}g(z)=\nabla\mathcal S_hg(z),
\]
and for $g=f\in C^1$ the last claim of \cref{lem:ball-gradient} gives the
second expression for $G_h$. Each column of a uniformly random orthonormal
basis is uniform on $\Sph$.
\end{proof}

Since $\sum_ja_jt_j=1$ by \eqref{eq:stencil-moments}, $G_h$ is an average
of $\nabla f$ over balls with signed weights of total mass one.

\subsection{Bias}

\begin{alemma}[Bias]
\label{lem:bias}
For all $z$ and $h>0$:
\begin{enumerate}[label=(\alph*)]
\item if $[f]_{\beta,\op}\le L_\beta$, then $\norm{G_h(z)-\nabla f(z)}\le c_1L_\beta h^{\beta-1}$;
\item if $[f]_{\beta,\HS}\le L_\beta$ and $q=2k+1$ is odd, then
$\norm{G_h(z)-\nabla f(z)}\le\frac{q!\,c_0}{2^kk!}\,L_\beta\,d^{-k/2}h^{\beta-1}$;
\item if $[f]_{\beta,\HS}\le L_\beta$ and $q=2k$ is even, then
$\norm{G_h(z)-\nabla f(z)}\le C_\beta\,L_\beta\,d^{-(\beta-1)/4}h^{\beta-1}$.
\end{enumerate}
Hence $\norm{G_h(z)-\nabla f(z)}\le\bfC L_\beta d^{-\tau/2}h^{\beta-1}$ in both
geometries, where $\bfC$ depends only on $\beta$, and $\bfC=c_1$ in the
operator case.
\end{alemma}

Part (a) is the argument of \citet[Lemma~6]{Akhavan2024}, written for the
stencil.

\begin{proof}
In (a) and (b) we expand the ball averages in \eqref{eq:mean} around $z$.
Terms proportional to $s^{2i}$ with $1\le i\le m-1$ cancel at $s=ht_j$ after
summation with the weights $a_jt_j$, because
$\sum_ja_jt_j(ht_j)^{2i}=h^{2i}\sum_ja_jt_j^{2i+1}=0$, and the constant term
gives $\sum_ja_jt_j\nabla f(z)=\nabla f(z)$.

\emph{(a) Operator norm.} By \cref{lem:taylor}(c),
$\nabla f(z+v)=T_{q-1}\nabla f(z;v)+\rho(v)$ with
$\norm{\rho(v)}\le L_\beta\norm v^{\beta-1}/(q-1)!$. The ball average of the term
of degree $r$ of $T_{q-1}\nabla f(z;sW)$ vanishes for odd $r$ and is
proportional to $s^r$ for even $r$; since $r\le q-1\le2m-1$, only $r=0$
survives. As $\E\norm W^{\beta-1}\le1$, the remainder contributes at most
$\sum_j\abs{a_j}t_jL_\beta(ht_j)^{\beta-1}/(q-1)!=c_1L_\beta h^{\beta-1}$.

\emph{(b) Hilbert--Schmidt norm, $q$ odd.} Let $g=\nabla f\in C^{2k}$ and
$H=\Delta^kg$. For a unit vector $e$,
$\inner{e}{H(x)-H(y)}=\inner{D^qf(x)-D^qf(y)}{e\otimes I^{\otimes k}}_{\HS}$,
so \cref{lem:tensor-facts}(b) gives
\begin{equation}
\label{eq:trace-holder}
\norm{H(x)-H(y)}\le d^{k/2}L_\beta\norm{x-y}^{\gamma}.
\end{equation}
By \eqref{eq:pizzetti-sphere} with $k$ steps, $S_g(s)$ is a polynomial in
$s^2$ of degree $k$ with constant term $\nabla f(z)$, plus
$J_d^k[S_H-H(z)](s)$. Ball averaging keeps the polynomial a polynomial in
$s^2$, and since $k=m-1$, the weights leave only $\nabla f(z)$. By
positivity of $J_d$, \eqref{eq:trace-holder} and
$(J_d^k1)(s)\le s^{2k}/(2^kk!\,d^k)$,
\[
\norm{J_d^k[S_H-H(z)](s)}\le d^{k/2}L_\beta s^\gamma(J_d^k1)(s)\le
\frac{L_\beta d^{-k/2}s^{\beta-1}}{2^kk!},
\]
and ball averaging preserves this bound. Summation over $j$ with $s=ht_j$
gives (b).

\emph{(c) Hilbert--Schmidt norm, $q$ even.} Here $m=k$, and by
\cref{lem:tensor-facts}(b) the scalar function $\Delta^kf$ is
$\gamma$-H\"older with constant $d^{k/2}L_\beta$. Let $E_h=\mathcal S_hf-f$. By
\eqref{eq:pizzetti-ball} with $k$ steps, the terms with $1\le i\le k-1$ cancel
and the term $i=0$ gives $f(z)$, so
$E_h(z)=\sum_ja_jt_j\,\mathcal R(ht_j;z)$, where $\mathcal R(s;z)$ is the
remainder of \eqref{eq:pizzetti-ball}. For two centers $z,z'$,
$\mathcal R(s;z)-\mathcal R(s;z')$ is a positive linear functional, of
total mass $s^{2k}/(2^kk!\prod_{i=1}^k(d+2i))$, applied to
$\Delta^kf(z+\cdot)-\Delta^kf(z'+\cdot)$. Hence
\begin{equation}
\label{eq:Eh-holder}
\abs{E_h(z)-E_h(z')}\le C_1L_\beta d^{-k/2}h^{2k}\norm{z-z'}^\gamma ,\qquad
C_1=\frac1{2^kk!}\sum_j\abs{a_j}t_j^{2k+1}.
\end{equation}
Fix a unit vector $e$ and let $h'=hd^{-1/4}$. Since
$[f]_{\beta,\op}\le L_\beta$ and $[\mathcal S_hf]_{\beta,\op}\le A_1L_\beta$,
\cref{lem:exact} along the line $z+se$ gives
$\abs{D_{h',e}v(z)-\partial_ev(z)}\le c_0L_\beta h'^{\beta-1}$ for $v=f$ and
$\le c_0A_1L_\beta h'^{\beta-1}$ for $v=\mathcal S_hf$. By \eqref{eq:Eh-holder},
\[
\abs{D_{h',e}E_h(z)}\le\frac1{2h'}\sum_j\abs{a_j}\abs{E_h(z+h't_je)-E_h(z-h't_je)}
\le S_1C_1L_\beta d^{-k/2}h^{2k}h'^{\gamma-1}.
\]
Since $D_{h',e}$ is linear and $\partial_e\mathcal S_hf=\inner{e}{G_h}$,
\[
\abs{\inner{e}{G_h(z)-\nabla f(z)}}
\le c_0(1+A_1)L_\beta h'^{\beta-1}+S_1C_1L_\beta d^{-k/2}h^{2k}h'^{\gamma-1}
=\bigl(c_0(1+A_1)+S_1C_1\bigr)L_\beta d^{-(\beta-1)/4}h^{\beta-1},
\]
because $2k+\gamma-1=\beta-1$. Take the supremum over $e$.

Finally, $\tau=0$ in (a), $\tau_\beta=k$ in (b) and
$\tau_\beta=(\beta-1)/2$ in (c).
\end{proof}

\subsection{Second moment}

Fix $z$ and $h$ and let $e_f(u)=D_{h,u}f(z)-\inner{\nabla f(z)}{u}$ for
$u\in\R^d$, an odd $C^q$ function. For a uniformly random orthonormal basis
$(U_i)$ let $Z=\sum_ie_f(U_i)U_i$ be the noiseless error of the estimate. By
orthonormality, \cref{lem:mean} and $d\,\E[UU^\top]=I$,
\begin{equation}
\label{eq:frame-moments}
\E\norm{Z}^2=d\,\E e_f(U)^2,\qquad \E Z=d\,\E[e_f(U)U]=G_h(z)-\nabla f(z).
\end{equation}

\begin{alemma}[Derivatives of the stencil error]
\label{lem:diff-cancel}
For $\norm u=1$: if $[f]_{\beta,\op}\le L_\beta$, then
$\norm{\nabla e_f(u)}\le c_1L_\beta h^{\beta-1}$; if $[f]_{\beta,\HS}\le L_\beta$, then
$\norm{D^je_f(u)}_{\HS}\le c_jL_\beta h^{\beta-1}$ for $1\le j\le q$.
\end{alemma}

\begin{proof}
Differentiating \eqref{eq:directional} $j$ times in $u$,
\[
D^je_f(u)=\frac{h^{j-1}}2\sum_la_lt_l^j\bigl(D^jf(z+ht_lu)-(-1)^jD^jf(z-ht_lu)\bigr)
-\mathbf 1\{j=1\}\nabla f(z).
\]
Replace $D^jf(z\pm v)$ by its Taylor polynomial of degree $q-j$. Its term
of degree $r$ survives the signed sum only if $j+r$ is odd, with weight
$h^{j+r-1}\sum_la_lt_l^{j+r}$. This weight vanishes for $3\le j+r\le q$ by
\eqref{eq:stencil-moments}, and the term $j+r=1$ cancels $\nabla f(z)$. The
two Taylor remainders are bounded by \cref{lem:taylor}(b),(c).
\end{proof}

\begin{alemma}[Means of odd derivatives]
\label{lem:tensor-means}
Let $[f]_{\beta,\HS}\le L_\beta$, let $3\le j<q$ be odd and $k=\lfloor(q-j)/2\rfloor$.
Then $\norm{\E D^je_f(U)}_{\HS}\le C_\beta L_\beta h^{\beta-1}d^{-k/2}$.
\end{alemma}

\begin{proof}
For odd $j\ge3$ the formula in the proof of \cref{lem:diff-cancel} gives
$\E D^je_f(U)=h^{j-1}\sum_la_lt_l^jS_{D^jf}(ht_l)$. If $k\ge1$, expand
$S_{D^jf}$ by \eqref{eq:pizzetti-sphere} with $k$ steps; this uses
$2k\le q-j$ derivatives of $D^jf$. The polynomial terms are proportional to
$s^{2r}$, $r\le k-1$, and receive the weights
$h^{j+2r-1}\sum_la_lt_l^{j+2r}$, which vanish because $j+2r$ is odd and
$3\le j+2r\le q$. In the remainder (for $k=0$, in $S_{D^jf}$ itself), the
tensor $\Delta^kD^jf$ is a $k$-fold trace of $D^{j+2k}f$, where
$q-j-2k\in\{0,1\}$. Subtract from it its Taylor polynomial at $z$ of degree
$q-j-2k$. A linear term has zero spherical mean, and the constant term
$\Delta^kD^jf(z)$ is multiplied by $(J_d^k1)(s)\propto s^{2k}$ and receives
the weight $\sum_la_lt_l^{j+2k}=0$, because $j+2k\in\{q-1,q\}$ is odd and
at least $3$. By \cref{lem:taylor}(b) and \cref{lem:tensor-facts}(b), the
rest has Hilbert--Schmidt norm at most $d^{k/2}L_\beta\norm{v}^{\beta-j-2k}$ at
$z+v$. Positivity of $J_d$ bounds the remaining term by
\[
d^{k/2}L_\beta s^{\beta-j-2k}(J_d^k1)(s)\le L_\beta s^{\beta-j}d^{-k/2}/(2^kk!),
\]
and $s=ht_l$ gives the claim.
\end{proof}

\begin{aproposition}[Noiseless error]
\label{prop:noiseless}
There is a constant $C_Z$, depending only on $\beta$, with $C_Z=2c_1^2$
in the operator case, such that $\E\norm Z^2\le C_ZL_\beta^2d^{-\tau}h^{2\beta-2}$.
\end{aproposition}

\begin{proof}
For $d=1$ the error $Z=e_f(1)$ is deterministic, since
$e_f(-1)\cdot(-1)=e_f(1)$, and $\abs{e_f(1)}\le c_0L_\beta h^{\beta-1}$ by
\cref{lem:exact}. Let $d\ge2$. We apply the Poincar\'e inequality
\eqref{eq:poincare} once in the operator case, and repeatedly along the
derivatives of $e_f$ in the Hilbert--Schmidt case.

\emph{Operator norm.} Since $e_f$ is odd, $\E e_f(U)=0$, and
\eqref{eq:poincare} with \cref{lem:diff-cancel} gives
$\E\norm Z^2=d\,\E e_f(U)^2\le\frac{d}{d-1}c_1^2L_\beta^2h^{2\beta-2}\le2c_1^2L_\beta^2h^{2\beta-2}$.

\emph{Hilbert--Schmidt norm.} Write
$\E\norm Z^2=\norm{\E Z}^2+\E\norm{Z-\E Z}^2$; the first term is at most
$\bfC^2L_\beta^2d^{-\tau}h^{2\beta-2}$ by \eqref{eq:frame-moments} and
\cref{lem:bias}. The vector $\mathsf a_*=\E Z=d\,\E[e_f(U)U]$ is the
coefficient of the $L^2(\Sph)$ projection of $e_f$ onto linear functions,
and $Z-\E Z=\sum_i(e_f(U_i)-\inner{\mathsf a_*}{U_i})U_i$. Hence, with
$\mathsf a_0=\E\nabla e_f(U)$ and $\Lambda_j=\E\norm{D^je_f(U)}_{\HS}^2$,
\[
\E\norm{Z-\E Z}^2=d\,\E\bigl(e_f(U)-\inner{\mathsf a_*}{U}\bigr)^2
\le d\,\E\bigl(e_f(U)-\inner{\mathsf a_0}{U}\bigr)^2\le\frac{d\,\Lambda_2}{(d-1)^2},
\]
where the last step applies \eqref{eq:poincare} to the odd function
$e_f-\inner{\mathsf a_0}{\cdot}$ and then to $\nabla e_f-\mathsf a_0$. For
$j<q$, \eqref{eq:poincare} applied to $D^je_f-\E D^je_f$ gives
$\Lambda_j\le\norm{\E D^je_f(U)}_{\HS}^2+\Lambda_{j+1}/(d-1)$, and
$\E D^je_f(U)=0$ for even $j$, since $D^je_f$ is then odd. Iterating,
\[
\E\norm{Z-\E Z}^2\le\sum_{\substack{3\le j<q\\ j\ \mathrm{odd}}}
\frac{d\,\norm{\E D^je_f(U)}_{\HS}^2}{(d-1)^j}+\frac{d\,\Lambda_q}{(d-1)^q}.
\]
By \cref{lem:diff-cancel,lem:tensor-means} and $d/(d-1)^j\le2^jd^{1-j}$,
the $j$-th term is at most
$C_\beta L_\beta^2h^{2\beta-2}d^{-(j-1+\lfloor(q-j)/2\rfloor)}$, and the last
is at most $C_\beta L_\beta^2h^{2\beta-2}d^{-(q-1)}$. Both exponents are at least
$\lceil q/2\rceil\ge\tau_\beta$, since
$j-1+\lfloor(q-j)/2\rfloor=\lfloor(q+j-2)/2\rfloor\ge\lfloor(q+1)/2\rfloor$
for $j\ge3$, and $q-1\ge\lceil q/2\rceil$ for $q\ge2$.
\end{proof}

\section{{The random-frame estimate under the unified noise}}
\label{app:estimator-unified}

This appendix proves \cref{prop:unified-estimator}, which decomposes the
estimate of one round of \cref{alg:stencil} under the noise of
\cref{def:oracle}. For $d\ge2$, its parts (i)--(iii) are
\cref{prop:estimator}. \Cref{rem:centered-filtration,rem:adaptive-lipschitz}
show that two features of \cref{def:oracle} cannot be dropped: in (C) the
$\sigma$-algebra $\mathcal E_t$ contains the error components of the earlier
queries of the round, and in (L) the function $\eta$ is fixed before the
run.

\subsection{Setting}
\label{app:unified-setting}

We use the notation of \cref{app:estimator}, in particular
$D_{h,u}$, $\mathcal S_h$, $G_h$, $e_f$, the constants \eqref{eq:A1},
$\bfC$ of \cref{lem:bias} and $C_Z$ of \cref{prop:noiseless}.

\paragraph{One round.}
Fix a round $k$. Its center $z$ and radius $h>0$ are
$\mathcal F_{k-1}$-measurable, and $\Theta_k$ consists of the $B\ge1$
independent frames $U^{(\ell)}=(U^{(\ell)}_1,\dots,U^{(\ell)}_d)$ of
\cref{alg:stencil}. The query at
$z\pm ht_jU^{(\ell)}_i$ is labeled $(\ell,i,j,\pm)$, and the queries are
ordered lexicographically; any order that does not depend on $\Theta_k$
would do. Its observed value is $Y^\pm_{ij\ell}$, with error components
$\xi^{\mathrm c,\pm}_{ij\ell}$, $\xi^{\mathrm o,\pm}_{ij\ell}$ and
$\delta^\pm_{ij\ell}$, and $\boldsymbol\xi^{\mathrm o}_k$ is the vector of the
oblivious errors of the round. All query points are functions of
$(z,h,\Theta_k)$. Some results of \cref{app:geometry} use, in addition, the
conditional sub-Gaussian condition
\begin{equation}
\label{eq:SG}\tag{SG}
\E\bigl[\exp(s\xi^{\mathrm c}_t)\bigm|\mathcal E_t\bigr]\le\exp(s^2\sigma^2/2)
\qquad(s\in\R)\quad\text{for all queries }t,
\end{equation}
which implies (C). Independent $N(0,\sigma^2)$ errors satisfy it.

\paragraph{Localization.}
Since $h$ may be small with positive probability, the random vectors below
need not be integrable. Let
$E_n=\{1/n\le h\le n,\ \norm z\le n\}\in\mathcal F_{k-1}$; these events
increase to the whole space. On $E_n$ all query points lie in $B(2n)$, where
$f$ and $\eta$ are bounded, $1/h\le n$, and all error components are square
integrable by (C), (O) and (A). A statement such as
$\E[\zeta\mid\mathcal F_{k-1}]=0$ means that
$\E[\mathbf 1_{E_n}\zeta\mid\mathcal F_{k-1}]=0$ for every $n$, and a
conditional second moment is the increasing limit of those of
$\mathbf 1_{E_n}\zeta$. This agrees with the usual definitions for
integrable variables, since
$\mathbf 1_{E_n}\E[X\mid\mathcal F_{k-1}]=\E[\mathbf 1_{E_n}X\mid\mathcal F_{k-1}]$.
We prove every statement on a fixed $E_n$ without further mention. For a
random vector $X$ let $\mathsf N(X)=(\E[\norm X^2\mid\mathcal F_{k-1}])^{1/2}$;
then $\mathsf N(X+X')\le\mathsf N(X)+\mathsf N(X')$ almost surely.

\begin{alemma}[Independence within a round]
\label{lem:round-facts}\leavevmode
\begin{enumerate}[label=(\alph*)]
\item $\Theta_k$ is independent of $\mathcal F_{k-1}\vee\sigma(\boldsymbol\xi^{\mathrm o}_k)$.
Hence, conditionally on $\mathcal F_{k-1}\vee\sigma(\boldsymbol\xi^{\mathrm o}_k)$,
the frames are independent uniformly random orthonormal bases, and every
column $U^{(\ell)}_i$ is uniform on $\Sph$.
\item Let $\mathcal H$ be a $\sigma$-algebra independent of $\Theta_k$,
let $X$ be $\mathcal H$-measurable, and let $\phi$ be jointly measurable
with $\phi(X,\Theta_k)$ integrable. Then
$\E[\phi(X,\Theta_k)\mid\mathcal H]=\Phi(X)$ almost surely, where
$\Phi(x)=\E\phi(x,\Theta_k)$.
\end{enumerate}
\end{alemma}

\begin{proof}
(a) $\Theta_k$ is independent of $\mathcal F_{k-1}$, and by (O) it is
conditionally independent of $\boldsymbol\xi^{\mathrm o}_k$ given
$\mathcal F_{k-1}$. Hence, for $F\in\mathcal F_{k-1}$ and bounded measurable
$\phi,\psi$,
$\E[\mathbf 1_F\phi(\boldsymbol\xi^{\mathrm o}_k)\psi(\Theta_k)]
=\E[\mathbf 1_F\phi(\boldsymbol\xi^{\mathrm o}_k)]\,\E\psi(\Theta_k)$, which
proves (a). Part (b) is the substitution rule for conditional
expectations \citep[Section~4.1]{Durrett2019}.
\end{proof}

\subsection{The systematic error}

\begin{alemma}[The systematic error]
\label{lem:lipschitz-part}
Let $\eta$ be $M$-Lipschitz, $z\in\R^d$, $h>0$, and
$e_\eta(u)=D_{h,u}\eta(z)$ for $u\in\R^d$.
\begin{enumerate}[label=(\alph*)]
\item $e_\eta$ is odd and $(A_1M)$-Lipschitz on $\R^d$.
\item $\mathcal S_h\eta$ is $(A_1M)$-Lipschitz, so
$\norm{\nabla\mathcal S_h\eta(z)}\le A_1M$.
\item If $d\ge2$, then $\E e_\eta(U)=0$ and
$\E e_\eta(U)^2\le A_1^2M^2/(d-1)$.
\item If $d\ge2$, then for a uniformly random orthonormal basis $(U_i)$ of
$\R^d$
\[
\E\Bigl\|\sum_{i=1}^de_\eta(U_i)U_i\Bigr\|^2=d\,\E e_\eta(U)^2
\le\frac d{d-1}\,A_1^2M^2 .
\]
If $d=1$, then $\sum_ie_\eta(U_i)U_i=e_\eta(1)$ is deterministic.
\end{enumerate}
\end{alemma}

\begin{proof}
(a) Oddness is clear, and
$\abs{e_\eta(u)-e_\eta(u')}\le\frac1{2h}\sum_j\abs{a_j}\,2M ht_j\norm{u-u'}
=A_1M\norm{u-u'}$.
(b) $\mathcal S_h\eta=\sum_ja_jt_j\mathcal A_{ht_j}\eta$, and each
$\mathcal A_{ht_j}\eta$ is $M$-Lipschitz by \cref{lem:ball-gradient}.
(c) and (d) are \eqref{eq:odd-frame} with $e=e_\eta$ and $\ell=A_1M$. For
$d=1$, $e_\eta(U_1)U_1=e_\eta(1)$ for both $U_1=\pm1$, by oddness.
\end{proof}

\subsection{The batched estimate}

\begin{aproposition}[Batched stencil estimate under the unified noise]
\label{prop:unified-estimator}
Let $f\in C^q$ satisfy $[f]_{\beta,\op}\le L_\beta$ (with $\tau=0$) or
$[f]_{\beta,\HS}\le L_\beta$ (with $\tau=\tau_\beta$), and let the noise be
admissible at levels $(M,\sigma,\delta)$. In round $k$ let $z$ and
$h>0$ be $\mathcal F_{k-1}$-measurable, and let $\bar g$ be the output of
\cref{alg:stencil} with $B\ge1$ frames. Then
\[
\bar g=\nabla f(z)+b(z)+\zeta+\Delta ,
\]
where $b(z)$ is given in (i), $\zeta$ and $\Delta$ are defined in
\eqref{eq:unified-decomp} below, and the following hold.
\begin{enumerate}[label=(\roman*)]
\item $b(z)=G_h(z)-\nabla f(z)+\nabla\mathcal S_h\eta(z)$ is
$\mathcal F_{k-1}$-measurable and
$\norm{b(z)}\le \bfC L_\beta d^{-\tau/2}h^{\beta-1}+A_1M$.
\item $\E[\zeta\mid\mathcal F_{k-1}]=0$, and for $d\ge2$
\[
\E\bigl[\norm\zeta^2\bigm|\mathcal F_{k-1}\bigr]\le\frac1B\Bigl(
\Cf L_\beta^2d^{-\tau}h^{2\beta-2}+C_\eta M^2
+C_\sigma\frac{d\sigma^2}{h^2}\Bigr),
\]
with $\Cf=2C_Z$, $C_\eta=4A_1^2$ and
$C_\sigma=(S_1+(S_2/2)^{1/2})^2$. For $d=1$,
$\E[\norm\zeta^2\mid\mathcal F_{k-1}]\le C_\sigma\sigma^2/(Bh^2)$.
\item $\norm\Delta\le\sqrt d\,S_1\delta/h$ almost surely.
\item The vector $\zeta+\Delta=\bar g-\nabla f(z)-b(z)$ is
$\mathcal F_k$-measurable; if $\delta=0$, then $\Delta=0$ and $\zeta$ is
$\mathcal F_k$-measurable. The vector
$\tilde\zeta=\zeta+\Delta-\E[\Delta\mid\mathcal F_{k-1}]$ is
$\mathcal F_k$-measurable, and
\[
\E[\tilde\zeta\mid\mathcal F_{k-1}]=0,\qquad
\mathsf N(\tilde\zeta)\le\mathsf N(\zeta)+\sqrt d\,S_1\delta/h,\qquad
\norm{\E[\Delta\mid\mathcal F_{k-1}]}\le\sqrt d\,S_1\delta/h.
\]
In particular, the vectors $\tilde\zeta$ of different rounds form a
martingale difference sequence with respect to $(\mathcal F_k)_k$.
\end{enumerate}
\end{aproposition}

When $\delta>0$, the vector $\zeta$ need not be $\mathcal F_k$-measurable,
and the vectors $\zeta$ of different rounds need not be orthogonal. An
adversarial error may cancel a centered error, which then stays hidden, and
a centered error of a later round, conditioned only on the past and on the
errors of its own round, may be correlated with the hidden one. The later
appendices therefore use $\zeta$ only within one round, through
$\E[\inner\zeta x\mid\mathcal F_{k-1}]=0$ for $\mathcal F_{k-1}$-measurable
random vectors $x$ and through (ii). Sums over rounds use the decomposition
$\bar g=\nabla f(z)+b(z)+\E[\Delta\mid\mathcal F_{k-1}]+\tilde\zeta$ with the
adapted martingale differences $\tilde\zeta$ of (iv).

\begin{table}[!htbp]
\centering
\small
\renewcommand{\arraystretch}{1.45}
\caption{Constants of \cref{prop:unified-estimator}; $C_\eta=4A_1^2$ and
$C_\sigma=(S_1+(S_2/2)^{1/2})^2$.}
\label{tab:unified-constants}
\begin{tabular}{ccccccc}
\toprule
$m$ & $(a_j)$ & $S_1$ & $S_2$ & $A_1$ & $C_\eta$ & $C_\sigma$\\
\midrule
$1$ & $1$ & $1$ & $1$ & $1$ & $4$ & $2.914$\\
$2$ & $\frac83,-\frac13$ & $3$ & $\frac{65}9$ & $\frac53$ & $\frac{100}9$ & $24.01$\\
$3$ & $\frac92,-\frac9{10},\frac1{10}$ & $\frac{11}2$ & $\frac{2107}{100}$ & $\frac{11}5$ & $\frac{484}{25}$ & $76.49$\\
$4$ & $\frac{32}5,-\frac85,\frac{32}{105},-\frac1{35}$ & $\frac{25}3$ & $\frac{480841}{11025}$ & $\frac{93}{35}$ & $\frac{34596}{1225}$ & $169.1$\\
\bottomrule
\end{tabular}
\end{table}

\begin{proof}
Condition on $\mathcal F_{k-1}$. Let $e_\eta(u)=D_{h,u}\eta(z)$, and for
frame $\ell$ put
\[
Y^{(\ell)}=\sum_i(e_f+e_\eta)(U_i^{(\ell)})U_i^{(\ell)}-b(z),
\]
with $b(z)$ as in (i). For the other three error components put
\[
\omega_i^{(\ell)}=\frac1{2h}\sum_j a_j(\xi^{\mathrm o,+}_{ij\ell}-\xi^{\mathrm o,-}_{ij\ell}),\quad
\chi_i^{(\ell)}=\frac1{2h}\sum_j a_j(\xi^{\mathrm c,+}_{ij\ell}-\xi^{\mathrm c,-}_{ij\ell}),\quad
\varrho_i^{(\ell)}=\frac1{2h}\sum_j a_j(\delta^+_{ij\ell}-\delta^-_{ij\ell}),
\]
and let $\Xi^{\mathrm o,(\ell)}=\sum_i\omega_i^{(\ell)}U_i^{(\ell)}$,
$\Xi^{\mathrm c,(\ell)}=\sum_i\chi_i^{(\ell)}U_i^{(\ell)}$ and
$\Delta^{(\ell)}=\sum_i\varrho_i^{(\ell)}U_i^{(\ell)}$. A bar denotes the
average over the $B$ frames. Since
$\sum_i\inner{\nabla f(z)}{U^{(\ell)}_i}U^{(\ell)}_i=\nabla f(z)$,
\begin{gather}
\label{eq:batched-estimate}
\bar g=\nabla f(z)+b(z)+\zeta+\Delta,\\
\label{eq:unified-decomp}
\zeta=\bar Y+\bar\Xi^{\mathrm o}+\bar\Xi^{\mathrm c},\qquad
\Delta=\frac1B\sum_{\ell=1}^B\Delta^{(\ell)}.
\end{gather}

\emph{(i) and (iii).} The surrogate gradients are continuous in $(z,h)$, so
$b(z)$ is $\mathcal F_{k-1}$-measurable. Its bound follows from
\cref{lem:bias} and \cref{lem:lipschitz-part}(b). Since
$\abs{\varrho_i^{(\ell)}}\le S_1\delta/h$, orthonormality gives
$\norm{\Delta^{(\ell)}}\le\sqrt d\,S_1\delta/h$, hence (iii).

\emph{The frame part.} By \cref{lem:mean}, applied to $f$ and to $\eta$,
$\E Y^{(\ell)}=0$, and the $Y^{(\ell)}$ are independent. For $d\ge2$,
\cref{prop:noiseless}, \cref{lem:lipschitz-part}(d) and $d/(d-1)\le2$ give
\begin{equation}\label{eq:unified-basis-part}
\E[\norm{\bar Y}^2\mid\mathcal F_{k-1}]
\le B^{-1}(2C_ZL_\beta^2d^{-\tau}h^{2\beta-2}+4A_1^2M^2).
\end{equation}
For $d=1$, $(e_f+e_\eta)(U_1)U_1=(e_f+e_\eta)(1)$ is deterministic, hence
$Y^{(\ell)}=0$.

\emph{Random noise.} Conditionally on
$\mathcal F_{k-1}\vee\sigma(\boldsymbol\xi_k^{\mathrm o})$, the coefficients
$\omega_i^{(\ell)}$ are fixed and the frames are independent and uniform
(\cref{lem:round-facts}). So the vectors $\Xi^{\mathrm o,(\ell)}$ have mean
zero and are orthogonal in conditional $L^2$ across frames. Since
$\mathsf N(\omega_i^{(\ell)})\le S_1\sigma/h$, orthonormality yields
\begin{equation}\label{eq:unified-tsybakov}
\E[\norm{\bar\Xi^{\mathrm o}}^2\mid\mathcal F_{k-1}]
\le dS_1^2\sigma^2/(Bh^2).
\end{equation}
The centered part is $\bar\Xi^{\mathrm c}=\sum_tw_t\xi^{\mathrm c}_t$, a sum
over the queries $t=(\ell,i,j,\pm)$ of the round with
$w_t=\pm a_jU_i^{(\ell)}/(2hB)$. Both $w_t$ and the centered errors of the
earlier queries are $\mathcal E_t$-measurable, so by (C) the terms
$w_t\xi^{\mathrm c}_t$ are centered and pairwise orthogonal, and
$\E[\norm{\bar\Xi^{\mathrm c}}^2\mid\mathcal F_{k-1}]\le\sigma^2\sum_t\norm{w_t}^2
=dS_2\sigma^2/(2Bh^2)$. The two parts need not be orthogonal to each other;
Minkowski's inequality gives
\begin{equation}\label{eq:unified-noise-part}
\E[\norm{\bar\Xi^{\mathrm o}+\bar\Xi^{\mathrm c}}^2\mid\mathcal F_{k-1}]
\le (S_1+\sqrt{S_2/2})^2d\sigma^2/(Bh^2).
\end{equation}

\emph{Cross terms.} Since $\inner{\bar Y}{w_t}$ is $\mathcal E_t$-measurable,
(C) gives $\E\inner{\bar Y}{w_t\xi^{\mathrm c}_t}=0$. For the oblivious part,
condition on $\mathcal F_{k-1}\vee\sigma(\boldsymbol\xi_k^{\mathrm o})$.
Different frames give zero by independence, and within frame $\ell$
\[
\E\inner{Y^{(\ell)}}{U_i^{(\ell)}}
=\E(e_f+e_\eta)(U_i^{(\ell)})-\inner{b(z)}{\E U_i^{(\ell)}}=0
\]
by oddness. Thus $\bar Y$ is orthogonal to both noise parts, and adding
\eqref{eq:unified-basis-part} and \eqref{eq:unified-noise-part} proves (ii).
The same computations give $\E[\zeta\mid\mathcal F_{k-1}]=0$.

\emph{(iv).} The vector $\zeta+\Delta=\bar g-\nabla f(z)-b(z)$ is a function
of the observed values, $\Theta_k$ and $\mathcal F_{k-1}$-measurable
quantities, hence $\mathcal F_k$-measurable. Subtracting the $\mathcal F_{k-1}$-measurable vector
$\E[\Delta\mid\mathcal F_{k-1}]$ gives $\tilde\zeta$, which is centered, and
\[
\mathsf N(\Delta-\E[\Delta\mid\mathcal F_{k-1}])^2
=\E[\norm\Delta^2\mid\mathcal F_{k-1}]
-\norm{\E[\Delta\mid\mathcal F_{k-1}]}^2
\le dS_1^2\delta^2/h^2
\]
proves (iv) by Minkowski's inequality. Martingale differences adapted to
$(\mathcal F_k)_k$ are orthogonal across rounds by the tower property.
\end{proof}

\begin{aremark}[The $\sigma$-algebra in the centered-noise condition]
\label{rem:centered-filtration}
\emph{The modified condition.} Replace $\mathcal E_t$ in (C) by the smaller
$\sigma$-algebra $\mathcal G_t$ generated by $\mathcal F_{k-1}$, $\Theta_k$, the
queries of round $k$, and the observed values of the queries $s<t$ of round
$k$. The decomposition, the bias and the centering in the proof of
\cref{prop:unified-estimator} use only
$\sigma(\mathcal F_{k-1},\Theta_k)\subset\mathcal G_t$, so (i), (iii), (iv)
and $\E[\zeta\mid\mathcal F_{k-1}]=0$ remain valid. The variance argument
loses the orthogonality of $\xi^{\mathrm c}_s$ and $\xi^{\mathrm c}_t$ for
$s<t$, because
$\xi^{\mathrm c}_s=Y_s-f(X_s)-\eta(X_s)-\xi^{\mathrm o}_s-\delta_s$ need not
be $\mathcal G_t$-measurable. Bounding
$\mathsf N(\Xi^{\mathrm c,(\ell)})\le\sqrt d\,S_1\sigma/h$ by Minkowski's
inequality, as for the oblivious part, gives only
\[
\E\bigl[\norm\zeta^2\bigm|\mathcal F_{k-1}\bigr]\le\frac1B\bigl(
\Cf L_\beta^2d^{-\tau}h^{2\beta-2}+C_\eta M^2\bigr)
+\bigl(1+B^{-1/2}\bigr)^2S_1^2\frac{d\sigma^2}{h^2}.
\]
The factor $1/B$ in front of $d\sigma^2/h^2$ cannot be kept with a constant
depending only on $\beta$: an oblivious error can hide a centered error from
the observation, and a later centered error of the round, which may depend on
all frames, can be correlated with the hidden part. The following example
loses a factor of order $\log^2B$.

\emph{Example: the vectors.} Let $k=1$, $z=0$, $h=1$, $f\equiv0$,
$\eta\equiv0$ and $\delta=0$. Fix a node $j_*$ with $a_{j_*}\ne0$, and let
only the queries $t_{j_*}U^{(\ell)}_i$ carry errors. In the sequence space
$\mathsf H=\ell^2(\mathbb Z)$ with unit vectors $\mathbf e_n$ put
\[
\vartheta_n=\frac1{\pi(n+\frac12)}\quad(n\in\mathbb Z),\qquad
\mathbf u_\ell=\sum_n\vartheta_{n-\ell}\mathbf e_n,\quad \ell=1,\dots,B.
\]
These vectors are orthonormal. First,
$\sum_n\vartheta_n^2=8\pi^{-2}\sum_{n\ge0}(2n+1)^{-2}=1$. Second, for
$r\ne0$,
\[
\sum_n\vartheta_n\vartheta_{n+r}=\frac1{\pi^2r}\sum_n\Bigl(\frac1{n+\frac12}
-\frac1{n+r+\frac12}\Bigr)=0;
\]
the series converges absolutely, and its partial sum over $\abs n\le K$
reduces to $2\abs r$ terms of order $1/K$. Let $P_\ell$ be the orthogonal
projection onto $\mathrm{span}\{\mathbf u_1,\dots,\mathbf u_{\ell-1}\}$, and
put
\[
\mathbf p_\ell=(I-P_\ell)\mathbf e_\ell,\qquad
\alpha_\ell=\inner{\mathbf u_\ell}{\mathbf p_\ell},\qquad
r_\ell=\alpha_\ell+(\alpha_\ell^2+1-\norm{\mathbf p_\ell}^2)^{1/2}\ge0 .
\]
The vectors
\[
\mathbf y_\ell=r_\ell\mathbf u_\ell-\mathbf p_\ell,\qquad
\mathbf v_\ell=\mathbf y_\ell+\mathbf e_\ell=r_\ell\mathbf u_\ell+P_\ell\mathbf e_\ell
\]
satisfy
$\norm{\mathbf y_\ell}^2=(r_\ell-\alpha_\ell)^2-\alpha_\ell^2+\norm{\mathbf p_\ell}^2=1$,
$\mathbf y_\ell\perp\mathbf u_1,\dots,\mathbf u_{\ell-1}$ (both $\mathbf u_\ell$ and
$\mathbf p_\ell$ are orthogonal to these vectors), and
$\mathbf v_{\ell'}\in\mathrm{span}\{\mathbf u_1,\dots,\mathbf u_{\ell'}\}$. Hence
$\inner{\mathbf y_\ell}{\mathbf v_{\ell'}}=0$ for $\ell'<\ell$.

\emph{Example: the errors.} Let $\mathbf o^{(1)},\dots,\mathbf o^{(B)}$ be
independent $N(0,I_d)$ vectors, independent of $\Theta_1$, and write
$U^{(\ell)}$ also for the orthogonal matrix with columns $U^{(\ell)}_i$. For
$1\le p\le d$ and $\mathbf x\in\mathsf H$ let
$\mathcal I_p(\mathbf x)=\sum_nx_n\gamma_{p,n}$, where
$\gamma_{p,n}=(U^{(n)}\mathbf o^{(n)})_p$ for $1\le n\le B$ and the other
$\gamma_{p,n}$ are independent $N(0,1)$ variables, independent of
everything else. Conditionally on $\Theta_1$ all $\gamma_{p,n}$ are
independent $N(0,1)$, by rotation invariance, so
\[
\mathrm{Cov}\bigl(\mathcal I_p(\mathbf x),\mathcal I_{p'}(\mathbf x')\bigm|
\Theta_1\bigr)=\mathbf 1\{p=p'\}\inner{\mathbf x}{\mathbf x'}.
\]
At the query $t_{j_*}U^{(\ell)}_i$ set
\[
\xi^{\mathrm o}=\sigma\mathbf o^{(\ell)}_i,\qquad
\xi^{\mathrm c}=\sigma\sum_pU^{(\ell)}_{pi}\mathcal I_p(\mathbf y_\ell),
\]
and let all other errors vanish. Since
$\mathbf o^{(\ell)}_i=\sum_pU^{(\ell)}_{pi}(U^{(\ell)}\mathbf o^{(\ell)})_p
=\sum_pU^{(\ell)}_{pi}\mathcal I_p(\mathbf e_\ell)$, the observed error is
$\nu_{\ell,i}=\sigma\sum_pU^{(\ell)}_{pi}\mathcal I_p(\mathbf v_\ell)$. Condition
(O) holds, since the oblivious errors are functions of
$\mathbf o^{(1)},\dots,\mathbf o^{(B)}$. Conditionally on $\Theta_1$, the error
$\xi^{\mathrm c}$ at $(\ell,i)$ and the observed errors $\nu_{\ell',i'}$ of the
earlier queries are jointly Gaussian with covariance
$\sigma^2\inner{U^{(\ell)}_i}{U^{(\ell')}_{i'}}\inner{\mathbf y_\ell}{\mathbf v_{\ell'}}$.
This covariance vanishes for $\ell'<\ell$ by the orthogonality above, and for
$\ell'=\ell$, $i'<i$ because $U^{(\ell)}_i\perp U^{(\ell)}_{i'}$. So, given
$\mathcal G_t$, $\xi^{\mathrm c}$ is $N(0,\sigma^2)$: (C) holds with
$\mathcal G_t$ in place of $\mathcal E_t$. With $\mathcal E_t$ it cannot hold
for large $B$, since otherwise (ii) would apply.

\emph{Example: the second moment.} Here $b(z)=0$ and $\Delta=0$, so
$\zeta=\bar g$. Since $\sum_iU^{(\ell)}_{p'i}U^{(\ell)}_{pi}=\mathbf 1\{p=p'\}$,
the $p$-th coordinate of $\bar g$ equals
$\frac{a_{j_*}\sigma}{2B}\mathcal I_p(\mathbf v_\Sigma)$, where
$\mathbf v_\Sigma=\sum_\ell\mathbf v_\ell\in\mathrm{span}\{\mathbf u_1,\dots,\mathbf u_B\}$,
and therefore
$\E\norm\zeta^2=a_{j_*}^2\sigma^2d\norm{\mathbf v_\Sigma}^2/(4B^2)$. The inner
product $\inner{\mathbf v_{\ell'}}{\mathbf u_\ell}$ equals $0$ for $\ell'<\ell$,
$r_\ell$ for $\ell'=\ell$, and
$\inner{\mathbf e_{\ell'}}{\mathbf u_\ell}=\vartheta_{\ell'-\ell}$ for $\ell'>\ell$, since
then $P_{\ell'}\mathbf u_\ell=\mathbf u_\ell$. Hence, with
$\sum_{n=1}^K(n+\frac12)^{-1}\ge\log\frac{2K+3}3$,
\[
\inner{\mathbf v_\Sigma}{\mathbf u_\ell}=r_\ell+\sum_{n=1}^{B-\ell}\vartheta_n
\ge\frac1\pi\log\frac{2(B-\ell)+3}3 .
\]
Keeping the at least $B/2$ indices with $2(B-\ell)+3\ge B+2$, we obtain
\[
\E\norm\zeta^2=\frac{a_{j_*}^2\sigma^2d}{4B^2}\sum_{\ell=1}^B
\inner{\mathbf v_\Sigma}{\mathbf u_\ell}^2
\ge\frac{a_{j_*}^2}{8\pi^2}\log^2\Bigl(\frac{B+2}3\Bigr)
\frac{d\sigma^2}{Bh^2}.
\]
For any constant $K$ the right-hand side exceeds $Kd\sigma^2/(Bh^2)$ once
$B$ is large, although $L_\beta=M=0$.

\emph{Example: exact values.} The lower bound above is crude; the norm is
explicit. Since $\alpha_\ell=\vartheta_0=2/\pi$ and
$\norm{\mathbf p_\ell}^2=1-\sum_{n=1}^{\ell-1}\vartheta_n^2$,
\[
\norm{\mathbf v_\Sigma}^2=\sum_{\ell=1}^B\Bigl(r_\ell+\sum_{n=1}^{B-\ell}
\vartheta_n\Bigr)^2,\qquad r_\ell=\frac2\pi+\Bigl(\frac4{\pi^2}
+\sum_{n=1}^{\ell-1}\vartheta_n^2\Bigr)^{1/2}.
\]
A direct computation gives $\norm{\mathbf v_\Sigma}^2/B=3.77$, $5.59$, $7.83$,
$10.47$, $11.94$ for $B=16$, $64$, $256$, $1024$, $2048$; it is $2$ when the
centered and oblivious errors are independent. For $m=1$ the bound of (ii),
$C_\sigma d\sigma^2/(Bh^2)$ with $C_\sigma=(1+2^{-1/2})^2$, is violated from
$B=1797$ on; the displayed lower bound shows this only for $B$ of order
$10^7$.
\end{aremark}

\subsection{Lipschitz noise chosen after the frames}

\begin{aremark}[Lipschitz noise must not see the frames]
\label{rem:adaptive-lipschitz}
Condition (L) requires $\eta$ to be fixed before the run. Suppose instead
that $\eta$ may be chosen after $\Theta_k$ is revealed. We construct such
an $\eta$ that shifts the estimate of one frame by order $\sqrt d\,M$.
Let
$\varkappa_d=d\,\E\abs{\inner{U}{e_1}}=d\,\Gamma(d/2)/(\sqrt\pi\,\Gamma((d+1)/2))$.
The Cauchy--Schwarz inequality and Gautschi's inequality
\citep{Gautschi1959} give $\sqrt{2d/\pi}\le\varkappa_d\le\sqrt d$.

Let $B=1$, $\xi^{\mathrm c}=\xi^{\mathrm o}=0$ and
$\delta=0$, fix a unit vector $v$, and let $\sigma_i=\sgn\inner{U_i}v$
(nonzero almost surely). On the finite set
$\mathcal P=\{z+ht_jsU_i:\ 1\le i\le d,\ 1\le j\le m,\ s=\pm1\}$ of query
points put $\eta(z+x)=M\norm x\sgn\inner xv$, that is,
$\eta(z+ht_jsU_i)=M ht_js\sigma_i$. This assignment is
$\sqrt2M$-Lipschitz on $\mathcal P$. For two points on the same line,
$x=rU_i$ and $y=r'U_i$, the values differ by
$M\abs{r-r'}=M\norm{x-y}$. For $x=rU_i$ and $y=r'U_{i'}$ with
$i\ne i'$, orthogonality gives $\norm{x-y}=(r^2+r'^2)^{1/2}$, and the values
differ by at most $M(\abs r+\abs{r'})\le\sqrt2M\norm{x-y}$. The
McShane extension \citep{McShane1934}
$\eta(x)=\min_{x'\in\mathcal P}(\eta(x')+\sqrt2M\norm{x-x'})$ is
$\sqrt2M$-Lipschitz on $\R^d$ and agrees with the assignment on
$\mathcal P$. Since $\sum_ja_jt_j=1$, we get $e_\eta(U_i)=M\sigma_i$.
The $f$-part of the estimate still has mean $G_h(z)$, so
\[
\E\bar g-G_h(z)=M\sum_i\E\bigl[\sgn\inner{U_i}v\,U_i\bigr]
=M\varkappa_d\,v,\qquad
\norm{\E\bar g-G_h(z)}\ge\sqrt{2/\pi}\,\sqrt d\,M .
\]
For a $\sqrt2M$-Lipschitz $\eta$ fixed before the frames the same norm is at most
$\sqrt2A_1M$ by \cref{prop:unified-estimator}(i). After rescaling
$\eta$ to be $M$-Lipschitz, the lower bound becomes
$\sqrt{d/\pi}\,M$.

\end{aremark}

\section{Lower bounds}
\label{app:lower}

This appendix proves the lower bounds for Gaussian noise and for Lipschitz
noise. \Cref{thm:radial-lower,thm:ridge-lower} give the bound of
\cref{thm:main}(b) for every budget (\cref{app:lower-proofs}), and
\cref{cor:normalized} the resulting minimax rate for normalized parameters.
\Cref{cor:convex} gives the factor $d^2$ for convex objectives used in
\cref{thm:geometry-lower}, and \cref{prop:lipschitz-floor} bounds the
admissible level of Lipschitz noise in every $\ell_p$ geometry; its
Euclidean case enters \cref{thm:main,thm:convex-lower-main}.

\paragraph{Algorithms and minimax risk.}
The lower bounds use independent $N(0,\sigma^2)$ noise or exact values and
allow adaptive randomized algorithms with queries and output anywhere in
$\R^d$. Such an algorithm draws a random seed $\Theta_0$ at the start; each
query is a measurable function of $\Theta_0$ and of the earlier observations,
and the output is a measurable function of $\Theta_0$ and of all
observations. Algorithms with rounds (\cref{sec:setup}) are the special case
$\Theta_0=(\Theta_1,\Theta_2,\dots)$, since the Gaussian noise is independent
of all $\Theta_k$. For a class $\mathcal C$ of functions on $\R^d$, the
Gaussian minimax risk $\Risk_T^{\mathrm G}(\mathcal C)$ is the infimum of
$\sup_{f\in\mathcal C}\E[f(\hat x)-\fstar]$ over these algorithms with at most
$T$ calls under Gaussian noise. It is nonincreasing in $T$, since extra calls
can be ignored.

\subsection{Two localized profiles}

Fix a smooth $\chi:\R\to[0,1]$ with $\chi(s)=1$ for $s\le1/4$ and
$\chi(s)=0$ for $s\ge1$, for instance
$\chi(s)=\phi(1-s)/(\phi(1-s)+\phi(s-1/4))$ with
$\phi(t)=e^{-1/t}\mathbf 1\{t>0\}$; all derivatives of $\chi$ are bounded. For
a unit vector $v$ let
\begin{equation}
\label{eq:radial-profile}
\psi_v(x)=\inner{v}{x}\,\chi(\norm x^2).
\end{equation}

\begin{alemma}[Radial profile]
\label{lem:radial-profile}
Let $q\ge2$, $\gamma\in(0,1]$, $\beta=q+\gamma$. There are constants
$K_\beta,K_2\ge1$, depending only on $\beta$ and $\chi$, such that for all
$d$, all unit $v$ and all $x,y$:
\begin{align*}
\norm{D^q\psi_v(x)-D^q\psi_v(y)}_{\op}&\le K_\beta\norm{x-y}^\gamma,\\
\norm{D^q\psi_v(x)-D^q\psi_v(y)}_{\HS}&\le K_\beta\,d^{\lfloor q/2\rfloor/2}\norm{x-y}^\gamma,\\
\sup_x\norm{\nabla^2\psi_v(x)}_{\op}&\le K_2 .
\end{align*}
\end{alemma}

\begin{proof}
Let $s(x)=\norm x^2$, so that $Ds(x)[u]=2\inner xu$, $D^2s[u,w]=2\inner uw$
and $D^3s=0$. By the chain rule, $D^j(\chi\circ s)$ is a sum, over the
partitions of $\{1,\dots,j\}$ into $i$ pairs and $j-2i$ singletons, of index
permutations of
$2^{j-i}\chi^{(j-i)}(\norm x^2)\,x^{\otimes(j-2i)}\otimes I^{\otimes i}$.
These terms vanish unless $\norm x\le1$, and by \cref{lem:tensor-facts} a term
with $i$ pairs has Hilbert--Schmidt norm at most $C_jd^{i/2}$ and operator
norm at most $C_j$, where $C_j$ depends only on $j$ and $\chi$.

By the product rule, $D^q\psi_v$ is the sum of $\inner vx D^q(\chi\circ s)$
and of the terms in which $v$ fills one slot of $D^{q-1}(\chi\circ s)$; each
of its terms has at most $\lfloor q/2\rfloor$ identity factors. A further
derivative in a unit direction $e$ replaces $\inner vx$ by $\inner ve$, turns
a coefficient $\chi^{(r)}(\norm x^2)$ into $2\inner xe\chi^{(r+1)}(\norm x^2)$,
or replaces a factor $x$ by $e$; none of these creates an identity factor.
Hence, with $B_q,E_q$ depending only on $q$ and $\chi$,
\[
\sup_x\norm{D^q\psi_v(x)}_{\HS}\le B_qd^{\lfloor q/2\rfloor/2},
\qquad
\sup_{x,\norm e=1}\norm{D^{q+1}\psi_v(x)[e]}_{\HS}\le E_qd^{\lfloor q/2\rfloor/2}.
\]
Integration along segments gives
$\norm{D^q\psi_v(x)-D^q\psi_v(y)}_{\HS}\le
d^{\lfloor q/2\rfloor/2}\min\{2B_q,E_q\norm{x-y}\}$, and
$\min\{A,Bt\}\le A^{1-\gamma}(Bt)^\gamma$ gives the Hilbert--Schmidt bound. On
unit vectors every factor $\inner xu$ or $\inner uw$ of these tensors is at
most one in absolute value, which gives the operator bound and the Hessian
bound without a power of $d$.
\end{proof}

The second profile replaces the linear factor by a narrow ridge. Let
$\vartheta:\R\to\R$ be smooth and odd, equal to its argument on
$[-1/2,1/2]$ and vanishing outside $[-1,1]$ (for instance
$\vartheta(s)=s\chi(s^2)$), and let $L_\vartheta=\sup\abs{\vartheta'}$. For a
unit vector $v$ put $\tilde\psi_v(x)=\vartheta(d^{1/4}\inner vx)\,\chi(\norm x^2)$,
a ridge of width $d^{-1/4}$.

\begin{alemma}[Ridge profile]
\label{lem:ridge-profile}
For every integer $j\ge0$ there is $\widetilde K_j$, independent of $d$
and $v$, with $\sup_x\norm{D^j\tilde\psi_v(x)}_{\HS}\le\widetilde K_jd^{j/4}$.
Consequently
\[
\norm{D^q\tilde\psi_v(x)-D^q\tilde\psi_v(y)}_{\HS}\le
\widetilde K_\beta d^{\beta/4}\norm{x-y}^\gamma,
\qquad
\widetilde K_\beta=(2\widetilde K_q)^{1-\gamma}\widetilde K_{q+1}^\gamma.
\]
\end{alemma}

\begin{proof}
In a product-rule term of order $j$ let $l$ derivatives fall on the ridge,
giving $d^{l/4}\vartheta^{(l)}(d^{1/4}\inner vx)v^{\otimes l}$, and the rest on
$\chi(\norm x^2)$, giving terms with $i$ identity factors and Hilbert--Schmidt
norm $O(d^{i/2})$ as in the proof of \cref{lem:radial-profile}, where
$l+2i\le j$. Since $d^{l/4}d^{i/2}\le d^{j/4}$, each term is $O(d^{j/4})$. The
H\"older bound follows as there, from
$\min\{2\widetilde K_qd^{q/4},\widetilde K_{q+1}d^{(q+1)/4}t\}\le
(2\widetilde K_qd^{q/4})^{1-\gamma}(\widetilde K_{q+1}d^{(q+1)/4}t)^\gamma$.
\end{proof}

\subsection{The radial family}

Let $0<\alpha<L$, $\rho=\min\{\alpha,(L-\alpha)/2\}$ and $\mu=\alpha+\rho$, and
let $\nu=0$ for the operator seminorm and $\nu=\lfloor q/2\rfloor$ for the
Hilbert--Schmidt seminorm. For $\omega\in\{-1,1\}^d$ put $v_\omega=\omega/\sqrt d$,
and for $a,h>0$ let
\begin{equation}
\label{eq:radial-family}
f_\omega(x)=\frac\mu2\norm x^2+ah\,\psi_{v_\omega}(x/h)
=\frac\mu2\norm x^2+a\inner{v_\omega}{x}\chi(\norm x^2/h^2),
\qquad a=\frac{L_\beta h^{\beta-1}}{K_\beta d^{\nu/2}} .
\end{equation}

\begin{atheorem}[Radial lower bound]
\label{thm:radial-lower}
Let $\beta>2$, $L_\beta,R,\sigma>0$ and $0<\alpha<L$, let the class
$\Fsc(\beta,L_\beta;\alpha,L,R)$ be defined with the operator or the
Hilbert--Schmidt seminorm, and let $\rho$ and $\nu$ be as above. Let
$K_\beta,K_2$ be the constants of \cref{lem:radial-profile},
$\bar K_2=\max\{1,K_2\}$ and
\begin{equation}
\label{eq:radial-H}
H=\min\Bigl\{\Bigl(\frac{\rho K_\beta d^{\nu/2}}{\bar K_2L_\beta}\Bigr)^{\frac1{\beta-2}},
\Bigl(\frac{\alpha RK_\beta d^{\nu/2}}{L_\beta}\Bigr)^{\frac1{\beta-1}}\Bigr\}.
\end{equation}
For every $T\ge1$,
\begin{equation}
\label{eq:radial-lower}
\Risk_T^{\mathrm G}\bigl(\Fsc(\beta,L_\beta;\alpha,L,R)\bigr)\ge\min\Bigl\{c_\beta\,\frac{L_\beta^{2/\beta}}{\alpha}\,
d^{\,2-\frac{\nu+2}{\beta}}\Bigl(\frac{\sigma^2}{T}\Bigr)^{\frac{\beta-1}{\beta}},\
\frac{3a_H^2}{64\alpha}\Bigr\},
\qquad c_\beta=\frac{3}{64}\,16^{-\frac{\beta-1}\beta}K_\beta^{-\frac2\beta},
\end{equation}
where $a_H=L_\beta H^{\beta-1}/(K_\beta d^{\nu/2})
=\min\{(\rho/\bar K_2)^{\frac{\beta-1}{\beta-2}}(K_\beta d^{\nu/2}/L_\beta)^{\frac1{\beta-2}},\alpha R\}$.
The first term is the smaller one if and only if
$T\ge K_\beta^2d^{\nu+2}\sigma^2/(16L_\beta^2H^{2\beta})$.
\end{atheorem}

\begin{proof}
We show that the functions \eqref{eq:radial-family} belong to the class,
bound the error from below by the number of wrong signs of an estimate of
$\omega$, and apply Assouad's lemma \eqref{eq:assouad}.

\emph{Step 1: membership.} By \cref{lem:radial-profile} and scaling, the
$q$-th derivative of the perturbation $ah\psi_{v_\omega}(\cdot/h)$ has H\"older
constant at most $K_\beta d^{\nu/2}ah^{1-\beta}=L_\beta$ in the chosen norm, and
the quadratic does not contribute since $q\ge2$; hence $[f_\omega]_\beta\le L_\beta$.
Choose
\begin{equation}
\label{eq:radial-h}
h=\min\{H,h_\star\},\qquad
h_\star^{2\beta}=\frac{K_\beta^2d^{\nu+2}\sigma^2}{16L_\beta^2T},
\qquad\text{so that}\qquad a^2h^2\le\frac{d^2\sigma^2}{16T}.
\end{equation}
Since $h\le H$, we have $\bar K_2a/h\le\rho$ and $a\le\alpha R$. The Hessian of the perturbation has operator norm at most
$K_2a/h\le\rho$, so $\alpha I\preceq\nabla^2f_\omega\preceq(\alpha+2\rho)I\preceq LI$.
The point $x_\omega^\star=-(a/\mu)v_\omega$ has norm $a/\mu\le\rho h/\mu\le h/2$,
so $\chi=1$ and $\chi'=0$ at $\norm{x^\star_\omega}^2/h^2$, and
$\nabla f_\omega(x^\star_\omega)=\mu x^\star_\omega+av_\omega=0$. By strong
convexity $x_\omega^\star$ is the minimizer, and
$\norm{x^\star_\omega}\le a/\alpha\le R$. Thus
$f_\omega\in\Fsc(\beta,L_\beta;\alpha,L,R)$.

\emph{Step 2: from the error to signs.} For an output $\hat x$ let
$\hat\omega_j=-\sgn(\hat x_j)$. A wrong sign means
$\abs{\hat x_j-x^\star_{\omega,j}}\ge a/(\mu\sqrt d)$, and strong convexity
with $\mu\le2\alpha$ gives
\begin{equation}
\label{eq:sign-loss}
f_\omega(\hat x)-f_\omega(x_\omega^\star)\ge\frac{\alpha}{2}\norm{\hat x-x_\omega^\star}^2
\ge\frac{a^2}{8\alpha d}\sum_{j=1}^d\mathbf 1\{\hat\omega_j\ne\omega_j\}.
\end{equation}

\emph{Step 3: information.} Let $P_\omega$ be the law of the seed and the
transcript under $f_\omega$, let $\omega^{(j)}$ be $\omega$ with its $j$-th
coordinate flipped, and let $X_{t,j}$ be the $j$-th coordinate of the $t$-th
query $X_t$. By \cref{lem:testing}(b),
\[
\KL(P_\omega\Vert P_{\omega^{(j)}})=\frac{2a^2}{d\sigma^2}\E_\omega\sum_t
X_{t,j}^2\chi(\norm{X_t}^2/h^2)^2 .
\]
Since $\sum_jX_{t,j}^2\chi(\norm{X_t}^2/h^2)^2\le h^2$ for every query,
\[
\frac1{d2^d}\sum_{\omega,j}\KL(P_\omega\Vert P_{\omega^{(j)}})\le
\frac{2Ta^2h^2}{d^2\sigma^2}.
\]
Pinsker's and Jensen's inequalities bound the average total variation by
$ah\sqrt T/(d\sigma)\le1/4$. By \eqref{eq:assouad} the average number of wrong
signs is at least $3d/8$, and \eqref{eq:sign-loss} gives
\[
\max_\omega\E_\omega\bigl[f_\omega(\hat x)-f_\omega(x^\star_\omega)\bigr]\ge3a^2/(64\alpha).
\]
Since $a=L_\beta h^{\beta-1}/(K_\beta d^{\nu/2})$ is increasing in $h$,
substituting $h=h_\star$ and $h=H$ gives the two terms of
\eqref{eq:radial-lower}, and $h_\star\le H$ is the stated condition on $T$.
\end{proof}

\subsection{The ridge family}

\begin{atheorem}[Ridge lower bound]
\label{thm:ridge-lower}
Let $\beta>2$, $L_\beta,R,\sigma>0$ and $0<\alpha<L$, and let the class
$\Fsc(\beta,L_\beta;\alpha,L,R)$ be defined with the Hilbert--Schmidt seminorm.
Let $\rho=\min\{\alpha,(L-\alpha)/2\}$, let $\widetilde K_\beta,\widetilde K_2$
be the constants of \cref{lem:ridge-profile}, $L_\vartheta=\sup\abs{\vartheta'}$,
$\bar K'_2=\max\{1,\widetilde K_2\}$ and
\[
H_0=\min\bigl\{(\rho\widetilde K_\beta/(\bar K'_2L_\beta))^{1/(\beta-2)},
(\alpha R\widetilde K_\beta/L_\beta)^{1/(\beta-1)}\bigr\}.
\]
For every $T\ge1$,
\[
\Risk_T^{\mathrm G}\bigl(\Fsc(\beta,L_\beta;\alpha,L,R)\bigr)\ge\min\Bigl\{\tilde c_\beta\,\frac{L_\beta^{2/\beta}}{\alpha}\,
d^{\,\frac32-\frac3{2\beta}}\Bigl(\frac{\sigma^2}{T}\Bigr)^{\frac{\beta-1}{\beta}},\
\frac{3\tilde a_H^2}{64\alpha}\Bigr\},
\qquad\tilde c_\beta=\frac3{64}(16L_\vartheta^2)^{-\frac{\beta-1}\beta}\widetilde K_\beta^{-\frac2\beta},
\]
where $\tilde a_H=L_\beta H_0^{\beta-1}/\widetilde K_\beta
=\min\{(\rho/\bar K'_2)^{\frac{\beta-1}{\beta-2}}(\widetilde K_\beta/L_\beta)^{\frac1{\beta-2}},\alpha R\}$.
The first term is the smaller one if and only if
$T\ge\widetilde K_\beta^2d^{3/2}\sigma^2/(16L_\vartheta^2L_\beta^2H_0^{2\beta})$.
\end{atheorem}

\begin{proof}
We follow the proof of \cref{thm:radial-lower}; only the functions, their
membership and the information bound change. With $\mu,v_\omega$ as there,
$h>0$ and the ridge width $w=hd^{-1/4}$, let
\[
\tilde f_\omega(x)=\frac\mu2\norm x^2+A\,\vartheta\Bigl(\frac{\inner{v_\omega}x}{w}\Bigr)
\chi(\norm x^2/h^2),\qquad
A=\frac{L_\beta w^\beta}{\widetilde K_\beta},\qquad
a=\frac{A}{w}.
\]

\emph{Membership.} The perturbation equals $A\,\tilde\psi_{v_\omega}(x/h)$. By
\cref{lem:ridge-profile}, its $q$-th derivative has Hilbert--Schmidt H\"older
constant at most $Ah^{-\beta}\widetilde K_\beta d^{\beta/4}
=A\widetilde K_\beta w^{-\beta}=L_\beta$, so $[\tilde f_\omega]_{\beta,\HS}\le L_\beta$, and
its Hessian has operator norm at most
$Ah^{-2}\widetilde K_2d^{1/2}=\widetilde K_2a/w$. Choose $w=\min\{H_0,w_\star\}$ with
\[
w_\star^{2\beta}=\widetilde K_\beta^2d^{3/2}\sigma^2/(16L_\vartheta^2L_\beta^2T).
\]
Since $w\le H_0$, we have $\bar K_2'a/w\le\rho$ and $a\le\alpha R$, hence $\alpha I\preceq\nabla^2\tilde f_\omega\preceq LI$. The
point $-(a/\mu)v_\omega$ has norm at most $w/2$; there $\vartheta'=1$,
$\chi=1$ and $\chi'=0$, so the gradient vanishes, and this point is the
minimizer, of norm at most $a/\alpha\le R$. Thus
$\tilde f_\omega\in\Fsc(\beta,L_\beta;\alpha,L,R)$, and \eqref{eq:sign-loss} holds
with $a=A/w$.

\emph{Information.} Since $\vartheta$ is $L_\vartheta$-Lipschitz,
\[
\abs{\tilde f_\omega(x)-\tilde f_{\omega^{(j)}}(x)}\le
\frac{2L_\vartheta a}{\sqrt d}\abs{x_j}\chi(\norm x^2/h^2),
\]
and Step~3 of the proof of \cref{thm:radial-lower} bounds the average total
variation by $L_\vartheta ah\sqrt T/(d\sigma)$. The choice of $w$ gives
$L_\vartheta^2a^2h^2\le d^2\sigma^2/(16T)$, so this bound is at most $1/4$, and
the error is at least $3a^2/(64\alpha)$ with $a^2=L_\beta^2w^{2\beta-2}/\widetilde K_\beta^2$;
the two values $w=w_\star$ and $w=H_0$ give the two terms.
\end{proof}

\subsection{Consequences for the strongly convex and convex classes}
\label{app:lower-proofs}

Let $\beta>2$, $L_\beta,R,\sigma>0$ and $0<\alpha<L$, let the seminorm be either
one, let $\tau$ be as in \eqref{eq:tau} and $\rho=\min\{\alpha,(L-\alpha)/2\}$,
and put
\begin{equation}
\label{eq:Phi0}
\Phi_0=\min\Bigl\{\alpha R^2,\ \frac{\rho^{2(\beta-1)/(\beta-2)}}{\alpha L_\beta^{2/(\beta-2)}}\Bigr\},
\qquad
T_{\mathrm{lb}}=T_\star(\Phi_0)
=\frac{d^{\,2-\tau/(\beta-1)}\sigma^2L_\beta^{2/(\beta-1)}}{(\alpha\Phi_0)^{\beta/(\beta-1)}},
\end{equation}
so that $\Phi_T\le\Phi_0$ if and only if $T\ge T_{\mathrm{lb}}$. For fixed
$\beta,\alpha,L,R,L_\beta,\sigma$, $T_{\mathrm{lb}}$ is of order
$d^{\,2-\tau/(\beta-1)}\le d^2$.

\begin{proof}[Proof of \cref{thm:main}(b)]
In the operator case and in the Hilbert--Schmidt case with odd $q$,
\cref{thm:radial-lower} with $\nu=\tau$ applies, because then
$\tau=\lfloor q/2\rfloor$ or $\tau=0$ and its exponent $2-(\nu+2)/\beta$ is that
of \eqref{eq:rate}; in the Hilbert--Schmidt case with even $q$,
\cref{thm:ridge-lower} applies, because $\tau_\beta=(\beta-1)/2$ and
$\frac32-\frac3{2\beta}=2-(2+\tau_\beta)/\beta$. Since $d^{\nu/2}\ge1$, both
$a_H^2/\alpha$ and $\tilde a_H^2/\alpha$ are at least a constant depending only on
$\beta$ times $\Phi_0$. Hence
$\Risk_T^{\mathrm G}(\Fsc(\beta,L_\beta;\alpha,L,R))\ge c_\beta\min\{\Phi_T,\Phi_0\}$
for every $T\ge1$. The hard functions satisfy
$\alpha I\preceq\nabla^2f\preceq(\alpha+2\rho)I\preceq3\alpha I$; for $L=3\alpha$ one
has $\rho=\alpha$, and the bound holds on $\Fsc(\beta,L_\beta;\alpha,3\alpha,R)$. For
the term $M^2/\alpha$, \cref{prop:lipschitz-floor}(b) with $p=2$ and the same
Gaussian errors gives $M^2/(2\alpha)$: its quadratic instances have Hessian
$\alpha I$, zero H\"older seminorm and minimizer norm $M/\alpha\le R$. The
maximum of the two bounds is at least half their sum.
\end{proof}

With the parameters normalized, the two bounds of \cref{thm:main} match for
every dimension and every budget.

\begin{acorollary}[Normalized minimax rate]
\label{cor:normalized}
Let $\beta>2$, $\alpha=1$, $L=3$, $R=L_\beta=\sigma=1$, let the seminorm be
either one, and put {$p=(\beta-1)/\beta$}. There are
constants $0<c_\beta\le C_\beta$ such that for all integers $d\ge2$ and $T\ge1$
\[
c_\beta\min\Bigl\{1,\Bigl(\frac{d^{{s}}}T\Bigr)^{p}\Bigr\}
\le\Risk_T^{\mathrm G}\bigl(\Fsc(\beta,1;1,3,1)\bigr)
\le C_\beta\min\Bigl\{1,\Bigl(\frac{d^{{s}}}T\Bigr)^{p}\Bigr\},
\]
and the upper bound is attained under every noise admissible at levels
$(0,1,0)$. Consequently there is $\eps_\beta>0$ such that for
$0<\eps\le\eps_\beta$ the least number of calls of an $\eps$-accurate method
under Gaussian noise is of order {$d^s\eps^{-\beta/(\beta-1)}$}, uniformly in
$d\ge2$. Here {$s=2$} for the operator seminorm, {$s=\frac32$} under
Hilbert--Schmidt smoothness with even $q$, and $\frac32<s\le\frac53$ with odd
$q$.
\end{acorollary}

\begin{proof}
Here $\rho=1$, $\Phi_0=1$ and $\Phi_T=d^{2-(2+\tau)/\beta}T^{-p}={(d^s/T)^{p}}$,
so the lower bound is \cref{thm:main}(b). For the upper bound, $\Phi_1\ge1$,
so the budget condition of \cref{thm:main}(a) holds if $T\ge D_\beta d\log(2d)$
for a constant $D_\beta$ depending only on $\beta$, and in particular if
$T\ge D'_\beta d^{{s}}$, since {$s\ge\frac32$}; then the error is at most
$C_\beta\Phi_T$ and $\Phi_T\le1$. For $T<D'_\beta d^{{s}}$ the method that returns
$0$ without calls has error $f(0)-\fstar\le\frac32\norm\xstar^2\le\frac32$, while
$\min\{1,(d^{{s}}/T)^p\}\ge(D'_\beta)^{-p}$. For the number of calls, take
$\eps_\beta=c_\beta/2$: an error $\eps\le\eps_\beta$ excludes the first term of
the lower bound and forces $T\ge d^{{s}}(c_\beta/\eps)^{1/p}$, while
$T=\lceil d^{{s}}(C_\beta/\eps)^{1/p}\rceil\le2d^{{s}}(C_\beta/\eps)^{1/p}$ {oracle calls}
suffice. For odd $q$, $\tau=(q-1)/2$ and $q<\beta\le q+1$ give
$\frac32<s\le2-\frac{q-1}{2q}\le\frac53$.
\end{proof}

\begin{acorollary}[Convex objectives: the factor $d^2$]
\label{cor:convex}
Let $\beta>2$ and $L_\beta,R,\sigma>0$, and let $\mathcal C$ be the class of convex
$q$ times continuously differentiable functions $f$ on $\R^d$ with
$[f]_{\beta,\op}\le L_\beta$ and a minimizer in $B(R)$. There are
$c_\beta,c'_\beta>0$ such that every algorithm whose expected error is at
most $\eps$ for every $f\in\mathcal C$ under independent $N(0,\sigma^2)$ noise,
where $\eps\le c'_\beta L_\beta R^\beta$, uses $T\ge c_\beta d^2\sigma^2/\eps^2$
calls. For the Hilbert--Schmidt seminorm (that is, with $[f]_{\beta,\HS}\le L_\beta$
in the definition of $\mathcal C$) the same holds for
$\eps\le c'_\beta L_\beta R^\beta d^{-\lfloor q/2\rfloor/2}$. All hard functions have
$\nabla^2f\preceq C_\beta L_\beta R^{\beta-2}I$.
\end{acorollary}

For the operator seminorm and $d\ge256$, \cref{thm:convex-lower-main}(a) is
stronger, up to constants by the factor
$\min\{d^{1/2},(L_\beta R^\beta/\eps)^{2/\beta}\}$. \Cref{cor:convex} also covers
$d<256$ and the Hilbert--Schmidt seminorm, and it is the bound transferred in
\cref{thm:geometry-lower}(b).

\begin{proof}[Proof of \cref{cor:convex}]
\emph{The functions.} Take $\nu$ as in \cref{thm:radial-lower} and the
curvature $\mu=2\bar K_2a/h$ in \eqref{eq:radial-family}. The perturbation has
Hessian of operator norm at most $K_2a/h\le\mu/2$, so
$(\mu/2)I\preceq\nabla^2f_\omega\preceq(3\mu/2)I$; in particular $f_\omega$ is
convex. As in Step~1 of the proof of \cref{thm:radial-lower}, the point
$-(a/\mu)v_\omega$, of norm $h/(2\bar K_2)\le h/2$, is the minimizer; its norm
is at most $R$ if $h\le2\bar K_2R$. Then also
$a/h=L_\beta h^{\beta-2}/(K_\beta d^{\nu/2})\le L_\beta(2\bar K_2R)^{\beta-2}/K_\beta$,
which gives the Hessian bound with $C_\beta=3\bar K_2(2\bar K_2)^{\beta-2}/K_\beta$.
As in \eqref{eq:sign-loss}, strong convexity with modulus $\mu/2$ gives
\[
f_\omega(\hat x)-\min f_\omega\ge\frac{a^2}{4\mu d}\sum_j\mathbf1\{\hat\omega_j\ne\omega_j\}.
\]

\emph{Large budgets.} Choose $h$ by $ah=d\sigma/(4\sqrt T)$, that is,
$h^\beta=K_\beta d^{\nu/2+1}\sigma/(4L_\beta\sqrt T)$. Then $h\le2\bar K_2R$ if and only
if $T\ge T_1:=\bigl(K_\beta d^{\nu/2+1}\sigma/(4L_\beta(2\bar K_2R)^\beta)\bigr)^2$. For
such $T$, Step~3 of the proof of \cref{thm:radial-lower} bounds the average
total variation by $1/4$, and
\[
\Risk_T^{\mathrm G}(\mathcal C)\ge\frac{3a^2}{32\mu}=\frac{3ah}{64\bar K_2}=
\frac{3d\sigma}{256\bar K_2\sqrt T}.
\]

\emph{All budgets.} Put $\eps_1=3d\sigma/(256\bar K_2\sqrt{T_1})
=c''_\beta L_\beta R^\beta d^{-\nu/2}$ with $c''_\beta=3(2\bar K_2)^\beta/(64\bar K_2K_\beta)$.
An algorithm with $T<T_1$ calls has at most $T_1$ calls, so by monotonicity
its worst-case error is at least $\Risk^{\mathrm G}_{T_1}(\mathcal C)\ge\eps_1$.
Hence worst-case error $\eps<\eps_1$ forces $T\ge T_1$, and then
$\eps\ge3d\sigma/(256\bar K_2\sqrt T)$, that is,
$T\ge(3/(256\bar K_2))^2d^2\sigma^2/\eps^2$. The corollary follows with
$c_\beta=(3/(256\bar K_2))^2$ and $c'_\beta=c''_\beta/2$.
\end{proof}

\subsection{Lipschitz noise}
\label{app:lipschitz-floor}

The hard instances for Lipschitz noise are pairs $(f_+,\eta_+)$,
$(f_-,\eta_-)$ with $f_++\eta_+=f_-+\eta_-$: a Lipschitz noise cannot be
distinguished from a Lipschitz perturbation of the objective. Here
$1\le p\le\infty$, the Lipschitz noise is measured in $\norm\cdot_p$, that is,
$\abs{\eta(x)-\eta(y)}\le M\norm{x-y}_p$, and $B_p(R)=\{x:\norm x_p\le R\}$.

\begin{aproposition}[Lipschitz noise: lower bounds]
\label{prop:lipschitz-floor}
Let $1\le p\le\infty$, $\beta>2$, $L_\beta\ge0$, $R>0$ and $M>0$. Let the
centered, oblivious and adversarial components of the noise be given by a rule
that does not depend on $(f,\eta)$, for instance exact values, or independent
$N(0,\sigma^2)$ noise and no other component. Consider any algorithm, with any
number of calls and rounds and with queries and output anywhere in $\R^d$.
\begin{enumerate}[label=(\alph*),leftmargin=*]
\item \emph{(Convex objectives.)} Let $L>0$ and $M\le LRd^{-\abs{1-2/p}}$.
There are convex quadratic functions $f_+,f_-$ with
$0\le\nabla^2f_\pm(x)[v,v]\le L\norm v_p^2$ for all $x,v$ and minimizers in
$B_p(R)$, and linear functions $\eta_\pm$ with
$\abs{\eta_\pm(x)-\eta_\pm(y)}\le M\norm{x-y}_p$, such that
\[
\max_\pm\E\bigl[f_\pm(\hat x)-\min f_\pm\bigr]\ge\frac{M R}2 .
\]
\item \emph{(Strongly convex objectives.)} Let $\alpha>0$, $M\le\alpha R$
and $L\ge\alpha d^{\abs{2/p-1}}$. There are quadratic functions $f_+,f_-$ with
$\alpha\norm v_p^2\le\nabla^2f_\pm(x)[v,v]\le
L\norm v_p^2$ for all $x,v$ and minimizers in $B_p(R)$, and linear functions
$\eta_\pm$ with $\abs{\eta_\pm(x)-\eta_\pm(y)}\le M\norm{x-y}_p$, such that
\[
\max_\pm\E\bigl[f_\pm(\hat x)-\min f_\pm\bigr]\ge\frac{M^2}{2\alpha} .
\]
For $p=2$ the pair is $f_\pm(x)=\frac\alpha2\norm{x\mp(M/\alpha)v}^2$,
$\eta_\pm(x)=\pm M\inner vx$ with a unit vector $v$.
\end{enumerate}
In both parts $D^qf_\pm$ is constant, so
$[f_\pm]_{\beta,p}=[f_\pm]_{\beta,\op}=[f_\pm]_{\beta,\HS}=0$, where
$[f]_{\beta,p}$ is the H\"older seminorm of $D^qf$ with respect to
$\norm\cdot_p$ (\cref{app:geometry}).
Consequently, if the Lipschitz noise is measured in $\norm\cdot_p$ and
$\sigma,\delta\ge0$ are arbitrary, then for every number of calls $T$:
\begin{itemize}
\item if $\eps<LR^2d^{-\abs{1-2/p}}/2$, then $M_{\max}(\eps,T)\le2\eps/R$ on
every class that contains the convex quadratic functions $f$ with
$\nabla^2f[v,v]\le L\norm v_p^2$ for all $v$ and a minimizer in $B_p(R)$;
\item if $L\ge\alpha d^{\abs{2/p-1}}$ and $\eps<\alpha R^2/2$, then
$M_{\max}(\eps,T)\le\sqrt{2\alpha\eps}$ on every class that contains the
quadratic functions $f$ with $\alpha\norm v_p^2\le\nabla^2f[v,v]\le
L\norm v_p^2$ for all $v$ and a minimizer in $B_p(R)$.
\end{itemize}
Such classes are the convex $q$ times continuously differentiable $f$ with
$[f]_{\beta,p}\le L_\beta$, $\nabla^2f(x)[v,v]\le L\norm v_p^2$ for all $x,v$ and a
minimizer in $B_p(R)$ in the first case, the analogous class with
$\alpha\norm v_p^2\le\nabla^2f(x)[v,v]\le L\norm v_p^2$ in the second case,
and, for $p=2$, $\Fcvx(\beta,L_\beta;L,R)$ and $\Fsc(\beta,L_\beta;\alpha,L,R)$ with
either seminorm.
\end{aproposition}

\begin{proof}
\emph{Two indistinguishable instances.} For $\mu>0$ and $w\in\R^d$ put
$f_\pm(x)=\frac\mu2\norm{x\mp w}^2$ and $\eta_\pm(x)=\pm\mu\inner wx$. Then
$f_\pm+\eta_\pm=\frac\mu2(\norm x^2+\norm w^2)$ for both signs. As the
remaining noise does not depend on $(f,\eta)$, induction over the queries
shows that the joint law of all queries, observations and the output $\hat x$
is the same under both instances. Since $\min f_\pm=f_\pm(\pm w)=0$, the
parallelogram law gives
\[
\frac12\sum_\pm\E f_\pm(\hat x)=\frac\mu2\E\bigl[\norm{\hat x}^2+\norm w^2\bigr]
\ge\frac\mu2\norm w^2 .
\]
The functions $f_\pm$ have Hessian $\mu I$ and constant $D^qf_\pm$, and by
H\"older's inequality $\eta_\pm$ is $\mu\norm w_{p^*}$-Lipschitz with respect
to $\norm\cdot_p$, where $p^*=p/(p-1)$.

\emph{Choice of $\mu$ and $w$.} With $\mathbf 1=(1,\dots,1)$, take
\[
\begin{array}{lll}
&p\le2&p\ge2\\[0.3ex]
\text{(a)}\quad&w=Rd^{-1/p}\mathbf 1,\ \ \mu=M d^{2/p-1}/R\qquad
&w=Re_1,\ \ \mu=M/R\\[0.3ex]
\text{(b)}\quad&w=(M/\alpha)d^{-1/p}\mathbf 1,\ \ \mu=\alpha d^{2/p-1}\qquad
&w=(M/\alpha)e_1,\ \ \mu=\alpha
\end{array}
\]
In each case $\mu\norm w_{p^*}=M$; moreover $\norm w_p=R$ and
$\mu\norm w^2/2=M R/2$ in (a), and $\norm w_p=M/\alpha\le R$ and
$\mu\norm w^2/2=M^2/(2\alpha)$ in (b). Since
$\norm v\le\norm v_p\le d^{1/p-1/2}\norm v$ for $p\le2$ and
$d^{1/p-1/2}\norm v\le\norm v_p\le\norm v$ for $p\ge2$, the assumption on
$M$ in (a) gives $\mu\norm v^2\le L\norm v_p^2$, and the assumption on $L$
in (b) gives $\alpha\norm v_p^2\le\mu\norm v^2\le L\norm v_p^2$.

\emph{Admissible levels.} Let the centered, oblivious and adversarial
components vanish; this is admissible at levels $(M,\sigma,\delta)$ for
all $\sigma,\delta\ge0$. In the convex case let $M>2\eps/R$ and
$M'=\min\{M,LRd^{-\abs{1-2/p}}\}$. Then $M'R/2>\eps$ because
$\eps<LR^2d^{-\abs{1-2/p}}/2$, an $M'$-Lipschitz noise is admissible at
level $M$, and the functions of (a) with $M'$ belong to every class
of the first kind; so no method is $\eps$-accurate at level $M$ on such a
class. The strongly convex case is the same, with $M>\sqrt{2\alpha\eps}$,
$M'=\min\{M,\alpha R\}$ and (b).
\end{proof}

For $p=2$ the bound in (b) is attained by the exact minimizer of the observed
function: if $\hat x$ minimizes $f+\eta$, then $\norm{\nabla f(\hat x)}\le M$
because $\eta$ is $M$-Lipschitz, and strong convexity gives
$f(\hat x)-\fstar\le M^2/(2\alpha)$.

\section{Strongly convex objectives: {oracle calls}, rounds and Lipschitz noise}
\label{app:rounds}

This appendix proves the upper bound of \cref{thm:main}(a) and the
achievable levels in \cref{thm:main}(c). Both methods of \cref{alg:descent}
call the estimate of {Procedure~\ref{alg:stencil}} once per round. The
two-phase method is analyzed in \cref{thm:sc-rounds}; it is projected onto
$B(R)$ and uses an upper bound on the Hessian only there, which
\cref{app:gs} needs. The accelerated method is analyzed in
\cref{app:accelerated-parameters}. The optimization lemmas hold for any
oracle with a bounded bias and conditionally centered noise.

\begin{algorithm}[!htbp]
\caption{Strongly convex methods with random-frame estimates}
\label{alg:descent}
\begin{algorithmic}[1]
\Require radius $h$; $\alpha,L,R$; $\kappa=L/\alpha$; choose (a) with $N,B$,
 or (b) with $J,(B_j)_{j=1}^J$ from \cref{thm:acc-explicit}
\State \textbf{(a) two phases} ($N$ even): $z_1=0$, $k_0=16\kappa$, $P=\Proj_{B(R)}$
\For{$k=1,\dots,N$} \Comment{one round each}
\State step $\nu_k=\frac1{4L}$ if $k\le\frac N2$, else $\nu_k=\frac4{\alpha(k-N/2+k_0)}$
\State $z_{k+1}=P\bigl(z_k-\nu_k\,\gbar_{h,B}(z_k)\bigr)$
\EndFor
\State \Return $\hat x=\sum_{k=N/2+1}^{N}w_kz_k\big/\sum_{k=N/2+1}^{N}w_k$, $w_k=k-\frac N2+k_0-1$
\State \textbf{(b) accelerated} ($J$ stages): $x^{(0)}=0$, $\theta=\frac{\sqrt\kappa-1}{\sqrt\kappa+1}$, $N_0=\lceil5\sqrt\kappa\,\rceil$
\For{stages $j=1,\dots,J$} \Comment{the number of frames $B_j$ doubles}
\State $x_{-1}=x_0=x^{(j-1)}$
\For{$k=0,\dots,N_0-1$}
\State $y_k=x_k+\theta(x_k-x_{k-1})$, $x_{k+1}=y_k-\gbar_{h,B_j}(y_k)/L$
\EndFor
\State $x^{(j)}=x_{N_0}$
\EndFor
\State \Return $\hat x=x^{(J)}$
\end{algorithmic}
\end{algorithm}

\paragraph{Setting.}
We use the class $\Fsc(\beta,L_\beta;\alpha,L,R)$ with either seminorm,
$\kappa=L/\alpha$, and the noise of \cref{def:oracle}. The methods also know
$d$, $T$ and the stencil constants. Since $h_T^{2\beta}=d^{2+\tau}\sigma^2/(L_\beta^2T)$ by \eqref{eq:rate},
\begin{equation}
\label{eq:hT-identities}
L_\beta^2d^{-\tau}h_T^{2\beta-2}=\frac{d^2\sigma^2}{Th_T^2}=\alpha\Phi_T ,
\end{equation}
and $\Phi_T=\Phi_1T^{-(\beta-1)/\beta}$, where $\Phi_1$ is the value of
$\Phi_T$ at $T=1$.

\paragraph{The estimate.}
Consider round $k$ of {Procedure~\ref{alg:stencil}} with an
$\mathcal F_{k-1}$-measurable center $z$, a deterministic radius $h>0$ and
$B\ge1$ frames; it makes $2mdB$ calls. By \cref{prop:unified-estimator}, it
returns $\bar g=\nabla f(z)+b(z)+\zeta+\Delta$, where:
\begin{enumerate}[label=(E\arabic*),leftmargin=*]
\item\label{E:bias} $b(z)$ is $\mathcal F_{k-1}$-measurable and
$\norm{b(z)}\le b_f+b_\eta$, with $b_f:=\bfC L_\beta d^{-\tau/2}h^{\beta-1}$ and
$b_\eta:=A_1M$;
\item\label{E:var} $\E[\zeta\mid\mathcal F_{k-1}]=0$ and
$\E[\norm\zeta^2\mid\mathcal F_{k-1}]\le v_h^2/B$, with
$v_h^2:=\Cf L_\beta^2d^{-\tau}h^{2\beta-2}+C_\eta M^2+C_\sigma d\sigma^2/h^2$
(also for $d=1$);
\item\label{E:adv} $\norm\Delta\le b_\delta:=\sqrt d\,S_1\delta/h$ almost surely.
\end{enumerate}
The constants are those of \cref{prop:unified-estimator}. Only the sum
$\zeta+\Delta$ is $\mathcal F_k$-measurable. As explained after
\cref{prop:unified-estimator}, we use $\zeta$ only through \ref{E:var} and
through $\E\inner{\zeta}{w}=0$ for square-integrable
$\mathcal F_{k-1}$-measurable random vectors $w$.

\subsection{Projected descent in two phases}

The first phase contracts the initial distance with a constant step, and the
second uses decreasing steps and weighted averaging, as in
\citet[Section~3.2]{LacosteJulien2012}, with an offset that keeps the steps
small. The following computation is shared by all two-phase methods of the
paper.

\begin{alemma}[Two-phase recursion]
\label{lem:recursion}
Let $a,c>0$, $\mathcal B,\mathcal V\ge0$ and $k_0\ge1$, and let numbers
$\Delta_k,F_k\ge0$ and steps $\mu_k>0$ satisfy, for the indices $k$ specified
below,
\begin{equation}
\label{eq:recursion}
\Delta_{k+1}\le(1-a\mu_k)\Delta_k-c\mu_kF_k+\mu_k\mathcal B+\mu_k^2\mathcal V .
\end{equation}
\begin{enumerate}[label=(\alph*)]
\item Let \eqref{eq:recursion} hold for $k=1,\dots,N$ with
$\mu_k=2/(a(k+k_0))$, and let $w_k=k+k_0-1$ and
$\mathcal W=\sum_{k=1}^Nw_k$. Then
\[
\frac1{\mathcal W}\sum_{k=1}^Nw_kF_k\le\frac{ak_0^2\Delta_1}{cN^2}
+\frac{\mathcal B}c+\frac{4\mathcal V}{acN}.
\]
\item Let \eqref{eq:recursion} hold for $k=1,\dots,N_1+N_2$, with a constant
step $\mu_k=\mu$ for $k\le N_1$, where $a\mu\le\min\{1,2/k_0\}$, and
$\mu_{N_1+k}=2/(a(k+k_0))$ for $1\le k\le N_2$. If $N_2\ge k_0$, then, with
$w_k$ and $\mathcal W$ as in (a) for $N=N_2$,
\[
\frac1{\mathcal W}\sum_{k=1}^{N_2}w_kF_{N_1+k}\le
\frac{ak_0^2\Delta_1}{cN_2^2}\,e^{-a\mu N_1}+\frac{2\mathcal B}c
+\frac{6\mathcal V}{acN_2}.
\]
\end{enumerate}
\end{alemma}

\begin{proof}
(a) Since $1/(c\mu_k)=a(w_k+1)/(2c)$, dividing \eqref{eq:recursion} by
$c\mu_k$ gives
\[
F_k\le\frac a{2c}\bigl((w_k-1)\Delta_k-(w_k+1)\Delta_{k+1}\bigr)
+\frac{\mathcal B}c+\frac{\mu_k\mathcal V}c .
\]
Multiply by $w_k$ and sum. The distance terms telescope because
$w_{k+1}(w_{k+1}-1)=w_k(w_k+1)$, and at most
$\frac a{2c}k_0(k_0-1)\Delta_1$ remains. Moreover $w_k\mu_k\le2/a$ and
$\mathcal W=N(N+2k_0-1)/2\ge N^2/2$.

(b) In the first phase $0\le1-a\mu<1$, so the recursion can be iterated;
dropping $F_k\ge0$ and summing the geometric series gives
$\Delta_{N_1+1}\le(1-a\mu)^{N_1}\Delta_1+\mathcal B/a+\mu\mathcal V/a$,
and $\mu\mathcal V/a\le2\mathcal V/(a^2k_0)$. Insert this into (a), applied to
the second phase, and use $(1-a\mu)^{N_1}\le e^{-a\mu N_1}$ and
$N_2\ge k_0$.
\end{proof}

\begin{alemma}[Two-phase projected descent with a biased oracle]
\label{lem:two-phase}
Let $\mathcal K\subseteq\R^d$ be closed and convex, and let $f:\R^d\to\R$ be
twice continuously differentiable and $\alpha$-strongly convex, with its
minimizer $\xstar$ in $\mathcal K$ and $\nabla^2f(x)\preceq LI$ for
$x\in\mathcal K$; put $\kappa=L/\alpha$. Let
$(\mathcal F_k)_{k\ge0}$ be a filtration and let $z_1\in\mathcal K$ be
$\mathcal F_0$-measurable with $\norm{z_1-\xstar}\le R$. Suppose that at every
$\mathcal F_{k-1}$-measurable center $z_k\in\mathcal K$ an oracle returns an
$\mathcal F_k$-measurable vector
\[
g_k=\nabla f(z_k)+b_k+\zeta_k+\Delta_k,
\quad
\norm{b_k}\le b,\ \ \norm{\Delta_k}\le b_D\ \text{a.s.},\ \
\E[\zeta_k\mid\mathcal F_{k-1}]=0,\ \
\E[\norm{\zeta_k}^2\mid\mathcal F_{k-1}]\le v^2 .
\]
Fix integers $N_1\ge0$ and $N_2\ge16\kappa$, let $k_0=16\kappa$, and run,
with the Euclidean projection $\Proj_{\mathcal K}$ onto $\mathcal K$,
\[
z_{k+1}=\Proj_{\mathcal K}\bigl(z_k-\nu_k\,g_k\bigr),\qquad
\nu_k=\begin{cases}\dfrac1{4L},&k=1,\dots,N_1,\\[1.2ex]
\dfrac4{\alpha(k-N_1+k_0)},&k=N_1+1,\dots,N_1+N_2 .\end{cases}
\]
Output
\[
\hat x=\mathcal W^{-1}\sum_{k=N_1+1}^{N_1+N_2}w_kz_k,
\qquad w_k=k-N_1+k_0-1,
\qquad \mathcal W=\sum_{k=N_1+1}^{N_1+N_2}w_k .
\]
Then
\begin{equation}
\label{eq:tp-final}
\E[f(\hat x)-\fstar]\le\frac{\alpha R^2}{2}\,e^{-N_1/(8\kappa)}
+\frac{7(b+b_D)^2}{\alpha}+\frac{48\,v^2}{\alpha N_2}.
\end{equation}
\end{alemma}

\begin{proof}
Write $D_k=\E\norm{z_k-\xstar}^2$, $F_k=\E[f(z_k)-\fstar]$,
$u_k=b_k+\Delta_k$ and $\bar b=b+b_D$, so that $\norm{u_k}\le\bar b$. By
induction, $z_k\in\mathcal K$ is $\mathcal F_{k-1}$-measurable and square
integrable, so $\E\inner{\zeta_k}{z_k-\xstar}=0$.

For $z\in\mathcal K$ the segment from $\xstar$ to $z$ lies in $\mathcal K$,
so $\nabla f(z)=Q(z-\xstar)$ with
$Q=\int_0^1\nabla^2f(\xstar+t(z-\xstar))\,dt$ and $\alpha I\preceq Q\preceq LI$.
Hence $Q^2\preceq LQ$, and strong convexity gives
\begin{equation}
\label{eq:tp-facts}
\norm{\nabla f(z)}^2\le L\inner{\nabla f(z)}{z-\xstar},\qquad
\tfrac32\inner{\nabla f(z)}{z-\xstar}\ge f(z)-\fstar+\alpha\norm{z-\xstar}^2 ,
\end{equation}
the second by adding $\inner{\nabla f(z)}{z-\xstar}\ge f(z)-\fstar+\frac\alpha2\norm{z-\xstar}^2$
to half of $\inner{\nabla f(z)}{z-\xstar}\ge\alpha\norm{z-\xstar}^2$.
Since $\xstar\in\mathcal K$, $D_{k+1}\le\E\norm{z_k-\nu_kg_k-\xstar}^2$, and
$\norm{g_k}^2\le2\norm{\nabla f(z_k)}^2+4\bar b^2+4\norm{\zeta_k}^2$. Every
step satisfies $\nu_k\le1/(4L)$, since $4/(\alpha(k_0+1))<1/(4L)$, so the
coefficient of $\E\inner{\nabla f(z_k)}{z_k-\xstar}\ge0$ is at most
$-2\nu_k+2L\nu_k^2\le-\frac32\nu_k$. Young's inequality
$2\bar b\norm{z_k-\xstar}\le\alpha\norm{z_k-\xstar}^2/2+2\bar b^2/\alpha$ and
\eqref{eq:tp-facts} give
\begin{equation}
\label{eq:tp-onestep}
D_{k+1}\le\Bigl(1-\frac{\alpha\nu_k}2\Bigr)D_k-\nu_kF_k
+\frac{2\nu_k\bar b^2}{\alpha}+4\nu_k^2(\bar b^2+v^2).
\end{equation}
This is \eqref{eq:recursion} with $a=\alpha/2$, $c=1$,
$\mathcal B=2\bar b^2/\alpha$, $\mathcal V=4(\bar b^2+v^2)$ and the steps of
\cref{lem:recursion}(b) with $\mu=1/(4L)$, for which $a\mu=2/k_0=1/(8\kappa)\le1$. Since
$D_1\le R^2$, $N_2\ge k_0$ and $N_2\ge16$, \cref{lem:recursion}(b) and the
convexity of $f$ give \eqref{eq:tp-final}.
\end{proof}

\subsection{{Oracle calls} and rounds without acceleration}

\begin{atheorem}[Strongly convex objectives: {oracle calls} and rounds]
\label{thm:sc-rounds}
Let $f\in\Fsc$ and let the noise be admissible at levels
$(M,\sigma,\delta)$. Let $N\ge32\kappa$ be even, $B\ge1$ an integer and
$h>0$. Run \cref{lem:two-phase} with $\mathcal K=B(R)$, $z_1=0$ and
$N_1=N_2=N/2$, where the oracle of round $k$ is the estimate of
{Procedure~\ref{alg:stencil}} with center $z_k$, radius $h$ and $B$ frames; this
is variant (a) of \cref{alg:descent}. The method uses $N$ rounds and
$T'=2mdNB$ calls, and
\begin{equation}
\label{eq:sc-general}
\begin{aligned}
\E[f(\hat x)-\fstar]\le{}&\frac{\alpha R^2}2e^{-N/(16\kappa)}
+\frac1\alpha\Bigl[(21C_{\mathrm b}^2+3C_{\mathrm f})L_\beta^2d^{-\tau}h^{2\beta-2}
+\frac{192mC_\sigma d^2\sigma^2}{T'h^2}\Bigr]\\
&+(21A_1^2+3C_\eta)\frac{M^2}{\alpha}
+21S_1^2\frac{d\delta^2}{\alpha h^2}.
\end{aligned}
\end{equation}
\begin{enumerate}[label=(\alph*),leftmargin=*]
\item Let $\sigma>0$ and $T\ge2mdN$. With $B=\lfloor T/(2mdN)\rfloor$ and $h=h_T$
the method uses at most $T$ calls, and
\begin{equation}
\label{eq:sc-rate}
\E[f(\hat x)-\fstar]\le\frac{\alpha R^2}2e^{-N/(16\kappa)}+C_\Phi\Phi_T
+C_M\frac{M^2}{\alpha}+C_\delta\frac{d\delta^2}{\alpha h_T^2},
\end{equation}
where
\[
C_\Phi=21C_{\mathrm b}^2+3C_{\mathrm f}+384mC_\sigma,\qquad
C_M=21A_1^2+3C_\eta,\qquad
C_\delta=21S_1^2 .
\]
Moreover $d\delta^2/(\alpha h_T^2)=(T\delta^2/(d\sigma^2))\Phi_T$.
\item Let $\sigma>0$ and
\[
N_T=2\bigl\lceil\max\{16\kappa,\ 8\kappa\log(\alpha R^2/(2\Phi_T))\}\bigr\rceil .
\]
If $T\ge2mdN_T$, then (a) with $N=N_T$ uses
$N_T\le2+32\kappa+16\kappa\log^+(\alpha R^2/(2\Phi_T))$ rounds, at most $T$ calls,
and
\[
\E[f(\hat x)-\fstar]\le(C_\Phi+1)\Phi_T+C_M\frac{M^2}{\alpha}
+C_\delta\frac{d\delta^2}{\alpha h_T^2}.
\]
The condition $T\ge2mdN_T$, which is equivalent to $B\ge1$, holds whenever
\begin{equation}
\label{eq:sc-minT}
T\ge8md+64md\kappa\Bigl[1+\log(64md\kappa)+\log^+\frac{\alpha R^2}{2\Phi_1}\Bigr].
\end{equation}
\end{enumerate}
The same holds if the upper bound $\nabla^2f\preceq LI$ is assumed only on
$B(R)$.
\end{atheorem}

For $\delta=0$, $L>\alpha$ and $T\ge T_{\mathrm{lb}}$ (see \eqref{eq:Phi0}),
\cref{thm:main}(b) shows that the bound in (b) is optimal up to a factor
depending only on $\beta$; the same holds for the accelerated bound
\eqref{eq:acc-a} below.

\begin{proof}
Apply \cref{lem:two-phase} with $b=b_f+b_\eta$, $b_D=b_\delta$,
$v^2=v_h^2/B$ and $N_1=N_2=N/2$. Its hypotheses follow from
\ref{E:bias}--\ref{E:adv}, since each center is measurable before its round,
and they involve the Hessian only on $\mathcal K=B(R)$; the estimate itself
uses no Hessian bound. Using $(b_f+b_\eta+b_\delta)^2\le3(b_f^2+b_\eta^2+b_\delta^2)$ and
$96/(NB)=192md/T'\le3$ gives \eqref{eq:sc-general}.

For (a), $\lfloor x\rfloor\ge x/2$ for $x\ge1$ gives $T/2\le T'\le T$;
substitute \eqref{eq:hT-identities}. For (b), the definition of $N_T$ makes
the initial-error term at most $\Phi_T$. Moreover
$\Phi_T\ge\Phi_1/T$ and
$\log T\le T/(64md\kappa)+\log(64md\kappa)-1$ imply
$2mdN_T\le T/2+4md+32md\kappa
[1+\log(64md\kappa)+\log^+(\alpha R^2/(2\Phi_1))]$, so \eqref{eq:sc-minT}
suffices.
\end{proof}

\subsection{An accelerated method with inexact gradients}

The method of the next lemma is the constant-momentum scheme of
\citet[Section~2.2]{Nesterov2004}, run with inexact gradients.

\begin{alemma}[Accelerated method with inexact gradients]
\label{lem:agd-noisy}
Let $\varphi:\R^d\to\R$ be differentiable and $\mu$-strongly convex with an
$L_\varphi$-Lipschitz gradient. Put $\kappa_\varphi=L_\varphi/\mu$,
$\omega=\kappa_\varphi^{-1/2}$, $\theta=(1-\omega)/(1+\omega)$, and let
$x_\varphi^\star$ be the minimizer of $\varphi$ and
$\varphi^\star=\varphi(x_\varphi^\star)$. Let $(\mathcal F_k)_{k\ge0}$ be a
filtration, let $x_0$ be $\mathcal F_0$-measurable and square integrable,
$x_{-1}=x_0$, and for $k\ge0$
\[
y_k=x_k+\theta(x_k-x_{k-1}),\qquad
x_{k+1}=y_k-\frac{g_k}{L_\varphi},\qquad
g_k=\nabla\varphi(y_k)+u_k+\zeta_k ,
\]
where $g_k$ is $\mathcal F_{k+1}$-measurable, $\norm{u_k}\le b$ almost surely,
$\E[\zeta_k\mid\mathcal F_k]=0$ and $\E[\norm{\zeta_k}^2\mid\mathcal F_k]\le v^2$.
Let $\mathcal V_0=\varphi(x_0)-\varphi^\star+\frac\mu2\norm{x_0-x_\varphi^\star}^2$.
Then for every $K\ge0$
\begin{equation}
\label{eq:agd-bound}
\E[\varphi(x_K)-\varphi^\star]\le\Bigl(1-\frac{1}{2\sqrt{\kappa_\varphi}}\Bigr)^K
\E\mathcal V_0+\frac{3b^2}{\mu}+\frac{3\sqrt{\kappa_\varphi}\,v^2}{L_\varphi}.
\end{equation}
\end{alemma}

\begin{proof}
Write $\xstar=x_\varphi^\star$, $\nabla_k=\nabla\varphi(y_k)$ and
$e_k=u_k+\zeta_k$. Put $v_0=x_0$,
$m_k=(1-\omega)v_k+\omega y_k$ and
$v_{k+1}=m_k-(\omega/\mu)g_k$. By induction, $(1+\omega)y_k=x_k+\omega v_k$, so
$m_k-\xstar=y_k-\xstar+(1-\omega)(y_k-x_k)/\omega$. The vectors $y_k$, $v_k$
and $m_k$ are $\mathcal F_k$-measurable and square integrable.
For $\mathcal V_k=\varphi(x_k)-\varphi^\star
+(\mu/2)\norm{v_k-\xstar}^2$, the estimate-sequence argument of
\citet[Section~2.2]{Nesterov2004}, with the error retained, gives
\begin{equation}
\label{eq:agd-onestep}
\mathcal V_{k+1}\le(1-\omega)\mathcal V_k
-\omega\inner{e_k}{v_{k+1}-\xstar}.
\end{equation}
Indeed, smoothness gives
$\varphi(x_{k+1})\le\varphi(y_k)-\norm{\nabla_k}^2/(2L_\varphi)
+\norm{e_k}^2/(2L_\varphi)$. Expand $\norm{v_{k+1}-\xstar}^2$, bound
$\norm{m_k-\xstar}^2$ by convexity of the squared norm, and use
$\omega^2/\mu=1/L_\varphi$, the formula for $m_k$, strong convexity at $y_k$
and convexity between $y_k$ and $x_k$. This leaves
\[
\mathcal V_{k+1}\le(1-\omega)\mathcal V_k
+\frac{\norm{g_k}^2-\norm{\nabla_k}^2+\norm{e_k}^2}{2L_\varphi}
-\omega\inner{e_k}{m_k-\xstar}.
\]
The numerator equals $2\inner{e_k}{g_k}$, and
$g_k/L_\varphi=\omega(m_k-v_{k+1})$; this proves \eqref{eq:agd-onestep}.

Since $v_{k+1}-\xstar=(m_k-\xstar-(\omega/\mu)\nabla_k)-(\omega/\mu)e_k$, where
the first vector is $\mathcal F_k$-measurable, conditional centering and
the Cauchy--Schwarz inequality give
\[
-\omega\E\inner{\zeta_k}{v_{k+1}-\xstar}
=\frac{\E\inner{\zeta_k}{u_k}+\E\norm{\zeta_k}^2}{L_\varphi}
\le\frac{bv+v^2}{L_\varphi}.
\]
Pathwise, Young's inequality gives
$-\omega\inner{u_k}{v_{k+1}-\xstar}
\le(\omega/2)\mathcal V_{k+1}+\omega b^2/\mu$. Hence
$(1-\omega/2)\E\mathcal V_{k+1}\le(1-\omega)\E\mathcal V_k+c$,
where $c=\omega b^2/\mu+(bv+v^2)/L_\varphi$. With
$\rho=(1-\omega)/(1-\omega/2)\le1-\omega/2$, iteration yields
\[
\E\mathcal V_K\le\rho^K\E\mathcal V_0+2c/\omega
\le(1-\omega/2)^K\E\mathcal V_0
+3b^2/\mu+3v^2/(\omega L_\varphi),
\]
where we used $2bv\le b^2+v^2$ and $\omega L_\varphi\ge\mu$. The claim
follows because the function gap is at most $\mathcal V_K$.
\end{proof}

\subsection{Restarts with growing batches}

Restarting an accelerated method in stages is a standard way to obtain
linear convergence for strongly convex problems
\citep[cf.\ the multistage method of][]{GhadimiLan2013}. In the next lemma the
batch size doubles from stage to stage, and the bias and the residual
variance are unknown to the method.

\begin{alemma}[Restarts with growing batches]
\label{lem:multistage}
Let $\varphi,\mu,L_\varphi,\kappa_\varphi$ be as in \cref{lem:agd-noisy}, let
$x_0\in\R^d$ be deterministic, and let $\bar{\mathcal V}_0$ be a known number with
$\bar{\mathcal V}_0\ge\varphi(x_0)-\varphi^\star+\frac\mu2\norm{x_0-x_\varphi^\star}^2$.
Let $0<\eps<\bar{\mathcal V}_0$, let $v_{\mathrm{sam}}^2\ge0$ be a known number, and put
\[
J=\Bigl\lceil\log_2\frac{\bar{\mathcal V}_0}{\eps}\Bigr\rceil,\quad
N_0=\lceil5\sqrt{\kappa_\varphi}\,\rceil,\quad
\eps_j=2^{-j}\bar{\mathcal V}_0,\quad
B_j=\max\Bigl\{1,\Bigl\lceil\frac{12\sqrt{\kappa_\varphi}\,v_{\mathrm{sam}}^2}{L_\varphi\eps_j}\Bigr\rceil\Bigr\}.
\]
Let $x^{(0)}=x_0$. For $j=1,\dots,J$ (stage $j$), run $N_0$ iterations of the method
of \cref{lem:agd-noisy} started at $x_0=x_{-1}=x^{(j-1)}$, each iteration using
a batch of $B_j$ samples, and let $x^{(j)}$ be its last iterate. Number the
iterations of all stages consecutively by $i=0,1,\dots,JN_0-1$. Let
$(\mathcal F_i)_{i\ge0}$ be a filtration such that the gradient of iteration $i$
is $\mathcal F_{i+1}$-measurable and, if iteration $i$ belongs to stage $j$, its
error is $u_i+\zeta_i$ with
\[
\norm{u_i}\le b\text{ almost surely},\qquad
\E[\zeta_i\mid\mathcal F_i]=0,\qquad
\E[\norm{\zeta_i}^2\mid\mathcal F_i]\le v_{\mathrm{sam}}^2/B_j+r^2,
\]
where $b,r\ge0$ are unknown. Then
\begin{equation}
\label{eq:ms-bound}
\E[\varphi(x^{(J)})-\varphi^\star]\le\frac34\eps+\frac{4b^2}{\mu}
+\frac{4\sqrt{\kappa_\varphi}\,r^2}{L_\varphi}.
\end{equation}
The procedure uses $JN_0\le(1+\log_2(\bar{\mathcal V}_0/\eps))(1+5\sqrt{\kappa_\varphi})$
iterations (rounds) and $\sum_{j=1}^JN_0B_j\le JN_0+288v_{\mathrm{sam}}^2/(\mu\eps)$ samples. In
particular, if $r=0$ and $b^2\le\mu\eps/16$, then
$\E[\varphi(x^{(J)})-\varphi^\star]\le\eps$.
\end{alemma}

\begin{proof}
Let $E_j=\E[\varphi(x^{(j)})-\varphi^\star]$ and
$G=3b^2/\mu+3\sqrt{\kappa_\varphi}\,r^2/L_\varphi$. The starting point of
each stage is measurable before the stage, so \cref{lem:agd-noisy} applies to
every stage. In it,
$(1-1/(2\sqrt{\kappa_\varphi}))^{N_0}\le e^{-5/2}<1/8$, the batch size bounds
the sampling term by $\eps_j/4$, and, after the first stage, strong
convexity bounds the initial potential by twice the starting gap. Thus
\[
E_1\le\bar{\mathcal V}_0/8+G+\eps_1/4,\qquad
E_j\le E_{j-1}/4+G+\eps_j/4\quad(j\ge2).
\]
Unrolling and using $\eps_j=2^{J-j}\eps_J$ gives
$E_J\le\eps_J/4+(4/3)G+\eps_J/2$, which is \eqref{eq:ms-bound} since
$\eps_J\le\eps$.

The round count is $JN_0$. Moreover
$B_j\le1+12\sqrt{\kappa_\varphi}v_{\mathrm{sam}}^2\,2^j/(L_\varphi\bar{\mathcal V}_0)$,
$\sum_{j=1}^J2^j<4\bar{\mathcal V}_0/\eps$ and
$N_0\le6\sqrt{\kappa_\varphi}$. Therefore
$\sum_jN_0B_j\le JN_0+288v_{\mathrm{sam}}^2/(\mu\eps)$.
For $r=0$ and $b^2\le\mu\eps/16$, the error bound is at most $\eps$.
\end{proof}

\subsection{Accelerated rounds with {random-frame estimates}}
\label{app:accelerated-parameters}

The batch sizes are chosen from the stochastic part
$v_{\mathrm{sam}}^2=C_\sigma d\sigma^2/h^2$ of the variance in \ref{E:var}. The remaining
part $r^2=\Cf L_\beta^2d^{-\tau}h^{2\beta-2}+C_\eta M^2$ is not divided by the
batch size in the analysis, which keeps the batch sizes independent of the
unknown level $M$.

\begin{atheorem}[Accelerated method with explicit constants]
\label{thm:acc-explicit}
Let $\beta>2$, $L_\beta,\alpha,R>0$, $L\ge\alpha$, $\kappa=L/\alpha$ and $\sigma>0$,
and let $\Fsc=\Fsc(\beta,L_\beta;\alpha,L,R)$ be defined with either seminorm. Put
\[
A=1152mC_\sigma,\qquad C^{\mathrm{acc}}_\Phi=A+12C_{\mathrm b}^2+4C_{\mathrm f},
\qquad C^{\mathrm{acc}}_M=12A_1^2+4C_\eta,\qquad
C^{\mathrm{acc}}_\delta=12S_1^2,
\]
and $\bar{\mathcal V}_0=(L+\alpha)R^2/2$.
\begin{enumerate}[label=(\alph*),leftmargin=*]
\item \emph{Fixed budget.} For $T\ge1$, put $h=h_T$, $\eps=A\Phi_T$ and
$v_{\mathrm{sam}}^2=C_\sigma d\sigma^2/h_T^2$. If $\bar{\mathcal V}_0\le\eps$, output
$\hat x=0$ and set $J=0$. Otherwise run \cref{lem:multistage} with
$\varphi=f$, $\mu=\alpha$, $L_\varphi=L$, $x_0=0$ and
these $\bar{\mathcal V}_0,\eps,v_{\mathrm{sam}}^2$, that is, variant (b) of
\cref{alg:descent} with
\[
N_0=\lceil5\sqrt\kappa\,\rceil,\quad
J=\lceil\log_2(\bar{\mathcal V}_0/\eps)\rceil,\quad
\eps_j=2^{-j}\bar{\mathcal V}_0,\quad
B_j=\max\{1,\lceil12\sqrt\kappa\,v_{\mathrm{sam}}^2/(L\eps_j)\rceil\};
\]
each gradient estimate is {Procedure~\ref{alg:stencil}} at the center $y_k$,
with radius $h_T$ and $B_j$ frames. If $T\ge4mdN_0J$, the method uses at
most $T$ calls and
$N=N_0J\le(1+5\sqrt\kappa)(1+\log_2^+(LR^2/(A\Phi_T)))$ rounds. For every
$f\in\Fsc$ and noise admissible at levels $(M,\sigma,\delta)$,
\begin{equation}
\label{eq:acc-a}
\E[f(\hat x)-\fstar]\le C^{\mathrm{acc}}_\Phi\Phi_T
+C^{\mathrm{acc}}_M\frac{M^2}{\alpha}
+C^{\mathrm{acc}}_\delta\frac{d\delta^2}{\alpha h_T^2}.
\end{equation}
A sufficient condition for $T\ge4mdN_0J$ is
\begin{equation}
\label{eq:acc-minT}
T\ge48md\sqrt\kappa\,\Bigl[\log_2(51md\sqrt\kappa)
+\log_2^+\frac{\bar{\mathcal V}_0}{A\Phi_1}\Bigr].
\end{equation}
\item \emph{Target accuracy and admissible errors.} Let $\sigma\ge0$,
$0<\eps<LR^2$ and
\[
h=\Bigl(\frac{\alpha\eps\,d^\tau}{8(12C_{\mathrm b}^2+4C_{\mathrm f})L_\beta^2}\Bigr)^{\frac1{2\beta-2}} .
\]
Run \cref{lem:multistage} as in (a), but with target $\eps/2$, radius $h$ and
$v_{\mathrm{sam}}^2=C_\sigma d\sigma^2/h^2$. The method uses $N_0J$ rounds, where
$J=\lceil\log_2((L+\alpha)R^2/\eps)\rceil\le2+\log_2(LR^2/\eps)$, and
\begin{equation}
\label{eq:accelerated-target-cost}
T\le2mdN_0J+\frac{1152mC_\sigma d^2\sigma^2}{\alpha\eps h^2}
\le C_\beta\Bigl(T_\star(\eps)+d\sqrt\kappa\Bigl[1+\log\frac{LR^2}{\eps}\Bigr]\Bigr)
\end{equation}
calls. It satisfies $\E[f(\hat x)-\fstar]\le\eps$ for every $f\in\Fsc$ and
noise admissible at levels $(M,\sigma,\delta)$ with
\[
M\le\Bigl(\frac{\alpha\eps}{8C^{\mathrm{acc}}_M}\Bigr)^{1/2},\qquad
\delta\le h\Bigl(\frac{\alpha\eps}{8C^{\mathrm{acc}}_\delta d}\Bigr)^{1/2}
=c_\beta\delta_\eps ,
\]
where $c_\beta$ depends only on $\beta$. The batch sizes do not depend on
$M$ or $\delta$.
\end{enumerate}
\end{atheorem}

The accelerated method is not projected: its centers, queries and output may
lie outside $B(R)$, as \cref{def:oracle} allows, and the proof uses the
Hessian and estimator bounds on all of $\R^d$.

\begin{proof}[Proof of \cref{thm:acc-explicit}(a) and of the upper bound in \cref{thm:main}(a)]
Iteration $i$ of \cref{lem:multistage} is round $k=i+1$, so the filtration
of the lemma is $(\mathcal F_{k-1})_{k\ge1}$ of \cref{def:oracle}. By
\ref{E:bias}--\ref{E:adv}, the oracle of a round $k$ of stage $j$, with
center $y$, has error $u+\zeta$, where $u=b(y)+\Delta$ satisfies $\norm u\le b_f+b_\eta+b_\delta$,
$\E[\zeta\mid\mathcal F_{k-1}]=0$, and
$\E[\norm\zeta^2\mid\mathcal F_{k-1}]\le v_{\mathrm{sam}}^2/B_j+r^2$, where
$r^2=C_{\mathrm f}L_\beta^2d^{-\tau}h_T^{2\beta-2}+C_\eta M^2$.
By \eqref{eq:hT-identities},
\begin{equation}
\label{eq:acc-identities}
\frac{b_f^2}{\alpha}=C_{\mathrm b}^2\Phi_T,\quad
\frac{b_\delta^2}{\alpha}=S_1^2\frac{d\delta^2}{\alpha h_T^2},\quad
\frac{v_{\mathrm{sam}}^2}{\alpha}=C_\sigma\frac Td\Phi_T,\quad
\frac{r^2}{\alpha}=C_{\mathrm f}\Phi_T+C_\eta\frac{M^2}{\alpha}.
\end{equation}
Smoothness bounds the initial potential by $\bar{\mathcal V}_0$.
If $\bar{\mathcal V}_0\le\eps$, then $f(0)-\fstar\le\eps=A\Phi_T$, which implies
\eqref{eq:acc-a}. Otherwise
\cref{lem:multistage}, $\sqrt\kappa/L\le1/\alpha$ and
$(b_f+b_\eta+b_\delta)^2\le3(b_f^2+b_\eta^2+b_\delta^2)$ prove
\eqref{eq:acc-a}. The same lemma bounds the number of frames by
$N_0J+288v_{\mathrm{sam}}^2/(\alpha A\Phi_T)=N_0J+T/(4md)$, so the method makes at most
$2mdN_0J+T/2\le T$ calls when $T\ge4mdN_0J$. Since
$\bar{\mathcal V}_0\le LR^2$, $J\le1+\log_2^+(LR^2/(A\Phi_T))$.

A direct computation with $N_0\le6\sqrt\kappa$, $\Phi_T\ge\Phi_1/T$ and
$\log T\le T/a+\log a-1$ for $a=48md\sqrt\kappa/\log2$ shows that
\eqref{eq:acc-minT} implies $T\ge4mdN_0J$.

\Cref{thm:main}(a) follows with
$C_\beta\ge\max\{51m,C^{\mathrm{acc}}_\Phi,C^{\mathrm{acc}}_M,
C^{\mathrm{acc}}_\delta\}$: since $A\ge1$ and $\bar{\mathcal V}_0\le LR^2$,
its budget condition implies \eqref{eq:acc-minT}, and
$d\delta^2/(\alpha h_T^2)=(T\delta^2/(d\sigma^2))\Phi_T$ by
\eqref{eq:hT-identities}.
\end{proof}

\begin{proof}[Proof of \cref{thm:acc-explicit}(b) and of the achievable levels in \cref{thm:main}(c)]
As in the proof of (a), with target $\eps/2$ and radius $h$, the bound
\eqref{eq:ms-bound} is at most $3\eps/8$ plus
$(12C_{\mathrm b}^2+4C_{\mathrm f})L_\beta^2d^{-\tau}h^{2\beta-2}/\alpha=\eps/8$,
the contribution of the smoothing bias and of its residual variance, plus
$C^{\mathrm{acc}}_MM^2/\alpha+C^{\mathrm{acc}}_\delta d\delta^2/(\alpha h^2)\le\eps/4$;
the sum is less than $\eps$. The sample bound of \cref{lem:multistage} gives
\eqref{eq:accelerated-target-cost}, since $h/h_\eps$ depends only on $\beta$
and {$T_\star(\eps)=d^2\sigma^2
(\alpha\eps)^{-1}h_\eps^{-2}$} by \eqref{eq:Tstar} and
\eqref{eq:heps}; for the same reason
$h(\alpha\eps/d)^{1/2}=(h/h_\eps)\delta_\eps$. The converse bounds in
\cref{thm:main}(c) are proved in \cref{prop:lipschitz-floor} and
\cref{app:adversarial-unified}.
\end{proof}

\section{Lower bounds on the number of rounds}
\label{app:rounds-lower}

This appendix bounds from below the number of rounds needed on $\Fsc$ by
algorithms that request at most $n$ values per round, with exact values or
with Gaussian noise. \Cref{app:rl-one-round} treats one round and
\cref{app:rl-nested} treats $N$ rounds, both with the tilts of
\cref{app:rl-tilts}; \cref{app:rl-threshold} deduces the necessary number of
rounds (\cref{cor:rl-threshold}), and \cref{app:rl-discussion} compares it
with the upper bound. For $\beta=2$ and noisy values, \citet{DRY2018} proved
the one-round bound $n^{-2/(d+4)}$ on smooth strongly convex functions, which
is the exponent of \cref{thm:rl-one-round} at $\beta=2$.

\paragraph{Setting.}
Throughout this appendix $\beta>2$, $q=\lceil\beta\rceil-1\ge2$ and
$\gamma=\beta-q\in(0,1]$; the numbers $L_\beta,\alpha,R>0$ and $L\ge3\alpha$ are
fixed, and $\Fsc=\Fsc(\beta,L_\beta;\alpha,L,R)$ is defined with the operator
seminorm (the Hilbert--Schmidt seminorm is treated in \cref{rem:rl-hs}). We
write $\mathrm{vol}$ for the Lebesgue measure on $\R^d$ and $\fstar=\min f$.

\begin{adefinition}[Rounds with a bounded number of calls]
\label{def:rl-rounds}
Let $N,n\ge1$ be integers. An \emph{algorithm with $N$ rounds of at most $n$
calls} draws independent random elements $\Theta_1,\dots,\Theta_N$ with laws
fixed by the algorithm. In round $k$ it makes at most $n$ calls at points of
$\R^d$ that are measurable functions of $\Theta_1,\dots,\Theta_k$ and of the
answers of rounds $1,\dots,k-1$. The output $\hat x$ is a measurable function
of $\Theta_1,\dots,\Theta_N$ and of all answers. The answer to a call at $X_t$
is either the exact value $Y_t=f(X_t)$, or $Y_t=f(X_t)+\xi_t$ with independent
$N(0,\sigma^2)$ variables $\xi_t$, independent of $\Theta_1,\dots,\Theta_N$
(the \emph{Gaussian model}). For a class $\mathcal C$ and $\sigma\ge0$ let
\[
\Risk_{N,n}(\mathcal C;\sigma)=\inf\,\sup_{f\in\mathcal C}\E\bigl[f(\hat x)-\fstar\bigr],
\]
where the infimum is over all algorithms with $N$ rounds of at most $n$ calls,
with exact values if $\sigma=0$ and in the Gaussian model if $\sigma>0$.
\end{adefinition}

This is the round structure of \cref{sec:setup}, and the Gaussian model is
the oracle of \cref{def:oracle} with only the centered component. Calls whose
answers are ignored do not change the law of the output, so in the proofs we
may assume that every round has exactly $n$ calls. An algorithm with $N$
rounds and $T$ calls in total has at most $T$ calls per round, so every lower
bound on $\Risk_{N,n}$ below also holds, with $n=T$, for such algorithms.

\subsection{Tilts}
\label{app:rl-tilts}

Let $\chi$ be the cutoff and $K_\beta,K_2\ge1$ the constants of
\cref{lem:radial-profile}. For $z,v\in\R^d$ and $r>0$ the \emph{tilt} with
center $z$, slope $v$ and radius $r$ is
\begin{equation}
\label{eq:rl-tilt}
P_{z,v,r}(x)=-\inner{v}{x-z}\,\chi\bigl(\norm{x-z}^2/r^2\bigr).
\end{equation}

\begin{alemma}[Translated and rescaled tilts]
\label{lem:rl-tilts}
Let $z,v\in\R^d$ and $r>0$. The function $P_{z,v,r}$ is infinitely
differentiable and vanishes where $\norm{x-z}\ge r$; where
$\norm{x-z}\le r/2$ it satisfies $P_{z,v,r}(x)=-\inner{v}{x-z}$ and
$\nabla P_{z,v,r}(x)=-v$. Moreover
\begin{equation}
\label{eq:rl-tilt-bounds}
[P_{z,v,r}]_{\beta,\op}\le K_\beta\norm v\,r^{1-\beta},
\qquad
\sup_{x\in\R^d}\norm{\nabla^2P_{z,v,r}(x)}_{\op}\le K_2\norm v/r .
\end{equation}
\end{alemma}

\begin{proof}
The support and the local form follow from $\chi=0$ on $[1,\infty)$ and
$\chi=1$, $\chi'=0$ on $(-\infty,1/4]$. For $v\ne0$ let $u=v/\norm v$; with the
profile $\psi_u$ of \eqref{eq:radial-profile},
\begin{equation}
\label{eq:rl-tilt-scaling}
P_{z,v,r}(x)=-\norm v\,r\,\psi_u\bigl((x-z)/r\bigr),\qquad
D^kP_{z,v,r}(x)=-\norm v\,r^{1-k}\,D^k\psi_u\bigl((x-z)/r\bigr),
\end{equation}
and the operator-norm bounds of \cref{lem:radial-profile} give
\eqref{eq:rl-tilt-bounds}.
\end{proof}

\subsection{One round}
\label{app:rl-one-round}

\begin{atheorem}[One round with Gaussian noise]
\label{thm:rl-one-round}
Let $\sigma>0$ and let $n\ge1$ be an integer. Put
\[
\bar r=\min\Bigl\{\frac R2,\Bigl(\frac{\alpha K_\beta}{K_2L_\beta}\Bigr)^{\frac1{\beta-2}}\Bigr\},
\qquad
r=\min\Bigl\{\bar r,\Bigl(\frac{K_\beta^2\sigma^2R^d}{16\cdot2^d\,nL_\beta^2}\Bigr)^{\frac1{2\beta+d}}\Bigr\}.
\]
Then
\begin{equation}
\label{eq:rl-one-round}
\Risk_{1,n}(\Fsc;\sigma)\ge\frac{3L_\beta^2}{64\,\alpha K_\beta^2}\,r^{2(\beta-1)} .
\end{equation}
Consequently $\Risk_{1,n}(\Fsc;\sigma)\ge c_1n^{-2(\beta-1)/(2\beta+d)}$ for
every $n\ge1$, where
\[
c_1=\frac{3L_\beta^2}{64\,\alpha K_\beta^2}\min\Bigl\{\bar r,
\Bigl(\frac{K_\beta^2\sigma^2R^d}{16\cdot2^dL_\beta^2}\Bigr)^{\frac1{2\beta+d}}\Bigr\}^{2(\beta-1)}
\]
does not depend on $n$, and $2(\beta-1)/(2\beta+d)<(\beta-1)/\beta$ for
every $d\ge1$.
\end{atheorem}

\begin{proof}
We hide a sign perturbation at a uniformly random point of $B(R/2)$, bound
the information that one round gives about the sign, and reduce the error to
a test of the sign.

\emph{Step 1: hard functions.} Let $a=L_\beta r^{\beta-1}/K_\beta$. For
$z\in B(R/2)$ and $\omega\in\{-1,1\}$ let
\[
f_{z,\omega}(x)=\alpha\norm{x-z}^2+P_{z,-\omega ae_1,r}(x)
=\alpha\norm{x-z}^2+\omega a\inner{e_1}{x-z}\,\chi\bigl(\norm{x-z}^2/r^2\bigr).
\]
By \cref{lem:rl-tilts}, $[f_{z,\omega}]_{\beta,\op}\le K_\beta ar^{1-\beta}=L_\beta$,
and since $r\le\bar r$ the Hessian of the tilt has operator norm at most
$K_2a/r=K_2L_\beta r^{\beta-2}/K_\beta\le\alpha$. Hence
$\alpha I\preceq\nabla^2f_{z,\omega}\preceq3\alpha I\preceq LI$, and
$a\le\alpha r/K_2\le\alpha r$. The point
$x^\star_{z,\omega}=z-\omega ae_1/(2\alpha)$ satisfies
$\norm{x^\star_{z,\omega}-z}=a/(2\alpha)\le r/2$, so by \cref{lem:rl-tilts}
$\nabla f_{z,\omega}(x^\star_{z,\omega})=2\alpha(x^\star_{z,\omega}-z)+\omega ae_1=0$.
By strong convexity $x^\star_{z,\omega}$ is the minimizer, and
$\norm{x^\star_{z,\omega}}\le R/2+r/2\le3R/4$. Thus $f_{z,\omega}\in\Fsc$.

\emph{Step 2: information.} Fix an algorithm with one round of $n$ calls at
$X_1,\dots,X_n$, which are functions of $\Theta_1$. Let $Z$ be uniform on
$B(R/2)$ and $\omega$ uniform on $\{-1,1\}$, independent of each other, of
$\Theta_1$ and of the noise, and let $P_{z,\omega}$ be the joint law of
$\Theta_1$ and of the answers when the objective is $f_{z,\omega}$. Since
\[
\abs{f_{z,1}(x)-f_{z,-1}(x)}=2a\abs{\inner{e_1}{x-z}}\chi(\norm{x-z}^2/r^2)
\le2ar\,\mathbf 1\{\norm{x-z}<r\}
\]
and the law of $(X_1,\dots,X_n)$ does not depend on $(z,\omega)$,
\cref{lem:testing}(b) gives
\[
\KL(P_{z,1}\Vert P_{z,-1})\le\frac{2a^2r^2}{\sigma^2}\sum_{i=1}^n
\PP\bigl(\norm{X_i-z}<r\bigr).
\]
For every fixed $x$, $\PP(\norm{x-Z}<r)\le\mathrm{vol}(B(r))/\mathrm{vol}(B(R/2))=(2r/R)^d$,
and $Z$ is independent of the queries. Let $\E_Z$ denote the average over $z$
distributed as $Z$. By Pinsker's inequality (\cref{lem:testing}(a)) and
Jensen's inequality,
\[
\E_Z\TV(P_{Z,1},P_{Z,-1})\le\Bigl(\tfrac12\E_Z\KL(P_{Z,1}\Vert P_{Z,-1})\Bigr)^{1/2}
\le\frac{ar}{\sigma}\Bigl(n\Bigl(\frac{2r}{R}\Bigr)^d\Bigr)^{1/2}
=\Bigl(\frac{2^dnL_\beta^2r^{2\beta+d}}{K_\beta^2\sigma^2R^d}\Bigr)^{1/2}\le\frac14
\]
by the choice of $r$.

\emph{Step 3: from the error to a test.} Given $z$, let $\hat\omega=1$ if
$\hat x_1<z_1$ and $\hat\omega=-1$ otherwise. The first coordinate of
$x^\star_{z,\omega}$ is $z_1-\omega a/(2\alpha)$, so $\hat\omega\ne\omega$
implies $\norm{\hat x-x^\star_{z,\omega}}\ge a/(2\alpha)$, and then
$f_{z,\omega}(\hat x)-\min f_{z,\omega}\ge a^2/(8\alpha)$ by strong convexity.
For fixed $z$, $\hat\omega$ is a test between $P_{z,1}$ and $P_{z,-1}$, so by
\cref{lem:testing}(a) and Step~2,
$\PP(\hat\omega\ne\omega)\ge\frac12(1-\E_Z\TV(P_{Z,1},P_{Z,-1}))\ge\frac38$. The
worst-case error over $\Fsc$ is at least the average over $(Z,\omega)$, so
\[
\sup_{f\in\Fsc}\E\bigl[f(\hat x)-\fstar\bigr]
\ge\frac{a^2}{8\alpha}\cdot\frac38=\frac{3L_\beta^2r^{2\beta-2}}{64\,\alpha K_\beta^2}.
\]
The algorithm was arbitrary, which proves \eqref{eq:rl-one-round}.

\emph{Step 4: the power of $n$.} Since $n^{-1/(2\beta+d)}\le1$,
$r\ge\min\{\bar r,(K_\beta^2\sigma^2R^d/(16\cdot2^dL_\beta^2))^{1/(2\beta+d)}\}
\,n^{-1/(2\beta+d)}$, which gives the bound with $c_1$. Finally
$2(\beta-1)/(2\beta+d)<(\beta-1)/\beta$ is equivalent to $d>0$.
\end{proof}

The queries of one round cannot concentrate near a perturbation at a
uniformly random location.

\subsection{Several rounds: nested tilts}
\label{app:rl-nested}

The construction adapts the smooth nested perturbations of
\citet[Appendix~E]{Dorn2026GradientFree} to rounds and to a fixed H\"older seminorm.

\begin{atheorem}[Several rounds]
\label{thm:rl-nested}
Let $N,n\ge1$ be integers and put
\[
r_1=\min\Bigl\{R,\Bigl(\frac{\alpha K_\beta}{8K_2L_\beta}\Bigr)^{\frac1{\beta-2}}\Bigr\},
\qquad
\vartheta_N=\frac{L_\beta}{2\alpha K_\beta(N+1)},
\]
and define recursively
\begin{equation}
\label{eq:rl-recursion}
\varrho_i=\vartheta_Nr_i^{\beta-1}\quad(1\le i\le N+1),\qquad
r_{i+1}=(4Nn)^{-1/d}\varrho_i\quad(1\le i\le N).
\end{equation}
Then, for every $\sigma\ge0$,
\begin{equation}
\label{eq:rl-nested}
\Risk_{N,n}(\Fsc;\sigma)\ge\Risk_{N,n}(\Fsc;0)\ge\frac{3\alpha}{8}\,\varrho_{N+1}^2,
\qquad
\varrho_{N+1}=r_1\bigl(\vartheta_Nr_1^{\beta-2}\bigr)^{\frac{(\beta-1)^{N+1}-1}{\beta-2}}
(4Nn)^{-\frac{(\beta-1)((\beta-1)^N-1)}{d(\beta-2)}} .
\end{equation}
In particular $\Risk_{N,n}(\Fsc;\sigma)\ge c_Nn^{-\nu_N}$ for every $n\ge1$,
where
\begin{equation}
\label{eq:rl-exponent}
\nu_N=\frac{2(\beta-1)\bigl((\beta-1)^N-1\bigr)}{d(\beta-2)},\qquad
c_N=\frac{3\alpha}{8}\,r_1^2\bigl(\vartheta_Nr_1^{\beta-2}\bigr)^{\frac{2((\beta-1)^{N+1}-1)}{\beta-2}}(4N)^{-\nu_N},
\end{equation}
and $c_N$ does not depend on $n$.
\end{atheorem}

\begin{proof}
We hide the minimizer behind $N+1$ tilts with nested random centers
$Z_1,Z_2,\dots$. With probability at least $3/4$, every query of round $k$
misses the ball around $Z_k$ in which $f$ differs from the sum $F_k$ of the
first $k$ tilts. The algorithm then behaves as if it saw only $F_k$ in round
$k$, and it learns nothing about the sign of the last tilt.

\emph{Step 1: scales.} By the choice of $r_1$, and since $K_2\ge1$ and
$N\ge1$,
\begin{equation}
\label{eq:rl-small}
\vartheta_Nr_1^{\beta-2}\le\frac{1}{16K_2(N+1)}\le\frac1{32}.
\end{equation}
If $r_i\le r_1$, then $\varrho_i/r_i=\vartheta_Nr_i^{\beta-2}\le1/32$ and
$r_{i+1}=(4Nn)^{-1/d}\varrho_i\le\varrho_i\le r_i/32\le r_1$. By induction,
$r_{i+1}\le\varrho_i\le r_i/32$ for $1\le i\le N$ and $\varrho_{N+1}\le r_{N+1}/32$;
hence $\varrho_{i+1}\le\varrho_i/32$, and for $1\le i\le N+1$
\begin{equation}
\label{eq:rl-tails}
\sum_{j=i}^{N+1}\varrho_j\le\frac{32}{31}\varrho_i\le\frac{r_i}{31}.
\end{equation}

\emph{Step 2: hard functions.} Let $W_1,\dots,W_N$ be uniform on $B(1)$ and
$\omega$ uniform on $\{-1,1\}$, all independent. Put $V_i=\varrho_iW_i$ for
$i\le N$, $V_{N+1}=\omega\varrho_{N+1}e_1$, $Z_0=0$ and $Z_i=Z_{i-1}+V_i$ for
$1\le i\le N+1$, and
\begin{equation}
\label{eq:rl-nested-f}
F_k(x)=\alpha\norm x^2+\sum_{j=1}^kP_{Z_{j-1},2\alpha V_j,r_j}(x)
\quad(0\le k\le N+1),\qquad f=F_{N+1}.
\end{equation}
Since $\norm{V_j}\le\varrho_j$, \cref{lem:rl-tilts} bounds the seminorm of the
$j$-th tilt by $2\alpha K_\beta\varrho_jr_j^{1-\beta}=2\alpha K_\beta\vartheta_N=L_\beta/(N+1)$,
and the quadratic contributes nothing, so $[f]_{\beta,\op}\le L_\beta$. The Hessian
of the $j$-th tilt has operator norm at most
$2\alpha K_2\varrho_j/r_j=2\alpha K_2\vartheta_Nr_j^{\beta-2}\le\alpha/(8(N+1))$
by \eqref{eq:rl-small}, so $\alpha I\preceq\nabla^2f\preceq3\alpha I\preceq LI$.
By \eqref{eq:rl-tails}, $\norm{Z_{N+1}-Z_{j-1}}\le\sum_{l=j}^{N+1}\varrho_l\le r_j/31$,
so by \cref{lem:rl-tilts} the $j$-th tilt has gradient $-2\alpha V_j$ at
$Z_{N+1}$, and
$\nabla f(Z_{N+1})=2\alpha(Z_{N+1}-\sum_{j=1}^{N+1}V_j)=0$. Thus every
realization of $f$ has the unique minimizer $Z_{N+1}$, of norm at most
$r_1/31\le R$, and lies in $\Fsc$.

\emph{Step 3: localization.} For $1\le k\le N$,
\begin{equation}
\label{eq:rl-local}
f(x)=F_k(x)\qquad\text{whenever}\qquad\norm{x-Z_k}\ge r_{k+1}.
\end{equation}
The difference $f-F_k$ is the sum of the tilts $j=k+1,\dots,N+1$. The tilt $j=k+1$
has center $Z_k$ and radius $r_{k+1}$. For $j\ge k+2$, \eqref{eq:rl-tails}
gives $\norm{Z_{j-1}-Z_k}\le r_{k+1}/31$, hence
$\norm{x-Z_{j-1}}\ge\frac{30}{31}r_{k+1}>r_{k+1}/32\ge r_{k+2}\ge r_j$. All
these tilts vanish at $x$.

\emph{Step 4: a truncated run.} Fix an algorithm with $N$ rounds of $n$ calls
and exact values, with $\Theta_1,\dots,\Theta_N$ independent of
$(W_1,\dots,W_N,\omega)$. We allow its queries in round $k$ to depend also on
$V_1,\dots,V_{k-1}$ and its output on $V_1,\dots,V_N$, which only enlarges the
class of algorithms. In the \emph{truncated run}, with the same
$\Theta_1,\dots,\Theta_N$ and $V_1,\dots,V_N$, the calls of round $k$ are
answered with the values of $F_k$ instead of $f$; let
$\tilde X_{k,1},\dots,\tilde X_{k,n}$ be its queries in round $k$ and $\tilde x$
its output. Since $F_s$ depends only on $V_1,\dots,V_s$, induction on $k$
shows that $\tilde X_{k,i}$ is a function of
$(\Theta_1,\dots,\Theta_k,V_1,\dots,V_{k-1})$ and $\tilde x$ a function of
$(\Theta_1,\dots,\Theta_N,V_1,\dots,V_N)$.

For $1\le k\le N$ let $E_k$ be the event that
$\norm{\tilde X_{k,i}-Z_k}<r_{k+1}$ for some $i\le n$. Given
$\Theta_1,\dots,\Theta_N,W_1,\dots,W_{k-1}$, which determine $\tilde X_{k,i}$
and $Z_{k-1}$, the vector $W_k$ is uniform on $B(1)$ and
$Z_k=Z_{k-1}+\varrho_kW_k$; hence each query satisfies
$\norm{\tilde X_{k,i}-Z_k}<r_{k+1}$ with conditional probability at most
$(r_{k+1}/\varrho_k)^d=1/(4Nn)$. A union bound over the $Nn$ queries gives
\begin{equation}
\label{eq:rl-hit}
\PP\Bigl(\bigcup_{k=1}^NE_k\Bigr)\le\frac14 .
\end{equation}
On $\mathcal G=\bigcap_{k=1}^NE_k^c$ the actual and the truncated runs
coincide: if they agree before round $k$, the queries of round $k$ agree, and
on $E_k^c$ the answers agree by \eqref{eq:rl-local}. In particular
$\hat x=\tilde x$ on $\mathcal G$.

\emph{Step 5: the error with exact values.} By strong convexity,
$f(\hat x)-\fstar\ge\frac\alpha2\norm{\hat x-Z_{N+1}}^2$, where
$Z_{N+1}=Z_N+\omega\varrho_{N+1}e_1$. The variables $\mathbf 1_{\mathcal G}$,
$\tilde x$ and $Z_N$ are functions of $(\Theta_1,\dots,\Theta_N,W_1,\dots,W_N)$,
which is independent of $\omega$, and
$\frac12\sum_{o\in\{-1,1\}}\norm{y-o\varrho_{N+1}e_1}^2=\norm y^2+\varrho_{N+1}^2$
for every $y\in\R^d$. Conditioning on these variables and using
\eqref{eq:rl-hit},
\[
\E\bigl[f(\hat x)-\fstar\bigr]
\ge\frac\alpha2\E\bigl[\mathbf 1_{\mathcal G}\norm{\tilde x-Z_N-\omega\varrho_{N+1}e_1}^2\bigr]
\ge\frac\alpha2\varrho_{N+1}^2\,\PP(\mathcal G)\ge\frac{3\alpha}8\varrho_{N+1}^2 .
\]
The left-hand side averages the error over realizations of $f$ in $\Fsc$,
so the worst-case error is at least $\frac{3\alpha}8\varrho_{N+1}^2$. As the
algorithm was arbitrary, $\Risk_{N,n}(\Fsc;0)\ge\frac{3\alpha}8\varrho_{N+1}^2$.

\emph{Step 6: Gaussian noise.} An algorithm with exact values can simulate a
given algorithm for the Gaussian model with the same rounds and calls: it
draws the noise variables as part of $\Theta_k$ and adds them to the exact
answers. Its queries and output have the same joint law, hence the same error
on every $f$, so $\Risk_{N,n}(\Fsc;0)\le\Risk_{N,n}(\Fsc;\sigma)$.

\emph{Step 7: the endpoint.} By \eqref{eq:rl-recursion},
$\log\varrho_{i+1}=(\beta-1)\log\varrho_i+\log\vartheta_N-\frac{\beta-1}d\log(4Nn)$
and $\log\varrho_1=\log\vartheta_N+(\beta-1)\log r_1$. Solving this linear
recursion gives
\[
\varrho_{N+1}=\vartheta_N^{\frac{(\beta-1)^{N+1}-1}{\beta-2}}\,r_1^{(\beta-1)^{N+1}}\,
(4Nn)^{-\frac{(\beta-1)((\beta-1)^N-1)}{d(\beta-2)}},
\]
which is \eqref{eq:rl-nested} because
$(\beta-1)^{N+1}=1+(\beta-2)\frac{(\beta-1)^{N+1}-1}{\beta-2}$. Squaring gives
$\frac{3\alpha}8\varrho_{N+1}^2=c_Nn^{-\nu_N}$.
\end{proof}

The shifts $\varrho_i=\vartheta_Nr_i^{\beta-1}$ are much smaller than the radii
$r_i$, so the radii shrink faster than geometrically. This is forced by the
fixed H\"older seminorm: by \eqref{eq:rl-tilt-scaling}, a tilt with slope of
norm $2\alpha\varrho$ and radius $r$ has seminorm
$2\alpha\varrho r^{1-\beta}[\psi_u]_{\beta,\op}$, where $[\psi_u]_{\beta,\op}>0$
does not depend on $u$, so shifts proportional to the radii would give tilts
with seminorms growing like $r_i^{2-\beta}$. For $\sigma>0$ and $N=1$,
\cref{thm:rl-one-round} is stronger than \cref{thm:rl-nested}, whose exponent
is $\nu_1=2(\beta-1)/d>2(\beta-1)/(2\beta+d)$; with exact values,
\cref{thm:rl-nested} is the only bound of this appendix.

\subsection{Rounds needed for the sequential rate}
\label{app:rl-threshold}

\begin{acorollary}[A necessary number of rounds]
\label{cor:rl-threshold}
Let $\sigma>0$, let $d$, $\beta$, $L_\beta$, $\alpha$, $R$ and $L\ge3\alpha$ be
fixed, and let $N\ge1$ be an integer such that $N=1$ or
\begin{equation}
\label{eq:rl-threshold}
(\beta-1)^N<1+\frac{d(\beta-2)}{2\beta}.
\end{equation}
Then for every $k\ge0$,
\[
\lim_{n\to\infty}n^{\frac{\beta-1}\beta}\bigl(\log(n+1)\bigr)^{-k}\,
\Risk_{N,n}(\Fsc;\sigma)=\infty .
\]
Consequently, if there are $C,k\ge0$ such that for all sufficiently large $n$
some algorithm with $N$ rounds of at most $n$ calls satisfies
$\sup_{f\in\Fsc}\E[f(\hat x)-\fstar]\le Cn^{-(\beta-1)/\beta}(\log(n+1))^k$,
then
\begin{equation}
\label{eq:rl-Nmin}
N\ge N_\beta(d):=\max\Bigl\{2,\Bigl\lceil\log_{\beta-1}\Bigl(1+\frac{d(\beta-2)}{2\beta}\Bigr)\Bigr\rceil\Bigr\}.
\end{equation}
\end{acorollary}

\begin{proof}
If $N=1$, \cref{thm:rl-one-round} gives
$\Risk_{1,n}(\Fsc;\sigma)\ge c_1n^{-2(\beta-1)/(2\beta+d)}$, with an exponent
smaller than $(\beta-1)/\beta$. Otherwise \eqref{eq:rl-threshold} is
equivalent to $\nu_N<(\beta-1)/\beta$ for the exponent $\nu_N$ of
\eqref{eq:rl-exponent}, and \cref{thm:rl-nested} gives
$\Risk_{N,n}(\Fsc;\sigma)\ge c_Nn^{-\nu_N}$. In both cases the risk is at least
a positive constant, independent of $n$, times $n^{-e}$ with
$e<(\beta-1)/\beta$, and $n^{(\beta-1)/\beta-e}(\log(n+1))^{-k}\to\infty$.

Under the assumption of the second claim,
$\Risk_{N,n}(\Fsc;\sigma)\le Cn^{-(\beta-1)/\beta}(\log(n+1))^k$ for all large
$n$, so neither $N=1$ nor \eqref{eq:rl-threshold} holds. Thus $N\ge2$ and
$N\ge\log_{\beta-1}(1+d(\beta-2)/(2\beta))$, which gives \eqref{eq:rl-Nmin}
because $N$ is an integer.
\end{proof}

For $\beta=3$, for instance, $N_\beta(d)=\max\{2,\lceil\log_2(1+d/6)\rceil\}$,
which is $2$ for $d\le18$ and $5$ for $d=100$. For fixed $\beta$, $N_\beta(d)$
grows like $\log d/\log(\beta-1)$, and as $\beta\downarrow2$ the logarithm in
\eqref{eq:rl-Nmin} tends to $d/4$. By the remark after
\cref{def:rl-rounds}, the corollary also holds with $n$ replaced by the total
number of calls $T$.

\subsection{Discussion}
\label{app:rl-discussion}

\begin{aremark}[Hilbert--Schmidt seminorm]
\label{rem:rl-hs}
For a $q$-tensor $A$, $\norm A_{\HS}\le d^{(q-1)/2}\norm A_{\op}$
(\cref{lem:tensor-facts}(a)). Hence the class $\Fsc(\beta,L_\beta;\alpha,L,R)$
defined with the Hilbert--Schmidt seminorm contains the class defined with the
operator seminorm and the constant $L_\beta d^{-(q-1)/2}$, and its minimax risk is at
least as large. So \cref{thm:rl-one-round,thm:rl-nested} hold for the
Hilbert--Schmidt seminorm with $L_\beta$ replaced by $L_\beta d^{-(q-1)/2}$ in all
parameters. For fixed $d$ this changes only the constants $c_1$ and $c_N$,
and \cref{cor:rl-threshold} holds unchanged.
\end{aremark}

\begin{aremark}[Size of the constants]
\label{rem:rl-comparison}
The constants of \cref{thm:rl-nested} are small: by \eqref{eq:rl-small}, the
factor $(\vartheta_Nr_1^{\beta-2})^{2((\beta-1)^{N+1}-1)/(\beta-2)}$ in
\eqref{eq:rl-exponent} is at most $32^{-2((\beta-1)^{N+1}-1)/(\beta-2)}$, which
is below $10^{-45}$ for $\beta=3$ and $N=3$.
\end{aremark}

\section{Adversarial and Lipschitz errors}
\label{app:adversarial-unified}

This appendix proves the bound on $\delta_{\max}$ in \cref{thm:main}(c) and
the adversarial levels for convex objectives stated at the end of
\cref{sec:convex}. Achievable levels follow from the descent results of
\cref{app:rounds,app:convex-upper}: for strongly convex objectives they are
\cref{thm:acc-explicit}(b). The impossibility bounds use a hiding argument in
which the adversarial error is a fixed function of the query point:
\cref{thm:adv-unified} for strongly convex and
\cref{prop:adv-convex-upper} for convex objectives.
\Cref{app:adv-twopoint} treats two-point feedback and a probability
criterion.

\subsection{Setting}
\label{app:adv-setting}

\paragraph{Admissible levels.}
Let $\mathcal C$ be a class, $\eps>0$, $T\ge1$ and $M,\sigma\ge0$. Then
$\delta_{\max}(\eps,T)$ is the supremum of the $\delta\ge0$ for which some
algorithm with at most $T$ oracle calls, in any number of rounds, is $\eps$-accurate
on $\mathcal C$ at levels $(M,\sigma,\delta)$. The level
$M_{\max}(\eps,T)$ is defined in the same way, with $\sigma$ and $\delta$
fixed. For $M=0$, respectively $\delta=0$, these are the admissible
levels of \cref{sec:setup}. Noise admissible at some levels is admissible at
all larger levels, so both quantities are nonincreasing in the levels held
fixed.

\paragraph{Classes and scales.}
Let $\beta>2$, $q=\lceil\beta\rceil-1$ and $m=\lceil q/2\rceil$, and let
$[f]_\beta$ and $\tau$ be as in \cref{sec:setup}. We use
$\Fsc=\Fsc(\beta,L_\beta;\alpha,L,R)$ with $0<\alpha<L$ and $\kappa=L/\alpha$, and
$\Fcvx=\Fcvx(\beta,L_\beta;L,R)$. For $\eps>0$, the quantities $T_\star(\eps)$,
$h_\eps$ and $\delta_\eps$ of \eqref{eq:Tstar} and \eqref{eq:heps} satisfy
\begin{equation}
\label{eq:adv-scales}
{T_\star(\eps)=d^2\sigma^2(\alpha\eps)^{-1}h_\eps^{-2}},\qquad
L_\beta^2d^{-\tau}h_\eps^{2\beta-2}=\alpha\eps,\qquad
\frac{d\delta_\eps^2}{\alpha h_\eps^2}=\eps,
\end{equation}
and for $\sigma>0$, \eqref{eq:rate} gives $\Phi_{T_\star(\eps)}=\eps$ and
$h_T=h_{\Phi_T}$.

\subsection{Strongly convex objectives}
\label{app:adv-sc}

\paragraph{Constants.}
Put $\rho=\min\{\alpha,(L-\alpha)/2\}$; let $\chi$,
$K_\beta$ and $K_2$ be as in \cref{lem:radial-profile}, and $\vartheta$,
$L_\vartheta$, $\widetilde K_\beta$ and $\widetilde K_2$ as in
\cref{lem:ridge-profile}.
\begin{itemize}
\item In the operator case, and in the Hilbert--Schmidt case with odd $q$, put
$r_\beta=(\sqrt{32}K_\beta)^{\frac1{\beta-1}}$, $\ell_*=1$, $K_*=\max\{1,K_2\}$ and
$\nu_*=\tau$.
\item In the Hilbert--Schmidt case with even $q$, put
$r_\beta=(\sqrt{32}\widetilde K_\beta)^{\frac1{\beta-1}}$, $\ell_*=L_\vartheta$,
$K_*=\max\{1,\widetilde K_2\}$ and $\nu_*=0$.
\end{itemize}
Define
\begin{equation}
\label{eq:adv-eps0}
\eps_0=\min\Bigl\{\frac{\alpha R^2}{32},\ \frac1\alpha
\Bigl(\frac{\rho\,r_\beta}{\sqrt{32}\,K_*}\Bigr)^{\frac{2(\beta-1)}{\beta-2}}
\frac{d^{\nu_*/(\beta-2)}}{L_\beta^{2/(\beta-2)}}\Bigr\},\qquad
C_{\mathrm{hid}}=2\sqrt{32}\,\ell_*r_\beta .
\end{equation}
\begin{atheorem}[Hiding adversarial errors, strongly convex case]
\label{thm:adv-unified}
Let $\beta>2$, $0<\alpha<L$, $R,L_\beta>0$ and $\mathcal C=\Fsc(\beta,L_\beta;\alpha,L,R)$
with either seminorm, and let $\eps_0$ and $C_{\mathrm{hid}}$ be as in
\eqref{eq:adv-eps0}. For every $\eps\in(0,\eps_0]$, every $T\ge1$ and all
$\sigma\ge0$,
\[
\delta_{\max}(\eps,T)\le C_{\mathrm{hid}}\min\bigl\{\sqrt d,\sqrt{\log(4T)}\bigr\}\,
\delta_\eps .
\]
The adversarial error that achieves this is a fixed function of the query
point, and the other noise components vanish.
\end{atheorem}

By \cref{thm:acc-explicit}(b), the level $c_\beta\delta_\eps$ is admissible
with at most ${C_\beta}\bigl(T_\star(\eps)+d\sqrt\kappa\,[1+\log(LR^2/\eps)]\bigr)$
oracle calls, so the
bound is attained up to the factor $C_\beta\min\{\sqrt d,\sqrt{\log(4T)}\}$.

\begin{proof}[Proof of \cref{thm:adv-unified}]
\emph{Hard functions.} Put $a=\sqrt{32\alpha\eps}$, $\mu=\alpha+\rho$,
$f_0(x)=\mu\norm x^2/2$ and $\bar h=r_\beta h_\eps$, and let $v\in\Sph$. In the
operator case and in the Hilbert--Schmidt case with odd $q$, let
\[
f_v(x)=f_0(x)+a\inner vx\chi(\norm x^2/\bar h^2),
\qquad a=\frac{L_\beta\bar h^{\beta-1}}{K_\beta d^{\tau/2}}.
\]
In the Hilbert--Schmidt case with even $q$, let $\bar w=\bar hd^{-1/4}$ and
\[
f_v(x)=f_0(x)+a\bar w\,\vartheta(\inner vx/\bar w)\chi(\norm x^2/\bar h^2),
\qquad a=\frac{L_\beta\bar w^{\beta-1}}{\widetilde K_\beta}.
\]
In both cases the second expression for $a$ equals $\sqrt{32\alpha\eps}$, by
\eqref{eq:adv-scales} and the choice of $r_\beta$, and the condition
$\eps\le\eps_0$ is equivalent to $a\le\alpha R$ together with
$K_*a/\bar h\le\rho$ in the first case, or $K_*a/\bar w\le\rho$ in the second.
As in the membership proofs of \cref{thm:radial-lower,thm:ridge-lower},
$[f_v]_\beta\le L_\beta$, $\alpha I\preceq\nabla^2f_v\preceq(\alpha+2\rho)I\preceq LI$
and $\xstar_v=-(a/\mu)v\in B(R)$. Both families satisfy
$f_{-v}-f_0=-(f_v-f_0)$, $\xstar_{-v}=-\xstar_v$ and
\begin{equation}
\label{eq:adv-unified-diff}
\abs{f_v(x)-f_0(x)}\le \ell_*a\abs{\inner vx}\,
\mathbf1\{\norm x\le\bar h\}\qquad(x\in\R^d).
\end{equation}

\emph{Hiding.} Let $\theta=\min\{1,\sqrt{2\log(4T)/(d-1)}\}$ for $d\ge2$ and
$\theta=1$ for $d=1$, let $A_v=\{x:\abs{\inner vx}>\theta\norm x\}$, and put
$\delta_{\mathrm{hid}}=\ell_*a\theta\bar h$. The oracle returns $Y_v(x)=f_0(x)$ if
$\abs{f_v(x)-f_0(x)}\le\delta_{\mathrm{hid}}$ and $Y_v(x)=f_v(x)$ otherwise.
Its error $Y_v-f_v$ is a Borel function of the query point bounded by
$\delta_{\mathrm{hid}}$, hence admissible at level $\delta_{\mathrm{hid}}$ for any
round structure; the other noise components are zero. By
\eqref{eq:adv-unified-diff}, $Y_v=f_0$ outside $A_v$.

Let $v$ be uniform on $\Sph$, independent of the seed $\Theta$ of the
algorithm. Run the algorithm on the exact values of $f_0$, and let
$X^0_1,\dots,X^0_T$ be its queries and $\tilde x$ its output; they depend only
on $\Theta$. On $\Omega_v=\{X^0_t\notin A_v\text{ for every }t\le T\}$, the run
on $Y_v$ has the same queries and answers, by induction over rounds, so
$\hat x=\tilde x$. If $\theta=1$, then $A_v$ is empty. Otherwise
\cref{lem:caps}, conditionally on $\Theta$, gives
$\PP(X^0_t\in A_v)\le2e^{-(d-1)\theta^2/2}=1/(2T)$. In either case a union
bound gives $\PP(\Omega_v)\ge1/2$.

\emph{Error.} The event $\Omega_v$ does not change under $v\mapsto-v$, and
$(\norm{y-x}^2+\norm{y+x}^2)/2\ge\norm x^2$. Applied with $x=\xstar_v$ and
$y=\tilde x$, strong convexity and $\mu\le2\alpha$ give, with $\E_v$ the
average over $v$,
\[
\E_v\E[f_v(\hat x)-f_v(\xstar_v)]
\ge\frac\alpha2\frac{a^2}{\mu^2}\PP(\Omega_v)
\ge\frac{a^2}{16\alpha}=2\eps,
\]
so some instance has error larger than $\eps$. Finally
$\theta\le2\min\{1,\sqrt{\log(4T)/d}\}$ for every $d\ge1$ and
$a\bar h=\sqrt{32}r_\beta\sqrt d\,\delta_\eps$, hence
$\delta_{\max}(\eps,T)\le\delta_{\mathrm{hid}}
\le C_{\mathrm{hid}}\min\{\sqrt d,\sqrt{\log(4T)}\}\delta_\eps$.
\end{proof}

\subsection{Convex objectives: achievable levels}
\label{app:adv-cvx}

For convex objectives the gradient must be known to accuracy of order
$\eps/R$ instead of $\sqrt{\alpha\eps}$, and both methods of
\cref{thm:convex-upper} are $\eps$-accurate for adversarial levels of order
$\eps h/(R\sqrt d)$, the gradient accuracy times the radius divided by
$\sqrt d$. At the gradient scale $h\asymp(\eps d^{\tau/2}/(L_\beta R))^{1/(\beta-1)}$, and
at the value scale $h=(\eps/(8V))^{1/2}$ with $V$ as in \cref{sec:convex}. For
$\tau=0$ and $u=\eps/(L_\beta R^\beta)\le1$ the first radius is of order
$Ru^{1/(\beta-1)}$. For $\beta\in(2,3]$, $V\le L_\beta(2R)^{\beta-2}$ gives a value-scale
radius of at least $(2^{\beta+1})^{-1/2}Ru^{1/2}$, which is larger for
$\beta\in(2,3)$ and $u<1$, so there the value scale tolerates the larger
level up to constants. For $\beta>3$ only $V=L$ is available, and the value
scale tolerates the larger level if and only if $L\lesssim\eps/h^2$, with $h$
the gradient-scale radius.

\subsection{Convex objectives: an impossibility bound}
\label{app:adv-cvx-upper}

For minimization over a ball with queries in the ball, the obstruction of
\citet{Pasechnyuk2026MALN}, built from linear functions, gives a bound of the
same order for every class that contains linear functions. Linear functions
have no minimizer on $\R^d$, so for $\Fcvx$ we combine the localized functions
of \cref{cor:convex} with the hiding argument of \cref{thm:adv-unified}.

\begin{aproposition}[An impossibility bound for convex objectives]
\label{prop:adv-convex-upper}
Let $\beta>2$, $L,R,L_\beta>0$ and $\mathcal C=\Fcvx(\beta,L_\beta;L,R)$, with either the
operator seminorm ($\nu=0$) or the Hilbert--Schmidt seminorm
($\nu=\lfloor q/2\rfloor$). Let $K_\beta,K_2$ be the constants of
\cref{lem:radial-profile}, $\bar K_2=\max\{1,K_2\}$, and
\[
\eps_c=\min\Bigl\{\frac{L_\beta(2\bar K_2R)^\beta}{32\bar K_2K_\beta d^{\nu/2}},\
\Bigl(\frac{L}{96\bar K_2^2}\Bigr)^{\frac\beta{\beta-2}}
\Bigl(\frac{32\bar K_2K_\beta d^{\nu/2}}{L_\beta}\Bigr)^{\frac2{\beta-2}}\Bigr\}.
\]
For every $\eps\in(0,\eps_c]$, every $T\ge1$, every number of rounds and all
$M,\sigma\ge0$,
\[
\delta_{\max}(\eps,T)\le64\bar K_2\,\eps\min\Bigl\{1,\sqrt{\frac{\log(4T)}d}\Bigr\}
=64\bar K_2\min\bigl\{\sqrt d,\sqrt{\log(4T)}\bigr\}\frac{\eps}{\sqrt d}.
\]
\end{aproposition}

\begin{proof}
Put $\bar h=(32\bar K_2K_\beta d^{\nu/2}\eps/L_\beta)^{1/\beta}$,
$a=L_\beta\bar h^{\beta-1}/(K_\beta d^{\nu/2})$ and $\mu=2\bar K_2a/\bar h$, so that
$a\bar h=32\bar K_2\eps$ and $\mu=64\bar K_2^2\eps/\bar h^2$. For $v\in\Sph$ let
$f_0(x)=\mu\norm x^2/2$ and $f_v(x)=f_0(x)+a\inner vx\chi(\norm x^2/\bar h^2)$.
As in the proof of \cref{cor:convex}, with $v$ in place of $v_\omega$ and
$\bar h$ in place of $h$, $[f_v]_\beta\le L_\beta$,
$(\mu/2)I\preceq\nabla^2f_v\preceq(3\mu/2)I$, and $f_v$ has the minimizer
$\xstar_v=-(a/\mu)v$, of norm $\bar h/(2\bar K_2)$. The two conditions in
$\eps\le\eps_c$ are $\norm{\xstar_v}\le R$ and $3\mu/2\le L$, so
$f_v\in\mathcal C$.

The proof of \cref{thm:adv-unified} applies with the same
$\theta,A_v,Y_v$ and $\delta_{\mathrm{hid}}=a\theta\bar h$. Strong convexity
with modulus $\mu/2$ in its last step gives
\[
\E_v\E[f_v(\hat x)-f_v(\xstar_v)]
\ge\frac\mu4\frac{a^2}{\mu^2}\frac12
=\frac{a\bar h}{16\bar K_2}=2\eps .
\]
Since $a\bar h=32\bar K_2\eps$ and $\theta\le2\min\{1,\sqrt{\log(4T)/d}\}$,
the asserted bound follows.
\end{proof}

\subsection{Two-point feedback and the probability criterion}
\label{app:adv-twopoint}

\begin{aremark}[Two-point oracles]
\label{rem:adv-twopoint}
In two-point feedback a call specifies two points $x,x'$ and returns
$f(x)+\eta(x)+\xi+\delta_t$ and $f(x')+\eta(x')+\xi'+\delta'_t$, where
$\xi,\xi'$ are stochastic (centered plus oblivious) errors and the
adversarial errors $\delta_t,\delta'_t$ are bounded by $\delta$ and chosen as
in \eqref{eq:noise-adv}.

(i) \emph{Achievability.} Suppose that the stochastic part is shared,
$\xi=\xi'$, and that one call queries $z+ht_jU_i$ and $z-ht_jU_i$. The
stochastic error then cancels in every difference of the stencil, and the
Lipschitz and adversarial parts are as before, so the estimate is that of
\cref{prop:unified-estimator} at levels $(M,0,\delta)$. Hence
\cref{thm:acc-explicit}(b) and \cref{thm:convex-upper} hold with $\sigma=0$,
with $T$ counting oracle calls and $md$ oracle calls per frame. In particular, the levels
$c_\beta\sqrt{\alpha\eps}$ and $c_\beta\delta_\eps$ are attained with
$C_\beta md\sqrt\kappa\,[1+\log(LR^2/\eps)]$ oracle calls, and $\delta_\eps$ does not depend
on $\sigma$. Without sharing, a call consists of two one-point queries of the
same round, and the one-point results apply.

(ii) \emph{Impossibility.} The two instances of \cref{prop:lipschitz-floor}
have the same function $f_\pm+\eta_\pm$, so its proof applies to every
feedback that depends on the objective only through $f+\eta$ and whose other
noise components do not depend on $(f,\eta)$, in particular to two-point
feedback. The hard instances of \cref{thm:adv-unified} and
\cref{prop:adv-convex-upper} have $M=\sigma=0$ and an adversarial error
that is a fixed function of the point. Feedback that returns these values at
$r$ points per call is simulated by $r$ one-point queries, so the proofs apply
with $rT$ queries: for $d\ge2$ let
$\theta=\min\{1,\sqrt{2\log(4rT)/(d-1)}\}$, and $\theta=1$ for $d=1$. If
$\theta=1$, the cone $A_v$ is empty and $\PP(\Omega_v^c)=0$. Otherwise
$\theta^2=2\log(4rT)/(d-1)$, and the union bound for the event $\Omega_v$ of
the proof of \cref{thm:adv-unified} and \cref{lem:caps} give
\[
\PP(\Omega_v^c)\le2rTe^{-(d-1)\theta^2/2}=1/2 .
\]
Both bounds hold with $\log(4rT)$ in place of $\log(4T)$. For $r=2$,
$\log(8T)\le\frac32\log(4T)$, so the constants grow by at most $\sqrt{3/2}$.
\end{aremark}

\begin{aremark}[The probability criterion of \citet{Pasechnyuk2026MALN}]
\label{rem:adv-probability}
\citet{Pasechnyuk2026MALN} require
$\PP\{f(\hat x)-\inf f\le\eps\}\ge1-\varpi$ for a fixed $\varpi\in(0,1/2)$
(called $\beta$ there). The hiding argument also applies to this criterion,
with $r\ge1$ points per call as in \cref{rem:adv-twopoint}(ii). For $d\ge2$
take $\theta=\min\{1,\sqrt{2\log(4rT/(1-2\varpi))/(d-1)}\}$, and $\theta=1$
for $d=1$; then $\E_v\PP(\Omega_v^c)\le(1-2\varpi)/2$, by the same two cases
($A_v=\emptyset$ if $\theta=1$, the union bound otherwise), where $\PP$ refers to the
randomness of the algorithm for fixed $v$ and $\E_v$ averages over $v$
uniform on $\Sph$.

On $\Omega_v$ the outputs for $v$ and $-v$ coincide, and one point $x$ cannot
solve both instances. Otherwise strong convexity would give
$\norm{x-\xstar_v}^2,\norm{x+\xstar_v}^2\le2\eps/\alpha$ in
\cref{thm:adv-unified}, and $\le4\eps/\mu$ in \cref{prop:adv-convex-upper},
and the parallelogram identity would give $\norm{\xstar_v}^2\le2\eps/\alpha$,
respectively $\le4\eps/\mu$; in these constructions, however,
$\norm{\xstar_v}^2\ge8\eps/\alpha$, respectively $=16\eps/\mu$. Hence, for each
$v$, the success probabilities for $v$ and $-v$ sum to at most
$1+\PP(\Omega_v^c)$. Averaging over $v$, the mean success probability is at
most $(1+\E_v\PP(\Omega_v^c))/2\le(3-2\varpi)/4<1-\varpi$, so some instance is
solved with probability less than $1-\varpi$. Therefore the levels of
\cref{thm:adv-unified} and \cref{prop:adv-convex-upper}, with
$\log(4rT/(1-2\varpi))$ in place of $\log(4T)$, are not admissible for this
criterion. This covers the deterministic oracle of
\citet{Pasechnyuk2026MALN} with unrestricted queries.
\end{aremark}

\section{Hessian growing with the gradient}
\label{app:gs}

This appendix states and proves \cref{thm:gs}, which \cref{sec:hessian}
summarizes. Throughout, $f$ belongs to
$\Fsc(\beta,L_\beta;\alpha,\infty,R)$, with either seminorm, and satisfies
\begin{equation}
\label{eq:gs-condition}
\norm{\nabla^2f(x)}_{\op}\le L_0+L_1\norm{\nabla f(x)}\qquad(x\in\R^d)
\end{equation}
with $L_0>0$ and $L_1\ge0$, the condition of \citet{Zhang2020Clipping}. At
the minimizer $\nabla f(\xstar)=0$ and $\nabla^2f(\xstar)\succeq\alpha I$, so
$\alpha\le L_0$. The method is the projected two-phase method of
\cref{thm:sc-rounds}, which needs an upper bound on the Hessian only on
$B(R)$; the two lemmas below provide it.

\begin{atheorem}[Hessian growing with the gradient]
\label{thm:gs}
Let $\beta>2$, $\alpha,L_\beta,R,L_0,\sigma>0$, $L_1\ge0$, and let
$f\in\Fsc(\beta,L_\beta;\alpha,\infty,R)$, with either seminorm, satisfy
\eqref{eq:gs-condition}. Put
$L_{\mathrm{eff}}=\min\{L_0e^{2L_1R},C_\beta(L_0(1+L_1R)^{q-2}+L_\beta R^{\beta-2})\}$ and
$\kappa_{\mathrm{eff}}=L_{\mathrm{eff}}/\alpha$. If
$T\ge C_\beta d\kappa_{\mathrm{eff}}(1+\log(d\kappa_{\mathrm{eff}})+\log^+(\alpha R^2/\Phi_1))$,
the projected two-phase method with $L_{\mathrm{eff}}$ in place of $L$ and radius
$h_T$ satisfies \eqref{eq:main-upper} for every noise admissible at levels
$(M,\sigma,\delta)$, in at most
$C_\beta\kappa_{\mathrm{eff}}(1+\log^+(\alpha R^2/\Phi_T))$ rounds. If
$L_0>\alpha$, the lower bounds of \cref{thm:main}(b), (c) hold with
$L=L_0$.
\end{atheorem}

\begin{alemma}[Growth along rays]
\label{lem:gs-rays}
For every $x$, with $s=\norm{x-\xstar}$,
\[
\norm{\nabla f(x)}\le\frac{L_0}{L_1}\bigl(e^{L_1s}-1\bigr)\quad(L_0s\ \text{if}\ L_1=0),
\qquad
\norm{\nabla^2f(x)}_{\op}\le L_0e^{L_1s}.
\]
\end{alemma}

\begin{proof}
Write $x=\xstar+su$ with $\norm u=1$ and put $g(t)=\norm{\nabla f(\xstar+tu)}$.
Integration of the Hessian along the segment and \eqref{eq:gs-condition} give
$g(t)\le w(t):=L_0t+L_1\int_0^tg$. Then $w(0)=0$ and $w'\le L_0+L_1w$, and
multiplication by $e^{-L_1t}$ and integration bound $w$, hence $g$, by the
first claim. Inserting it into \eqref{eq:gs-condition} gives the second.
\end{proof}

\begin{alemma}[Polynomial growth]
\label{lem:gs-poly}
For every $x$, with $s=\norm{x-\xstar}$: if $q=2$, then
$\norm{\nabla^2f(x)}_{\op}\le L_0+L_\beta s^{\beta-2}$; if $q\ge3$, then, with
$n=q-2$ and $C_n=(n+1)n^n$,
\[
\norm{\nabla^2f(x)}_{\op}\le eC_nL_0\max\{1,L_1s\}^{n}+\frac{C_n+1}{n!}\,L_\beta s^{\beta-2}.
\]
\end{alemma}

\begin{proof}
Both seminorms bound $[f]_{\beta,\op}$. If $q=2$, then
$\norm{\nabla^2f(\xstar)}_{\op}\le L_0$ by \eqref{eq:gs-condition}, and
$\norm{\nabla^2f(x)-\nabla^2f(\xstar)}_{\op}\le L_\beta s^{\beta-2}$. Let $q\ge3$
and $s>0$, write $x=\xstar+su$ with $\norm u=1$, fix unit vectors $v,w$ and
put $\varphi(t)=D^2f(\xstar+tu)[v,w]$. Then $\varphi\in C^n$, and
$\varphi^{(n)}(t)=D^qf(\xstar+tu)[v,w,u,\dots,u]$ is $\gamma$-H\"older with
constant $L_\beta$. By \cref{lem:taylor}(a) in dimension one, the Taylor
polynomial $P$ of degree $n$ of $\varphi$ at $0$ satisfies
$\abs{\varphi(t)-P(t)}\le St^{\,n+\gamma}$ for $t\ge0$, with $S=L_\beta/n!$, and
$n+\gamma=\beta-2$. Let $r=\min\{s,1/L_1\}$ ($r=s$ if $L_1=0$). On $[0,r]$,
\cref{lem:gs-rays} gives $\abs\varphi\le eL_0$, hence
$\abs P\le eL_0+Sr^{n+\gamma}$. Lagrange interpolation at the nodes $ir/n$,
$0\le i\le n$, whose basis polynomials are bounded at $s\ge r$ by
$(ns/r)^n$, gives $\abs{P(s)}\le(n+1)(ns/r)^n(eL_0+Sr^{n+\gamma})$. Since
$(s/r)^nr^{n+\gamma}\le s^{n+\gamma}$ and $s/r=\max\{1,L_1s\}$,
\[
\abs{\varphi(s)}\le\abs{P(s)}+Ss^{n+\gamma}
\le eC_nL_0\max\{1,L_1s\}^n+(C_n+1)Ss^{\beta-2}.
\]
The supremum over $v,w$ is the operator norm of the symmetric matrix
$\nabla^2f(x)$. For $s=0$ the claim follows from \eqref{eq:gs-condition}.
\end{proof}

For $x\in B(R)$ we have $\norm{x-\xstar}\le2R$, so the two lemmas give
$\nabla^2f\preceq L_{\mathrm{eff}}I$ on $B(R)$ with
\begin{equation}
\label{eq:gs-Leff}
L_{\mathrm{eff}}=\min\Bigl\{L_0e^{2L_1R},\ eC_nL_0\max\{1,2L_1R\}^{n}
+\frac{C_n+1}{n!}L_\beta(2R)^{\beta-2}\Bigr\}
\end{equation}
for $q\ge3$, and with the second term replaced by $L_0+L_\beta(2R)^{\beta-2}$ for
$q=2$. Since $\max\{1,2t\}\le2(1+t)$, this is at most the value
$\min\{L_0e^{2L_1R},C_\beta(L_0(1+L_1R)^{q-2}+L_\beta R^{\beta-2})\}$ used in
\cref{thm:gs}, and both are at least $L_0\ge\alpha$.

\begin{proof}[Proof of \cref{thm:gs}]
\emph{Upper bound.} Apply \cref{thm:sc-rounds}(b) with $L_{\mathrm{eff}}$ in place of $L$ and
$\kappa_{\mathrm{eff}}=L_{\mathrm{eff}}/\alpha\ge1$ in place of $\kappa$. By the
remark after that theorem, it needs the Hessian bound only on $B(R)$, which
\eqref{eq:gs-Leff} provides; the estimate of {Procedure~\ref{alg:stencil}} uses
no bound on the Hessian, even at queries outside $B(R)$. The budget
condition \eqref{eq:sc-minT}, the number of rounds and the bound
$(C_\Phi+1)\Phi_T+C_MM^2/\alpha+C_\delta d\delta^2/(\alpha h_T^2)$ take the form
stated in \cref{thm:gs} after enlarging $C_\beta$, since
$d\delta^2/(\alpha h_T^2)=(T\delta^2/(d\sigma^2))\Phi_T$.

\emph{Lower bounds.} A function with $\nabla^2f\preceq L_0I$ satisfies \eqref{eq:gs-condition}
for every $L_1\ge0$. If $L_0>\alpha$, the class therefore contains
$\Fsc(\beta,L_\beta;\alpha,L_0,R)$, on which \cref{thm:main}(b) and the converse
bounds of \cref{thm:main}(c) hold.
\end{proof}

\begin{aremark}[The power of $L_1R$]
\label{rem:gs-sharp}
For $4<\beta\le5$, so that $q=4$ and $n=2$, the power $n$ in
\cref{lem:gs-poly} cannot be lowered. For $k>0$ let
$f_k(x)=\frac12\norm x^2+\frac{k^2}{12}x_1^4$. Then
$\nabla^2f_k(x)=I+k^2x_1^2e_1e_1^\top\succeq I$, $D^4f_k$ is constant, so
$[f_k]_\beta=0$, and the minimizer is $0$. With $t=k\abs{x_1}$,
$\norm{\nabla^2f_k(x)}_{\op}=1+t^2\le2+3t+t^3\le2+3k\norm{\nabla f_k(x)}$, so
\eqref{eq:gs-condition} holds with $L_0=2$ and $L_1=3k$, while at
$x=Re_1$ the Hessian has norm $1+(L_1R/3)^2$. Hence no bound of the form
$C\max\{1,L_1R\}^a+C'$ on $B(R)$ with $a<2$ and $C,C'$ independent of $L_1$
holds on these classes.
\end{aremark}

\begin{aremark}[Rounds]
\label{rem:gs-rounds}
The accelerated method of \cref{thm:acc-explicit} is not projected: its
extrapolated points may leave $B(R)$, where \eqref{eq:gs-Leff} does not
bound the Hessian. Whether a round bound proportional to
$\sqrt{\kappa_{\mathrm{eff}}}\log(1/\eps)$ suffices under
\eqref{eq:gs-condition} is open.
\end{aremark}

\section{{Polynomial-pilot methods for strongly convex objectives}}
\label{app:no-hessian}

This appendix proves \cref{thm:nh-main}, except for the method with
logarithmically many rounds, which is treated in
\cref{app:batched-no-hessian}. In this class, the curvature of $f$ and the
coefficients of its Taylor polynomial of degree $q$ may have arbitrary
magnitude. The methods below estimate this polynomial from
function values and then minimize the remainder. The upper bounds are
\cref{thm:nh-upper-general} ($\beta>4$, $T\ge C_0d^q$) and
\cref{thm:nh-upper-quadratic} ($2<\beta\le4$, $T\ge C_0d^2$), where $C_0$
depends only on $\beta,\alpha,R,L_\beta,\sigma$; \cref{app:nh-proof} combines them
with the lower bound of \cref{thm:main}(b).

\subsection{Setting}
\label{app:nh-setting}

Throughout this appendix $\beta>2$, $q=\lceil\beta\rceil-1\ge2$,
$\gamma=\beta-q\in(0,1]$, $m=\lceil q/2\rceil$ and $\alpha,L_\beta,R>0$. We write
$B(r)=\{x:\norm x\le r\}$, $\bar R=R+1$, $\fstar=\min f$,
$K_j=\binom{d+j}{j}$ for the dimension of the space of polynomials of degree
at most $j$ on $\R^d$, and $p_\beta=(\beta-1)/\beta$, so that $\Phi_T$ of
\eqref{eq:rate} is proportional to $T^{-p_\beta}$. As in \cref{sec:setup},
either $L_\beta$ bounds $[f]_{\beta,\op}$ and $\tau=0$, or $L_\beta$ bounds
$[f]_{\beta,\HS}$ and $\tau=\tau_\beta$. In both cases
$[f]_{\beta,\op}\le L_\beta$, and only the stencil bounds \eqref{eq:nh-bh} depend
on the choice.

\begin{adefinition}[Class, centered noise and minimax risk]
\label{def:nh-class}\label{def:nh-noise}
The class $\Fsc(\beta,L_\beta;\alpha,\infty,R)$ is the class of \cref{sec:setup}
with $L=\infty$: $\nabla^2f\succeq\alpha I$, $[f]_\beta\le L_\beta$ and
$\xstar\in B(R)$, while derivatives of orders $2,\dots,q$ may have arbitrary
magnitude. The noise
is \emph{centered at level $\sigma\ge0$} if $\eta\equiv0$,
$\xi^{\mathrm o}\equiv0$, $\delta_t\equiv0$ and the centered component
satisfies \eqref{eq:noise-centered}; a query $X_t$ then returns
$Y_t=f(X_t)+\xi_t$ with
\[
\E[\xi_t\mid\mathcal E_t]=0,\qquad\E[\xi_t^2\mid\mathcal E_t]\le\sigma^2 .
\]
For a class $\mathcal C$, $\Risk_T(\mathcal C;\sigma)=\inf\sup\E[f(\hat x)-\fstar]$,
where the infimum is over all algorithms that know $d,T,\beta,\alpha,R,L_\beta,\sigma$
and make at most $T$ oracle calls in any number of rounds, and the supremum is over
$f\in\mathcal C$ and all noises centered at level $\sigma$.
\end{adefinition}

A function with $\nabla^2f\succeq\alpha I$ is $\alpha$-strongly convex, so its
minimizer exists and is unique. Adding a polynomial of degree at most $q$ to
$f$ adds a constant to $D^qf$ and changes neither $[f]_{\beta,\op}$ nor
$[f]_{\beta,\HS}$. Exact values and independent $N(0,\sigma^2)$ errors that
are independent of $\Theta_1,\Theta_2,\dots$ are centered noise at level
$\sigma$. For centered noise the error $\xi_t=Y_t-f(X_t)$ is a function of
the query and the answer. Hence $\mathcal F_{k-1}$ contains the errors of the
rounds before round $k$, and $\mathcal E_t$ contains $\mathcal F_{k-1}$,
$\Theta_k$ and the errors of the earlier queries of round $k$.

\paragraph{The stencil under centered noise.}
Let the center $z$ of round $k$ be bounded and $\mathcal F_{k-1}$-measurable,
and let the radius $h>0$ be deterministic. We write $\gbar_h(z)=\gbar_{h,1}(z)$
for one round of \cref{alg:stencil} with one frame ($2md$ oracle calls).
\Cref{prop:unified-estimator}(i), (ii) and (iv), with $M=\delta=0$,
$\xi^{\mathrm o}\equiv0$ and $B=1$, give
\begin{equation}
\label{eq:nh-estimate}
\gbar_h(z)=\nabla f(z)+b(z)+\zeta,\qquad
\norm{b(z)}\le b_h,\qquad
\E[\zeta\mid\mathcal F_{k-1}]=0,\qquad
\E\bigl[\norm\zeta^2\bigm|\mathcal F_{k-1}\bigr]\le\varsigma_h^2,
\end{equation}
where $b(z)=G_h(z)-\nabla f(z)$ is $\mathcal F_{k-1}$-measurable, $\zeta$ is
$\mathcal F_k$-measurable, and
\begin{equation}
\label{eq:nh-bh}
b_h=\bfC L_\beta d^{-\tau/2}h^{\beta-1},\qquad
\varsigma_h^2=\Cf L_\beta^2d^{-\tau}h^{2\beta-2}+C_\sigma\frac{d\sigma^2}{h^2};
\end{equation}
for $d=1$ the variance bound holds a fortiori. Since
$\E[\norm{b(z)+\zeta}^2\mid\mathcal F_{k-1}]=\norm{b(z)}^2
+\E[\norm\zeta^2\mid\mathcal F_{k-1}]$, for every bounded
$\mathcal F_{k-1}$-measurable random vector $w$ and every event
$G\in\mathcal F_{k-1}$,
\begin{equation}
\label{eq:nh-estimate-moments}
\E\bigl[\inner{\zeta}{w}\mathbf 1_G\bigr]=0,\qquad
\E\bigl[\norm{\gbar_h(z)-\nabla f(z)}^2\mathbf 1_G\bigr]\le
(b_h^2+\varsigma_h^2)\,\PP(G).
\end{equation}
With $B$ frames, \eqref{eq:nh-estimate} and \eqref{eq:nh-estimate-moments}
hold with $\varsigma_h^2/B$ in place of $\varsigma_h^2$, and so does
\cref{lem:nh-fbf-step} below.

\paragraph{Conventions.}
Minimizers and projections of known continuous functions over compact sets
are chosen by \cref{lem:nh-selection}, so they are Borel functions of the
data; computation is not charged (\cref{rem:nh-computation}). We write
$\norm X_{L^v}=(\E\abs X^v)^{1/v}$ and
$\mathrm{clip}_{[-b,b]}(x)=\max\{-b,\min\{b,x\}\}$ for $b>0$. In statements,
$C_\beta$ and $c_\beta$ depend only on $\beta$. In proofs, $C$ and $c$ denote
positive numbers that depend only on $\beta,\alpha,R,L_\beta,\sigma$ and on
integers fixed in the proof as functions of $\beta$; they never depend on $d$
or $T$, and their values may change from line to line.

\subsection{Tools}
\label{app:nh-tools}

\begin{alemma}[Borel choice of minimizers]
\label{lem:nh-selection}
Let $X$ be a metric space, $Y\subset\R^n$ compact,
$\mathcal K\subseteq X\times Y$ closed with nonempty sections
$\mathcal K_x=\{y:(x,y)\in\mathcal K\}$, and $\psi:\mathcal K\to\R$
continuous. There is a Borel map $\pi:X\to Y$ with $\pi(x)\in\mathcal K_x$
and $\psi(x,\pi(x))=\min_{y\in\mathcal K_x}\psi(x,y)$ for every $x$.
\end{alemma}

\begin{proof}
Extend $\psi$ to a continuous function on $X\times Y$ by Tietze's theorem.
For closed $F\subseteq Y$, the set $\{x:\mathcal K_x\cap F\ne\emptyset\}$ is
the projection of the closed set $\mathcal K\cap(X\times F)$ along the
compact factor $Y$, hence closed. Every open subset of $Y$ is a countable
union of closed sets, so $x\mapsto\mathcal K_x$ is weakly measurable with
nonempty compact values. The measurable maximum theorem
\citep[Theorem~18.19]{AliprantisBorder2006}, applied to $-\psi$ on the Borel
$\sigma$-algebra of $X$, gives $\pi$.
\end{proof}

The next lemma concerns the median of block means
\citep{NemirovskiYudin1983,HsuSabato2016}, for errors that are martingale
differences rather than independent.

\begin{alemma}[Medians of block means under martingale noise]
\label{lem:nh-median}\leavevmode
\begin{enumerate}[label=(\alph*)]
\item Let $r$ be odd, $j=(r+1)/2$, and let $Z_1,\dots,Z_r$ be real random
variables adapted to a filtration $(\mathcal G_i)_{i=0}^r$ with
$\E[Z_i^2\mid\mathcal G_{i-1}]\le v$ for all $i$. Let $\mathrm{med}$ be their
median. Then
\[
\abs{\mathrm{med}}^{2j}\le\sum_S\prod_{i\in S}Z_i^2,
\]
the sum being over all $S\subseteq\{1,\dots,r\}$ with $\abs S=j$. Hence
$\E\abs{\mathrm{med}}^{2j}\le\binom rjv^j$, and
$\PP(\abs{\mathrm{med}}>u)\le\binom rj(v/u^2)^j\le2^r(v/u^2)^j$ for $u>0$.
\item Let $\mathrm{med}_1,\dots,\mathrm{med}_J$ be as in (a), each with its
own filtration, and $E_0=\max_i\abs{\mathrm{med}_i}$. For $u>0$ and $0<a<2j$,
\[
\PP(E_0>u)\le J2^r\Bigl(\frac v{u^2}\Bigr)^j,\qquad
\E\bigl[E_0^a\mathbf 1\{E_0>u\}\bigr]\le J2^rv^j\,\frac{2j}{2j-a}\,u^{a-2j}.
\]
\item Let the noise be centered at level $\sigma$, and let the
queries of round $k$ be made in an order that is fixed before the round.
Let $I_1,\dots,I_r$ be disjoint sets of queries of the round such that every
query of $I_i$ precedes every query of $I_{i+1}$, let $w_t$ be deterministic
numbers, and put $Z_i=\sum_{t\in I_i}w_t\xi_t$. Let $\mathcal G_i$ be the
$\sigma$-algebra generated by $\mathcal F_{k-1}$, $\Theta_k$ and the errors
of all queries of the round up to the last query of $I_i$, and
$\mathcal G_0=\mathcal F_{k-1}\vee\sigma(\Theta_k)$. Then $(Z_i)$ is adapted
to $(\mathcal G_i)$, $\E[Z_i\mid\mathcal G_{i-1}]=0$ and
$\E[Z_i^2\mid\mathcal G_{i-1}]\le\sigma^2\sum_{t\in I_i}w_t^2$.
\end{enumerate}
\end{alemma}

\begin{proof}
(a) At least $j$ of the $Z_i$ have absolute value at least
$\abs{\mathrm{med}}$, which gives the pointwise bound. For
$S=\{i_1<\cdots<i_j\}$, conditioning successively on
$\mathcal G_{i_j-1},\ldots,\mathcal G_{i_1-1}$ gives
$\E\prod_{i\in S}Z_i^2\le v^j$; the integrands are nonnegative. Summation
over $S$ and Markov's inequality give the moment and tail bounds.

(b) The first bound is a union bound, and the second follows from it and
$\E[E_0^a\mathbf1\{E_0>u\}]=u^a\PP(E_0>u)+\int_u^\infty at^{a-1}\PP(E_0>t)\,dt$.

(c) For $t\in I_i$ we have $\mathcal G_{i-1}\subseteq\mathcal E_t$, so the
tower property gives $\E[\xi_t\mid\mathcal G_{i-1}]=0$ and
$\E[\xi_t^2\mid\mathcal G_{i-1}]\le\sigma^2$. For $t<t'$ in $I_i$, $\xi_t$ is
$\mathcal E_{t'}$-measurable, so $\E[\xi_t\xi_{t'}\mid\mathcal G_{i-1}]=0$.
Expanding the square proves the claim.
\end{proof}

\begin{alemma}[Maximum-volume design]
\label{lem:nh-maxvol}
Let $\phi=(\phi_1,\dots,\phi_{K})$ be a basis of the polynomials of degree
at most $q$ on $\R^d$, $K=K_q$, let $\mathcal D\subset\R^d$ be a closed
ball of positive radius, and let $J\ge K$. Let $x_1,\dots,x_J\in\mathcal D$
maximize $\det\mathbf G$, where $\mathbf G=\sum_{i=1}^J\phi(x_i)\phi(x_i)^\top$.
Then $\mathbf G\succ0$ and
$\sup_{x\in\mathcal D}\phi(x)^\top\mathbf G^{-1}\phi(x)\le K/(J-K+1)$.
\end{alemma}

\begin{proof}
The maximum exists by compactness, and it is positive because the vectors
$\phi(x)$, $x\in\mathcal D$, span $\R^K$: a polynomial that vanishes on a
ball is zero. Fix $x\in\mathcal D$ and append the row $\phi(x)^\top$ to the
$J\times K$ matrix with rows $\phi(x_i)^\top$. By maximality, replacing one
of the first $J$ rows by the last one does not increase $\det\mathbf G$, so
the first $J$ rows form a dominant submatrix, and
\citet[Theorem~4.2]{MikhalevOseledets2018} gives $c\in\R^J$ with
$\phi(x)=\sum_ic_i\phi(x_i)$ and $\norm c^2\le K/(J-K+1)$. The smallest
$\norm c^2$ among such $c$ equals $\phi(x)^\top\mathbf G^{-1}\phi(x)$.
\end{proof}

\begin{alemma}[Polynomials on balls]
\label{lem:nh-poly}
Let $p$ be a polynomial of degree at most $q$ on $\R^d$ and $\rho>0$.
\begin{enumerate}[label=(\alph*)]
\item $\norm{\nabla^2p(x)}_{\op}\le q^4\rho^{-2}\sup_{B(2\rho)}\abs p$ for
$\norm x\le\rho$.
\item $\sup_{B(\rho)}\abs p\le C^{\rm pol}_q(\rho/a)^q\sup_{B(a)}\abs p$ for
$0<a\le\rho$, where $C^{\rm pol}_q=\sum_{j=0}^qq^{2j}/j!$.
\end{enumerate}
\end{alemma}

\begin{proof}
By Markov's inequality \citep[Chapter~5]{BorweinErdelyi1995}, a univariate
polynomial $\varphi$ of degree at most $q$ satisfies
$\sup_{[-t,t]}\abs{\varphi'}\le(q^2/t)\sup_{[-t,t]}\abs\varphi$ for $t>0$.
(a) For a unit vector $u$, apply it twice to $\varphi(t)=p(x+tu)$ on
$[-\rho,\rho]$; the points $x+tu$ lie in $B(2\rho)$, and
$\varphi''(0)=u^\top\nabla^2p(x)u$. (b) For a unit vector $v$, applied $j$
times to $\varphi(t)=p(tv)$ on $[-a,a]$ it gives
$\abs{\varphi^{(j)}(0)}\le(q^2/a)^j\sup_{B(a)}\abs p$. Insert this into the
Taylor expansion of $\varphi$ at $0$ and use $(\rho/a)^j\le(\rho/a)^q$.
\end{proof}

For data $y\in\R^J$ at the nodes of the design of \cref{lem:nh-maxvol} let
$\Pi y=\mathbf c^\top\phi$, with $\mathbf c=\mathbf G^{-1}\sum_iy_i\phi(x_i)$,
be the least-squares fit by a polynomial of degree at most $q$. It is linear
in $y$ and reproduces polynomials of degree at most $q$.

\begin{alemma}[Stability of the fit]
\label{lem:nh-fit}
Let the design of \cref{lem:nh-maxvol} be taken on $\mathcal D=B(s)$, $s>0$,
with $J=2K_q$, and put $\Gamma_0=\sqrt{2K_q}$. If
$\max_i\abs{y_i}\le\bar y$, then $\sup_{\mathcal D}\abs{\Pi y}\le\Gamma_0\bar y$.
If $s=2\bar R$, then also $\norm{\nabla^2\Pi y(x)}_{\op}\le\Gamma_2\bar y$ for
$\norm x\le\bar R$, where $\Gamma_2=\Gamma_0q^4/\bar R^2$.
\end{alemma}

\begin{proof}
The fitted values $(\Pi y(x_i))_i$ are the orthogonal projection of $y$ onto
a subspace of $\R^J$, so
$\mathbf c^\top\mathbf G\mathbf c=\sum_i\Pi y(x_i)^2\le J\bar y^2$. By the
Cauchy--Schwarz inequality and \cref{lem:nh-maxvol},
$\abs{\Pi y(x)}^2\le(\mathbf c^\top\mathbf G\mathbf c)\,
\phi(x)^\top\mathbf G^{-1}\phi(x)\le2K_q\bar y^2$ on $\mathcal D$. The second
claim follows from \cref{lem:nh-poly}(a) with $\rho=\bar R$.
\end{proof}

\subsection{Pilot fit, guard and corrected residual method}
\label{app:nh-pilot}

\paragraph{The pilot.}
Fix a radius $s>0$, the design $x_1,\dots,x_J$ of \cref{lem:nh-fit} on
$B(s)$ ($J=2K_q$ nodes), an odd integer $r\ge3$ and an integer $\ell\ge1$.
In one round, query every node $r\ell$ times, node after node, and split the
$r\ell$ values at $x_i$ into $r$ consecutive blocks of $\ell$ values. Let
$\tilde f_i$ be the median of the $r$ block means at $x_i$, let
$P=\Pi\tilde f$ be the fit of these medians, and put
$\nu_i=\tilde f_i-f(x_i)$ and $E_0=\max_i\abs{\nu_i}$. The pilot uses
$n_0=2K_qr\ell$ oracle calls in one round, at points fixed in advance. Outside
\cref{app:batched-no-hessian} we take $s=2\bar R$. Let
$\varrho_j=L_\beta(2\bar R)^{\beta-j}/(q-j)!$ for $j=0,2$, and for a tolerance
$\theta\ge1$ put
\begin{equation}
\label{eq:nh-UL}
\Upsilon_\theta=\varrho_0+\Gamma_0(\varrho_0+\theta),\qquad
\Lambda_\theta=\varrho_2+\Gamma_2(\varrho_0+\theta).
\end{equation}
Since $\bar R\ge1$ and $\theta\ge1$, both numbers are at most
$C\sqrt{K_q}\,\theta$.

\begin{alemma}[Accuracy of the pilot]
\label{lem:nh-pilot}
Let $f\in C^q$ with $[f]_{\beta,\op}\le L_\beta$, and let $\mathfrak t=T_qf(0;\cdot)$
be its Taylor polynomial of degree $q$ at $0$ (\cref{app:taylor}). On every
realization, $\sup_{B(s)}\abs{P-\mathfrak t}\le\Gamma_0(L_\beta s^\beta/q!+E_0)$.
If $s=2\bar R$, then on every realization
\begin{equation}
\label{eq:nh-fit-error}
\sup_{B(R)}\abs{f-P}\le\varrho_0+\Gamma_0(\varrho_0+E_0),\qquad
\sup_{x\in B(R)}\norm{\nabla^2(f-P)(x)}_{\op}\le\varrho_2+\Gamma_2(\varrho_0+E_0);
\end{equation}
in particular, on the event $\{E_0\le\theta\}$ the residual $g=f-P$
satisfies $\abs g\le\Upsilon_\theta$ and
$\norm{\nabla^2g}_{\op}\le\Lambda_\theta$ on $B(R)$. If the noise is
centered at level $\sigma$, then for $u>0$
\begin{equation}
\label{eq:nh-pilot-tails}
\PP(E_0>u)\le J2^r\Bigl(\frac{\sigma^2}{\ell u^2}\Bigr)^{(r+1)/2},\qquad
\E\bigl[E_0\mathbf 1\{E_0>u\}\bigr]\le2J2^r\Bigl(\frac{\sigma^2}{\ell}\Bigr)^{(r+1)/2}u^{-r}.
\end{equation}
\end{alemma}

\begin{proof}
The fit is linear and reproduces $\mathfrak t$, so
$P-\mathfrak t=\Pi\bigl((f-\mathfrak t)(x_i)\bigr)_i+\Pi\nu$. By
\cref{lem:taylor}(a), $\abs{f-\mathfrak t}\le L_\beta\norm x^\beta/q!$, which is at
most $L_\beta s^\beta/q!$ at the nodes, and \cref{lem:nh-fit} gives the first claim.
For $s=2\bar R$ write $f-P=(f-\mathfrak t)-(P-\mathfrak t)$. On $B(R)$,
$\abs{f-\mathfrak t}\le\varrho_0$ and, by \cref{lem:taylor}(c) with $j=2$ and
$\nabla^2\mathfrak t(x)=T_{q-2}D^2f(0;x)$,
$\norm{\nabla^2(f-\mathfrak t)}_{\op}\le\varrho_2$; \cref{lem:nh-fit} bounds
$P-\mathfrak t$ and its Hessian. This gives \eqref{eq:nh-fit-error}.

For the tails, $\nu_i$ is the median of the $r$ block means of the errors at
$x_i$. By \cref{lem:nh-median}(c) with weights $1/\ell$, these means satisfy
the assumption of \cref{lem:nh-median}(a) with $v=\sigma^2/\ell$.
\Cref{lem:nh-median}(b) with $j=(r+1)/2$ and $a=1$, where
$2j/(2j-1)\le2$, gives \eqref{eq:nh-pilot-tails}.
\end{proof}

\begin{alemma}[Polynomial-value guard]
\label{lem:nh-guard}
Let $P$ be continuous, let $x_P$ minimize $P$ over $B(R)$, let
$z\in B(R)$ and $\Upsilon>0$. Output $\hat x=z$ if $P(z)\le P(x_P)+2\Upsilon$,
and $\hat x=x_P$ otherwise. If $\sup_{B(R)}\abs{f-P}\le\Upsilon$, then
$f(\hat x)\le f(z)$. On every realization,
$f(\hat x)-\fstar\le2\Upsilon+2\sup_{B(R)}\abs{f-P}$.
\end{alemma}

\begin{proof}
If $\sup_{B(R)}\abs{f-P}\le\Upsilon$ and $z$ is rejected, then
$f(z)-f(x_P)\ge P(z)-P(x_P)-2\Upsilon>0$. In both cases $\hat x\in B(R)$
and $P(\hat x)\le P(x_P)+2\Upsilon\le P(\xstar)+2\Upsilon$, because
$\xstar\in B(R)$. Comparing $f$ with $P$ at $\hat x$ and at $\xstar$ gives
the second claim.
\end{proof}

\paragraph{A corrected residual method.}
Let $P$ be a polynomial of degree at most $q$ that is
$\mathcal F_{k_1}$-measurable, where $k_1$ is the last round of the pilot,
let $g=f-P$, and fix a radius $h>0$. At a center $x\in B(R)$ the
\emph{residual oracle} makes one round of \cref{alg:stencil} with radius $h$
and returns $\hat v(x)=\gbar_h(x)-\nabla P(x)$. By \eqref{eq:nh-estimate},
$\hat v(x)=\nabla g(x)+b(x)+\zeta$: the residual gradient is estimated with
the error of $\gbar_h$, and no coefficient of $P$ enters the bounds. With
step sizes $\mu_k>0$ and a starting point $z_1\in B(R)$ that is measurable
with respect to the rounds before the first iteration, iteration $k$ uses two
rounds:
\begin{equation}
\label{eq:nh-fbf}
\begin{aligned}
y_k&\in\argmin_{x\in B(R)}\Bigl\{P(x)+\inner{\hat v(z_k)}{x}
+\frac{\norm{x-z_k}^2}{2\mu_k}\Bigr\},\\
z_{k+1}&=\Proj_{B(R)}\bigl(y_k-\mu_k(\hat v(y_k)-\hat v(z_k))\bigr).
\end{aligned}
\end{equation}
The first line treats the known polynomial exactly; the second corrects the
step with a fresh estimate at $y_k$, as in the forward--backward--forward
method of \citet{Tseng2000FBF}. By \cref{lem:nh-selection}, $y_k$ is a Borel
function of $P$, $\hat v(z_k)$, $z_k$ and $\mu_k$. Hence $z_k$ is measurable
with respect to the rounds before iteration $k$, $y_k$ with respect to these
and the first round of iteration $k$, and both lie in $B(R)$. If
$\norm{\nabla^2g}_{\op}\le\Lambda$ on $B(R)$ and $\mu_k\Lambda\le1/2$, the
objective of the first line has Hessian
$\nabla^2f-\nabla^2g+\mu_k^{-1}I\succeq(\alpha+\Lambda)I$ on $B(R)$, so $y_k$
is unique although $P$ need not be convex.

\begin{alemma}[One corrected step]
\label{lem:nh-fbf-step}
Let $f\in C^q$ with $[f]_\beta\le L_\beta$, $\nabla^2f\succeq\alpha I$ for some
$\alpha\ge0$, and a minimizer $\xstar\in B(R)$. Let $G\in\mathcal F_{k_1}$ be
an event on which $\norm{\nabla^2g(x)}_{\op}\le\Lambda$ for all $x\in B(R)$,
where $\Lambda>0$ is a number. Let $\Delta_k=\E[\norm{z_k-\xstar}^2\mathbf 1_G]$,
$F_k=\E[(f(y_k)-\fstar)\mathbf 1_G]$ and $V_h=b_h^2+\varsigma_h^2$. If
$\mu_k\Lambda\le1/2$, then
\[
\Delta_{k+1}\le\Delta_k-2\mu_kF_k+4\mu_kRb_h+8\mu_k^2V_h ,
\]
and if moreover $\alpha>0$ and $\alpha\mu_k\le1$, then
\[
\Delta_{k+1}\le\Bigl(1-\frac{\alpha\mu_k}4\Bigr)\Delta_k-2\mu_kF_k
+\frac{2\mu_kb_h^2}{\alpha}+8\mu_k^2V_h .
\]
\end{alemma}

\begin{proof}
Omit the index $k$ and write $u=\xstar$, $w=z-y$, $c=\mu(\hat v(y)-\hat v(z))$,
$e_z=\hat v(z)-\nabla g(z)$ and $e_y=\hat v(y)-\nabla g(y)$. Steps~1 and~2
hold on every realization in $G$.

\emph{Step 1: the first line of \eqref{eq:nh-fbf}.} The optimality condition
at $y$ over the convex set $B(R)\ni u$ gives
$\inner{w}{y-u}\ge\mu\inner{\nabla P(y)+\hat v(z)}{y-u}$. Add
$\inner c{y-u}$ to both sides. Since $\nabla P+\nabla g=\nabla f$ and
$\nabla^2f\succeq\alpha I$,
\[
\inner{w+c}{y-u}\ge\mu\inner{\nabla f(y)+e_y}{y-u}
\ge\mu\Bigl(f(y)-f(u)+\frac\alpha2\norm{y-u}^2+\inner{e_y}{y-u}\Bigr);
\]
the error $e_z$ has canceled.

\emph{Step 2: the second line.} Since $y-c-u=(z-u)-(w+c)$, we have
$\norm{y-c-u}^2=\norm{z-u}^2-2\inner{w+c}{y-u}-\norm w^2+\norm c^2$. The
projection onto $B(R)\ni u$ is nonexpansive, so Step~1 gives
\[
\norm{z_{k+1}-u}^2\le\norm{z-u}^2-2\mu\bigl(f(y)-f(u)\bigr)
-\alpha\mu\norm{y-u}^2-2\mu\inner{e_y}{y-u}-\norm w^2+\norm c^2 .
\]
On $G$, $\nabla g$ is $\Lambda$-Lipschitz on $B(R)$, and $\mu\Lambda\le1/2$
gives $\norm c^2\le2\mu^2\Lambda^2\norm w^2+4\mu^2(\norm{e_y}^2+\norm{e_z}^2)
\le\frac12\norm w^2+4\mu^2(\norm{e_y}^2+\norm{e_z}^2)$.

\emph{Step 3: expectations.} Multiply by $\mathbf 1_G$ and take
expectations. The centers $z$ and $y$ are bounded and measurable with respect
to the past of their rounds, and $G$ belongs to that past. By
\eqref{eq:nh-estimate} and \eqref{eq:nh-estimate-moments},
$e_y=b(y)+\zeta_y$ with $\E[\inner{\zeta_y}{y-u}\mathbf 1_G]=0$, and
$\E[(\norm{e_y}^2+\norm{e_z}^2)\mathbf 1_G]\le2V_h$. For the first claim,
bound $-2\mu\inner{b(y)}{y-u}\le4\mu Rb_h$ and drop the remaining nonpositive
terms. For the second, bound
$-2\mu\inner{b(y)}{y-u}\le\frac{\alpha\mu}2\norm{y-u}^2+\frac{2\mu b_h^2}{\alpha}$;
this leaves the term
$-\E[(\frac{\alpha\mu}2\norm{y-u}^2+\frac12\norm w^2)\mathbf 1_G]$, and
$\norm{z-u}^2\le2\norm{y-u}^2+2\norm w^2$ with $\alpha\mu\le1$ gives
$\frac{\alpha\mu}2\norm{y-u}^2+\frac12\norm w^2\ge\frac{\alpha\mu}4\norm{z-u}^2$.
\end{proof}

Both the restarted method below and the metric method of
\cref{app:nh-quadratic} lead to the recursion \eqref{eq:recursion}, which is
solved by \cref{lem:recursion}.

\begin{aproposition}[Restarted corrected method]
\label{prop:nh-restart}
Let $f\in\Fsc(\beta,L_\beta;\alpha,\infty,R)$, let $G$ and $\Lambda$ be as in
\cref{lem:nh-fbf-step}, and put $k_0=\max\{7,16\Lambda/\alpha\}$. Run $N_1\ge0$
iterations of \eqref{eq:nh-fbf} from $z_1=0$ with the step
$8/(\alpha(k_0+1))$, and then $N_2\ge1$ iterations with the steps
$8/(\alpha(k+k_0))$, $k=1,\dots,N_2$. Output
$\bar y=\mathcal W^{-1}\sum_{k=1}^{N_2}w_ky_{N_1+k}$, where $w_k=k+k_0-1$ and
$\mathcal W=\sum_kw_k$. The method uses $2(N_1+N_2)$ rounds of $2md$ oracle calls
each, and if $N_2\ge k_0$, then
\[
\E\bigl[(f(\bar y)-\fstar)\mathbf 1_G\bigr]\le
\frac{\alpha k_0^2R^2}{8N_2^2}\,e^{-2N_1/(k_0+1)}
+\frac{2b_h^2}{\alpha}+\frac{96(b_h^2+\varsigma_h^2)}{\alpha N_2}.
\]
\end{aproposition}

\begin{proof}
All steps are at most $8/(\alpha(k_0+1))$, hence $\mu_k\Lambda\le1/2$ and
$\alpha\mu_k\le1$. \Cref{lem:nh-fbf-step} gives \eqref{eq:recursion} with
$a=\alpha/4$, $c=2$, $\mathcal B=2b_h^2/\alpha$ and
$\mathcal V=8(b_h^2+\varsigma_h^2)$, and the steps are those of
\cref{lem:recursion}(b), with $a\mu=2/(k_0+1)\le\min\{1,2/k_0\}$. Since $\Delta_1\le R^2$ and, by convexity,
$f(\bar y)-\fstar\le\mathcal W^{-1}\sum_kw_k(f(y_{N_1+k})-\fstar)$,
\cref{lem:recursion}(b) gives the claim.
\end{proof}

\subsection{Budget of order \texorpdfstring{$d^q$}{d to the q} for \texorpdfstring{$\beta>4$}{beta > 4}}
\label{app:nh-general}

\begin{algorithm}[t]
\caption{Pilot fit and corrected residual method ($\beta>4$)}
\label{alg:nh-general}
\begin{algorithmic}[1]
\Require budget $T$; $d,\beta,\alpha,R,L_\beta,\sigma$; the stencil of \cref{alg:stencil}
\State $q\gets\lceil\beta\rceil-1$, $\omega\gets(q-2)/(4q)$, $\theta\gets T^{\omega}$
\State choose an odd $r\ge3$ with $\omega(r-1)\ge p_\beta+3$, and $\ell\gets\max\{1,\lceil16\sigma^2\rceil\}$
\State run the pilot of \cref{app:nh-pilot} on $B(2\bar R)$ (one round, $n_0=2K_qr\ell$ oracle calls)
\State $P\gets$ fit of the medians of block means
\State $\Lambda\gets\Lambda_\theta$, $\Upsilon\gets\Upsilon_\theta$ from \eqref{eq:nh-UL}, $k_0\gets\max\{7,16\Lambda/\alpha\}$
\State $N_{\mathrm{it}}\gets\lfloor(T-n_0)/(4md)\rfloor$, $N_1\gets\lfloor N_{\mathrm{it}}/2\rfloor$, $N_2\gets N_{\mathrm{it}}-N_1$
\State $h\gets\bigl(d^{1+\tau}\sigma^2/(L_\beta^2N_2)\bigr)^{1/(2\beta)}$
\State run \cref{prop:nh-restart} with the residual oracle $\hat v(x)=\gbar_h(x)-\nabla P(x)$; obtain $\bar y$
\State \Return the output of the guard of \cref{lem:nh-guard} with candidate $\bar y$ and level $\Upsilon$
\end{algorithmic}
\end{algorithm}

\begin{atheorem}[{Polynomial-pilot upper bound, $\beta>4$}]
\label{thm:nh-upper-general}
Let $\beta>4$, $\alpha,R,L_\beta,\sigma>0$, and let the noise be centered
at level $\sigma$. There is $C_0$, depending only on
$\beta,\alpha,R,L_\beta,\sigma$, such that for every $d\ge1$ and every
$T\ge C_0d^q$, \cref{alg:nh-general} uses at most $T$ oracle calls in
$1+2N_{\mathrm{it}}$ rounds and satisfies
\[
\E[f(\hat x)-\fstar]\le C_\beta\,\Phi_T\qquad\text{for every }
f\in\Fsc(\beta,L_\beta;\alpha,\infty,R),
\]
for the operator seminorm with $\tau=0$ and for the Hilbert--Schmidt
seminorm with $\tau=\tau_\beta$.
\end{atheorem}

\begin{proof}
Write $p=p_\beta$. Here $q\ge4$, so $\omega=(q-2)/(4q)\in(0,1/4)$ and
$\theta=T^\omega\ge1$. The integer $r$ depends only on $\beta$, $\ell$ only
on $\sigma$, and $\sigma^2/\ell\le1/16$. Let $G=\{E_0\le\theta\}$; it is
determined by the pilot round.

\emph{Step 1: oracle calls.} Each factor of $K_q=\prod_{i=1}^q(d+i)/i$ is at most
$d+1$, so $K_q\le2^qd^q$ and $n_0\le2^{q+1}r\ell d^q$. Let
$C_0\ge\max\{2^{q+2}r\ell,32m\}$. Then $n_0\le T/2$ and $T/(8md)\ge4$, so
$N_{\mathrm{it}}\ge T/(16md)\ge2$, $N_2\ge T/(32md)$ and
$N_1\ge N_{\mathrm{it}}/4$. The method uses at most
$n_0+4mdN_{\mathrm{it}}\le T$ oracle calls.

\emph{Step 2: the radius.} Put $X=L_\beta^2d^{-\tau}h^{2\beta-2}$. The choice of
$h$ gives $X=d\sigma^2/(N_2h^2)=L_\beta^{2/\beta}d^{-\tau/\beta}(d\sigma^2/N_2)^p$,
and $N_2\ge T/(32md)$ gives $X\le(32m)^p\alpha\Phi_T$, because
$2p-\tau/\beta=2-(2+\tau)/\beta$. By \eqref{eq:nh-bh}, $b_h^2=\bfC^2X$ and
$\varsigma_h^2=(\Cf+C_\sigma N_2)X$. Hence
\begin{equation}
\label{eq:nh-rate-terms}
\frac{2b_h^2}{\alpha}+\frac{96(b_h^2+\varsigma_h^2)}{\alpha N_2}
\le C_1(32m)^p\,\Phi_T,\qquad
C_1=98\bfC^2+96\Cf+96C_\sigma .
\end{equation}

\emph{Step 3: the event $G$.} By \cref{lem:nh-pilot}, on $G$ the residual
$g=f-P$ satisfies $\norm{\nabla^2g}_{\op}\le\Lambda_\theta$ and
$\abs g\le\Upsilon_\theta$ on $B(R)$. Since
$\Lambda_\theta\le C\sqrt{K_q}\,\theta$, we have
$k_0+1\le Cd^{q/2}T^\omega$, and $T\ge C_0d^q$ gives $d\le T^{1/q}$. Hence
\[
\frac{N_{\mathrm{it}}}{k_0+1}\ge c\,\frac{T^{1-\omega}}{d^{1+q/2}}
\ge c\,T^{\frac12-\frac1q-\omega}=c\,T^{\omega},
\]
because $\frac12-\frac1q=2\omega$. After increasing $C_0$, $N_2\ge k_0$ and
$2N_1/(k_0+1)\ge cT^\omega$. On $G$ the guard does not increase the value
(\cref{lem:nh-guard}), so \cref{prop:nh-restart} with $\Lambda=\Lambda_\theta$
and \eqref{eq:nh-rate-terms} give
\[
\E\bigl[(f(\hat x)-\fstar)\mathbf 1_G\bigr]\le
\frac{\alpha R^2}8e^{-cT^\omega}+C_1(32m)^p\,\Phi_T .
\]

\emph{Step 4: the complement.} On $G^c$ we have $E_0>\theta\ge1$, so
\cref{lem:nh-guard} and \eqref{eq:nh-fit-error} give
$f(\hat x)-\fstar\le2\Upsilon_\theta+2\varrho_0+2\Gamma_0(\varrho_0+E_0)
\le C\sqrt{K_q}\,E_0$. By \eqref{eq:nh-pilot-tails} with
$\sigma^2/\ell\le1/16$, $\E[E_0\mathbf 1_{G^c}]\le J2^{-r}\theta^{-r}
\le2K_q\theta^{-r}$. Since $K_q\le2^qd^q\le2^qT$ and $\omega r\ge p+3$,
\[
\E\bigl[(f(\hat x)-\fstar)\mathbf 1_{G^c}\bigr]\le CK_q^{3/2}\theta^{-r}
\le CT^{3/2-\omega r}\le CT^{-p-1}.
\]

\emph{Step 5: conclusion.} The exponent $2-(2+\tau)/\beta$ of $d$ in
$\Phi_T$ is positive, because $\tau\le(\beta-1)/2$; hence
$\Phi_T\ge\alpha^{-1}L_\beta^{2/\beta}\sigma^{2p}T^{-p}$. After increasing $C_0$,
the remainders of Steps~3 and~4 are at most $\Phi_T$ in total, and
$\E[f(\hat x)-\fstar]\le(C_1(32m)^p+1)\Phi_T$, where the constant depends only
on $\beta$. The pilot uses only $[f]_{\beta,\op}\le L_\beta$, which holds in both
cases, and $\tau$ enters only through \eqref{eq:nh-bh}.
\end{proof}

\subsection{Budget of order \texorpdfstring{$d^2$}{d squared} for \texorpdfstring{$2<\beta\le4$}{2 < beta <= 4}}
\label{app:nh-quadratic}

For $q\in\{2,3\}$ the fit of \cref{app:nh-general} needs of order $d^q$
values, and its curvature bound $\Lambda_\theta$ grows like $\sqrt{K_q}$,
which is too large when $T$ is of order $d^2$. We replace it by an estimate
of $H_0=\nabla^2f(0)$ whose error in operator norm grows more slowly than
$d$. For $2<\beta\le3$ the quadratic model built from this estimate is
accurate on $B(R)$, and the method of \cref{app:nh-general} applies. For
$3<\beta\le4$ the cubic part of $f$ is not estimated. Instead, convexity
bounds the variation of $\nabla^2f$ relative to $H_0$
(\cref{lem:nh-relative-curvature}), and projected steps in a metric
calibrated by the estimate of $H_0$ replace the corrected method.

\begin{alemma}[Robust estimate of a symmetric matrix]
\label{lem:nh-robust-matrix}
Let $M_\star$ be a deterministic symmetric $d\times d$ matrix, $\sigma_0>0$,
and $v\ge2$ an integer. Suppose that observations $Y_1,Y_2,\dots$ are made in
an order fixed in advance, each referring to a pair $(i,j)$ with $i\le j$,
and that $Y_n=(M_\star)_{i_nj_n}+\tilde\xi_n$, where $(\tilde\xi_n)$ is adapted
to a filtration $(\mathcal H_n)$ with $\E[\tilde\xi_n\mid\mathcal H_{n-1}]=0$
and $\E[\tilde\xi_n^2\mid\mathcal H_{n-1}]\le\sigma_0^2$. From $2v$
observations of every pair one can compute a symmetric matrix $\hat M$ with
\[
\bigl(\E\norm{\hat M-M_\star}_{\op}^v\bigr)^{1/v}\le C_v\,\sigma_0\sqrt{d\log(2d)},
\]
where $C_v$ depends only on $v$.
\end{alemma}

\begin{proof}
Let $\ell_d=\log(2d)$ and $b=\sigma_0\sqrt{d/\ell_d}$. For every pair, let
$c_{ij}$ be the median of the first $2v-1$ observations, let $Y'_{ij}$ be the
last one, and put
\[
\hat M_{ij}=\hat M_{ji}=c_{ij}+\mathrm{clip}_{[-b,b]}(Y'_{ij}-c_{ij}).
\]
We write $\hat M-M_\star$ as a bias matrix plus a matrix martingale.

\emph{Step 1: one entry.} By \cref{lem:nh-median}(a) with $r=2v-1$,
$\omega_{ij}=c_{ij}-(M_\star)_{ij}$ satisfies
$\E\omega_{ij}^{2v}\le\binom{2v-1}{v}\sigma_0^{2v}$. Let $\mathcal H'_{ij}$
be the $\sigma$-algebra before $Y'_{ij}$ and $\tilde\xi$ the error of
$Y'_{ij}$. Then
$\hat M_{ij}-(M_\star)_{ij}=\omega_{ij}+\mathrm{clip}_{[-b,b]}(\tilde\xi-\omega_{ij})
=\varpi_{ij}+\chi_{ij}$, where $\varpi_{ij}$ is the conditional mean given
$\mathcal H'_{ij}$. Clipping is $1$-Lipschitz, and a $1$-Lipschitz function
of a real random variable has at most its variance; applied to the
conditional law of $\tilde\xi$, this gives
$\E[\chi_{ij}^2\mid\mathcal H'_{ij}]\le\sigma_0^2$. Moreover
$\abs{\chi_{ij}}\le2b$. Since $\E[\tilde\xi\mid\mathcal H'_{ij}]=0$,
\[
\varpi_{ij}=\E\bigl[\mathrm{clip}_{[-b,b]}(\tilde\xi-\omega_{ij})
-(\tilde\xi-\omega_{ij})\bigm|\mathcal H'_{ij}\bigr].
\]
If $\abs{\omega_{ij}}\le b/2$, then
$\abs{\mathrm{clip}_{[-b,b]}(x)-x}=(\abs x-b)_+$ and
$(\abs{\tilde\xi-\omega_{ij}}-b)_+\le(\abs{\tilde\xi}-b/2)_+\le2\tilde\xi^2/b$,
so $\abs{\varpi_{ij}}\le2\sigma_0^2/b$. Otherwise
$\abs{\varpi_{ij}}\le\abs{\omega_{ij}}+b\le3\abs{\omega_{ij}}$.

\emph{Step 2: the bias matrix.} The entries $\varpi_{ij}$ with
$\abs{\omega_{ij}}\le b/2$ form a matrix of operator norm at most
$2d\sigma_0^2/b$. The others form a matrix of operator norm at most
$F=(\sum_{i,j}9\omega_{ij}^2\mathbf 1\{\abs{\omega_{ij}}>b/2\})^{1/2}$, the
sum running over ordered pairs. Since
$\abs\omega^v\mathbf 1\{\abs\omega>b/2\}\le(2/b)^v\abs\omega^{2v}$,
Minkowski's inequality in $L^{v/2}$ gives
$\norm F_{L^v}^2\le36\binom{2v-1}{v}^{2/v}d^2\sigma_0^4/b^2$. Hence the bias
matrix $\varpi$ satisfies $\norm{\norm{\varpi}_{\op}}_{L^v}\le C_vd\sigma_0^2/b$.

\emph{Step 3: the martingale part.} Let $E_{ii}=e_ie_i^\top$ and
$E_{ij}=e_ie_j^\top+e_je_i^\top$ for $i<j$, and
$\mathbf Y=\sum_{i\le j}\chi_{ij}E_{ij}$, summed in the order of the
observations $Y'_{ij}$. This is a martingale of symmetric matrices whose
increments have operator norm at most $2b$, and its predictable quadratic
variation is at most $\sigma_0^2\sum_{i\le j}E_{ij}^2=d\sigma_0^2I$. The
matrix Freedman inequality \citep[Theorem~1.2]{Tropp2011MatrixFreedman},
applied to $\mathbf Y$ and to $-\mathbf Y$, gives
\[
\PP(\norm{\mathbf Y}_{\op}\ge u)\le2d\exp\Bigl(-\frac{u^2}{2d\sigma_0^2+4bu/3}\Bigr)
\le2de^{-u^2/(4d\sigma_0^2)}+2de^{-3u/(8b)}.
\]
With $u_0=2\sigma_0\sqrt{d\ell_d}+3b\ell_d$, the right-hand side at
$u=u_0+s$, $s\ge0$, is at most $e^{-s^2/(4d\sigma_0^2)}+e^{-3s/(8b)}$, and
integrating the tail gives
$\norm{\norm{\mathbf Y}_{\op}}_{L^v}\le C_v(u_0+\sigma_0\sqrt d+b)$.

\emph{Step 4: conclusion.} We have $\hat M-M_\star=\varpi+\mathbf Y$,
$d\sigma_0^2/b=b\ell_d=\sigma_0\sqrt{d\ell_d}$ and
$\sigma_0\sqrt d+b\le C\sigma_0\sqrt{d\ell_d}$.
\end{proof}

\begin{alemma}[Hessian pilot]
\label{lem:nh-hessian-pilot}
Let $2<\beta\le4$, $f\in C^q$ with $[f]_{\beta,\op}\le L_\beta$, $H_0=\nabla^2f(0)$,
let the noise be centered at level $\sigma>0$, and let $v\ge2$ be
an integer. For $\rho>0$ and $i<j$ consider the difference quotients
\[
\frac{Y(\rho e_i)-2Y(0)+Y(-\rho e_i)}{\rho^2},\qquad
\frac{Y(\rho(e_i+e_j))-Y(\rho(e_i-e_j))-Y(\rho(e_j-e_i))+Y(-\rho(e_i+e_j))}{4\rho^2},
\]
where $Y(x)$ is a fresh value at $x$ each time, and let $\hat H$ be the
matrix of \cref{lem:nh-robust-matrix} computed from them with
$\rho=(\sigma/L_\beta)^{1/\beta}(\log(2d)/d)^{1/(2\beta)}$. This uses at most
$8vd^2$ oracle calls in one round, and
\[
\bigl(\E\norm{\hat H-H_0}_{\op}^v\bigr)^{1/v}\le C_vd^{a'},\qquad
a'=\frac34+\frac1{2\beta}<1,
\]
where $C_v$ depends only on $v,\beta,L_\beta,\sigma$.
\end{alemma}

\begin{proof}
The query points do not depend on any observation, so they form one round
with an order fixed in advance; there are $2v\,d(d+1)/2$ quotients of at most
$4$ oracle calls each. The error of a quotient is a combination of at most four
errors of consecutive queries whose squared coefficients sum to at most
$6/\rho^4$. Hence \cref{lem:nh-median}(c) gives the assumptions of
\cref{lem:nh-robust-matrix} with $\sigma_0^2=6\sigma^2/\rho^4$.

The means of the quotients are
\[
(M_\star)_{ii}=\int_{-1}^1(1-\abs s)\,\partial_{ii}f(s\rho e_i)\,ds,\qquad
(M_\star)_{ij}=\frac14\int_{-1}^1\int_{-1}^1\partial_{ij}f(\rho(se_i+te_j))
\,ds\,dt .
\]
For $q=2$, $\norm{\nabla^2f(y)-H_0}_{\op}\le L_\beta\norm y^{\beta-2}$. For $q=3$,
\cref{lem:taylor}(c) gives
$\norm{\nabla^2f(y)-H_0-D^3f(0)[y]}_{\op}\le L_\beta\norm y^{\beta-2}$, and the
linear term averages to zero, because both weightings are symmetric under
$y\mapsto-y$. The points have norm at most $\sqrt2\rho$ and $\beta-2\le2$, so
every entry of $M_\star-H_0$ is at most $2L_\beta\rho^{\beta-2}$ in absolute value
and $\norm{M_\star-H_0}_{\op}\le2dL_\beta\rho^{\beta-2}$. With
\cref{lem:nh-robust-matrix},
\[
\bigl(\E\norm{\hat H-H_0}_{\op}^v\bigr)^{1/v}\le C_v\frac{\sqrt6\,\sigma}{\rho^2}
\sqrt{d\log(2d)}+2dL_\beta\rho^{\beta-2}.
\]
For the given $\rho$ both terms are constants times
$L_\beta^{2/\beta}\sigma^{1-2/\beta}d^{a}(\log(2d))^{1-a}$ with
$a=\frac12+\frac1\beta<1$, and $\log(2d)\le2\sqrt d$ gives
$d^a(\log(2d))^{1-a}\le2d^{(1+a)/2}=2d^{a'}$.
\end{proof}

The next lemma, used for $3<\beta\le4$, is a H\"older version of an
inequality of \citet[Lemme~I]{glaeser1963racine}; a related bound is
\citet[Lemma~3]{nesterov2021implementable}.

\begin{alemma}[Relative curvature]
\label{lem:nh-relative-curvature}
Let $3<\beta\le4$, so that $q=3$, and let $f\in C^3$ with
$\nabla^2f\succeq\alpha I$ and $[f]_{\beta,\op}\le L_\beta$. Put
$\varkappa=\beta-2\in(1,2]$, $\upsilon=1-1/\varkappa\in(0,1/2]$,
$H_0=\nabla^2f(0)$ and, for a unit vector $w$, $g_w(x)=w^\top\nabla^2f(x)w$.
\begin{enumerate}[label=(\alph*)]
\item $\norm{\nabla g_w(x)}\le2L_\beta^{1/\varkappa}g_w(x)^{\upsilon}$ for all $x$.
\item For $\varrho>0$ and $\norm x\le\varrho$,
$C_\varrho^{-1}H_0\preceq\nabla^2f(x)\preceq C_\varrho H_0$ with
$C_\varrho=\bigl[1+\frac{2\varrho}\varkappa(L_\beta/\alpha)^{1/\varkappa}\bigr]^\varkappa$.
\item For $\norm x\le R$,
$\frac12H_0-\bar cI\preceq\nabla^2f(x)\preceq\frac32H_0+\bar cI$ with
$\bar c=3L_\beta R^\varkappa$.
\end{enumerate}
\end{alemma}

\begin{proof}
(a) The gradient $\nabla g_w(x)=D^3f(x)[w,w,\cdot]$ is
$(\varkappa-1)$-H\"older with constant $L_\beta$, so for a unit vector $u$ and
$s\ge0$,
\[
g_w(x+su)\le g_w(x)+s\inner{\nabla g_w(x)}u+\frac {L_\beta}\varkappa s^\varkappa .
\]
If $\nabla g_w(x)\ne0$, take $u=-\nabla g_w(x)/\norm{\nabla g_w(x)}$ and
$s=(g_w(x)/L_\beta)^{1/\varkappa}$. Since $g_w>0$, this gives
$s\norm{\nabla g_w(x)}\le(1+1/\varkappa)g_w(x)\le2g_w(x)$, which is (a).

(b) By (a), $\norm{\nabla(g_w^{1/\varkappa})}=\varkappa^{-1}g_w^{-\upsilon}
\norm{\nabla g_w}\le\frac2\varkappa L_\beta^{1/\varkappa}$, so
$\abs{g_w(x)^{1/\varkappa}-g_w(0)^{1/\varkappa}}\le\frac{2\varrho}\varkappa
L_\beta^{1/\varkappa}$ for $\norm x\le\varrho$. Since $g_w\ge\alpha$, this gives
$g_w(x)\le C_\varrho g_w(0)$ and $g_w(0)\le C_\varrho g_w(x)$ for every unit
$w$.

(c) Let $\norm x\le R$ and $h_w=g_w(0)$. By the H\"older bound and (a),
\[
\abs{g_w(x)-h_w}\le\abs{\inner{\nabla g_w(0)}x}+\frac {L_\beta}\varkappa R^\varkappa
\le2L_\beta^{1/\varkappa}Rh_w^\upsilon+L_\beta R^\varkappa .
\]
With $z=h_w/(L_\beta R^\varkappa)$ we have
$2L_\beta^{1/\varkappa}Rh_w^\upsilon=L_\beta R^\varkappa\cdot2z^\upsilon$, and
$2z^\upsilon\le z/2+2$ for $z\ge0$, because $z^\upsilon\le\max\{1,\sqrt z\}$
and $2\sqrt z\le z/2+2$. Hence $\abs{g_w(x)-h_w}\le h_w/2+3L_\beta R^\varkappa$ for
every unit $w$, which is (c).
\end{proof}

\begin{alemma}[Spectral initialization]
\label{lem:nh-spectral}
Let $f$ be as in \cref{lem:nh-relative-curvature} with minimizer
$\xstar\in B(R)$. Let $A\succeq\alpha I$ be a symmetric matrix with
eigenvalues $a_1\ge\dots\ge a_d$ and orthonormal eigenvectors
$o_1,\dots,o_d$, let $E_A=\norm{A-H_0}_{\op}$, write $x_i=\inner x{o_i}$, and
let $\mathcal C_A=\{x:\abs{x_i}\le R\ \forall i\}$, which contains $\xstar$.
Put $R_d=R\sqrt d$ and $\Xi_d=1+2\sqrt2L_\beta^{1/\varkappa}R_d+(L_\beta/\varkappa)R_d^\varkappa$.
Starting from $x_0=0$, make $d$ updates
\[
x_{t+1,i}=\mathrm{clip}_{[-R,R]}\bigl(x_{t,i}-a_i^{-1}\hat g_i(x_t)\bigr),
\qquad i=1,\dots,d,
\]
where $\hat g(x_t)=\nabla f(x_t)+\zeta_t$ is any gradient estimate and
$\hat g_i=\inner{\hat g}{o_i}$. Let $u$ be the projection of $x_d$ onto $B(R)$
in the norm $\norm y_A=(y^\top Ay)^{1/2}$. Then, on every realization,
\[
\norm{u-\xstar}_A^2\le80R^2d^3\Xi_d^2+12R^2d+12R^2dE_A^2
+4d\sum_{t<d}\norm{\zeta_t}^2 .
\]
\end{alemma}

\begin{proof}
We bound the entries of $\nabla^2f-A$ on $\mathcal C_A$ in the rows of the
large eigenvalues, show by induction that these coordinates contract, and
add the coordinates with small $a_i$ at the end.

\emph{Step 1: entries of the Hessian.} For $i\ne j$, polarization with
$w_\pm=(o_i\pm o_j)/\sqrt2$ gives
\[
D^3f(0)[o_i,o_j,\cdot]=\tfrac12\bigl(\nabla g_{w_+}(0)-\nabla g_{w_-}(0)\bigr),
\qquad
g_{w_\pm}(0)\le w_\pm^\top Aw_\pm+E_A=\frac{a_i+a_j}2+E_A .
\]
With \cref{lem:nh-relative-curvature}(a) this gives, also for $i=j$,
$\norm{D^3f(0)[o_i,o_j,\cdot]}\le2L_\beta^{1/\varkappa}(\max\{a_i,a_j\}+E_A)^\upsilon$.
By the H\"older continuity of $D^3f$,
$\norm{\nabla^2f(y)-H_0-D^3f(0)[y]}_{\op}\le L_\beta\norm y^\varkappa/\varkappa$,
and every $y\in\mathcal C_A$ has $\norm y\le R_d$. Hence, in the basis
$(o_i)$,
\[
\abs{(\nabla^2f(y)-A)_{ij}}\le E_A+2L_\beta^{1/\varkappa}R_d(\max\{a_i,a_j\}+E_A)^\upsilon
+\frac {L_\beta}\varkappa R_d^\varkappa\qquad(y\in\mathcal C_A).
\]
Let $J_A=(1+E_A+2d\Xi_d)^2$ and let $k_A$ be the number of $i$ with
$a_i>J_A$. For $i\le k_A$ and any $j$, $a:=\max\{a_i,a_j\}>J_A\ge1$ gives
$E_A\le\sqrt a$, $a+E_A\le2a$ and, since $\upsilon\le\frac12$,
$(a+E_A)^\upsilon\le\sqrt{2a}$. Hence
\begin{equation}
\label{eq:nh-high-rows}
\abs{(\nabla^2f(y)-A)_{ij}}\le\Xi_d\sqrt{\max\{a_i,a_j\}}
\qquad(i\le k_A,\ y\in\mathcal C_A).
\end{equation}

\emph{Step 2: one update.} Let $e_t=x_t-\xstar$. Both $x_t$ and $\xstar$ lie
in the convex set $\mathcal C_A$, so $\nabla f(x_t)=Ae_t+\mathcal M_te_t$
with $\mathcal M_t=\int_0^1(\nabla^2f(\xstar+se_t)-A)\,ds$, whose entries
obey \eqref{eq:nh-high-rows} in the rows $i\le k_A$. Clipping fixes
$\xstar_i$ and is $1$-Lipschitz, so
$\abs{e_{t+1,i}}\le a_i^{-1}\abs{(\mathcal M_te_t)_i+\zeta_{t,i}}$. Hence,
with $\bar z_i^t=\sqrt{a_i}\abs{e_{t,i}}$, for $i\le k_A$
\[
\bar z_i^{t+1}\le\Xi_d\sum_{j\le k_A}\frac{\bar z_j^t}{\sqrt{\min\{a_i,a_j\}}}
+\Xi_d\sum_{j>k_A}\abs{e_{t,j}}+\norm{\zeta_t},
\]
where for $j>k_A$ we used $a_j\le a_i$, and $a_i>1$.

\emph{Step 3: induction.} Let $\bar B_0=0$ and
$\bar B_{t+1}=\bar B_t/2+2Rd\Xi_d+\norm{\zeta_t}$. We show by induction that
$\bar z_i^t\le\bar B_t$ for $i\le\min\{t,k_A\}$. Let $i\le\min\{t+1,k_A\}$.
In the bound of Step~2, the indices $j\le\min\{t,k_A\}$ satisfy
$\bar z_j^t\le\bar B_t$ and
$\Xi_d/\sqrt{\min\{a_i,a_j\}}\le\Xi_d/\sqrt{J_A}\le1/(2d)$, so they contribute
at most $\bar B_t/2$. The other indices satisfy $j\ge i$, hence
$a_j\le a_i$. Their terms are at most $\Xi_d\abs{e_{t,j}}\le2R\Xi_d$ (for
$j\le k_A$ use $\bar z_j^t/\sqrt{a_j}=\abs{e_{t,j}}$), and there are at most
$d$ of them.

\emph{Step 4: conclusion.} After $d$ updates,
$\bar B_d\le4Rd\Xi_d+\sum_{t<d}2^{-(d-1-t)}\norm{\zeta_t}$. By the
Cauchy--Schwarz inequality, the indices $i\le k_A$ contribute at most
$d\bar B_d^2\le32R^2d^3\Xi_d^2+4d\sum_t\norm{\zeta_t}^2$ to
$\norm{x_d-\xstar}_A^2$. The other indices have
$a_i\le J_A\le3(1+E_A^2+4d^2\Xi_d^2)$ and $\abs{e_{d,i}}\le2R$, so they
contribute at most $12R^2d+12R^2dE_A^2+48R^2d^3\Xi_d^2$. Finally, the
projection onto $B(R)\ni\xstar$ in the norm $\norm\cdot_A$ does not increase
the distance to $\xstar$.
\end{proof}

\begin{alemma}[Projected steps in a calibrated metric]
\label{lem:nh-metric}
Let $f\in\Fsc(\beta,L_\beta;\alpha,\infty,R)$ and let $k_2$ be a round. Let a
nonempty compact convex set $\mathcal T\subseteq B(R)$, a symmetric matrix
$\mathbf S$ and a point $x_1\in\mathcal T$ be $\mathcal F_{k_2}$-measurable,
with $\mathbf S\succeq\Lambda I$ for a number $\Lambda>0$. Let
$\mu_{\mathbf S}\in(0,1]$ be a number and $G\in\mathcal F_{k_2}$ an event on
which $\xstar\in\mathcal T$ and
$\mu_{\mathbf S}\mathbf S\preceq\nabla^2f(x)\preceq\frac32\mathbf S$ for all
$x\in B(R)$. Fix a radius $h>0$, let $k_0=\max\{7,24/\mu_{\mathbf S}\}$, and
after round $k_2$ run, with one round of \cref{alg:stencil} (one frame,
radius $h$) per iteration,
\[
x_{k+1}=\Proj^{\mathbf S}_{\mathcal T}\bigl(x_k-\mu_k\mathbf S^{-1}\gbar_h(x_k)\bigr),
\qquad k=1,\dots,N_1+N_2,
\]
where $\Proj^{\mathbf S}_{\mathcal T}$ is the projection onto $\mathcal T$ in
the norm $\norm y_{\mathbf S}=(y^\top\mathbf Sy)^{1/2}$,
$\mu_k=8/(\mu_{\mathbf S}(k_0+1))$ for $k\le N_1$ and
$\mu_{N_1+k}=8/(\mu_{\mathbf S}(k+k_0))$ for $1\le k\le N_2$. Output
$\hat x=\mathcal W^{-1}\sum_{k=1}^{N_2}w_kx_{N_1+k}$, where $w_k=k+k_0-1$ and
$\mathcal W=\sum_kw_k$. Then $\hat x\in\mathcal T$, and if $N_2\ge k_0$,
\[
\E\bigl[(f(\hat x)-\fstar)\mathbf 1_G\bigr]\le
\frac{\mu_{\mathbf S}k_0^2}{4N_2^2}e^{-2N_1/(k_0+1)}
\E\bigl[\norm{x_1-\xstar}_{\mathbf S}^2\mathbf 1_G\bigr]
+\frac{8b_h^2}{\mu_{\mathbf S}\Lambda}
+\frac{48(b_h^2+\varsigma_h^2)}{\mu_{\mathbf S}\Lambda N_2}.
\]
\end{alemma}

\begin{proof}
The centers lie in $\mathcal T$ and are measurable with respect to the past
of their rounds, so $\hat x\in\mathcal T$. Let
$\Delta_k=\E[\norm{x_k-\xstar}^2_{\mathbf S}\mathbf 1_G]$,
$F_k=\E[(f(x_k)-\fstar)\mathbf 1_G]$, $V_h=b_h^2+\varsigma_h^2$, and
$\norm y_{\mathbf S^{-1}}=(y^\top\mathbf S^{-1}y)^{1/2}$.

\emph{Step 1: two inequalities on $G$.} As in the proof of
\cref{lem:two-phase}, on $G$ and for $x\in B(R)$ we have
$\nabla f(x)=\bar H(x-\xstar)$ with
$\bar H=\int_0^1\nabla^2f(\xstar+s(x-\xstar))\,ds$, now with
$\mu_{\mathbf S}\mathbf S\preceq\bar H\preceq\frac32\mathbf S$. Hence
$\bar H^{1/2}\mathbf S^{-1}\bar H^{1/2}\preceq\frac32I$, and
\[
\norm{\nabla f(x)}_{\mathbf S^{-1}}^2\le\tfrac32\inner{\nabla f(x)}{x-\xstar},\qquad
\inner{\nabla f(x)}{x-\xstar}\ge f(x)-\fstar+\tfrac{\mu_{\mathbf S}}2
\norm{x-\xstar}_{\mathbf S}^2 .
\]

\emph{Step 2: one step.} Write $\gbar_h(x_k)=\nabla f(x_k)+b_k+\zeta_k$ as in
\eqref{eq:nh-estimate}. On $G$ the projection is nonexpansive in
$\norm\cdot_{\mathbf S}$ and $\xstar\in\mathcal T$, so
\[
\norm{x_{k+1}-\xstar}_{\mathbf S}^2\le\norm{x_k-\xstar}^2_{\mathbf S}
-2\mu_k\inner{\gbar_h(x_k)}{x_k-\xstar}+\mu_k^2\norm{\gbar_h(x_k)}^2_{\mathbf S^{-1}}.
\]
Multiply by $\mathbf 1_G$ and take expectations. We use Step~1,
\eqref{eq:nh-estimate-moments}, $\norm y_{\mathbf S^{-1}}^2\le\norm y^2/\Lambda$,
$\norm{\gbar_h}^2_{\mathbf S^{-1}}\le2\norm{\nabla f}^2_{\mathbf S^{-1}}
+2\norm{b_k+\zeta_k}^2_{\mathbf S^{-1}}$ and
\[
2\mu_k\norm{b_k}_{\mathbf S^{-1}}\norm{x_k-\xstar}_{\mathbf S}\le
\frac{\mu_k\mu_{\mathbf S}}4\norm{x_k-\xstar}_{\mathbf S}^2
+\frac{4\mu_k}{\mu_{\mathbf S}}\norm{b_k}^2_{\mathbf S^{-1}} .
\]
All steps satisfy $\mu_k\le8/(\mu_{\mathbf S}(k_0+1))\le1/3$, so the terms
with $\inner{\nabla f(x_k)}{x_k-\xstar}\ge0$ have total coefficient
$-2\mu_k+3\mu_k^2\le-\mu_k$. This gives \eqref{eq:recursion} with
$a=\mu_{\mathbf S}/4$, $c=1$, $\mathcal B=4b_h^2/(\mu_{\mathbf S}\Lambda)$ and
$\mathcal V=2V_h/\Lambda$, and the steps are those of
\cref{lem:recursion}(b), with $a\mu=2/(k_0+1)\le\min\{1,2/k_0\}$. That lemma and the convexity of $f$ on
$\mathcal T$ prove the claim.
\end{proof}

\paragraph{The method for $2<\beta\le4$.}
\Cref{alg:nh-quadratic} collects the two constructions. Let
$a'=\frac34+\frac1{2\beta}$ be as in \cref{lem:nh-hessian-pilot},
$\omega=(1-a')/4$ and $\theta=T^\omega$. With fresh values for every
quotient, \cref{lem:nh-hessian-pilot} with moment parameter $v$ uses
$n_H(v)=2v(3d+4\binom d2)=2vd(2d+1)$ oracle calls. The numbers $S_d\le Cd^{a'}$
and, for $\beta>3$, $F_d\le Cd^{\beta+1}$ are the explicit moment bounds of
Step~1 of the corresponding part of the proof of
\cref{thm:nh-upper-quadratic}; they depend only on $d$ and the known
parameters.

\begin{algorithm}[t]
\caption{Quadratic startup without a Hessian bound, \(2<\beta\le4\)}
\label{alg:nh-quadratic}
\begin{algorithmic}[1]
\Require \(d,\beta,\alpha,L_\beta,R,\sigma,T\), with the onset of
\cref{thm:nh-upper-quadratic}; the stencil of \cref{alg:stencil}
\If{\(2<\beta\le3\)}
  \State Choose an integer \(v\ge2\) with
  \(\omega(v-1)\ge a'/2+p_\beta+2\)
  \State In one round form \(\hat H\) by \cref{lem:nh-hessian-pilot}
  with parameter \(v\), the median \(\hat c\) of \(2v-1\) values at zero,
  and, for each \(i\), the median \(\hat g_i\) of \(2v-1\) fresh
  differences \((Y(e_i)-Y(-e_i))/2\)
  \State \(P(x)\gets\hat c+\inner{\hat g}{x}+x^\top\hat Hx/2\),
  \(n_0\gets n_H(v)+(2v-1)(2d+1)\),
  \(\Lambda\gets S_d\theta\), \(\Upsilon\gets S_d\theta\)
  \State \(N_{\rm it}\gets\lfloor(T-n_0)/(4md)\rfloor\),
  \(N_1\gets\lfloor N_{\rm it}/2\rfloor\),
  \(N_2\gets N_{\rm it}-N_1\),
  \(h\gets(d^{1+\tau}\sigma^2/(L_\beta^2N_2))^{1/(2\beta)}\)
  \State Run \cref{prop:nh-restart} from zero with this \(P,\Lambda,h,N_1,N_2\);
  \Return its averaged output after the guard of \cref{lem:nh-guard}
  at level \(\Upsilon\)
\Else
  \State Choose an integer \(v\ge2\) with
  \(\omega(v-1)\ge(\beta+1)/2+p_\beta+2\)
  \State In one round form \(\hat H\) by \cref{lem:nh-hessian-pilot}
  with parameter \(2v\); raise every eigenvalue below \(\alpha\) to
  \(\alpha\), obtaining \(A\) with ordered eigenbasis \((o_i)\)
  \State Run the \(d\) updates of \cref{lem:nh-spectral} from zero.
  In update \(t\), estimate coordinate \(i\) by the median of \(2v-1\)
  fresh directional stencils \eqref{eq:directional} at \(x_t\), with
  direction \(o_i\) and radius \(1\); query all coordinates in one round
  \State \(u\gets\Proj^A_{B(R)}x_d\),
  \(n_0\gets n_H(2v)+2m(2v-1)d^2\),
  \(\Lambda_0\gets S_d\theta\), \(K_{\rm tr}\gets F_d\theta\)
  \State \(\mathcal T\gets\{x\in B(R):\norm{x-u}_A^2\le K_{\rm tr}\}\),
  \(\Lambda\gets\alpha+3L_\beta R^{\beta-2}+\Lambda_0\),
  \(\mathbf S\gets A+\Lambda I\), \(\mu_{\mathbf S}\gets\alpha/(4\Lambda)\)
  \State \(N_{\rm it}\gets\lfloor(T-n_0)/(2md)\rfloor\),
  \(N_1\gets\lfloor N_{\rm it}/2\rfloor\),
  \(N_2\gets N_{\rm it}-N_1\),
  \(h\gets(d^{1+\tau}\sigma^2/(L_\beta^2N_2))^{1/(2\beta)}\)
  \State \Return the weighted average of \cref{lem:nh-metric}, run
  from \(x_1=u\) on \(\mathcal T\) with
  \(\mathbf S,\mu_{\mathbf S},h,N_1,N_2\)
\EndIf
\end{algorithmic}
\end{algorithm}

\begin{atheorem}[{Quadratic-pilot upper bound, $2<\beta\le4$}]
\label{thm:nh-upper-quadratic}
Let $2<\beta\le4$, $\alpha,R,L_\beta,\sigma>0$, and let the noise be centered
at level $\sigma$. There is $C_0$, depending only on
$\beta,\alpha,R,L_\beta,\sigma$, such that for every $d\ge1$ and every
$T\ge C_0d^2$, \cref{alg:nh-quadratic} uses at most $T$ oracle calls, in
$1+2N_{\mathrm{it}}$ rounds for $\beta\le3$ and in $1+d+N_{\mathrm{it}}$
rounds for $\beta>3$, and satisfies
\[
\E[f(\hat x)-\fstar]\le C_\beta\Phi_T\qquad\text{for every }
f\in\Fsc(\beta,L_\beta;\alpha,\infty,R),
\]
for the operator seminorm with $\tau=0$ and for the Hilbert--Schmidt
seminorm with $\tau=\tau_\beta$.
\end{atheorem}

The algorithm does not use the events $G$ of the proofs. For $\beta>3$ every output lies in the trust set
$\mathcal T$, and \eqref{eq:nh-trust-guard} below takes the place of the
polynomial-value guard.

\begin{proof}[Proof for $2<\beta\le3$]
Here $q=2$, $m=1$, $p=p_\beta$, and $v$ is the integer of
\cref{alg:nh-quadratic}, which depends only on $\beta$. We follow the proof of
\cref{thm:nh-upper-general}.

\emph{Step 1: the pilot.} It uses $n_0\le C_vd^2$ oracle calls in one round. By
\cref{lem:taylor}(a), $\abs{f(x)-T_2f(0;x)}\le L_\beta\norm x^\beta/2$, so the mean
of $(Y(e_i)-Y(-e_i))/2$ is within $L_\beta/2$ of $\partial_if(0)$, and its error
has conditional variance at most $\sigma^2/2$. By \cref{lem:nh-median} and
Minkowski's inequality, $\norm{\hat c-f(0)}_{L^{2v}}\le C\sigma$ and
$\norm{\norm{\hat g-\nabla f(0)}}_{L^v}\le\sqrt d(L_\beta/2+C\sigma)$. On $B(R)$,
\[
\abs{f-P}\le\frac{L_\beta R^\beta}2+\abs{\hat c-f(0)}+R\norm{\hat g-\nabla f(0)}
+\frac{R^2}2\norm{\hat H-H_0}_{\op},\qquad
\norm{\nabla^2(f-P)}_{\op}\le L_\beta R^{\beta-2}+\norm{\hat H-H_0}_{\op}.
\]
With \cref{lem:nh-hessian-pilot} and $\sqrt d\le d^{a'}$, Minkowski's
inequality gives an explicit number $S_d\le Cd^{a'}$ such that
\[
Z=S_d^{-1}\max\Bigl\{1,\sup_{B(R)}\abs{f-P},\sup_{B(R)}\norm{\nabla^2(f-P)}_{\op}\Bigr\}
\]
satisfies $\E Z^v\le1$.

\emph{Step 2: the event $G=\{Z\le\theta\}$.} It is determined by the pilot
round, and on $G$ the residual satisfies
$\norm{\nabla^2(f-P)}_{\op}\le S_d\theta=\Lambda$ and
$\abs{f-P}\le S_d\theta=\Upsilon$ on $B(R)$. As in Steps~1 and~2 of the
proof of \cref{thm:nh-upper-general}, after increasing $C_0$ the oracle calls are at
most $T$, $N_{\mathrm{it}}\ge T/(16md)$, $N_2\ge T/(32md)$, and
\eqref{eq:nh-rate-terms} holds. Now $k_0+1\le Cd^{a'}T^\omega$ and
$d\le\sqrt T$, so
\[
\frac{N_{\mathrm{it}}}{k_0+1}\ge c\,\frac{T^{1-\omega}}{d^{1+a'}}\ge
c\,T^{1-\omega-\frac{1+a'}2}=c\,T^{\omega},
\]
because $\frac{1-a'}2=2\omega$. As in Step~3 there,
$\E[(f(\hat x)-\fstar)\mathbf 1_G]\le\frac{\alpha R^2}8e^{-cT^\omega}
+C_1(32m)^p\Phi_T$.

\emph{Step 3: the complement.} On $G^c$ the guard gives
$f(\hat x)-\fstar\le2S_d\theta+2S_dZ\le4S_dZ$. Since
$Z\mathbf 1\{Z>\theta\}\le Z^v\theta^{1-v}$,
\[
\E\bigl[(f(\hat x)-\fstar)\mathbf 1_{G^c}\bigr]\le4S_d\theta^{1-v}\le
CT^{a'/2-\omega(v-1)}\le CT^{-p-2}.
\]
Both remainders are absorbed into $\Phi_T$ as in Step~5 of the proof of
\cref{thm:nh-upper-general}.
\end{proof}

\begin{proof}[Proof for $3<\beta\le4$]
Here $q=3$, $m=2$, $p=p_\beta$, $\varkappa,\upsilon,\bar c$ are as in
\cref{lem:nh-relative-curvature}, and $v$ is the integer of
\cref{alg:nh-quadratic}. A preliminary stage estimates $H_0$ and a point $u$
near $\xstar$; \cref{lem:nh-metric} then runs on a trust set around $u$, in a
metric calibrated by the estimate of $H_0$.

\emph{Step 1: the preliminary stage.} \Cref{lem:nh-hessian-pilot} with $2v$
in place of $v$ gives $\hat H$ in one round. By Weyl's inequality the
smallest eigenvalue of $\hat H$ is at least $\alpha-\norm{\hat H-H_0}_{\op}$,
so $\norm{A-\hat H}_{\op}\le\norm{\hat H-H_0}_{\op}$ and
$E_A\le2\norm{\hat H-H_0}_{\op}$; the eigenvectors are chosen by
\cref{lem:nh-selection}. In the updates of \cref{lem:nh-spectral}, each value
of the stencil \eqref{eq:directional} with $h=1$ in the direction $o_i$ has mean within
$c_0L_\beta$ of $\inner{\nabla f(x_t)}{o_i}$ (\cref{lem:exact}, with $c_0$ from
\eqref{eq:stencil-constants}), and its error has conditional second moment at
most $S_2\sigma^2/2$ (\cref{lem:nh-median}(c)). By \cref{lem:nh-median}(a)
and Minkowski's inequality, $\E[\norm{\zeta_t}^{2v}\mid\mathcal F]^{1/(2v)}\le
C\sqrt d$ given the past $\mathcal F$ of the round. The stage uses
$n_0\le16vd^2+2m(2v-1)d^2$ oracle calls in $1+d$ rounds. Since
$\Xi_d^2\le Cd^\varkappa$ and $d^3\Xi_d^2\le Cd^{\beta+1}$,
\cref{lem:nh-hessian-pilot,lem:nh-spectral} and Minkowski's inequality give
explicit numbers $S_d$ and $F_d$ with
\begin{equation}
\label{eq:nh-prelim-moments}
\norm{E_A}_{L^{2v}}\le S_d\le Cd^{a'},\qquad
\norm{W_A}_{L^v}\le F_d\le Cd^{\beta+1},\qquad W_A=\norm{u-\xstar}_A^2 .
\end{equation}

\emph{Step 2: the trust set.} Let $\Lambda_0=S_d\theta$,
$K_{\mathrm{tr}}=F_d\theta$ and
$\mathcal T=\{x\in B(R):\norm{x-u}_A^2\le K_{\mathrm{tr}}\}$, a nonempty
compact convex set. On the event
$G=\{E_A\le\Lambda_0,\ W_A\le K_{\mathrm{tr}}\}$ we have $\xstar\in\mathcal T$.
Let $C_R$ be the number $C_\varrho$ of \cref{lem:nh-relative-curvature}(b) at
$\varrho=R$, and $\norm y_{H_0}=(y^\top H_0y)^{1/2}$. For every
$x\in\mathcal T$ and on every realization,
\cref{lem:nh-relative-curvature}(b), $H_0\preceq A+E_AI$ and
$\norm{x-\xstar}_A^2\le2K_{\mathrm{tr}}+2W_A$ give
\begin{equation}
\label{eq:nh-trust-guard}
f(x)-\fstar\le\frac{C_R}2\norm{x-\xstar}_{H_0}^2\le
C_R\bigl(K_{\mathrm{tr}}+W_A+2R^2E_A\bigr).
\end{equation}
By \eqref{eq:nh-prelim-moments} and Markov's inequality,
$\PP(G^c)\le\theta^{-2v}+\theta^{-v}\le2\theta^{-v}$. H\"older's inequality
and \eqref{eq:nh-trust-guard} then give
$\E[(f(x)-\fstar)\mathbf 1_{G^c}]\le CF_d\theta^{1-v}$ for every
$\mathcal T$-valued random point $x$.

\emph{Step 3: the calibrated metric.} Let $\Lambda=\alpha+\bar c+\Lambda_0$,
$\mathbf S=A+\Lambda I$ and $\mu_{\mathbf S}=\alpha/(4\Lambda)\le\frac14$. On
$G$, \cref{lem:nh-relative-curvature}(c) and
$A-\Lambda_0I\preceq H_0\preceq A+\Lambda_0I$ give, for $x\in B(R)$,
$\nabla^2f(x)\preceq\frac32A+(\frac32\Lambda_0+\bar c)I\preceq\frac32\mathbf S$
and $\nabla^2f(x)\succeq\frac12A-(\frac12\Lambda_0+\bar c)I$. The convex
combination of the last bound and $\nabla^2f\succeq\alpha I$ with weights
$\lambda=\alpha/(2\Lambda)$ and $1-\lambda$ gives
\[
\nabla^2f(x)\succeq\frac\lambda2A+\bigl((1-\lambda)\alpha-\lambda(\tfrac12\Lambda_0
+\bar c)\bigr)I\succeq\mu_{\mathbf S}\mathbf S,
\]
because $\lambda/2=\mu_{\mathbf S}$ and
$(1-\lambda)\alpha-\lambda(\frac12\Lambda_0+\bar c)\ge\alpha-\lambda\Lambda
=\frac\alpha2\ge\mu_{\mathbf S}\Lambda$.

\emph{Step 4: optimization.} After increasing $C_0$, $n_0\le T/2$,
$N_{\mathrm{it}}\ge T/(8md)$ and $N_2\ge T/(16md)$, so the choice of $h$
gives, as in Step~2 of the proof of \cref{thm:nh-upper-general},
$b_h^2+\varsigma_h^2/N_2\le(\bfC^2+\Cf+C_\sigma)(16m)^p\alpha\Phi_T$. Since
$\mu_{\mathbf S}\Lambda=\alpha/4$, the last two terms of \cref{lem:nh-metric}
are at most $C_2\Phi_T$, where $C_2$ depends only on $\beta$. Here
$k_0=\max\{7,96\Lambda/\alpha\}\le Cd^{a'}T^\omega$ and $d\le\sqrt T$, so
$N_{\mathrm{it}}/(k_0+1)\ge cT^\omega$ as for $\beta\le3$, and $N_2\ge k_0$
after increasing $C_0$. On $G$,
$\norm{u-\xstar}_{\mathbf S}^2=W_A+\Lambda\norm{u-\xstar}^2\le F_d\theta+4\Lambda R^2$,
so the first term of \cref{lem:nh-metric} is a power of $T$ times
$e^{-cT^\omega}$.

\emph{Step 5: conclusion.} The output lies in $\mathcal T$, so by Step~2 and
$d\le\sqrt T$ the contribution of $G^c$ is at most
$CF_d\theta^{1-v}\le CT^{(\beta+1)/2-\omega(v-1)}\le CT^{-p-2}$. Both
remainders are absorbed into $\Phi_T$ after increasing $C_0$, as in the proof
of \cref{thm:nh-upper-general}. The preliminary stage uses only
$[f]_{\beta,\op}\le L_\beta$.
\end{proof}

\subsection{The minimax rate}
\label{app:nh-proof}

\begin{proof}[Proof of \cref{thm:nh-main}]
\emph{Upper bound.} For $2<\beta\le4$ it is \cref{thm:nh-upper-quadratic},
and for $\beta>4$ it is \cref{thm:nh-upper-general}. Both methods use at most
$C_\beta T/d$ rounds (\cref{rem:nh-rounds}).

\emph{Lower bound.} Apply \cref{thm:main}(b) with $L=3\alpha$ and $M=0$.
Its proof (\cref{app:lower-proofs}) gives $c_\beta>0$ and $C_0'$, depending
only on $\beta,\alpha,R,L_\beta,\sigma$, such that every algorithm with
$T\ge C_0'd^2$ oracle calls, in any number of rounds, has worst-case error at least
$c_\beta\Phi_T$ over the functions with $\alpha I\preceq\nabla^2f\preceq3\alpha I$,
$[f]_\beta\le L_\beta$ and minimizer in $B(R)$, under independent $N(0,\sigma^2)$
errors. These functions belong to $\Fsc(\beta,L_\beta;\alpha,\infty,R)$, and
these errors are centered noise. Taking $C_0\ge C_0'$ in
\cref{thm:nh-upper-general,thm:nh-upper-quadratic}, and since $d^q\ge d^2$,
both bounds hold under the stated budget condition.

\emph{Logarithmically many rounds.} This is \cref{thm:nh-batched}; the size
of its pilot is estimated after its proof.
\end{proof}

\subsection{Remarks}
\label{app:nh-remarks}

\begin{aremark}[Computation]
\label{rem:nh-computation}
The methods of this appendix use exact minimizers of known functions over
balls, and the oracle model does not charge this computation. The
maximum-volume design depends only on $d,q,R$ and is computed once; any
design with the conclusion of \cref{lem:nh-maxvol} can replace it. For
$\beta>4$ the guard needs a global minimizer of a polynomial of degree
$q\ge4$ over a ball, which is a nonconvex problem. On the event where the
analysis takes place, the objective of the first line of \eqref{eq:nh-fbf}
has Hessian at least $(\alpha+\Lambda)I$ on the ball; off that event an exact
global minimizer is still used. For $2<\beta\le3$ the polynomial $P$ is
quadratic, and both minimizations are trust-region subproblems. For
$3<\beta\le4$ the method uses eigendecompositions and projections onto $B(R)$
and onto the intersection of $B(R)$ with an ellipsoid, in the norms
$\norm\cdot_A$ and $\norm\cdot_{\mathbf S}$ of
\cref{lem:nh-spectral,lem:nh-metric}, and minimizes no polynomial of degree
three.
\end{aremark}

\begin{aremark}[Rounds]
\label{rem:nh-rounds}
The pilots for $\beta>4$ and for $2<\beta\le3$ are single rounds, because
their query points are fixed in advance; for $3<\beta\le4$ the preliminary
stage uses $1+d$ rounds. The method of \cref{prop:nh-restart} uses two rounds
of $2md$ oracle calls per iteration, and that of \cref{lem:nh-metric} one. Hence the
number of rounds is of order $T/(md)$: apart from the $2md$ oracle calls of one
stencil, the methods are sequential.
\end{aremark}

\begin{aremark}[Oblivious noise]
\label{rem:nh-oblivious}
The stencil tolerates oblivious noise (\cref{prop:unified-estimator}), but the
analysis of the one-round pilots does not. Condition \eqref{eq:noise-oblivious}
bounds the second moment of an oblivious error only given $\mathcal F_{k-1}$,
so the oblivious errors at all nodes of the pilot round may equal $\pm W$ for
one random variable $W$ with $\E W^2=\sigma^2$. The medians of block means
then have only the tail of $W$, \cref{lem:nh-median}(a) does not apply, and
the event $\{E_0>\theta\}$ is no longer negligible in the proof of
\cref{thm:nh-upper-general}. A bias of order $\sigma$ alone would be
harmless, since the proofs use only $E_0\le\theta$. \Cref{lem:nh-robust-matrix}
uses the centering both in its bias bound and in its martingale part. We
therefore treat centered noise only.
\end{aremark}

\section{{Logarithmic-round polynomial correction}}
\label{app:batched-no-hessian}

The methods of \cref{app:no-hessian} have a small polynomial startup, but
their optimization steps are sequential. Here a more accurate pilot, whose
residual has curvature bounded independently of $d$, allows epochs with
doubling numbers of frames: the statistical rate is the same, the number of
rounds is logarithmic, and the startup is larger. All startup calls are
counted in the budget. We also treat convex objectives without a Hessian
bound (\cref{app:convex-no-hessian}).

The notation is that of \cref{app:nh-setting}, $K=K_q$, and $L_\beta,\sigma,R>0$.
In \cref{app:batch-pilot,app:batch-rounds,app:convex-no-hessian} the noise is
centered at level $\sigma$ (\cref{def:nh-noise}).

\subsection{A pilot with dimension-free residual curvature}
\label{app:batch-pilot}

Let $\Gamma_0=\sqrt{2K}$ and $C_q^{\rm pol}$ be as in
\cref{lem:nh-fit,lem:nh-poly}, and define
\begin{align}
 H_q&=C_q^{\rm pol}(2\bar R)^q,&
 a&=\min\left\{1,\left(\frac{q!}{2L_\beta\Gamma_0H_q}\right)^{1/\gamma}\right\},&
 \theta&=\frac{a^q}{2\Gamma_0H_q},&
 \ell&=\max\{1,\lceil16\sigma^2/\theta^2\rceil\}.
 \label{eq:batch-pilot-design}
\end{align}
The pilot is that of \cref{app:nh-pilot} with design radius $s=a$, this
$\ell$ and an odd $r\ge3$. Let $P$ be its fit, $E_0$ the largest absolute
error of its medians, $G=\{E_0\le\theta\}$, and, with $\varrho_0,\varrho_2$
as in \cref{app:nh-pilot},
\begin{equation}
 U=\varrho_0+1,\qquad
 \Lambda=\varrho_2+\frac{q^4}{\bar R^2}.
 \label{eq:batch-pilot-constants}
\end{equation}
Unlike $\Upsilon_\theta$ and $\Lambda_\theta$ in \cref{app:nh-pilot}, $U$
and $\Lambda$ do not depend on $d$.

\begin{alemma}[A rescaled polynomial pilot]
\label{lem:batch-pilot}
For every $f\in C^q(\R^d)$ with $[f]_{\beta,\op}\le L_\beta$, this pilot
uses $Q_{\rm pil}=2Kr\ell$ calls in one round. On $G$,
\[
 \sup_{B(R)}|f-P|\le U,\qquad
 \sup_{B(R)}\|\nabla^2(f-P)\|_{\op}\le\Lambda.
\]
If $f$ also has a global minimizer in $B(R)$, then for every candidate
in $B(R)$ the output $\hat x$ of the guard of \cref{lem:nh-guard}, with
level $U$, satisfies $\E[(f(\hat x)-\fstar)\mathbf 1_{G^c}]\le4UK2^{-r}$.
\end{alemma}

\begin{proof}
Let $\mathfrak t=T_qf(0;\cdot)$. By \cref{lem:nh-pilot} with $s=a$,
$\sup_{B(a)}|P-\mathfrak t|\le\Gamma_0(L_\beta a^\beta/q!+E_0)$. Since
$a\le1\le2\bar R$, \cref{lem:nh-poly}(b) with $\rho=2\bar R$ multiplies this
bound by $H_qa^{-q}$, and \eqref{eq:batch-pilot-design} gives
$H_qa^{-q}\Gamma_0L_\beta a^\beta/q!=H_q\Gamma_0L_\beta a^\gamma/q!\le\frac12$ and
$H_qa^{-q}\Gamma_0=1/(2\theta)$. Hence, on every realization,
\begin{equation}
 \sup_{B(2\bar R)}|P-\mathfrak t|
 \le\frac12+\frac{E_0}{2\theta}.
 \label{eq:batch-pilot-pathwise}
\end{equation}
On $G$ this is at most one. On $B(R)$, \cref{lem:taylor}(a),(c) give
$|f-\mathfrak t|\le\varrho_0$ and $\|\nabla^2(f-\mathfrak t)\|_{\op}\le\varrho_2$,
and \cref{lem:nh-poly}(a) with $\rho=\bar R$ gives
$\|\nabla^2(P-\mathfrak t)\|_{\op}\le q^4/\bar R^2$ on $G$. This proves the
bounds on $G$.

By \eqref{eq:nh-pilot-tails} with $2K$ nodes and $\sigma^2/\ell\le\theta^2/16$,
$\E[E_0\mathbf1_{G^c}]\le K\theta2^{-r}$. By \cref{lem:nh-guard} and
\eqref{eq:batch-pilot-pathwise}, $f(\hat x)-\fstar\le4\varrho_0+3+E_0/\theta$
on every realization, which is at most $4UE_0/\theta$ on $G^c$.
\end{proof}

\subsection{A statistical rate in logarithmically many rounds}
\label{app:batch-rounds}

\Cref{alg:batched-polynomial-corrector} runs a pilot and then $J$ epochs of
the corrected method \eqref{eq:nh-fbf}. Epoch $j$ makes $n_j$ iterations
with $B_j$ frames per stencil, radius $h_j$, steps $\mu_{j,k}$ and averaging
weights $\omega_{j,k}$, and starts from the output of the previous epoch. The
algorithm uses $Q_{\rm pil}+4md\sum_{j=1}^Jn_jB_j$ calls in
$1+2\sum_{j=1}^Jn_j$ rounds: in each iteration the second center depends on
the first stencil, while the $B_j$ frames of one stencil are queried
together. Polynomial minimization and the guard need no calls.

\begin{algorithm}[t]
\caption{Polynomial pilot and batched residual correction}
\label{alg:batched-polynomial-corrector}
\begin{algorithmic}[1]
\Require $T,d,\beta,L_\beta,R,\sigma$ (also $\alpha$ in the strongly convex case);
the pilot, guard level $U$, starting point and schedule
$(n_j,B_j,h_j,\mu_{j,k},\omega_{j,k})$ of \cref{thm:nh-batched} or \cref{thm:convex-nh-direct}
\State In one round acquire all $Q_{\rm pil}$ pilot values and fit $P$
\State Select $x_P\in\argmin_{B(R)}P$; set $z^{(0)}=x_P$ (\cref{thm:nh-batched}) or $z^{(0)}=0$ (\cref{thm:convex-nh-direct})
\For{$j=1,\ldots,J$}
  \State $z_1\gets z^{(j-1)}$
  \For{$k=1,\ldots,n_j$}
    \State $\hat v_z\gets\gbar_{h_j,B_j}(z_k)-\nabla P(z_k)$ from a fresh stencil block
    \State $y_k\gets\argmin_{x\in B(R)}
       \{P(x)+\inner{\hat v_z}{x}+\norm{x-z_k}^2/(2\mu_{j,k})\}$
    \State $\hat v_y\gets\gbar_{h_j,B_j}(y_k)-\nabla P(y_k)$ from a fresh stencil block
    \State $z_{k+1}\gets\Proj_{B(R)}(y_k-\mu_{j,k}(\hat v_y-\hat v_z))$
  \EndFor
  \State $z^{(j)}\gets\sum_{k=1}^{n_j}\omega_{j,k}y_k/\sum_{k=1}^{n_j}\omega_{j,k}$
\EndFor
\State \Return $z^{(J)}$ if $P(z^{(J)})\le P(x_P)+2U$, and $x_P$ otherwise
\end{algorithmic}
\end{algorithm}

\begin{atheorem}[Logarithmically many rounds]
\label{thm:nh-batched}
Let $\beta>2$, $\alpha,L_\beta,\sigma,R>0$, and let
$f\in\Fsc(\beta,L_\beta;\alpha,\infty,R)$ with centered noise at level
$\sigma$. Run \cref{alg:batched-polynomial-corrector} with the pilot of
\cref{app:batch-pilot}, where $r$ is the smallest odd integer $r\ge3$ with
$4UK2^{-r}\le\Phi_T$, with $Q_{\rm pil}=2Kr\ell$, the guard level $U$, the
starting point $z^{(0)}=x_P$, and, for $k_0=\max\{7,16\Lambda/\alpha\}$ and
$n=\lceil2k_0\rceil$,
\[
 n_j=n,\qquad B_j=2^{j-1},\qquad
 h_j=\left(\frac{d^{1+\tau}\sigma^2}{L_\beta^2nB_j}\right)^{1/(2\beta)},\qquad
 \mu_{j,k}=\frac8{\alpha(k+k_0)},\qquad \omega_{j,k}=k+k_0-1,
\]
for the epochs $j=1,\dots,J$, where
\begin{equation}
 J=\left\lfloor\log_2\left(1+
                   \frac{T-Q_{\rm pil}}{4mdn}\right)\right\rfloor.
 \label{eq:batch-epochs}
\end{equation}
If
\begin{equation}
 Q_{\rm pil}\le T/2,\qquad T-Q_{\rm pil}\ge4mdn,
 \qquad 2U16^{-J}\le\Phi_T,
 \label{eq:batch-onset}
\end{equation}
then the algorithm uses at most $T$ calls in exactly $1+2nJ$ rounds, which
is at most a constant multiple of $1+\log T$, with the constant depending
only on $\beta,\alpha,L_\beta,R$, and its output satisfies
\[
 \E[f(\hat x)-\fstar]\le C_\beta\Phi_T.
\]
\end{atheorem}

\begin{proof}
The epochs use $4mdn(2^J-1)\le T-Q_{\rm pil}$ calls, and $n$ depends only on
$\beta,\alpha,L_\beta,R$. Write $p=p_\beta$ and
$A_j=\E[(f(z^{(j)})-\fstar)\mathbf1_G]$. On $G$, $|f-P|\le U$ on $B(R)$
(\cref{lem:batch-pilot}), so $A_0\le2U$.

Consider epoch $j$. Its starting point $z^{(j-1)}\in B(R)$ is measurable
with respect to the earlier rounds, and all its steps are at most
$8/(\alpha(k_0+1))$. Hence \cref{lem:nh-fbf-step}, with $B_j$ frames, gives
\eqref{eq:recursion} with $a=\alpha/4$, $c=2$,
$\mathcal B=2b_{h_j}^2/\alpha$ and $\mathcal V=8V_j$, where
$V_j=b_{h_j}^2+\varsigma_{h_j}^2/B_j$. \Cref{lem:recursion}(a) with
$N=n$ and the convexity of $f$ give
\begin{equation}
 A_j\le\frac{\alpha k_0^2}{8n^2}
       \E[\|z^{(j-1)}-\xstar\|^2\mathbf1_G]
       +\frac{b_{h_j}^2}{\alpha}+\frac{64V_j}{\alpha n}.
 \label{eq:batch-epoch-risk}
\end{equation}
Strong convexity and $n\ge2k_0$ bound the first term by $A_{j-1}/16$. For
$Q_j=4mdnB_j$, the radius $h_j$ gives
\[
 L_\beta^2d^{-\tau}h_j^{2\beta-2}
 =\frac{d\sigma^2}{nB_jh_j^2}
 =(4m)^{p}\alpha\Phi_{Q_j},
\]
so $A_j\le A_{j-1}/16+C_\beta\Phi_{Q_j}$. Since
$\Phi_{Q_{j-1}}=2^{p}\Phi_{Q_j}$ and $2^p/16\le1/8$, induction gives
$A_J\le2U16^{-J}+C_\beta\Phi_{Q_J}$ with a new $C_\beta$. Maximality of $J$
gives $Q_J>(T-Q_{\rm pil})/4\ge T/8$, so \eqref{eq:batch-onset} yields
$A_J\le C_\beta\Phi_T$. On $G$ the guard does not increase the value
(\cref{lem:nh-guard}), and on $G^c$ \cref{lem:batch-pilot} and the choice of
$r$ bound its contribution by $\Phi_T$.
\end{proof}

\paragraph{Size of the pilot.}
For fixed $\beta,\alpha,L_\beta,\sigma,R$ and sufficiently large $K$, the
quantities $a$, $\theta$ and $\ell$ in \eqref{eq:batch-pilot-design} are
bounded above and below by positive multiples of $K^{-1/(2\gamma)}$,
$K^{-(q/\gamma+1)/2}$ and $K^{q/\gamma+1}$, respectively; these multiples
depend only on the fixed parameters. Since $\Phi_T\ge cT^{-p_\beta}$, there is
also such a constant
$C$ for which $r\le C(1+\log(2+d+T))$. Consequently,
\[
 Q_{\rm pil}\le C K^{\,2+q/\gamma}\bigl(1+\log(2+d+T)\bigr).
\]
For integer $\beta$, the right-hand side is at most another such constant
times $d^{\beta^2-1}(1+\log(2+d+T))$, and the onset
conditions \eqref{eq:batch-onset} hold once $T$ exceeds a polynomial in $d$,
up to logarithmic factors. This startup is much larger than those of
\cref{thm:nh-upper-general,thm:nh-upper-quadratic}. The leading constant
$C_\beta$ depends only on $\beta$; the startup and the number of rounds also
depend on $\alpha,L_\beta,\sigma,R$.

\subsection{{A polynomial-pilot method for convex objectives}}
\label{app:convex-no-hessian}

Let $f\in C^q(\R^d)$ be convex with $[f]_\beta\le L_\beta$ and a global minimizer
$\xstar\in B(R)$; strong convexity is not assumed. We use the pilot of
\cref{app:nh-pilot} on $B(2\bar R)$ with tolerance $1$,
$\ell=\max\{1,\lceil16\sigma^2\rceil\}$ and an odd $r\ge3$, so that
$Q_{\rm pil}=2Kr\ell$. Let $U_1=\Upsilon_1$ and $\Lambda_1$ be the numbers
\eqref{eq:nh-UL} with $\theta=1$, and
\[
 A_{\rm pil}=2U_1+2\varrho_0+2\Gamma_0\varrho_0+2\Gamma_0 .
\]

\begin{atheorem}[{Polynomial-pilot upper bound for convex objectives}]
\label{thm:convex-nh-direct}
Under the assumptions just stated, let $N\ge1$ and
\[
 h=\left(\frac{d^{1+\tau}\sigma^2}{L_\beta^2N}\right)^{1/(2\beta)},\quad
 V_h=b_h^2+\varsigma_h^2,\quad
 \mu=\min\left\{\frac1{2\Lambda_1},
                        \frac{R}{\sqrt{8NV_h}}\right\},
\]
with $b_h,\varsigma_h$ from \eqref{eq:nh-bh}. Run
\cref{alg:batched-polynomial-corrector} with this pilot, the guard level
$U_1$, the starting point $z^{(0)}=0$ and one epoch ($J=1$) with $n_1=N$,
$B_1=1$, $h_1=h$, $\mu_{1,k}=\mu$ and $\omega_{1,k}=1$. The output satisfies
\begin{align}
 \E[f(\hat x)-\fstar]
 &\le C_\beta\left\{\frac{\Lambda_1R^2}{N}
 +RL_\beta^{1/\beta}\sigma^{(\beta-1)/\beta}
 d^{(\beta-1-\tau)/(2\beta)}N^{-(\beta-1)/(2\beta)}\right\}
 +A_{\rm pil}K2^{-r}.
 \label{eq:convex-nh-direct}
\end{align}
The algorithm uses $Q_{\rm pil}+4mdN$ calls in $1+2N$ rounds.
\end{atheorem}

\begin{proof}
Let $G=\{E_0\le1\}$, $\Delta_k=\E[\|z_k-\xstar\|^2\mathbf1_G]$ and
$F_k=\E[(f(y_k)-\fstar)\mathbf1_G]$. On $G$ the residual Hessian is at most
$\Lambda_1$ on $B(R)$ (\cref{lem:nh-pilot}) and $\mu\Lambda_1\le1/2$, so
\cref{lem:nh-fbf-step} with $\alpha=0$ gives
$\Delta_{k+1}\le\Delta_k-2\mu F_k+4\mu Rb_h+8\mu^2V_h$. Summing over
$k\le N$, and using $\Delta_1\le R^2$ and convexity, we obtain for the
average $\bar y=z^{(1)}$ of the $y_k$
\[
 \E[(f(\bar y)-\fstar)\mathbf1_G]
 \le\frac{R^2}{2\mu N}+2Rb_h+4\mu V_h
 \le\frac{\Lambda_1R^2}{N}
       +2\sqrt2R\sqrt{V_h/N}+2Rb_h .
\]
The radius gives $X:=L_\beta^2d^{-\tau}h^{2\beta-2}=d\sigma^2/(Nh^2)$, hence
$b_h^2=\bfC^2X$, $V_h/N\le(\bfC^2+\Cf+C_\sigma)X$ and
$\sqrt X=L_\beta^{1/\beta}\sigma^{(\beta-1)/\beta}d^{(\beta-1-\tau)/(2\beta)}
N^{-(\beta-1)/(2\beta)}$. On $G$ the guard does not increase the objective.
On $G^c$, \cref{lem:nh-guard} and
\eqref{eq:nh-fit-error} bound its gap by $A_{\rm pil}E_0$, and
\eqref{eq:nh-pilot-tails} gives $\E[E_0\mathbf1_{G^c}]\le K2^{-r}$.
\end{proof}

With the whole budget, let
\[
 S_T=RL_\beta^{1/\beta}\sigma^{(\beta-1)/\beta}
 d^{\,1-(2+\tau)/(2\beta)}T^{-(\beta-1)/(2\beta)},
\]
let $r$ be the smallest odd integer $r\ge3$ with $A_{\rm pil}K2^{-r}\le S_T$,
and $N=\lfloor(T-Q_{\rm pil})/(4md)\rfloor$. If $Q_{\rm pil}\le T/2$,
$T\ge16md$ and $\Lambda_1R^2d/T\le S_T$, then $T/4\le4mdN\le T-Q_{\rm pil}$,
and \eqref{eq:convex-nh-direct} gives $\E[f(\hat x)-\fstar]\le C_\beta S_T$
with at most $T$ calls.

\paragraph{Ellipsoidal norms.}
Let $\Sigma$ be a known invertible matrix and $\|x\|_\Sigma=\|\Sigma x\|$.
\Cref{thm:nh-batched,thm:convex-nh-direct} apply to
$\widetilde f(z)=f(\Sigma^{-1}z)$, with each query $z$ made at
$\Sigma^{-1}z$. Function values and their errors are unchanged, so these
results hold, with the same calls, rounds and noise levels, when strong
convexity, the localization of $\xstar$, the H\"older seminorms
($[f]_{\beta,\op,\Sigma}=[\widetilde f]_{\beta,\op}$ and
$[f]_{\beta,\HS,\Sigma}=[\widetilde f]_{\beta,\HS}$) and the Lipschitz
constant of $\eta$ are
measured in $\|\cdot\|_\Sigma$ and its dual norm. No factor depending on the
condition number of $\Sigma$ appears.

\section{Convex objectives: upper bounds}
\label{app:convex-upper}

This appendix proves the upper bounds of \cref{sec:convex} for minimization
over a compact convex set $Q$, where the minimizer may lie on the boundary
with a nonzero gradient; the class $\Fcvx(\beta,L_\beta;L,R)$ is the case $Q=B(R)$
(\cref{def:cvx-Q}). The gradient-scale bound, \cref{cor:direct-convex},
rests on projected methods with a bounded bias
(\cref{app:convex-first-order,app:convex-gradient-bias}). The value-scale
bound, \cref{thm:value-analysis}, rests on descent on the ball average of $f$
(\cref{app:convex-value}). \Cref{app:convex-proof-main} simplifies their
constants. \Cref{app:convex-matching} proves the upper bound of
\cref{cor:natural-curvature}, whose lower bound is proved in
\cref{app:convex-lower}, and \cref{app:convex-lipschitz} gives the
achievable level of Lipschitz noise.

\paragraph{Setting.}
For a nonempty closed convex set $K\subset\R^d$, $\Proj_K$ is the Euclidean
projection onto $K$, and for $r\ge0$
\[
K_r=K+rB(1)=\{x\in\R^d:\ \operatorname{dist}(x,K)\le r\} ;
\]
this set is closed and convex, and $K_r+sB(1)=K_{r+s}$ for $r,s\ge0$. The random
vector $W$ is uniform on $B(1)$ and $U$ is uniform on $\Sph$. For a function
$g$ defined near $x+hB(1)$ we write $g_h(x)=\E g(x+hW)$. With
$\gamma=\beta-q$ as in \cref{sec:setup}, we put
$c_\gamma=1/((1+\gamma)(2+\gamma))\le1/2$. Rounds and noise obey
\cref{def:oracle}.

\begin{adefinition}[Constrained problems]
\label{def:cvx-Q}
Let $Q\subset\R^d$ be nonempty, closed and convex with diameter
$D=\sup_{x,y\in Q}\norm{x-y}\in(0,\infty)$; thus $Q$ is compact. For a closed
convex set $K\subset\R^d$, let $\Fcvx^K(\beta,L_\beta;L)$ be the set of real
functions $f$ that are defined and $q$ times continuously differentiable on an
open set containing $K$ and satisfy
\[
0\preceq\nabla^2f(x)\preceq LI\quad(x\in K),\qquad
\norm{D^qf(x)-D^qf(y)}\le L_\beta\norm{x-y}^\gamma\quad(x,y\in K),
\]
where the norm of $D^qf(x)-D^qf(y)$ is $\norm\cdot_{\op}$ ($\tau=0$) or
$\norm\cdot_{\HS}$ ($\tau=\tau_\beta$). A \emph{constrained problem} consists
of $Q$, a starting point $x_0\in Q$, a number $R\in(0,D]$ and an objective $f$
with a minimizer $u\in\argmin_Qf$ such that $\norm{x_0-u}\le R$. Each result
below states the smoothness of $f$ near $Q$ that it uses. A method knows $Q$
through $\Proj_Q$, and it knows $x_0$, $R$, $D$ and the parameters of the
class. All its queries lie in $Q_h$, where $h$ is the radius of its stencil,
and it outputs a point $\hat x\in Q$; its error is $\E[f(\hat x)-\min_Qf]$.
\end{adefinition}

By integration along segments, the Hessian inequalities imply that $f$ is
convex with an $L$-Lipschitz gradient on every convex set on which they hold.
If the
projection onto $Q\cap B(x_0,R)$ is available, $Q$ may be replaced by this
intersection and $D$ by its diameter, which is at most $\min\{D,2R\}$.

\paragraph{Known parameters.}
The methods of this appendix know $\Proj_Q$, $x_0$, $\beta$, $L$, $R$, $D$,
$\sigma$, $d$, the target accuracy $\eps$ and the constants of
\cref{prop:unified-estimator} (explicit for $m=1$). They know $L_\beta$, except for
the method of \cref{thm:value-analysis}(b). They do not use $M$,
$\delta$, $f$ or $\eta$; the numbers $\bar M$ below are design values.
We assume $L>0$ and $L_\beta>0$: an affine objective satisfies the assumptions with
every $L>0$, and if $L_\beta=0$, every statement holds with any $L_\beta'>0$ in place of
$L_\beta$. The case $\sigma=0$ is included; the terms containing $\sigma$ then
vanish.

\paragraph{The estimate.}
The next lemma is \cref{prop:unified-estimator} for an objective that is smooth
only near the set that contains the centers.

\begin{alemma}[The estimate near a set]
\label{lem:cvx-local-estimate}
Let $K\subset\R^d$ be nonempty, closed and convex, let $h>0$, and put
$\vartheta=2$ if the H\"older bound is measured in the Hilbert--Schmidt norm
and $q$ is even, and $\vartheta=1$ otherwise. Let $f$ be $q$ times
continuously differentiable on an open set containing $K_{\vartheta h}$, with
$\norm{D^qf(x)-D^qf(y)}\le L_\beta\norm{x-y}^\gamma$ for all
$x,y\in K_{\vartheta h}$, and let the noise be admissible at levels
$(M,\sigma,\delta)$. Consider round $k$ of \cref{alg:stencil} with
$B\ge1$ frames, the deterministic radius $h$ and an
$\mathcal F_{k-1}$-measurable center $z$ with values in $K$. Then all queries
of the round lie in $K_h$, and the conclusions (i)--(iv) of
\cref{prop:unified-estimator} hold. In particular
\begin{equation}
\label{eq:cvx-estimate}
\gbar=\nabla f(z)+b(z)+\zeta+\Delta ,
\end{equation}
where $b(z)$ is $\mathcal F_{k-1}$-measurable, $\zeta+\Delta$ is
$\mathcal F_k$-measurable, $\E[\zeta\mid\mathcal F_{k-1}]=0$,
\[
\norm{b(z)}\le\bfC L_\beta d^{-\tau/2}h^{\beta-1}+A_1M,\qquad
\norm\Delta\le\frac{\sqrt d\,S_1\delta}h\ \text{almost surely},
\]
and
\begin{equation}
\label{eq:cvx-estimate-var}
\E\bigl[\norm\zeta^2\bigm|\mathcal F_{k-1}\bigr]\le\frac1B\Bigl(
\Cf L_\beta^2d^{-\tau}h^{2\beta-2}+C_\eta M^2
+C_\sigma\frac{d\sigma^2}{h^2}\Bigr).
\end{equation}
The bound \eqref{eq:cvx-estimate-var} also holds for $d=1$.
\end{alemma}

\begin{proof}
All queries lie in $K_h$. The Taylor, Pizzetti and Poincar\'e arguments
in \cref{app:estimator} use only the ball $z+hB(1)$, except for the
Hilbert--Schmidt bias with even $q$: its additional difference of radius
$h'=hd^{-1/4}$ uses $z+(h+h')B(1)\subseteq K_{2h}$. So the stated
neighborhood contains every point used by those proofs.
The surrogate gradient is continuous on $K$ by dominated convergence on
compact subsets, so its value at the predictable center is measurable.
The noise part of the proof of \cref{prop:unified-estimator} is unchanged,
and its stronger $d=1$ variance bound implies \eqref{eq:cvx-estimate-var}.
\end{proof}

For $m=1$, that is, for $\beta\in(2,3]$, the stencil is the central
difference: $t_1=a_1=1$, and $\mathcal S_hf=f_h$, $\mathcal S_h\eta=\eta_h$ in
\eqref{eq:mean}. Hence
\begin{equation}
\label{eq:cvx-estimate-m1}
\nabla f(z)+b(z)=\nabla f_h(z)+\nabla\eta_h(z),\qquad
\norm{\nabla\eta_h(z)}\le M ,
\end{equation}
and $S_1=S_2=A_1=1$, $C_\eta=4$ and $C_\sigma=(1+2^{-1/2})^2$. In the operator
geometry also $\bfC=c_1=1$ and $\Cf=4c_1^2=4$
(\cref{app:estimator-unified} and \eqref{eq:stencil-constants}).

\subsection{Projected methods with a bounded bias}
\label{app:convex-first-order}

In this subsection $Q$ is as in \cref{def:cvx-Q}, and $\varphi$ is
differentiable on an open set containing $Q$, convex on $Q$, with
$\norm{\nabla\varphi(x)-\nabla\varphi(y)}\le L\norm{x-y}$ for $x,y\in Q$.
Let $(\mathcal F_k)_{k\ge0}$ be a filtration; in the applications it is the
filtration of \cref{def:oracle}, and step $k$ is round $k$. In step $k$ an
$\mathcal F_{k-1}$-measurable point $z\in Q$, the \emph{center}, is submitted
to an oracle, which returns an $\mathcal F_k$-measurable vector
\begin{equation}
\label{eq:cvx-oracle}
g_k=\nabla\varphi(z)+b_k+\zeta_k,\qquad
\norm{b_k}\le b\ \text{a.s.},\quad
\E[\zeta_k\mid\mathcal F_{k-1}]=0,\quad
\E[\norm{\zeta_k}^2\mid\mathcal F_{k-1}]\le v_B^2 .
\end{equation}
In the applications $b_k$ contains the adversarial part of the estimate,
which may depend on $\zeta_k$; so $b_k$ and $\zeta_k$ are not assumed to be
$\mathcal F_k$-measurable. The comparator $u\in Q$ is an arbitrary fixed
point. In the applications it minimizes $f$ on $Q$, possibly at a boundary
point with $\nabla f(u)\ne0$, and in \cref{thm:value-analysis} it need not
minimize $\varphi=f_h$. For $x,y\in Q$ the segment $[x,y]$ lies in $Q$, so
integration of $\nabla\varphi$ along it and convexity give
\begin{equation}
\label{eq:cvx-descent}
\varphi(y)\le\varphi(x)+\inner{\nabla\varphi(x)}{y-x}+\frac L2\norm{y-x}^2,
\qquad
\varphi(y)\ge\varphi(x)+\inner{\nabla\varphi(x)}{y-x}.
\end{equation}

\begin{alemma}[Projected stochastic gradient]
\label{lem:convex-sgd}
Let $z_1=x_0\in Q$ and $0<\nu\le1/(2L)$. For $k=1,\dots,N$ submit the center
$z_k$, and let $z_{k+1}=\Proj_Q(z_k-\nu g_k)$. Let
$\bar z_N=\frac1N\sum_{k=2}^{N+1}z_k$. For every $u\in Q$,
\begin{equation}
\label{eq:cvx-sgd}
\E[\varphi(\bar z_N)-\varphi(u)]\le\frac{\norm{x_0-u}^2}{2\nu N}+\nu v_B^2
+Db .
\end{equation}
In particular, let $\norm{x_0-u}\le R$, $\bar v_B\ge v_B$ and
$\nu=\min\{1/(2L),R/(\bar v_B\sqrt{2N})\}$, with $\nu=1/(2L)$ if $\bar v_B=0$.
Then
\begin{equation}
\label{eq:cvx-sgd-tuned}
\E[\varphi(\bar z_N)-\varphi(u)]\le\frac{LR^2}{N}
+\frac{\sqrt2\,R\bar v_B}{\sqrt N}+Db .
\end{equation}
\end{alemma}

\begin{proof}
This is the standard analysis of projected stochastic gradient descent for
smooth objectives, with the bias kept. By induction, $z_k\in Q$ is
$\mathcal F_{k-1}$-measurable. Fix $k$. The projection inequality
$\inner{z_k-\nu g_k-z_{k+1}}{u-z_{k+1}}\le0$ gives
\begin{equation}
\label{eq:cvx-three-point}
\nu\inner{g_k}{z_{k+1}-u}\le\frac12\bigl(\norm{z_k-u}^2
-\norm{z_{k+1}-u}^2-\norm{z_{k+1}-z_k}^2\bigr).
\end{equation}
By \eqref{eq:cvx-descent} for the pairs $(z_k,z_{k+1})$ and $(z_k,u)$,
$\varphi(z_{k+1})-\varphi(u)\le\inner{\nabla\varphi(z_k)}{z_{k+1}-u}
+\frac L2\norm{z_{k+1}-z_k}^2$. Write $\nabla\varphi(z_k)=g_k-b_k-\zeta_k$,
bound $\inner{g_k}{z_{k+1}-u}$ by \eqref{eq:cvx-three-point},
$-\inner{b_k}{z_{k+1}-u}$ by $bD$, and
$-\inner{\zeta_k}{z_{k+1}-z_k}$ by
$\nu\norm{\zeta_k}^2+\norm{z_{k+1}-z_k}^2/(4\nu)$. Since $\nu\le1/(2L)$,
the terms with $\norm{z_{k+1}-z_k}^2$ have a nonpositive total coefficient,
and
\begin{equation}
\label{eq:cvx-sgd-step}
\varphi(z_{k+1})-\varphi(u)\le
\frac{\norm{z_k-u}^2-\norm{z_{k+1}-u}^2}{2\nu}
+\nu\norm{\zeta_k}^2+bD-\inner{\zeta_k}{z_k-u}.
\end{equation}
The last term has mean zero, because $z_k-u$ is bounded and
$\mathcal F_{k-1}$-measurable. Summing over $k=1,\dots,N$, taking
expectations and using the convexity of $\varphi$ on $Q$ proves
\eqref{eq:cvx-sgd}. For the tuned step, $1/\nu\le2L+\bar v_B\sqrt{2N}/R$
gives $R^2/(2\nu N)\le LR^2/N+R\bar v_B/\sqrt{2N}$, and
$\nu\bar v_B^2\le R\bar v_B/\sqrt{2N}$.
\end{proof}

The accelerated method below is a variant of the AC-SA method of
\citet{lan2012optimal} for convex objectives, run with a biased oracle.

\begin{alemma}[Accelerated projected method]
\label{lem:convex-agd}
Let $\Gamma_0=0$, $\Gamma_k=k(k+1)/2$ and $\rho_k=k/\Gamma_k=2/(k+1)$ for
$k\ge1$. Let $v_0=x_0\in Q$ and $0<\nu_0\le1/(4L)$. For $k=1,\dots,N$ submit
the center $z_k$, where
\[
z_k=(1-\rho_k)x_{k-1}+\rho_kv_{k-1},\qquad
v_k=\Proj_Q(v_{k-1}-\nu_0k\,g_k),\qquad
x_k=(1-\rho_k)x_{k-1}+\rho_kv_k .
\]
Then for every $u\in Q$,
\begin{equation}
\label{eq:cvx-agd}
\E[\varphi(x_N)-\varphi(u)]\le\frac{\norm{x_0-u}^2}{\nu_0N(N+1)}
+\frac{(2N+1)\,\nu_0v_B^2}{3}+Db .
\end{equation}
In particular, let $\norm{x_0-u}\le R$, $\bar v_B\ge v_B$ and
$\nu_0=\min\{1/(4L),R/(\bar v_BN^{3/2})\}$, with $\nu_0=1/(4L)$ if
$\bar v_B=0$. Then
\begin{equation}
\label{eq:cvx-agd-tuned}
\E[\varphi(x_N)-\varphi(u)]\le\frac{4LR^2}{N^2}+\frac{2R\bar v_B}{\sqrt N}
+Db .
\end{equation}
\end{alemma}

\begin{proof}
All iterates lie in $Q$, and $z_k,v_{k-1},x_{k-1}$ are
$\mathcal F_{k-1}$-measurable.
Write $\nu_k=\nu_0k$. Smoothness, convexity, and
$x_k-z_k=\rho_k(v_k-v_{k-1})$ give
\[
\varphi(x_k)-\varphi(u)\le(1-\rho_k)(\varphi(x_{k-1})-\varphi(u))
+\rho_k\inner{\nabla\varphi(z_k)}{v_k-u}
+\tfrac12L\rho_k^2\norm{v_k-v_{k-1}}^2.
\]
Split the oracle error as in the proof of \cref{lem:convex-sgd}: apply the
projection inequality to $\inner{g_k}{v_k-u}$, bound the bias term by $bD$,
and use $-\inner{\zeta_k}{v_k-v_{k-1}}\le\nu_k\norm{\zeta_k}^2
+\norm{v_k-v_{k-1}}^2/(4\nu_k)$.
The remaining coefficient of $\norm{v_k-v_{k-1}}^2$ is nonpositive because
$2L\rho_k\nu_k\le1$. Multiplication by $\Gamma_k$ yields
\begin{align*}
\Gamma_k(\varphi(x_k)-\varphi(u))\le{}&
\Gamma_{k-1}(\varphi(x_{k-1})-\varphi(u))
+\frac{\norm{v_{k-1}-u}^2-\norm{v_k-u}^2}{2\nu_0}\\
&+\nu_0k^2\norm{\zeta_k}^2+kbD
-k\inner{\zeta_k}{v_{k-1}-u}.
\end{align*}
The last term has mean zero. Summing, dropping the final nonpositive
distance, and using $\sum k^2=N(N+1)(2N+1)/6$ proves
\eqref{eq:cvx-agd}. Finally
$1/\nu_0\le4L+\bar v_BN^{3/2}/R$, $(2N+1)/3\le N$, and
$N\nu_0v_B^2\le R\bar v_B/\sqrt N$ give
\eqref{eq:cvx-agd-tuned}.
\end{proof}

\begin{aremark}[The bias does not accumulate]
\label{rem:no-accumulation}
In \cref{lem:convex-agd} the quadratic upper model uses the exact gradient
$\nabla\varphi(z_k)$, and the bias enters only through
$\rho_k\inner{b_k}{u-v_k}\le\rho_kbD$, because $v_k$ is projected onto the
bounded set $Q$. So the accelerated method pays the same bias term $Db$ as
the plain method; for an unbounded feasible set the argument gives no bound.
Treated as an inexact oracle in the sense of \citet{Devolder2014}, a
gradient error of norm $b$ on a set of diameter $D$ has accuracy of order
$bD$, and the fast gradient method accumulates it into a term of order
$NDb$. A bound of the same form as \eqref{eq:cvx-agd}, with noise weighted by
$k/\Gamma_N$ and multiplied by the diameter, is Proposition~3.5 of
\citet{Cohen2018}.
\end{aremark}

\subsection{The stencil as a biased gradient}
\label{app:convex-gradient-bias}

The next corollary uses only \cref{lem:cvx-local-estimate} and the lemmas
above. Its accelerated form is \cref{alg:projected-convex}, which with $m=1$
is also the method of \cref{thm:value-analysis}(d).

\begin{algorithm}[t]
\caption{Projected acceleration with batched scalar stencils}
\label{alg:projected-convex}
\begin{algorithmic}[1]
\Require $Q$, $\Proj_Q$, $x_0\in Q$; $h,m,N,B,\nu_0$ from \cref{cor:direct-convex}(b), or $m=1$ and $h,N,B,\nu_0$ from \cref{thm:value-analysis}(d)
\State $v_0\gets x_0$
\For{$k=1,\ldots,N$}
  \State $\rho_k\gets2/(k+1)$; $z_k\gets(1-\rho_k)x_{k-1}+\rho_kv_{k-1}$
  \State Acquire $\gbar_k$ at $z_k$ by \cref{alg:stencil}, radius $h$, $B$ frames
  \State $v_k\gets\Proj_Q(v_{k-1}-\nu_0k\gbar_k)$
  \State $x_k\gets(1-\rho_k)x_{k-1}+\rho_kv_k$
\EndFor
\State \Return $x_N\in Q$
\end{algorithmic}
\end{algorithm}

\begin{acorollary}[Direct projected methods]
\label{cor:direct-convex}
Let $Q$, $x_0$, $R$ and $D$ be as in \cref{def:cvx-Q}, let $\eps>0$, and put
\[
h=(4\bfC)^{-\frac1{\beta-1}}
\Bigl(\frac{\eps\,d^{\tau/2}}{L_\beta D}\Bigr)^{\frac1{\beta-1}},\quad
\bar M=\frac{\eps}{8A_1D},\quad
\bar v_h^2=\Cf L_\beta^2d^{-\tau}h^{2\beta-2}+C_\eta\bar M^2
+C_\sigma\frac{d\sigma^2}{h^2}.
\]
Let $\vartheta$ be as in \cref{lem:cvx-local-estimate}, and let $f$ be $q$
times continuously differentiable on an open set containing
$Q_{\vartheta h}$, with
\[
0\preceq\nabla^2f(x)\preceq LI\quad(x\in Q),\qquad
\norm{D^qf(x)-D^qf(y)}\le L_\beta\norm{x-y}^\gamma\quad(x,y\in Q_{\vartheta h}),
\]
and with a minimizer $u\in\argmin_Qf$ such that $\norm{x_0-u}\le R$; for
instance, $f\in\Fcvx^{Q_{\vartheta h}}(\beta,L_\beta;L)$ with such a minimizer. Let
the noise be admissible at levels $(M,\sigma,\delta)$. Consider the
following two methods; in both, the oracle of round $k$ is \cref{alg:stencil}
at the center $z_k$, with radius $h$ and $B$ frames.
\begin{enumerate}[label=(\alph*)]
\item Run \cref{lem:convex-sgd} with $\varphi=f$ and $z_1=x_0$, output
$\bar z_N$, and take
\[
N=\Bigl\lceil\frac{8LR^2}{\eps}\Bigr\rceil,\qquad
T_0=\frac{256\,mdR^2\bar v_h^2}{9\eps^2},\qquad
B=\Bigl\lceil\frac{T_0}{2mdN}\Bigr\rceil,\qquad
\nu=\min\Bigl\{\frac1{2L},\frac{R\sqrt B}{\bar v_h\sqrt{2N}}\Bigr\}.
\]
\item Run \cref{lem:convex-agd} with $\varphi=f$ and $v_0=x_0$, output $x_N$,
and take
\[
N=\lceil\sqrt{32LR^2/\eps}\,\rceil,\qquad
T_0=\frac{512\,mdR^2\bar v_h^2}{9\eps^2},\qquad
B=\Bigl\lceil\frac{T_0}{2mdN}\Bigr\rceil,\qquad
\nu_0=\min\Bigl\{\frac1{4L},\frac{R\sqrt B}{\bar v_hN^{3/2}}\Bigr\}.
\]
\end{enumerate}
Suppose $M\le\bar M$ and $\delta\le\eps h/(8S_1D\sqrt d)$. Then
both methods query only points of $Q_h$, both outputs $\hat x$ lie in $Q$ and
satisfy $\E[f(\hat x)-\min_Qf]\le\eps$, and the number of calls is
$T=2mdNB\le T_0+2mdN$, where, explicitly,
\begin{equation}
\label{eq:cvx-direct-T0}
T_0=\theta\Bigl[\frac{256\,m\,C_\sigma(4\bfC)^{\frac2{\beta-1}}}{9}\,
\frac{\sigma^2}{\eps^2}\,d^{\,2-\frac{\tau}{\beta-1}}R^2
\Bigl(\frac{L_\beta D}{\eps}\Bigr)^{\frac2{\beta-1}}
+md\,\frac{R^2}{D^2}\Bigl(\frac{16\Cf}{9\bfC^2}+\frac{4C_\eta}{9A_1^2}\Bigr)
\Bigr],
\end{equation}
with $\theta=1$ in (a) and $\theta=2$ in (b).
\end{acorollary}

\begin{proof}
Let $z_k$ be the center of round $k$; it is $\mathcal F_{k-1}$-measurable and
lies in $Q$. By \cref{lem:cvx-local-estimate} with $K=Q$, the queries lie in
$Q_h$, and the returned vector is $g_k=\nabla f(z_k)+b_k+\zeta_k$ with
$b_k=b(z_k)+\Delta_k$, $\E[\zeta_k\mid\mathcal F_{k-1}]=0$ and, since
$M\le\bar M$, $\E[\norm{\zeta_k}^2\mid\mathcal F_{k-1}]\le\bar v_h^2/B$.
The choice of $h$ gives $\bfC L_\beta d^{-\tau/2}h^{\beta-1}=\eps/(4D)$, so
\[
\norm{b_k}\le b:=\frac{\eps}{4D}+A_1M+\frac{\sqrt d\,S_1\delta}{h}
\le\frac{\eps}{4D}+\frac{\eps}{8D}+\frac{\eps}{8D}=\frac{\eps}{2D}.
\]
Thus the setting of \cref{app:convex-first-order} holds for $\varphi=f$, which
is convex with an $L$-Lipschitz gradient on $Q$, with $v_B^2=\bar v_h^2/B$ and
$Db\le\eps/2$.

In case (a), \eqref{eq:cvx-sgd-tuned} with $\bar v_B=\bar v_h/\sqrt B$ bounds
the error by $LR^2/N+\sqrt2R\bar v_h/\sqrt{NB}+Db$. Here $LR^2/N\le\eps/8$, and
$NB\ge T_0/(2md)=128R^2\bar v_h^2/(9\eps^2)$ gives
$\sqrt2R\bar v_h/\sqrt{NB}\le3\eps/8$. In case (b), \eqref{eq:cvx-agd-tuned}
bounds the error by $4LR^2/N^2+2R\bar v_h/\sqrt{NB}+Db$, where
$4LR^2/N^2\le\eps/8$ and $NB\ge256R^2\bar v_h^2/(9\eps^2)$ gives
$2R\bar v_h/\sqrt{NB}\le3\eps/8$. In both cases the sum is at most $\eps$, and
$f(u)=\min_Qf$.

Finally, $B\le T_0/(2mdN)+1$ gives $T\le T_0+2mdN$, and inserting
\[
L_\beta^2d^{-\tau}h^{2\beta-2}=\frac{\eps^2}{16\bfC^2D^2},\qquad
\bar M^2=\frac{\eps^2}{64A_1^2D^2},\qquad
\frac1{h^2}=(4\bfC)^{\frac2{\beta-1}}d^{-\frac{\tau}{\beta-1}}
\Bigl(\frac{L_\beta D}{\eps}\Bigr)^{\frac2{\beta-1}}
\]
into $T_0$ gives \eqref{eq:cvx-direct-T0}.
\end{proof}

\subsection{Value analysis of the ball average}
\label{app:convex-value}

The method of this subsection is projected stochastic gradient descent driven
by central differences, analyzed as descent on the ball average $f_h$. It uses
central differences for every $\beta$: for $\beta>3$ the stencil of
\cref{prop:unified-estimator} has $m\ge2$, and its surrogate $\mathcal S_hf$
has weights of both signs and need not be convex.

\begin{alemma}[Values of the ball average]
\label{lem:smoothing-values}
Let $Q$ and $D$ be as in \cref{def:cvx-Q}, and let $h>0$.
\begin{enumerate}[label=(\alph*)]
\item Let $\gamma\in(0,1]$ and $\beta=2+\gamma$, and let $f$ be twice
continuously differentiable on an open set containing $Q_h$, with
$\norm{\nabla^2f(x)-\nabla^2f(y)}_{\op}\le L_\beta\norm{x-y}^\gamma$ for all
$x,y\in Q_h$. Then for $x\in Q$,
\begin{equation}
\label{eq:cvx-value-identity}
f_h(x)=f(x)+\frac{h^2}{2(d+2)}\tr\nabla^2f(x)+r_h(x),\qquad
\abs{r_h(x)}\le\frac{c_\gamma d}{d+\beta}L_\beta h^\beta\le c_\gamma L_\beta h^\beta ,
\end{equation}
and $\abs{\tr\nabla^2f(x)-\tr\nabla^2f(y)}\le dL_\beta\norm{x-y}^\gamma$ for all
$x,y\in Q_h$.
\item Let $f$ be differentiable on an open set containing $Q_h$ and convex on
$Q_h$, with $\norm{\nabla f(x)-\nabla f(y)}\le L\norm{x-y}$ for
$x,y\in Q_h$. Then $f_h$ is defined and continuously differentiable on an
open set containing $Q$, with $\nabla f_h(x)=\E\nabla f(x+hW)$ for $x\in Q$. It is convex
on $Q$, $\norm{\nabla f_h(x)-\nabla f_h(y)}\le L\norm{x-y}$ for $x,y\in Q$,
and
\[
0\le f_h(x)-f(x)\le\frac{d}{2(d+2)}Lh^2\qquad(x\in Q).
\]
If moreover $f$ is twice continuously differentiable on an open set
containing $Q_h$, then $0\le\tr\nabla^2f(x)\le dL$ for $x\in Q$.
\item Under the assumptions of (b), for all $x,y\in Q$,
\begin{equation}
\label{eq:cvx-value-transfer-M}
f(x)-f(y)\le f_h(x)-f_h(y)+\frac{Lh^2}2 .
\end{equation}
Under the assumptions of (a) and (b), for all $x,y\in Q$,
\begin{equation}
\label{eq:cvx-value-transfer}
f(x)-f(y)\le f_h(x)-f_h(y)+S(h),\qquad
S(h):=\min\Bigl\{\frac{Lh^2}2,\ \frac{L_\beta D^\gamma h^2}2+2c_\gamma L_\beta h^\beta\Bigr\}.
\end{equation}
\end{enumerate}
In particular, under (a) and (b), for $x,y\in Q$,
\[
\abs{\tr\nabla^2f(x)-\tr\nabla^2f(y)}\le d\min\{L,L_\beta\norm{x-y}^\gamma\}
\le d\min\{L,L_\beta D^\gamma\}.
\]
\end{alemma}

\begin{proof}
(a) For $x\in Q$ and $\norm v\le h$, Taylor's formula with integral remainder
and the H\"older condition give
\[
\left|f(x+v)-f(x)-\inner{\nabla f(x)}v-\tfrac12\nabla^2f(x)[v,v]\right|
\le L_\beta\norm v^{2+\gamma}\int_0^1(1-t)t^\gamma\,dt
=c_\gamma L_\beta\norm v^\beta .
\]
Averaging over $v=hW$, with $\E W=0$, $\E WW^\top=I/(d+2)$ and
$\E\norm W^\beta=d/(d+\beta)$, proves the identity. The trace bound follows
from $\abs{\tr A}\le d\norm A_{\op}$.

(b) By \cref{lem:ball-gradient}, whose proof is local, $f_h$ is
continuously differentiable near $Q$ and $\nabla f_h(x)=\E\nabla f(x+hW)$ for
$x\in Q$. Since $x+hW\in Q_h$ for $x\in Q$, averaging
preserves convexity and the Lipschitz constant of the gradient on $Q$. Averaging
$0\le f(x+v)-f(x)-\inner{\nabla f(x)}v\le L\norm v^2/2$ over $v=hW$, with
$\E\norm W^2=d/(d+2)$, gives the value bounds. When the Hessian exists, its
eigenvalues lie in $[0,L]$ on $Q$, which proves the trace bounds.

(c) The difference $(f(x)-f(y))-(f_h(x)-f_h(y))$ equals
$(f_h-f)(y)-(f_h-f)(x)$. Part (b) bounds it by $Lh^2/2$, and part (a) by
$h^2dL_\beta D^\gamma/(2(d+2))+2c_\gamma L_\beta h^\beta$. The last claim combines the
trace bounds of (a) and (b).
\end{proof}

\begin{alemma}[Central differences for objectives with a Lipschitz gradient]
\label{lem:m1-smooth-frame}
Let $K\subset\R^d$ be nonempty, closed and convex, let $h>0$, let $f$ be
differentiable on an open set containing $K_h$ with
$\norm{\nabla f(x)-\nabla f(y)}\le L\norm{x-y}$ for $x,y\in K_h$, and let the
noise be admissible at levels $(M,\sigma,\delta)$. Consider a round of
\cref{alg:stencil} with $m=1$, an $\mathcal F_{k-1}$-measurable center $z$
with values in $K$, the deterministic radius $h$ and $B$ frames. Then the
queries lie in $K_h$,
\[
\gbar=\nabla f_h(z)+\nabla\eta_h(z)+\zeta+\Delta,\qquad
\norm{\nabla\eta_h(z)}\le M,\qquad
\norm\Delta\le\frac{\sqrt d\,\delta}h\ \text{a.s.},
\]
and $\E[\zeta\mid\mathcal F_{k-1}]=0$. For $d\ge2$,
\begin{equation}
\label{eq:cvx-m1-var}
\E\bigl[\norm\zeta^2\bigm|\mathcal F_{k-1}\bigr]\le\frac1B\Bigl(
\frac d{d-1}(Lh+M)^2+C_\sigma\frac{d\sigma^2}{h^2}\Bigr),
\qquad C_\sigma=(1+2^{-1/2})^2 ;
\end{equation}
for $d=1$ the term with $Lh+M$ is absent. Part (iv) of
\cref{prop:unified-estimator} holds with $\nabla f(z)+b(z)$ replaced by
$\nabla f_h(z)+\nabla\eta_h(z)$ and with $S_1=1$. If moreover
$\beta\in(2,3]$, $f$ is twice continuously differentiable on an open set
containing $K_h$ and
$\norm{\nabla^2f(x)-\nabla^2f(y)}_{\op}\le L_\beta\norm{x-y}^\gamma$ for
$x,y\in K_h$, then $\zeta$ is the vector of \cref{lem:cvx-local-estimate}
with $m=1$ and the operator norm, and \eqref{eq:cvx-estimate-var} holds for
it as well.
\end{alemma}

\begin{proof}
The mean, the stochastic-noise contribution, the adversarial bound and
the adapted decomposition are those of \cref{prop:unified-estimator}
with $m=1$; they use only continuity and oddness on $z+hB(1)$.
It remains to bound the frame variance without a bound on $\nabla f(z)$.

Put $e_f(w)=(f(z+hw)-f(z-hw))/(2h)-\inner{\nabla f(z)}w$.
On the unit ball,
$\nabla e_f(w)=(\nabla f(z+hw)+\nabla f(z-hw))/2-\nabla f(z)$
has norm at most $Lh$, so $e_f$ is odd and $Lh$-Lipschitz there.
Composition with the odd, nonexpansive projection onto that ball extends it
to an odd $Lh$-Lipschitz function on $\R^d$.
The Lipschitz-noise term $e_\eta(w)=(\eta(z+hw)-\eta(z-hw))/(2h)$ is odd
and $M$-Lipschitz, so $e_f+e_\eta$ has Lipschitz constant $Lh+M$.
By \eqref{eq:odd-frame}, the centered second moment of one frame is at most
$d(Lh+M)^2/(d-1)$. Independent frames divide it by $B$; for $d=1$ it
is zero. Adding the orthogonal noise term proves \eqref{eq:cvx-m1-var}.
For the final claim, both constructions define
$\zeta=\gbar-\nabla f_h(z)-\nabla\eta_h(z)-\Delta$; apply
\cref{lem:cvx-local-estimate}.
\end{proof}

\begin{atheorem}[Value analysis]
\label{thm:value-analysis}
Let $Q$, $x_0$, $R$ and $D$ be as in \cref{def:cvx-Q}, and let $h>0$. Let $f$
be differentiable on an open set containing $Q_h$ and convex on $Q_h$, with
$\norm{\nabla f(x)-\nabla f(y)}\le L\norm{x-y}$ for $x,y\in Q_h$, and let $f$
have a minimizer $u\in\argmin_Qf$ such that $\norm{x_0-u}\le R$. Every
$f\in\Fcvx^{Q_h}(\beta,L_\beta;L)$, $\beta>2$, with such a minimizer has these
properties. Let the noise be admissible at levels $(M,\sigma,\delta)$.
The method is projected stochastic gradient descent on $Q$ driven by
central differences:
\begin{itemize}
\item $z_1=x_0$;
\item in round $k=1,\dots,N$, run \cref{alg:stencil} with $m=1$ at the
center $z_k$, with radius $h$ and $B$ frames ($2dB$ calls), obtain
$\gbar_k$, and set $z_{k+1}=\Proj_Q(z_k-\nu\gbar_k)$;
\item output $\bar z_N=\frac1N\sum_{k=2}^{N+1}z_k$.
\end{itemize}
All queries lie in $Q_h$, and $\bar z_N\in Q$.
\begin{enumerate}[label=(\alph*)]
\item \emph{(Error bound.)} For all $N,B\in\N$ and $0<\nu\le1/(2L)$,
\begin{equation}
\label{eq:cvx-value-bound}
\E\Bigl[f(\bar z_N)-\min_Qf\Bigr]\le\frac{R^2}{2\nu N}+\frac{\nu v_h^2}{B}
+D\Bigl(M+\frac{\sqrt d\,\delta}{h}\Bigr)+S(h),
\end{equation}
where
\[
v_h^2=4L^2h^2+4M^2+C_\sigma d\sigma^2/h^2,\qquad S(h)=Lh^2/2 .
\]
If moreover $\beta\in(2,3]$ and $f$ is twice continuously differentiable on
an open set containing $Q_h$, with
$\norm{\nabla^2f(x)-\nabla^2f(y)}_{\op}\le L_\beta\norm{x-y}^\gamma$ for
$x,y\in Q_h$, where $\gamma=\beta-2$, then \eqref{eq:cvx-value-bound} also
holds with
\[
v_h^2=\min\{4L^2h^2,4L_\beta^2h^{2\beta-2}\}+4M^2+C_\sigma d\sigma^2/h^2,
\qquad
S(h)=\min\Bigl\{\frac{Lh^2}2,\ \frac{L_\beta D^\gamma h^2}2+2c_\gamma L_\beta h^\beta\Bigr\}.
\]
\item \emph{(Curvature branch, every $\beta$.)} Let $\eps>0$, and take
\begin{gather*}
h=\sqrt{\frac{\eps}{8L}},\qquad \bar M=\frac{\eps}{8D},\qquad
\bar v_h^2=4L^2h^2+4\bar M^2+C_\sigma\frac{d\sigma^2}{h^2},\\
N=\Bigl\lceil\frac{8LR^2}{\eps}\Bigr\rceil,\qquad
T_0=\frac{256\,dR^2\bar v_h^2}{9\eps^2},\qquad
B=\Bigl\lceil\frac{T_0}{2dN}\Bigr\rceil,\qquad
\nu=\min\Bigl\{\frac1{2L},\frac{R\sqrt B}{\bar v_h\sqrt{2N}}\Bigr\}.
\end{gather*}
If $M\le\bar M$ and $\delta\le\eps h/(8D\sqrt d)$, then
$\E[f(\bar z_N)-\min_Qf]\le\eps$, with $N$ rounds and
\begin{equation}
\label{eq:cvx-value-T-M}
T=2dNB\le\frac{2048\,C_\sigma}{9}\,\frac{d^2\sigma^2LR^2}{\eps^3}
+\frac{272}{9}\,\frac{dLR^2}{\eps}+\Bigl(2+\frac{16R^2}{9D^2}\Bigr)d
\end{equation}
calls. This method uses $L,R,D,\sigma,d,\eps$, but neither $L_\beta$ nor $\beta$.
\item \emph{(H\"older branch, $\beta\in(2,3]$.)} Suppose moreover that $f$
satisfies the H\"older condition of (a) on $Q_h$ and that
$0<\eps\le8L_\beta D^\beta$. Take
\[
h=\sqrt{\frac{\eps}{8L_\beta D^\gamma}},\qquad
\bar M=\frac{\eps}{8D},\qquad
\bar v_h^2=4L_\beta^2h^{2\beta-2}+4\bar M^2+C_\sigma\frac{d\sigma^2}{h^2},
\]
and define $N,T_0,B,\nu$ as in (b), with this $\bar v_h$. Then
$h\le D$, and the conclusion of (b) holds with
\begin{equation}
\label{eq:cvx-value-T-L}
T\le\frac{2048\,C_\sigma}{9}\,\frac{d^2\sigma^2}{\eps^2}\cdot
\frac{L_\beta D^\gamma R^2}{\eps}
+\frac{128}{9}\Bigl(\frac hD\Bigr)^{2\gamma}\frac{dL_\beta D^\gamma R^2}{\eps}
+\frac{16\,dLR^2}{\eps}+\Bigl(2+\frac{16R^2}{9D^2}\Bigr)d .
\end{equation}
\item \emph{(Acceleration.)} In (b) and (c), replace the update by the
method of \cref{lem:convex-agd} with $\varphi=f_h$ and $v_0=x_0$; it queries
at the centers $z_k$ and outputs $x_N\in Q$. Keep $h$ and $\bar v_h$, and take
\[
N=\lceil\sqrt{32LR^2/\eps}\,\rceil,\qquad
T_0=\frac{512\,dR^2\bar v_h^2}{9\eps^2},\qquad
B=\Bigl\lceil\frac{T_0}{2dN}\Bigr\rceil,\qquad
\nu_0=\min\Bigl\{\frac1{4L},\frac{R\sqrt B}{\bar v_hN^{3/2}}\Bigr\}.
\]
Under the same conditions, $\E[f(x_N)-\min_Qf]\le\eps$ with
$T=2dNB\le T_0+2dN$ calls. In the curvature branch this means
\[
T\le\frac{4096\,C_\sigma}{9}\,\frac{d^2\sigma^2LR^2}{\eps^3}
+\frac{256}{9}\,\frac{dLR^2}{\eps}+\frac{32R^2}{9D^2}\,d
+2d\Bigl\lceil\sqrt{32LR^2/\eps}\,\Bigr\rceil ,
\]
and in the H\"older branch
\[
T\le\frac{4096\,C_\sigma}{9}\,\frac{d^2\sigma^2}{\eps^2}\cdot
\frac{L_\beta D^\gamma R^2}{\eps}
+\frac{256}{9}\Bigl(\frac hD\Bigr)^{2\gamma}\frac{dL_\beta D^\gamma R^2}{\eps}
+\frac{32R^2}{9D^2}\,d+2d\Bigl\lceil\sqrt{32LR^2/\eps}\,\Bigr\rceil .
\]
\end{enumerate}
\end{atheorem}

\begin{proof}
(a) By \cref{lem:m1-smooth-frame} with $K=Q$, in round $k$
\[
\gbar_k=\nabla f_h(z_k)+b_k+\zeta_k,\qquad
b_k=\nabla\eta_h(z_k)+\Delta_k,\qquad
\norm{b_k}\le M+\frac{\sqrt d\,\delta}{h},
\]
where $\gbar_k$ is $\mathcal F_k$-measurable and
$\E[\zeta_k\mid\mathcal F_{k-1}]=0$. Since
$\frac d{d-1}(Lh+M)^2\le2(Lh+M)^2\le4L^2h^2+4M^2$ for
$d\ge2$, \eqref{eq:cvx-m1-var} gives
$\E[\norm{\zeta_k}^2\mid\mathcal F_{k-1}]\le v_h^2/B$. In the H\"older case,
\eqref{eq:cvx-estimate-var} with $m=1$, $\tau=0$ and $\Cf=C_\eta=4$ applies
to the same $\zeta_k$ by the last claim of \cref{lem:m1-smooth-frame}, so the
smaller of the two bounds holds. By \cref{lem:smoothing-values}(b),
$\varphi=f_h$ satisfies the setting of \cref{app:convex-first-order}, and
\cref{lem:convex-sgd} with the comparator $u$ gives
\[
\E[f_h(\bar z_N)-f_h(u)]\le\frac{R^2}{2\nu N}+\frac{\nu v_h^2}{B}
+D\Bigl(M+\frac{\sqrt d\,\delta}h\Bigr).
\]
Both $\bar z_N$ and $u$ lie in $Q$, so \eqref{eq:cvx-value-transfer-M} or
\eqref{eq:cvx-value-transfer} gives
$f(\bar z_N)-f(u)\le f_h(\bar z_N)-f_h(u)+S(h)$, and $f(u)=\min_Qf$. This
proves \eqref{eq:cvx-value-bound}.

(b) Here $S(h)=Lh^2/2=\eps/16$ and $v_h^2\le\bar v_h^2$. As in
\eqref{eq:cvx-sgd-tuned}, the choice of $\nu$ gives
$R^2/(2\nu N)+\nu v_h^2/B\le LR^2/N+\sqrt2R\bar v_h/\sqrt{NB}$. In this bound
$LR^2/N\le\eps/8$, and $NB\ge T_0/(2d)=128R^2\bar v_h^2/(9\eps^2)$ gives
$\sqrt2R\bar v_h/\sqrt{NB}\le3\eps/8$. The conditions on $M$ and
$\delta$ give $D(M+\sqrt d\delta/h)\le\eps/4$, so the error is at most
$13\eps/16$. For the calls, $B\le T_0/(2dN)+1$ and $N\le8LR^2/\eps+1$ give
$T\le T_0+16dLR^2/\eps+2d$. With $h^2=\eps/(8L)$ and
$\bar M=\eps/(8D)$, the parts $4L^2h^2=L\eps/2$,
$4\bar M^2=\eps^2/(16D^2)$ and
$C_\sigma d\sigma^2/h^2=8C_\sigma d\sigma^2L/\eps$ of $\bar v_h^2$, multiplied by
$256dR^2/(9\eps^2)$, give \eqref{eq:cvx-value-T-M}.

(c) Now $L_\beta D^\gamma h^2/2=\eps/16$ and $(h/D)^2=\eps/(8L_\beta D^\beta)\le1$. Since
$\beta/2=1+\gamma/2$,
\[
L_\beta h^\beta=\frac\eps8\Bigl(\frac{\eps}{8L_\beta D^\beta}\Bigr)^{\gamma/2}\le\frac\eps8 ,
\]
and $2c_\gamma\le1$ gives $S(h)\le3\eps/16$. The rest of the argument of (b)
applies, and the error is at most $15\eps/16$. For the calls, insert
\[
L_\beta^2h^{2\beta-2}=\frac\eps8\,L_\beta D^\gamma\Bigl(\frac hD\Bigr)^{2\gamma},
\qquad
C_\sigma\frac{d\sigma^2}{h^2}=\frac{8C_\sigma d\sigma^2L_\beta D^\gamma}{\eps}
\]
and $4\bar M^2=\eps^2/(16D^2)$ into $T_0$; this gives
\eqref{eq:cvx-value-T-L}.

(d) Replace \cref{lem:convex-sgd} by \cref{lem:convex-agd} for $\varphi=f_h$.
By \eqref{eq:cvx-agd-tuned} with $\bar v_B=\bar v_h/\sqrt B$, the optimization
error is at most $4LR^2/N^2+2R\bar v_h/\sqrt{NB}\le\eps/8+3\eps/8$, since
$NB\ge256R^2\bar v_h^2/(9\eps^2)$. The other terms are as in (b) and (c), and
the transfer inequalities apply because $x_N\in Q$. The contributions to
$T_0$ double, and $2dN=2d\lceil\sqrt{32LR^2/\eps}\,\rceil$.
\end{proof}

\subsection{Explicit call constants on a Euclidean ball}
\label{app:convex-proof-main}

\begin{proof}[Proof of the upper bounds of \cref{sec:convex}]
The gradient-scale bound follows from \cref{cor:direct-convex} and the
value-scale bound from \cref{thm:value-analysis}, applied with $Q=B(R)$,
$x_0=0$ and $D=2R$. We simplify the constants for a general $Q$, using
$R\le D$ and $\lceil y\rceil\le y+1$, and write $x=LR^2/\eps$.

\emph{Gradient scale.} \Cref{cor:direct-convex} uses $\lceil8x\rceil$
rounds, or $\lceil\sqrt{32x}\,\rceil$ rounds in the accelerated form; in
both forms $2mdN\le16md(1+x)$, because $\sqrt{32x}\le4+2x$. Since
$R^2(L_\beta D/\eps)^{2/(\beta-1)}=(L_\beta R^{\beta-1}D/\eps)^{2/(\beta-1)}$, the bound
\eqref{eq:cvx-direct-T0} with $\theta\le2$ gives
\[
T\le T_0+2mdN\le C_\beta\frac{\sigma^2}{\eps^2}\,d^{\,2-\frac\tau{\beta-1}}
\Bigl(\frac{L_\beta R^{\beta-1}D}{\eps}\Bigr)^{\frac2{\beta-1}}+C_\beta md(1+x)
\]
with
\[
C_\beta=\max\Bigl\{\frac{512\,mC_\sigma(4\bfC)^{\frac2{\beta-1}}}{9},\
\frac{32\Cf}{9\bfC^2}+\frac{8C_\eta}{9A_1^2}+16\Bigr\},
\]
which depends only on $\beta$. For $D=2R$, $L_\beta R^{\beta-1}D=2L_\beta R^\beta$,
and the radius and the noise conditions of \cref{cor:direct-convex} take the
form stated in \cref{sec:convex}.

\emph{Value scale.} Let $V=L$, or, for $\beta\in(2,3]$ and the operator norm,
$V=\min\{L,L_\beta D^{\beta-2}\}$. If $V=L$ we use \cref{thm:value-analysis}(b). If
$V=L_\beta D^{\beta-2}<L$ we use \cref{thm:value-analysis}(c), which requires
$\eps\le8L_\beta D^\beta$. Part (d) of that theorem gives the accelerated forms. In
all cases the radius is $h=\sqrt{\eps/(8V)}$ and the conditions on the noise
are $M\le\eps/(8D)$ and $\delta\le\eps h/(8D\sqrt d)$. In the H\"older
branch $h\le D$ and $V<L$, so the second term of \eqref{eq:cvx-value-T-L} is
at most $\frac{128}9dx$. Hence, in both branches,
\[
T\le\frac{2048\,C_\sigma}{9}\,\frac{d^2\sigma^2}{\eps^2}\,\frac{VR^2}{\eps}
+\frac{272}{9}\,dx+\frac{34}{9}\,d ,
\]
and for the accelerated form
\[
T\le\frac{4096\,C_\sigma}{9}\,\frac{d^2\sigma^2}{\eps^2}\,\frac{VR^2}{\eps}
+\frac{256}{9}\,dx+\frac{32}{9}\,d+2d\bigl(\sqrt{32x}+1\bigr).
\]
Both are at most
\[
\frac{2^{12}C_\sigma}{9}\,\frac{d^2\sigma^2}{\eps^2}\,\frac{VR^2}{\eps}+32d(1+x);
\]
for the second use $8\sqrt2\sqrt x\le32x/9+9$.
\end{proof}

\subsection{Matching bounds on the unconstrained class}
\label{app:convex-matching}

In this subsection and the next, $Q=B(R)$, $x_0=0$ and $D=2R$, so that the
problems of \cref{def:cvx-Q} are those of the class $\Fcvx(\beta,L_\beta;L,R)$ and
$\min_Qf=\fstar$.

\paragraph{Lower bounds.}
\Cref{thm:convex-lower-main}(a), proved in \cref{app:convex-lower}, gives
constants $c_\beta,c'_\beta>0$, depending only on $\beta$, with the following
property in the operator geometry. Let $d\ge256$, let the noise be
independent $N(0,\sigma^2)$, which is admissible at levels $(0,\sigma,0)$, let
$L\in(0,\infty]$ and $0<\eps\le c'_\beta\min\{L_\beta R^\beta,LR^2\}$. Then every
algorithm with $\E[f(\hat x)-\fstar]\le\eps$ for all
$f\in\Fcvx(\beta,L_\beta;L,R)$ makes
\begin{equation}
\label{eq:cvx-H1}
T\ge c_\beta\frac{\sigma^2}{\eps^2}
\min\Bigl\{d^{5/2},\ d^2\Bigl(\frac{L_\beta R^\beta}{\eps}\Bigr)^{2/\beta},\
\frac{d^2LR^2}{\eps}\Bigr\}
\end{equation}
calls. With $L_\eps=C_{\mathrm{cap}}L_\beta^{2/\beta}\eps^{1-2/\beta}$ as in its proof,
the last term is the smallest one if and only if $L\le L_\eps$ and
$LR^2/\eps\le C_{\mathrm{cap}}\sqrt d$.

\begin{acorollary}[Curvature matched to the hard instances]
\label{cor:natural-curvature}
Consider the operator geometry, let $c_\beta,c'_\beta,C_{\mathrm{cap}}$ be as in
\eqref{eq:cvx-H1}, let $d\ge256$ and $L,\sigma>0$, and for
$M\le\eps/(16R)$ and $\delta\le\eps^{3/2}/(16R\sqrt{8dL})$ let $T^*(\eps;L)$
be the smallest number of calls with which some algorithm achieves
$\E[f(\hat x)-\fstar]\le\eps$ for all $f\in\Fcvx(\beta,L_\beta;L,R)$ and all
noise admissible at levels $(M,\sigma,\delta)$.
\begin{enumerate}[label=(\alph*)]
\item If
\begin{equation}
\label{eq:cvx-window}
\max\Bigl\{\frac{LR^2}{C_{\mathrm{cap}}\sqrt d},\
\Bigl(\frac{L}{C_{\mathrm{cap}}L_\beta^{2/\beta}}\Bigr)^{\frac\beta{\beta-2}}\Bigr\}
\le\eps\le\min\bigl\{c'_\beta L_\beta R^\beta,\ c'_\beta LR^2,\ \sigma\sqrt d\bigr\},
\end{equation}
then
\[
\frac{c_\beta}{C_{\mathrm{cap}}}\,\frac{d^2\sigma^2LR^2}{\eps^3}\le T^*(\eps;L)
\le\frac{2048\,C_\sigma+294}{9}\,\frac{d^2\sigma^2LR^2}{\eps^3},
\]
with $C_\sigma=(1+2^{-1/2})^2$. The upper bound is attained by
\cref{thm:value-analysis}(b) with $\lceil8LR^2/\eps\rceil$ rounds, and with
twice the constant by its accelerated form (d) with
$\lceil\sqrt{32LR^2/\eps}\,\rceil$ rounds; neither uses $\beta$ or $L_\beta$.
\item If $L\ge L_\eps$ and $L_\beta R^\beta d^{-\beta/4}\le\eps\le c'_\beta L_\beta R^\beta$, then
\[
c_\beta\frac{d^2\sigma^2}{\eps^2}\Bigl(\frac{L_\beta R^\beta}{\eps}\Bigr)^{2/\beta}
\le T^*(\eps;L)\le\frac{2^{12}C_\sigma C_{\mathrm{cap}}}9\,\frac{V}{L_\eps}\,
\frac{d^2\sigma^2}{\eps^2}\Bigl(\frac{L_\beta R^\beta}{\eps}\Bigr)^{2/\beta}+32d\Bigl(1+\frac{LR^2}\eps\Bigr),
\]
where $V=L$, or $V=\min\{L,L_\beta(2R)^{\beta-2}\}$ if $2<\beta\le3$ and
$\eps\le8L_\beta(2R)^\beta$.
\end{enumerate}
The window \eqref{eq:cvx-window} is nonempty if and only if its left end is
at most its right end; this requires $LR^2\le c\,L_\beta R^\beta$ and
$d\ge c''$ for numbers $c,c''$ depending only on $\beta$.
\end{acorollary}

\begin{proof}
(a) The lower end of the window makes the last term of \eqref{eq:cvx-H1} the
smallest one, which gives the lower bound, since Gaussian noise is
admissible at every level. For the upper bound apply
\cref{thm:value-analysis}(b) on $Q=B(R)$, $D=2R$, with $h=\sqrt{\eps/(8L)}$;
its admissible levels are those defining $T^*(\eps;L)$, and
\eqref{eq:cvx-value-T-M} gives
$T\le\frac{2048C_\sigma}9\frac{d^2\sigma^2LR^2}{\eps^3}+\frac{272}9\frac{dLR^2}\eps+\frac{22}9d$.
Since $\eps\le\sigma\sqrt d$ and $LR^2/\eps\ge1/c'_\beta\ge1$, the last two
terms are at most $\frac{294}9d^2\sigma^2LR^2/\eps^3$. For acceleration use
$\lceil\sqrt{32x}\rceil\le16x+3/2$ with $x=LR^2/\eps$ in part (d) of that
theorem.

(b) Since $L\ge L_\eps$, the proof of \cref{thm:convex-lower-main}(a)
applies with $\ell=L_\beta$ whenever $\eps\le c'_\beta L_\beta R^\beta$, and the lower end
of the range makes the second term of \eqref{eq:cvx-H1} the smallest one. The upper bound is the value-scale bound
of \cref{app:convex-proof-main} with
$d^2\sigma^2VR^2/\eps^3=(V/L_\eps)\,C_{\mathrm{cap}}(d^2\sigma^2/\eps^2)(L_\beta R^\beta/\eps)^{2/\beta}$.
\end{proof}

\subsection{Admissible level of Lipschitz noise}
\label{app:convex-lipschitz}

\begin{acorollary}[Admissible level of Lipschitz noise]
\label{cor:mmax-cvx}
Let $\beta>2$, in either geometry, let $\eps>0$ and $\sigma\ge0$, and let
$M_{\max}(\eps,T)$ be the admissible level of \cref{sec:setup} on
$\Fcvx(\beta,L_\beta;L,R)$. Then $M_{\max}(\eps,T)\ge\eps/(16R)$ for every
\[
T\ge\frac{2048\,C_\sigma}{9}\,\frac{d^2\sigma^2LR^2}{\eps^3}
+\frac{272}{9}\,\frac{dLR^2}{\eps}+\frac{22}{9}\,d .
\]
If $\eps<LR^2/2$, then also $M_{\max}(\eps,T)\le2\eps/R$ for every $T$
(\cref{prop:lipschitz-floor}). The constants $1/16$ and $2$ do not depend on
$d$, $T$, $\sigma$, $\beta$ or the geometry.
\end{acorollary}

\begin{proof}
Apply \cref{thm:value-analysis}(b) with $Q=B(R)$, $x_0=0$, $D=2R$ and
$\delta=0$. It applies to every $f\in\Fcvx(\beta,L_\beta;L,R)$, for every $\beta$
and every $\eps>0$, with $\bar M=\eps/(8D)=\eps/(16R)$, and for $D=2R$
the right-hand side of \eqref{eq:cvx-value-T-M} is the stated number of
calls. The upper bound is the first consequence listed in
\cref{prop:lipschitz-floor}, with $p=2$.
\end{proof}

\section{The smooth convex lower bound}
\label{app:convex-lower}

This appendix proves \cref{thm:convex-lower-main}(a). It uses the hidden-map
construction of \citet{RajaramanHan2026}, whose hard functions are convex and
Lipschitz but not differentiable; theorem and lemma numbers of that paper
refer to its version~3. \Cref{sec:smooth-family} smooths the family: the
Euclidean norm in the hidden tube is replaced by $\sqrt{1+\norm{\cdot}^2}-1$,
and the hinge that joins the function to a conic branch outside the unit ball
\citep[Section~4.2]{RajaramanHan2026} is replaced by a hinge of class
$C^{q+2}$. \Cref{sec:information} carries the information argument of
\citet{RajaramanHan2026} over to this family with explicit constants, and
\cref{sec:convex-lower} rescales the family to the convex classes.

\subsection{A smooth hidden-map family}
\label{sec:smooth-family}

Throughout this appendix $d=2d_0$, $\beta>2$, $q=\lceil\beta\rceil-1$ and
$\gamma=\beta-q$. We write $B_2^n(R)$ for the closed Euclidean ball of
radius $R$ in $\R^n$ and $B_2^n=B_2^n(1)$. Norms of derivative tensors are
multilinear operator norms. The symbols $W$, $u$, $\pi$, $M_j$, $S_k$, $E_k$,
$K_\beta^{\mathrm{tube}}$ and $U_j$ are local to this appendix: $W$ is a
matrix, and the constants are unrelated to the Hessian bound $L$ and to the
stencil constants $S_1,S_2$.

\paragraph{The tube.}
For $0<\vartheta\le1/2$, $r\ge16\vartheta$, a $d_0\times d_0$ matrix $W$, a
vector $u\in\R^{d_0}$ and $z=(x,y)$ with $x,y\in\R^{d_0}$, set
\begin{align}
 B_W(x,y)&=Wx/r-y/(8\vartheta),&
 a_W(x,y)&=\sqrt{1+\norm{B_W(x,y)}^2}-1,\nonumber\\
 b_u(x,y)&=\tfrac12\norm{x}^2-\inner{x}{u},&
 f_{W,u}(x,y)&=\vartheta\{\log(e^{a_W(x,y)}+e^{b_u(x,y)})-\log2\}.
 \label{eq:smooth-tube}
\end{align}
With $\norm{B_W}$ in place of $a_W$ and $\eps=\vartheta$, this is the hard
function of \citet[Section~2]{RajaramanHan2026}. The bounded parameter set is
$E_b=\{\norm{W}_{\op}\le4,\ \norm{u}\le3/8\}$.

\begin{alemma}[Geometry of the smooth tube]\label{lem:tube-geometry}
The function $f=f_{W,u}$ is smooth and convex on $\R^{2d_0}$, and
$f(0)=0$. It has the global minimizer
$z_*=(u,(8\vartheta/r)Wu)$, and for every $z=(x,y)$,
\[
 f(z)-f(z_*)\ge
 \frac{\vartheta}{2(1+\exp(\norm{u}^2/2))}\norm{x-u}^2.
\]
On $E_b$, $\norm{z_*}\le\rho=3\sqrt5/8$ and $f$ is $1$-Lipschitz
on the unit ball.
\end{alemma}
\begin{proof}
The proof follows that of \citet[Lemmas~A.1 and A.2]{RajaramanHan2026}
for the tube $\norm{B_W}$. Both arguments of the log-sum-exp are convex, and log-sum-exp is
convex and coordinatewise increasing; the minima $0$ of $a_W$ and
$-\norm{u}^2/2$ of $b_u$ are both attained at $z_*$. Since $a_W\ge0$,
$b_u(z)+\norm{u}^2/2=\norm{x-u}^2/2$, and $h(t)=\log(1+e^t)$ satisfies
$h'(t)\ge(1+e^{\norm{u}^2/2})^{-1}$ for $t\ge-\norm{u}^2/2$, integrating $h'$
gives the error bound. On $E_b$, $\nabla f$ is a convex combination of
$\vartheta\nabla a_W(z)=\vartheta B_W^\top B_W(z)/(1+\norm{B_W(z)}^2)^{1/2}$,
where $B_W^\top$ is the adjoint of the linear map $B_W$, and of
$\vartheta(x-u,0)$. Their norms are at most
$\vartheta\norm{B_W}_{\op}\le(16\vartheta^2/r^2+1/64)^{1/2}\le\sqrt5/8$ and,
on the unit ball, $\vartheta(1+3/8)\le11/16$. Finally
$\norm{(8\vartheta/r)Wu}\le2\norm{u}$, so $\norm{z_*}\le\sqrt5\,\norm{u}\le\rho$.
\end{proof}

\paragraph{A smooth hinge.}
Let $n=q+1$ and $\lambda_0=1/64$, and define
\[
 \varphi_n(t)=\frac{(2n+1)!}{(n!)^2}t^n(1-t)^n\mathbf1_{[0,1]}(t),
 \quad H_n(t)=\int_{-\infty}^t\varphi_n(v)\,dv,
 \quad \Psi(s)=\lambda_0\int_{-\infty}^{(s+\lambda_0)/\lambda_0}H_n(t)\,dt.
\]
This hinge replaces the $C^2$ hinge of \citet[Appendix~C]{RajaramanHan2026}.
It satisfies $\Psi\in C^{n+1}=C^{q+2}$, $0\le\Psi'\le1$, $\Psi''\ge0$ and
$\Psi(s)=0$ for $s\le-\lambda_0$. Symmetry of $\varphi_n$ gives
$\Psi(s)=s+\lambda_0/2$ for $s\ge0$, and convexity gives
$\Psi(s)\ge s+\lambda_0/2$ for every $s$. Put
\[
 V_n=\max_t\varphi_n(t)=\frac{(2n+1)!}{(n!)^2 4^n},\qquad
 \vartheta_\beta=\min\{1/2,1/V_n\}.
\]
Since $\Psi''(s)=\varphi_n((s+\lambda_0)/\lambda_0)/\lambda_0$, we have
$\lambda_0\vartheta\norm{\Psi''}_\infty\le1$ whenever
$\vartheta\le\vartheta_\beta$. Put $U_1=1$ and, for $2\le j\le q+1$,
$U_j=\lambda_0^{1-j}\norm{\varphi_n^{(j-2)}}_{\infty,[0,1]}$; these are
maxima of fixed polynomials, and $\abs{\Psi^{(j)}}\le U_j$ for
$1\le j\le q+1$.

\paragraph{Global extension.}
Set $\Gamma(z)=\lambda_0+(\norm{z}-1)/2$ and define
\begin{equation}
 g_{W,u}(z)=
 \begin{cases}
 \Gamma(z)+\Psi(\lambda_0 f_{W,u}(z)-\Gamma(z)-2\lambda_0),&\norm{z}<1,\\
 \Gamma(z),&\norm{z}\ge1.
 \end{cases}
 \label{eq:smooth-extension}
\end{equation}
For arbitrary $(W,u)$, the same formula defines the mean $\mu_{W,u}$ of the
auxiliary experiment in \cref{sec:information}.

\begin{alemma}[Smooth extension and global decoding]\label{lem:extension}
For $(W,u)\in E_b$, $g$ is convex, globally $1/2$-Lipschitz, and
$C^{q+1}$ on $\R^{2d_0}$. It equals $\Gamma$ whenever
$\norm{z}\ge30/31$, and $z_*$ is a global minimizer. Let
$\Pi(z)=z/(1\vee\norm{z})$ be the radial projection and $\Pi^1(z)\in\R^{d_0}$
its $x$-component. Then
\begin{align}
 g(z)-g(z_*)&\ge\lambda_0\{f(\Pi(z))-f(z_*)\}\nonumber\\
 &\ge\frac{\lambda_0\vartheta}{2(1+e^{9/128})}
       \norm{\Pi^1(z)-u}^2.
 \label{eq:smooth-decoder}
\end{align}
\end{alemma}
\begin{proof}
Apart from the smoothness, these statements are
\citet[Lemma~4.3]{RajaramanHan2026} together with its proof in
\citet[Appendix~C.1]{RajaramanHan2026}, where the radius $30/31$, the
Lipschitz constant $1/2$ and the first inequality in \eqref{eq:smooth-decoder}
are obtained for their tube and hinge. That proof uses only the following
properties, which hold here by \cref{lem:tube-geometry} and the properties of
$\Psi$ listed above: $f$ is convex and $1$-Lipschitz on the unit ball,
$f(0)=0$, and $z_*$ minimizes $f$ with $\norm{z_*}\le\rho<29/33$; $\Psi$ is
convex, $0\le\Psi'\le1$, $\Psi=0$ on $(-\infty,-\lambda_0]$,
$\Psi(s)=s+\lambda_0/2$ on $[0,\infty)$ and $\Psi(s)\ge s+\lambda_0/2$.
In that proof $g$ is written as $\Gamma+\Psi(\lambda_0\widetilde f-\Gamma-2\lambda_0)$
with the convex $1$-Lipschitz extension
$\widetilde f(z)=\inf_{v\in B_2^{2d_0}}\{f(v)+\norm{z-v}\}$, which equals $f$ on
the unit ball. Since $\widetilde f(z)\le\norm{z}$, the hinge argument is at
most $(29-31\norm{z})/64$, hence at most $-\lambda_0$ for $\norm{z}\ge30/31$,
where the hinge term therefore vanishes; so this function agrees with
\eqref{eq:smooth-extension}. The second inequality in
\eqref{eq:smooth-decoder} follows from \cref{lem:tube-geometry} and
$\norm{u}^2/2\le9/128$.

It remains to prove smoothness. For $\norm{z}\le1/4$, $f(z)\ge-\norm{z}$, so
the hinge argument is at least $(29-33\norm{z})/64>0$ and
$g=\lambda_0f-3\lambda_0/2$; the nonsmooth norm in $\Gamma$ cancels. On
$0<\norm{z}<1$ all ingredients are $C^{q+1}$, and $g=\Gamma$ on the collar
$30/31\le\norm{z}\le1$ and outside it. Hence $g\in C^{q+1}(\R^{2d_0})$.
\end{proof}

\paragraph{Derivative constants.}
For $k\ge1$, let $(1/2)_{\underline{j}}$ denote the falling factorial and set
\[
 P_k=\sum_{a+2b=k}\frac{k!\,2^a}{a!b!}
       \abs{(1/2)_{\underline{a+b}}},\qquad M_j=2^{j-1}(j-1)!.
\]
Let $\mathfrak P_k$ be the set of partitions of $\{1,\ldots,k\}$, and define
\begin{align*}
 C_k&=P_k+\sum_{\pi\in\mathfrak P_k}M_{\abs{\pi}}
                   \prod_{A\in\pi}(P_{\abs{A}}+2),\\
 G_k&=\tfrac12P_k4^{k-1},\qquad S_k=\lambda_0 C_k+G_k,\\
 E_k&=G_k+\lambda_0 C_k+
       \sum_{\pi\in\mathfrak P_k}U_{\abs{\pi}}\prod_{A\in\pi}S_{\abs{A}}
       \quad(1\le k\le q+1),\\
 K_\beta^{\mathrm{tube}}&=(2E_q)^{1-\gamma}E_{q+1}^{\gamma}.
\end{align*}

\begin{alemma}[Global smoothness and Hessian cap]\label{lem:derivatives}
For $0<\vartheta\le\vartheta_\beta$, $r\ge16\vartheta$, and $(W,u)\in E_b$,
\[
 \sup_z\norm{D^k g(z)}_{\op}\le E_k\vartheta^{1-k}
 \quad(1\le k\le q+1),\qquad
 [g]_{\beta,\op}\le K_\beta^{\mathrm{tube}}\vartheta^{1-\beta},\qquad
 0\preceq\nabla^2g\preceq E_2\vartheta^{-1}I.
\]
The constants depend only on $\beta$.
\end{alemma}
\begin{proof}
We bound the derivatives of $f$ on the unit ball, then those of $g$ on the
regions $\norm{z}\le1/4$, $1/4\le\norm{z}\le1$ and $\norm{z}\ge1$, and
interpolate.

\emph{Step 1: derivatives of $f$.} Fa\`a di Bruno's formula for
$h(v)=\sqrt{1+\norm{v}^2}$ has only blocks of size one or two. There are
$k!/(a!b!2^b)$ partitions with $a$ singletons and $b$ pairs, and the sum of
their absolute contributions is at most the corresponding summand of $P_k$
times $(1+\norm{v}^2)^{(1-k)/2}$. Thus $\norm{D^kh(v)}_{\op}\le P_k$, and,
as $\norm{B_W}_{\op}\le\vartheta^{-1}$,
$\norm{D^ka_W}_{\op}\le P_k\vartheta^{-k}$. On the unit ball the nonzero
derivatives of $b_u$ have norms at most $2$. Write
\[
 f=\vartheta a_W+\vartheta h_0(b_u-a_W)-\vartheta\log2,\qquad
 h_0(t)=\log(1+e^t).
\]
Here $h_0'=p$ with $p(t)=e^t/(1+e^t)\in[0,1]$, and each further derivative
is obtained by applying $p(1-p)\,d/dp$ to the preceding polynomial in $p$.
By induction, $h_0^{(j)}$ is a polynomial in $p$ whose absolute
coefficients sum to at most $M_j$. Fa\`a di Bruno's formula for the
composition and $\vartheta\le1$ give
$\norm{D^kf}_{\op}\le C_k\vartheta^{1-k}$ on the unit ball.

\emph{Step 2: derivatives of $g$.} The same computation for $\norm{z}$ shows
$\norm{D^k\Gamma(z)}_{\op}\le G_k$ for $\norm{z}\ge1/4$. On the annulus
$1/4\le\norm{z}\le1$ the derivatives of the hinge argument are at most
$S_k\vartheta^{1-k}$. In Fa\`a di Bruno's formula for $\Psi$ composed with
the hinge argument, the term of a partition $\pi$ with $j=\abs{\pi}$ blocks is
at most $U_j\prod_{A\in\pi}S_{\abs{A}}$ times
$\vartheta^{j-k}\le\vartheta^{1-k}$. This proves the bound $E_k\vartheta^{1-k}$
on the annulus. The same bound holds in the inner ball, where
$g=\lambda_0 f-3\lambda_0/2$, and outside the unit ball, where $g=\Gamma$.

\emph{Step 3: H\"older and Hessian bounds.} With $t=\norm{z-z'}$, the uniform
bounds on $D^qg$ and $D^{q+1}g$ give
\[
 \norm{D^qg(z)-D^qg(z')}_{\op}
 \le\min\{2E_q\vartheta^{1-q},E_{q+1}\vartheta^{-q}t\}
 \le K_\beta^{\mathrm{tube}}\vartheta^{1-\beta}t^\gamma.
\]
Convexity and the case $k=2$ give the Hessian inequalities.
\end{proof}

\subsection{Information lower bound for the smooth tube}
\label{sec:information}

The next theorem is the estimation lower bound of
\citet[Theorem~4.2]{RajaramanHan2026} for the smooth family, with explicit
constants and with the amplitude and the noise level kept separate.

\begin{atheorem}[Smooth-tube information lower bound]
\label{thm:information}
Let $d_0\ge128$, $A,\sigma>0$, $0<\vartheta\le\vartheta_\beta$, and
$r\ge\max\{16\vartheta,d_0^{-1/4}\}$. Let $\pi$ be the prior under which
$\operatorname{vec}(W)\sim N(0,I_{d_0^2}/d_0)$ and
$u\sim N(0,I_{d_0}/(16d_0))$ are
independent, and let $\pi_b=\pi(\cdot\mid E_b)$. The unknown function is
$F_{W,u}=A g_{W,u}$, and a query at $z_t\in\R^{2d_0}$ returns
$F_{W,u}(z_t)+\sigma Z_t$, where $Z_1,Z_2,\dots$ are independent standard
Gaussian variables, independent of $(W,u)$ and of the seed of the algorithm.
Queries and the output $\widehat z$ are unrestricted, and the algorithm may
know $A,\sigma,\vartheta,r,\beta$ and $\pi$. Put
$\nu=A\lambda_0\vartheta/\sigma$. If
\begin{equation}
 T<\frac{1}{65536\nu^2}
       \min\{d_0^{5/2},d_0^2/r^2\},
 \label{eq:information-budget}
\end{equation}
then every adaptive randomized algorithm with at most $T$ queries has Bayes
error $\E[F_{W,u}(\widehat z)-\min F_{W,u}]$ under $\pi_b$ greater than
$A\lambda_0\vartheta/4096$. In particular, its expected error is greater than
$A\lambda_0\vartheta/4096$ for some $(W,u)\in E_b$.
\end{atheorem}

\begin{proof}
We follow the proof of the estimation bound of
\citet[Theorem~4.2]{RajaramanHan2026} in Appendix~C.3 there. It combines the
reduction of their Lemma~2.7 with their Lemma~C.1, which is the version of
their Lemmas~2.1--2.3 for the auxiliary mean. Dividing the observations by
$\sigma$ gives their model with the mean multiplied by $A/\sigma$, so that
$\nu$ plays the role of their $\lambda\eps$. Apart from the decoding
inequality \eqref{eq:smooth-decoder}, these proofs use the tube and the hinge
only through the parameter derivatives of the mean (Step~2) and an
anti-concentration bound for the weight of the second branch (Step~5). We
verify both for the smooth family and keep track of the constants.

Write $\theta=(W,u)$. Let $\mathcal H_t$ be generated by the seed of the
algorithm and the first $t$ query--response pairs, so that $z_t$ is
$\mathcal H_{t-1}$-measurable. For a density $\rho$ let
$J(\rho)=\int(\nabla\log\rho)(\nabla\log\rho)^\top\rho$. In the
\emph{auxiliary experiment} $\theta\sim\pi$, and a query at $z_t$ returns
$A\mu_\theta(z_t)+\sigma Z_t$; put $h_\theta=A\mu_\theta/\sigma$.

\emph{Step 1: reduction to the auxiliary experiment.} For $g\sim N(0,I_{d_0})$
and $a>1$ the Chernoff bound gives
$\PP(\norm{g}^2\ge a d_0)\le\exp\{-d_0(a-1-\log a)/2\}$, which is at most
$e^{-d_0/2}$ for $a=4$ and at most $e^{-d_0/8}$ for $a=9/4$. For a
$d_0\times d_0$
matrix $G$ with independent standard Gaussian entries and a $1/4$-net
$\mathcal N$ of the unit sphere with $\abs{\mathcal N}\le9^{d_0}$, we have
$\norm{G}_{\op}\le(4/3)\max_{v\in\mathcal N}\norm{Gv}$; the case $a=9$ and a
union bound give $\PP(\norm{G}_{\op}>4\sqrt{d_0})
\le e^{-d_0(4-3\log3)}\le e^{-d_0/2}$.
For $W=d_0^{-1/2}G$ and $u=g/(4\sqrt{d_0})$ this yields
\[
 \PP(E_b^c)\le e^{-d_0/2}+e^{-d_0/8}\le2e^{-d_0/8},\qquad
 \PP(\norm{u}>1/2)\le e^{-d_0/2}.
\]
Let $\widehat u=\Pi^1(\widehat z)$, so that $\norm{\widehat u}\le1$. Since
$\E\norm{u}^4=(1+2/d_0)/256$ and $d_0\ge128$, the Cauchy--Schwarz inequality gives
\begin{equation}
 \E\bigl[\norm{\widehat u-u}^2\mathbf1_{E_b^c}\bigr]
 \le2\PP(E_b^c)+2\bigl(\E\norm{u}^4\,\PP(E_b^c)\bigr)^{1/2}
 \le5e^{-d_0/16}\le1/512.
 \label{eq:explicit-tail-transfer}
\end{equation}
On $E_b$, $\mu_\theta=g_\theta$ at every point of $\R^{2d_0}$, so, as in
\citet[Lemma~2.7]{RajaramanHan2026}, the law of the whole trajectory of the
auxiliary experiment conditional on $E_b$ is that of the actual experiment
with prior $\pi_b$. With \eqref{eq:explicit-tail-transfer} this gives
$\E_{\pi,\mu}\norm{\widehat u-u}^2\le\E_{\pi_b,g}\norm{\widehat u-u}^2+1/512$,
where the subscripts indicate the prior and the mean function. On $E_b$,
\eqref{eq:smooth-decoder} and $e^{9/128}<3$ give
$F_\theta(z)-\min F_\theta\ge(A\lambda_0\vartheta/8)\norm{\Pi^1(z)-u}^2$.
Hence Bayes error at most $A\lambda_0\vartheta/4096$ under $\pi_b$ would imply
$\E_{\pi_b,g}\norm{\widehat u-u}^2\le1/512$ and
$\E_{\pi,\mu}\norm{\widehat u-u}^2\le1/256$. It remains to show that under
\eqref{eq:information-budget} no estimator has
$\E\norm{\widehat u-u}^2\le1/256$ in the auxiliary experiment; all
expectations below refer to this experiment.

\emph{Step 2: parameter derivatives.} At a query with $\norm{z_t}\ge1$,
$\mu_\theta(z_t)=\Gamma(z_t)$ does not depend on $\theta$. For $\norm{z_t}<1$
put
\[
 p_t=\frac{e^{b_u(z_t)}}{e^{a_W(z_t)}+e^{b_u(z_t)}},\qquad
 s_t=\lambda_0 f_\theta(z_t)-\Gamma(z_t)-2\lambda_0 .
\]
As in \citet[Appendix~C.2]{RajaramanHan2026}, with $\eps=\vartheta$ and
$\lambda=\lambda_0$,
\begin{align}
 \nabla_u\mu_\theta(z_t)&=-\lambda_0\vartheta\Psi'(s_t)p_t x_t,\nonumber\\
 \nabla_u^2\mu_\theta(z_t)&=\bigl[\lambda_0\vartheta\Psi'(s_t)p_t(1-p_t)
      +\lambda_0^2\vartheta^2\Psi''(s_t)p_t^2\bigr]x_tx_t^\top,\nonumber\\
 \norm{\nabla_W\mu_\theta(z_t)}_F&\le\lambda_0\vartheta/r;
 \label{eq:parameter-derivatives}
\end{align}
the last bound holds because $\partial f_\theta/\partial a_W\le\vartheta$ and
$\nabla_W a_W(z_t)=B_W(z_t)x_t^\top/(r(1+\norm{B_W(z_t)}^2)^{1/2})$ has
Frobenius norm at most $\norm{x_t}/r\le1/r$. Put $p_t=0$ if $\norm{z_t}\ge1$,
and
\[
 I_0=0,\qquad
 I_T=\nu^2\sum_{t=1}^Tp_t^2x_tx_t^\top,
 \qquad \overline I_T=\E[I_T\mid\mathcal H_T].
\]
Since $0\le\Psi'\le1$, $\lambda_0\vartheta\norm{\Psi''}_\infty\le1$ and
$\norm{x_t}\le1$, the bracket in \eqref{eq:parameter-derivatives} lies in
$[0,\lambda_0\vartheta p_t]$, and $B_t=\nabla_u^2h_\theta(z_t)$ satisfies
\begin{equation}
 (\nabla_u h_\theta)(\nabla_u h_\theta)^\top\preceq I_t-I_{t-1},
 \qquad B_t^2\preceq I_t-I_{t-1}.
 \label{eq:proxy-domination}
\end{equation}
These bounds, the Gaussian prior and the at most linear growth of
$\mu_\theta$ in $\theta$ justify the differentiations under the integral sign
and the integrations by parts below, first with compact cutoffs.

\emph{Step 3: posterior spread.} Let $\rho_{T,u}$ be the marginal posterior
density of $u$. Given \eqref{eq:proxy-domination}, the proof of
\citet[Lemma~3.1]{RajaramanHan2026} with $\eta=1$ and $\delta=1/4$ applies
verbatim. It gives
\begin{align}
 J(\rho_{T,u})&\preceq16d_0I_{d_0}+\overline I_T-\overline C_T,
 \qquad C_T=\sum_{t=1}^TZ_tB_t,
 \quad\overline C_T=\E[C_T\mid\mathcal H_T],
 \label{eq:posterior-hessian}\\
 J(\rho_{T,u})&\preceq17d_0I_{d_0}+\tfrac32\overline I_T
 \quad\text{with probability at least }3/4.
 \label{eq:posterior-fisher-explicit}
\end{align}
Here $16d_0I_{d_0}$ is the prior precision of $u$. For the second bound, that proof
applies the matrix supermartingale inequality
\citep[Lemma~2.1]{Tropp2011MatrixFreedman} to $X_t=-Z_tB_t$, using
$\E[e^{-Z_tB_t}\mid\theta,\mathcal H_{t-1},z_t]=e^{B_t^2/2}$ and the second
bound in \eqref{eq:proxy-domination}, and obtains
$-\overline C_T\preceq\overline I_T/2+\log(4d_0)I_{d_0}$ with probability at least
$3/4$; then $\log(4d_0)\le d_0$ gives \eqref{eq:posterior-fisher-explicit}.

Now suppose that $\E\norm{\widehat u-u}^2\le1/256$. The posterior mean
minimizes the conditional risk, so Markov's inequality gives
$\tr\operatorname{Cov}(u\mid\mathcal H_T)\le1/64$ with probability at least
$3/4$. Let $S_T$ be the intersection of this event with the event in
\eqref{eq:posterior-fisher-explicit}; then $\PP(S_T)\ge1/2$. The Cram\'er--Rao
inequality $\operatorname{Cov}(u\mid\mathcal H_T)\succeq J(\rho_{T,u})^{-1}$
(integration by parts gives $\E[(u-\E u)(\nabla\log\rho)^\top]=-I$, and a
Schur complement concludes) shows that on $S_T$ the trace of the inverse of
$17d_0I_{d_0}+\frac32\overline I_T$ is at most $1/64$. If fewer than $d_0/2$
eigenvalues of $\overline I_T$ were at least $d_0$, this trace would exceed
$(d_0/2)/((37/2)d_0)=1/37>1/64$. Thus on $S_T$ at least $d_0/2$ eigenvalues of
$\overline I_T$ are at least $d_0$.

\emph{Step 4: information about $W$.} Let $\rho_{t,W}$ be the marginal
posterior density of $W$ and
$K_t=d_0^{-2}\operatorname{Tr}_{\rm out}J(\rho_{t,W})$,
where $\operatorname{Tr}_{\rm out}$ sums the $d_0$ diagonal blocks indexed by
the output rows of $W$ \citep[Section~2]{RajaramanHan2026}. The proof of
\citet[Lemma~2.2(1)--(2)]{RajaramanHan2026} in Section~3.2 there applies
verbatim with $\norm{\nabla_Wh_\theta}_F\le\nu/r$ from
\eqref{eq:parameter-derivatives}: $K_0=I_{d_0}$,
$K_t\preceq\E[K_{t+1}\mid\mathcal H_t]$,
and
\begin{equation}
 \E\tr K_T\le d_0+\frac{\nu^2T}{r^2d_0^2}.
 \label{eq:map-information}
\end{equation}

\emph{Step 5: anti-concentration.} \Citet[Lemma~3.2]{RajaramanHan2026}
bounds small-ball probabilities by Fisher information with an unspecified
constant; we use an explicit version. Let $Z\in\R^{d_0}$, $d_0\ge3$, have a
differentiable density with scalar Fisher information $J_{\rm tr}<\infty$.
Integration by parts against $V_\delta(w)=(w-z)/(\norm{w-z}^2+\delta)$,
whose divergence is at least $(d_0-2)/(\norm{w-z}^2+\delta)$, and the
Cauchy--Schwarz inequality give
$(d_0-2)\E(\norm{Z-z}^2+\delta)^{-1}
\le(J_{\rm tr}\E(\norm{Z-z}^2+\delta)^{-1})^{1/2}$.
Letting $\delta\downarrow0$ gives Hardy's inequality
$\E\norm{Z-z}^{-2}\le J_{\rm tr}/(d_0-2)^2$, and Markov's inequality gives
\begin{equation}
 \PP(\norm{Z-z}\le s)
 \le \frac{s^2J_{\rm tr}}{(d_0-2)^2}
 \le \frac{9s^2J_{\rm tr}}{d_0^2}.
 \label{eq:hardy-small-ball}
\end{equation}

Fix a history and an in-ball query with $x_t\ne0$ (the case $x_t=0$ is
trivial), and let $Z=Wx_t$ with $W$ distributed according to
$\rho_{t-1,W}$. The argument in \citet[Section~3.2]{RajaramanHan2026} relates
the scores of $W$ and $Z$ and gives
\begin{equation}
 \frac{\tr J(Z)}{d_0^2}x_tx_t^\top\preceq K_{t-1},
 \qquad
 \frac{\tr J(Z)}{d_0^2}
       x_t^\top K_{t-1}^{-1}x_t\le1.
 \label{eq:score-contraction}
\end{equation}
The matrix $K_{t-1}$ is positive definite, since a positive differentiable
density whose Fisher matrix has a null direction is constant along lines in
that direction and hence not integrable.

Let $E=\{\norm{u}\le1/2\}$ and $H_t=x_t^\top K_{t-1}^{-1}x_t$. On $E$,
$b_u(z_t)\le1$, so $p_t^2\mathbf1_E\le\exp[-2(a_W(z_t)-1)_+]$, which depends
on $\theta$ only through $W$. This is the step that changes for the smooth
tube: for $s\ge0$,
\[
 \{a_W(z_t)\le s\}
 =\bigl\{\norm{Wx_t-r y_t/(8\vartheta)}\le r\sqrt{s^2+2s}\bigr\},
\]
whereas for the tube $\norm{B_W}$ the radius is $rs$. Using
$e^{-2(a-1)_+}=\int_1^\infty2e^{-2(s-1)}\mathbf1_{\{a\le s\}}\,ds$,
\eqref{eq:hardy-small-ball} and \eqref{eq:score-contraction}, we obtain
\begin{equation}
 H_t\E[p_t^2\mathbf1_E\mid\mathcal H_{t-1},z_t]
 \le9r^2\int_1^\infty2e^{-2(s-1)}(s^2+2s)\,ds
 =\tfrac{99}{2}r^2\le50r^2.
 \label{eq:smooth-anticoncentration}
\end{equation}

\emph{Step 6: projections and the contradiction.} Let
$\Lambda=16(1+\nu^2T/(r^2d_0^3))$. The construction of
\citet[Section~3.3]{RajaramanHan2026} gives $\mathcal H_t$-measurable
orthogonal projections $I_{d_0}=Q_0\succeq Q_1\succeq\cdots\succeq Q_T$ with
$Q_tK_tQ_t\preceq\Lambda Q_t$: $Q_t$ is obtained from $Q_{t-1}$ by removing
the eigendirections of $Q_{t-1}K_tQ_{t-1}$ with eigenvalues larger than
$\Lambda$. The submartingale property and \eqref{eq:map-information} give,
as there,
\begin{equation}
 \E\tr(I_{d_0}-Q_T)\le d_0/16.
 \label{eq:projection-codimension}
\end{equation}
Moreover $K_t^{-1}\succeq Q_t/\Lambda$: the bound
$\norm{K_t^{1/2}Q_t}_{\op}^2\le\Lambda$ gives
$K_t^{1/2}Q_tK_t^{1/2}\preceq\Lambda I$, and conjugation by $K_t^{-1/2}$
gives the claim. Hence
$x_t^\top Q_Tx_t\le x_t^\top Q_{t-1}x_t\le\Lambda H_t$, and
\eqref{eq:smooth-anticoncentration} gives
$\E[\tr(I_TQ_T)\mathbf1_E]\le50\Lambda\nu^2Tr^2$. Since
$\tr I_T\le\nu^2T$, $\PP(E^c)\le e^{-d_0/2}\le d_0^{-1/2}\le r^2$ and
$\E\tr(\overline I_TQ_T)=\E\tr(I_TQ_T)$, we get
\begin{equation}
 \E\tr(\overline I_TQ_T)
 \le801\left(\nu^2Tr^2+\frac{\nu^4T^2}{d_0^3}\right).
 \label{eq:explicit-tension}
\end{equation}

On the other hand, by \eqref{eq:projection-codimension} and Markov's
inequality, $\operatorname{rank}Q_T\ge3d_0/4$ with probability at least $3/4$.
On the intersection of this event with $S_T$, which has probability at least
$1/4$, the span of the eigenvectors of $\overline I_T$ with eigenvalues at
least $d_0$ meets $\operatorname{range}Q_T$ in dimension at least $d_0/4$, so
$\tr(\overline I_TQ_T)\ge d_0^2/4$ and
\begin{equation}
 \E\tr(\overline I_TQ_T)\ge d_0^2/16.
 \label{eq:target-information-lower}
\end{equation}
By \eqref{eq:information-budget}, $\nu^2Tr^2<2^{-16}d_0^2$ and
$\nu^4T^2/d_0^3<2^{-32}d_0^2$, so the right side
of~\eqref{eq:explicit-tension} is less than
$801(2^{-16}+2^{-32})d_0^2<d_0^2/16$, which contradicts
\eqref{eq:target-information-lower}. Hence no estimator has
$\E\norm{\widehat u-u}^2\le1/256$ in the auxiliary experiment, and Step~1
completes the proof.
\end{proof}

\subsection{Scaling to the convex classes}
\label{sec:convex-lower}

Let $\mathcal C_{\beta,L_\beta,R}$ denote the convex functions
$F:\R^d\to\R$ of class $C^q$ with $[F]_{\beta,\op}\le L_\beta$ and a global
minimizer in $B_2^d(R)$. A query at $X_t\in\R^d$ returns $F(X_t)+\xi_t$,
where the $\xi_t$ are independent $N(0,\sigma^2)$ variables, independent of
the seed of the algorithm. Queries and the output are unrestricted. The
algorithm may know $\beta,L_\beta,R,\sigma,T$, the hard family and its
construction parameters, and, on the subclass with $\nabla^2F\preceq LI$,
also $L$.

\begin{alemma}[A quantitative hard family]\label{thm:convex-lower}
Let $\beta>2$, $L_\beta,R,\sigma>0$, and $d=2d_0$ with $d_0\ge128$, and let
$K_\beta^{\mathrm{tube}},E_2,\vartheta_\beta$ be the constants of
\cref{sec:smooth-family}.
For every accuracy scale
\[
 0<\eps_0\le
 \frac{L_\beta R^\beta\vartheta_\beta^\beta}{K_\beta^{\mathrm{tube}}},\qquad
 \vartheta=\left(\frac{K_\beta^{\mathrm{tube}}\eps_0}
 {L_\beta R^\beta}\right)^{1/\beta},\qquad
 r=\max\{16\vartheta,d_0^{-1/4}\},
\]
every adaptive randomized algorithm using an integer budget
\begin{equation}
 T<\frac{\sigma^2}{16\eps_0^2}
       \min\left\{d_0^{5/2},\frac{d_0^2}{256\vartheta^2}\right\}
 \label{eq:exact-smooth-lower-budget}
\end{equation}
has expected error greater than $\eps_0/262144$ on some
$F\in\mathcal C_{\beta,L_\beta,R}$.
The hard functions also satisfy
\begin{equation}
 0\preceq\nabla^2F\preceq
 \frac{E_2\eps_0}{R^2\vartheta^2}I
 =\frac{E_2}{K_\beta^{\mathrm{tube}}}
 L_\beta R^{\beta-2}\vartheta^{\beta-2}I.
 \label{eq:hard-hessian-cap}
\end{equation}
The same assertion therefore holds on the subclass with a known Hessian
cap $L$ whenever the right-hand side of \eqref{eq:hard-hessian-cap} is at
most $L$. Odd dimensions are covered by $d_0=\lfloor d/2\rfloor$ and one
unused coordinate.
\end{alemma}
\begin{proof}
Set $A=\eps_0/\vartheta$ and $F_\theta(x)=A g_\theta(x/R)$ for $\theta\in E_b$.
By the condition on $\eps_0$, $\vartheta\le\vartheta_\beta$.

\emph{The class.} By \cref{lem:extension}, $F_\theta$ is convex and has a
global minimizer in $B_2^d(R)$. By \cref{lem:derivatives},
\[
 [F_\theta]_{\beta,\op}
 \le A R^{-\beta}K_\beta^{\mathrm{tube}}\vartheta^{1-\beta}
 = K_\beta^{\mathrm{tube}}\eps_0R^{-\beta}\vartheta^{-\beta}=L_\beta,
\]
and $\nabla^2F_\theta\preceq AR^{-2}E_2\vartheta^{-1}I$, which is
\eqref{eq:hard-hessian-cap}.

\emph{The budget.} A query of $F_\theta$ at $X$ is a query of $Ag_\theta$ at
$X/R$, and errors are preserved. \Cref{thm:information} excludes Bayes error
at most $A\lambda_0\vartheta/4096=\eps_0/262144$ whenever
\[
 T<2^{-16}\left(\frac{\sigma}{A\lambda_0\vartheta}\right)^2
       \min\{d_0^{5/2},d_0^2/r^2\}.
\]
Since $A\vartheta=\eps_0$, $\lambda_0=1/64$ and
$r^2=\max\{256\vartheta^2,d_0^{-1/2}\}$, this is
\eqref{eq:exact-smooth-lower-budget}. Bayes error greater than
$\eps_0/262144$ implies the same lower bound for the worst case.

\emph{Odd dimensions.} Extending $F_\theta$ to $\R^{2d_0+1}$ by a coordinate on
which it does not depend preserves the class, the errors and, after
projection of the queries onto the first $2d_0$ coordinates, the Gaussian
experiment.
\end{proof}

\begin{proof}[Proof of \cref{thm:convex-lower-main}(a)]
Let $C_{\mathrm{cap}}=2^{18(1-2/\beta)}E_2
(K_\beta^{\mathrm{tube}})^{-2/\beta}$ and
$c'_0=\vartheta_\beta^\beta/(2^{18}K_\beta^{\mathrm{tube}})$. For
$L\in(0,\infty]$ put
\[
\ell=\min\bigl\{L_\beta,\ (L/C_{\mathrm{cap}})^{\beta/2}\eps^{-(\beta-2)/2}\bigr\},
\]
with $\ell=L_\beta$ if $L=\infty$. Apply \cref{thm:convex-lower} with $\ell$ in
place of $L_\beta$, $\eps_0=2^{18}\eps$ and
$d_0=\lfloor d/2\rfloor\ge128$; the
condition on $\eps_0$ is $\eps\le c'_0\ell R^\beta$. The hard functions have
$[F]_{\beta,\op}\le\ell\le L_\beta$, and by the definition of $\vartheta$ the cap in
\eqref{eq:hard-hessian-cap} equals
\[
\frac{E_2\eps_0}{R^2\vartheta^2}
=E_2(K_\beta^{\mathrm{tube}})^{-2/\beta}
\ell^{2/\beta}\eps_0^{1-2/\beta}
=C_{\mathrm{cap}}\ell^{2/\beta}\eps^{1-2/\beta}\le L ,
\]
so they belong to $\Fcvx(\beta,L_\beta;L,R)$. An $\eps$-accurate method therefore
needs at least the right-hand side of \eqref{eq:exact-smooth-lower-budget}
calls; with $d_0\ge d/3$ and
$\vartheta^2=(2^{18}K_\beta^{\mathrm{tube}}\eps/(\ell R^\beta))^{2/\beta}$
this is at least
$c_\beta(\sigma^2/\eps^2)\min\{d^{5/2},d^2(\ell R^\beta/\eps)^{2/\beta}\}$, and
\[
\Bigl(\frac{\ell R^\beta}\eps\Bigr)^{2/\beta}
=\min\Bigl\{\Bigl(\frac{L_\beta R^\beta}\eps\Bigr)^{2/\beta},\ \frac{LR^2}{C_{\mathrm{cap}}\eps}\Bigr\}.
\]
The condition $\eps\le c'_0\ell R^\beta$ holds if $\eps\le c'_0L_\beta R^\beta$ and
$\eps\le(c'_0)^{2/\beta}LR^2/C_{\mathrm{cap}}$. The theorem follows with
$c'_\beta=\min\{c'_0,(c'_0)^{2/\beta}/C_{\mathrm{cap}}\}$ and $c_\beta$ divided by
$\max\{1,C_{\mathrm{cap}}\}$. For $L\ge L_\eps:=C_{\mathrm{cap}}L_\beta^{2/\beta}\eps^{1-2/\beta}$
one has $\ell=L_\beta$, and the third term of the minimum is not the smallest.
\end{proof}

\subsection{The separable bounds for convex functions}
\label{app:prior-convex}
\citet{Akhavan2020,Akhavan2024} state their bounds for strongly convex
functions. This subsection derives from them the lower and the upper bound on
$\Fcvx(\beta,L_\beta;L,R)$ in their row of \cref{tab:convex-comparison} and
compares the lower bound with \cref{thm:convex-lower-main}(a). Let
$A=L_\beta R^\beta/\eps$, $S=\sigma^2/\eps^2$ and $K=LR^2/\eps$.
The lower bound rescales the construction in the proof of
\citet[Theorem~22]{Akhavan2024} (see also \citealp[Theorem~6.1]{Akhavan2020}),
whose minimizers lie in the unit ball. Below the seminorm is the operator
seminorm, the errors are independent $N(0,\sigma^2)$ with $\sigma>0$, queries
may be adaptive and unrestricted, and all constants depend only on $\beta$ and
on the fixed function $\phi$.

\paragraph{The cited family.}
Let $\phi:\R\to\R$ be the bounded smooth function of the cited proof, with
$\phi'=1$ on $[-1/4,1/4]$ and $\phi'=0$ outside $[-1,1]$. For $h,a,k>0$ and
$\omega\in\{-1,1\}^d$, set
\[
 f_\omega(x)=\frac{k}{2}\norm{x}^2+a\sum_{i=1}^d\omega_i\phi(x_i/h).
\]
Its operator H\"older seminorm is at most $C_\beta ah^{-\beta}$.
If $k\ge C_\beta a/h^2$, the Hessian lies between $kI/2$ and $2kI$,
$(x_\omega^*)_i=-a\omega_i/(kh)$, and $\norm{x_\omega^*}=\sqrt d\,a/(kh)$.
The bounded differences between adjacent functions and the testing argument
of the cited proof give, whenever $T\le c_\beta\sigma^2/a^2$,
\begin{equation}
\label{eq:prior-family-risk}
 \sup_\omega\E_\omega[f_\omega(\hat x)-f_\omega(x_\omega^*)]
       \ge c_\beta da^2/(kh^2),
\end{equation}
where $\E_\omega$ denotes expectation under $f_\omega$.

\paragraph{Lower bound.}
Let $d\ge1$, $\beta>2$, $L_\beta,R>0$ and $0<\eps\le c_\beta L_\beta R^\beta$. Put
$h_1=(\eps/(dL_\beta))^{1/\beta}$ and $h_2=(\eps/(\sqrt d\,RL_\beta))^{1/(\beta-1)}$, and
choose
\[
 h=C_\beta\max\{h_1,h_2\},\qquad a=c_\beta L_\beta h^\beta,\qquad
 k=C_\beta L_\beta\max\{h^{\beta-2},\sqrt d\,h^{\beta-1}/R\}.
\]
The functions have seminorm at most $L_\beta$ and minimizers in $B(R)$, and the
right side of \eqref{eq:prior-family-risk} is at least a constant times
$\min\{dL_\beta h^\beta,\sqrt d\,RL_\beta h^{\beta-1}\}$, which exceeds $\eps$ for a
sufficiently large constant in $h$. Thus every $\eps$-accurate method
requires
\begin{equation}
\label{eq:prior-convex-calls}
 T\ge c_\beta\frac{\sigma^2}{\eps^2}
 \min\left\{d^2,\ d^{\beta/(\beta-1)}
                  \left(\frac{L_\beta R^\beta}{\eps}\right)^{2/(\beta-1)}\right\}.
\end{equation}
Since $h_1,h_2\le(\eps/L_\beta)^{1/\beta}$,
$\sqrt d\,L_\beta h_1^{\beta-1}/R=L_\beta^{1/\beta}\eps^{1-1/\beta}d^{1/\beta-1/2}/R$,
$\sqrt d\,L_\beta h_2^{\beta-1}/R=\eps/R^2$ and $\eps\le L_\beta R^\beta$, both terms of $k$
are at most $C_\beta L_\beta^{2/\beta}\eps^{1-2/\beta}$. Hence these functions belong to
$\Fcvx(\beta,L_\beta;L,R)$ whenever $L\ge C_\beta L_\beta^{2/\beta}\eps^{1-2/\beta}$. The
two terms of the minimum in \eqref{eq:prior-convex-calls} are equal at
$\eps=L_\beta R^\beta d^{1-\beta/2}$, and the term $d^2$ is the smaller one for
smaller $\eps$.

\paragraph{Upper bound by regularization.}
For $f\in\Fcvx(\beta,L_\beta;L,R)$ put $f_\mu=f+\mu\norm{\cdot}^2/2$ with
$\mu=\eps/R^2$. The minimizer of $f_\mu$ lies in $B(R)$, its seminorm is at
most $L_\beta$, and $\mu I\preceq\nabla^2f_\mu\preceq(L+\mu)I$; let $A\ge1$.
\Citet[Theorem~18 and Table~1]{Akhavan2024}, applied to $f_\mu$ at accuracy
$\eps/2$, gives error at most $\eps$ for $f$ with
\[
 T\le C_\beta\{d^2S A^{2/(\beta-1)}+d(1+K)^2\}
\]
calls. The third term of that theorem requires a budget of order $Q$, where
\[
 Q^{\beta+1}=d^{\beta+2}(1+K)^{2\beta}S A^{-2}
 \le [d(1+K)^2]^\beta\,[d^2S A^{2/(\beta-1)}],
\]
so it is absorbed by the two displayed terms. The leading statistical term
alone suffices when it dominates $d(1+K)^2$. The bound of
\citet[Corollary~19]{Akhavan2024} is not used here, because its constant
depends on $L^2/\sigma^2$.

\paragraph{Size of the improvement.}
Let $A=L_\beta R^\beta/\eps$ with $1\le A\le d^{\beta/4}$. On this range the lower
bound of \cref{thm:convex-lower-main}(a) is $Sd^2A^{2/\beta}$ up to a constant
factor, and its ratio to \eqref{eq:prior-convex-calls} is at least a constant
times $d^{(\beta-2)/\beta}$. Indeed, for $A\ge d^{(\beta-2)/2}$ the prior
minimum is $d^2$ and the ratio is $A^{2/\beta}$; for $A\le d^{(\beta-2)/2}$ the
prior minimum is $d^{\beta/(\beta-1)}A^{2/(\beta-1)}$ and the ratio is
$d^{(\beta-2)/(\beta-1)}A^{-2/(\beta(\beta-1))}$, which decreases in $A$ and is
$d^{(\beta-2)/\beta}$ at $A=d^{(\beta-2)/2}$ and
$d^{(2\beta-5)/(2\beta-2)}\ge d^{(\beta-2)/\beta}$ at $A=d^{\beta/4}$ when
$\beta>4$. For $A>d^{\beta/4}$ the new lower bound is $c_\beta Sd^{5/2}$ and
\eqref{eq:prior-convex-calls} is at most $c_\beta Sd^2$, so the ratio is at
least of order $d^{1/2}$. Both bounds hold on the same class
$\Fcvx(\beta,L_\beta;L,R)$ with $L\ge C_\beta L_\beta^{2/\beta}\eps^{1-2/\beta}$.

\section{Non-Euclidean geometry, constraints, and an \texorpdfstring{$\ell_1$}{l1} upper bound}
\label{app:geometry}

This appendix proves the results of \cref{sec:other}.
\Cref{sec:geo-pullback,sec:geo-transfer} transfer the Euclidean lower
bounds of \cref{app:lower,app:convex-lower} through image balls to
H\"older regularity measured in an arbitrary norm, to minimizers localized
in an arbitrary set and to constrained problems. \Cref{sec:geo-lp}
evaluates these bounds on $\ell_p$ balls. \Cref{sec:geo-l1} proves an upper
bound on the $\ell_1$ ball that is linear in the dimension up to
logarithmic factors for conditionally sub-Gaussian centered noise without
an oblivious component; in a range of accuracies the lower bound is linear
in the dimension as well.

\paragraph{Notation.}
Throughout, $\beta>2$, $q=\lceil\beta\rceil-1\ge2$ and
$\gamma=\beta-q\in(0,1]$. In this appendix $\norm{\cdot}$ is a norm on
$\R^d$ with closed unit ball $\mathcal B$ (written
$\norm\cdot_{\mathcal B}$ in \cref{sec:other}),
$\norm{u}_*=\sup_{\norm{v}\le1}\inner{u}{v}$ is its dual norm, and the
Euclidean norm is written $\norm{\cdot}_2$. We write $\norm{\cdot}_p$ for
the $\ell_p$ norm, $B_p^n(R)=\{y\in\R^n:\norm{y}_p\le R\}$ and
$B_p^n=B_p^n(1)$. For a $j$-linear form $A$ on $\R^d$ and for
$f\in C^q(\R^d)$ let
\[
\norm{A}_{\norm{\cdot}}=\sup\bigl\{\abs{A[u_1,\dots,u_j]}:\
\norm{u_1}\le1,\dots,\norm{u_j}\le1\bigr\},\qquad
[f]_{\beta,\norm{\cdot}}=\sup_{x\ne y}
\frac{\norm{D^qf(x)-D^qf(y)}_{\norm{\cdot}}}{\norm{x-y}^{\gamma}} ,
\]
so that $[f]_{\beta,\norm{\cdot}}=[f]_{\beta,\mathcal B}$ in the notation of
\cref{sec:other} and $[f]_{\beta,\norm{\cdot}_2}=[f]_{\beta,\op}$; we write
$[f]_{\beta,p}=[f]_{\beta,\norm{\cdot}_p}$. For a nonempty set
$K\subseteq\R^d$ let $\mathcal C(\beta,L_\beta;\norm{\cdot},K)$ be the class of
convex $f\in C^q(\R^d)$ with $[f]_{\beta,\norm{\cdot}}\le L_\beta$ and
$\argmin f\cap K\ne\emptyset$. For a linear map $P:\R^d\to\R^n$ and a
symmetric $d\times d$ matrix $H$ let
\[
\norm{P}_{\mathcal B\to2}=\max_{x\in\mathcal B}\norm{Px}_2,\qquad
\norm{H}_{\norm{\cdot}\to\norm{\cdot}_*}=\sup\{u^\top Hv:\norm{u}\le1,\norm{v}\le1\},
\]
and let $\norm{P}_{2\to2}$ be the spectral norm of $P$. If
$\norm{u}_2\le\kappa_0\norm{u}$ for all $u$, then
$\norm{A}_{\norm{\cdot}}\le\kappa_0^q\norm{A}_{\op}$ for every $q$-linear
form $A$ and $\norm{x-y}_2^\gamma\le\kappa_0^\gamma\norm{x-y}^\gamma$, hence
\begin{equation}
\label{eq:geo-norm-compare}
[f]_{\beta,\norm{\cdot}}\le\kappa_0^\beta\,[f]_{\beta,\op};
\end{equation}
for $\norm{\cdot}_p$ with $1\le p\le2$ one may take $\kappa_0=1$.

\begin{adefinition}[Image radius]
\label{def:image-radius}
Let $C\subseteq\R^d$ be nonempty and $1\le n\le d$. Image balls $(P,z,r)$
of $C$ with $P:\R^d\to\R^n$ are defined in \cref{def:image-ball} with
$\mathcal K=C$; their first condition is $\norm{P}_{\mathcal B\to2}\le1$.
The \emph{image radius} of $C$ is
\[
\mathcal R_n(\mathcal B,C)=\sup\bigl\{r>0:\ (P,z,r)\text{ is an image ball of
}C\text{ for some }P,z\bigr\},\qquad \sup\emptyset=0 .
\]
\end{adefinition}

This is the quantity of \citet{Dorn2026GradientFree}, who write the supremum as a
maximum. For compact $C$ with $\mathcal R_n(\mathcal B,C)>0$ the supremum is
attained, because $P$ and $z$ stay bounded along a maximizing sequence and a
limit of image balls of a compact set, with radii bounded away from $0$, is
again an image ball. Thus positive image radius is precisely the geometric
condition used by the transfer theorem. For centrally symmetric convex $C$
the center may be taken to be $0$
\citep[Lemma~J.1]{Dorn2026GradientFree}.

\paragraph{Oracle.}
We use the round model and the noise of \cref{def:oracle}; unless stated
otherwise, $\eta$ is Lipschitz in the Euclidean norm. The lower bounds use
independent $N(0,\sigma^2)$ noise and allow adaptation after every call, so
they hold for every number of rounds.

\subsection{Pull-back of hard functions}
\label{sec:geo-pullback}

\begin{alemma}[Pull-back]
\label{lem:pullback}
Let $K\subseteq\R^d$ be nonempty, let $1\le n\le d$, and let $(P,z,r)$ be an
image ball of $K$. Let $g:\R^n\to\R$ be convex and $q$ times continuously
differentiable, with $[g]_{\beta,\op}\le L_\beta$ and a minimizer
$y^\star\in z+rB_2^n$, and put $f=g\circ P$.
\begin{enumerate}[label=(\alph*)]
\item $f$ is convex and $q$ times continuously differentiable,
$D^jf(x)[u_1,\dots,u_j]=D^jg(Px)[Pu_1,\dots,Pu_j]$ for $1\le j\le q$, and
$[f]_{\beta,\norm{\cdot}}\le\norm{P}_{\mathcal B\to2}^\beta\,L_\beta\le L_\beta$.
\item $\nabla f(x)=P^\top\nabla g(Px)$ and
$\nabla^2f(x)=P^\top\nabla^2g(Px)P$. Hence, for all $x\in\R^d$,
\[
\norm{\nabla f(x)}_*\le\norm{\nabla g(Px)}_2,\qquad
\norm{\nabla^2f(x)}_{\norm{\cdot}\to\norm{\cdot}_*}\le\norm{\nabla^2g(Px)}_{\op}.
\]
If $0\preceq\nabla^2g\preceq LI$ on $\R^n$, then
$0\le u^\top\nabla^2f(x)u\le L\norm{u}^2$ and
$0\preceq\nabla^2f(x)\preceq L\norm{P}_{2\to2}^2I$ for all $x,u\in\R^d$.
\item There is $\xstar\in K$ with $P\xstar=y^\star$; it minimizes $f$,
$\min f=\min g$, and $f(x)-\min f=g(Px)-\min g$ for all $x$. In particular
$f\in\mathcal C(\beta,L_\beta;\norm{\cdot},K)$.
\item Let $\mathsf A$ be an algorithm on $\R^d$ with $T$ calls in $N$
rounds, and let $\mathsf A_P$ be the algorithm on $\R^n$ that uses the same
random elements, runs a copy of $\mathsf A$, answers each query $x$ of the
copy by querying $Px$, and outputs $P\hat x$; it has $T$ calls in $N$
rounds. For every noise admissible at levels $(M',\sigma,\delta)$ in
the run of $\mathsf A_P$ on $g$ there is a noise admissible at levels
$(M'\norm{P}_{2\to2},\sigma,\delta)$ in the run of $\mathsf A$ on $f$
under which the random elements, the queries of $\mathsf A$, the observed
values and the error components coincide pathwise with those of the copy;
independent $N(0,\sigma^2)$ noise corresponds to independent
$N(0,\sigma^2)$ noise, and $\E[g(P\hat x)-\min g]=\E[f(\hat x)-\min f]$.
Consequently, if $\mathsf A$ is $\eps$-accurate at levels
$(M,\sigma,\delta)$ on a class that contains $g\circ P$ for all $g$
in a class $\mathcal E$ of functions as above, then $\mathsf A_P$ is
$\eps$-accurate on $\mathcal E$ at levels
$(M/\norm{P}_{2\to2},\sigma,\delta)$, with the same numbers of calls
and rounds.
\end{enumerate}
\end{alemma}

\begin{proof}
For exact function values this is the image-ball reduction of
\citet{Dorn2026GradientFree}; we check the regularity of order $\beta$ and the
noise model. The chain rule gives $D^jf(x)=D^jg(Px)\circ P^{\otimes j}$,
so
\[
\norm{D^qf(x)-D^qf(y)}_{\norm\cdot}
\le\norm{P}_{\mathcal B\to2}^q\norm{D^qg(Px)-D^qg(Py)}_{\op}
\le L_\beta\norm{P}_{\mathcal B\to2}^{q+\gamma}\norm{x-y}^{\gamma}.
\]
The norm bounds in (b) follow from $\norm{Pu}_2\le\norm u$ and
$\norm{Pu}_2\le\norm P_{2\to2}\norm u_2$, and (c) from
$y^\star\in z+rB_2^n\subseteq P(K)$.

For (d), let the run of $\mathsf A$ on $f$ use the same random elements and
the same variables $\xi^{\mathrm c}_t$, $\xi^{\mathrm o}_t$, $\delta_t$ as the
run of $\mathsf A_P$ on $g$, and put $\eta=\eta'\circ P$, where $\eta'$ is
the Lipschitz component of the latter. Since $f+\eta=(g+\eta')\circ P$,
induction over the rounds shows that the two runs observe the same values.
The queries of a round are functions of the random elements and of the
values of the earlier rounds, so both runs generate the same
$\sigma$-algebras $\mathcal F_{k-1}$ and $\mathcal E_t$. Hence (C), (O) and
(A) carry over, $\eta$ is fixed before the run, and
$\abs{\eta(x)-\eta(y)}\le M'\norm{P(x-y)}_2\le M'\norm P_{2\to2}\norm{x-y}_2$.
Finally $f(\hat x)-\min f=g(P\hat x)-\min g$ by (c). If the Lipschitz
component on $\R^d$ is measured in $\norm{\cdot}$ instead, the same
computation gives the level $M'$.
\end{proof}

\subsection{Transfer of one-point lower bounds}
\label{sec:geo-transfer}

\paragraph{Euclidean lower bounds.}
Write $K_{\mathrm H}$, $E_{\mathrm H}$ and $s_{\mathrm H}$ for the
constants $K_\beta$, $E_2$ and $\vartheta_\beta\le1/2$ of
\cref{thm:convex-lower} (not those of \cref{lem:radial-profile}), and put
$C_{\mathrm H}=E_{\mathrm H}s_{\mathrm H}^{\beta-2}/K_{\mathrm H}$. For
$n\ge1$, $z\in\R^n$ and $r>0$ let $\mathcal E_n(z,r)$ be the class of
convex $g\in C^q(\R^n)$ with $[g]_{\beta,\op}\le L_\beta$ and a global minimizer
in $z+rB_2^n$, and let $\mathcal E^{\mathrm{cap}}_n(z,r)$ be its subclass
defined by $0\preceq\nabla^2g\preceq C_{\mathrm H}L_\beta r^{\beta-2}I$. The
translation $g\mapsto g(\cdot+z)$ maps these classes onto
$\mathcal E_n(0,r)$ and $\mathcal E^{\mathrm{cap}}_n(0,r)$ and preserves
optimization errors. We use two bounds for independent $N(0,\sigma^2)$
noise and adaptive randomized algorithms with unrestricted queries and
outputs. This noise is independent of the random elements of the
algorithm, so both bounds hold in the model of \cref{def:oracle}.

\emph{(H1)} For $n\ge256$ and $0<\eps\le c'_{\mathrm H}L_\beta r^\beta$, every
algorithm with
$\sup_{g\in\mathcal E^{\mathrm{cap}}_n(0,r)}\E[g(\hat y)-\min g]\le\eps$
uses
\[
T\ge c_{\mathrm H}\frac{\sigma^2}{\eps^2}\min\{n^{5/2},n^2(L_\beta r^\beta/\eps)^{2/\beta}\}
\]
calls, even when the Hessian cap and all parameters are known. This is
\cref{thm:convex-lower} with $R$ replaced by $r$, with $\eps_0=2^{18}\eps$
and with the half-dimension $m'=\lfloor n/2\rfloor\ge128$ (denoted $m$
there), so that $m'\ge\frac{255}{512}n$; it gives
\[
c'_{\mathrm H}=\frac{s_{\mathrm H}^\beta}{2^{18}K_{\mathrm H}},\qquad
c_{\mathrm H}=2^{-40}\Bigl(\frac{255}{512}\Bigr)^{5/2}
\min\bigl\{1,\,2^{-8}(2^{18}K_{\mathrm H})^{-2/\beta}\bigr\}.
\]
The scale parameter $\vartheta$ of these hard functions is at most
$s_{\mathrm H}$, so \eqref{eq:hard-hessian-cap} bounds their Hessians by
$C_{\mathrm H}L_\beta r^{\beta-2}$.

\emph{(V)} For $n\ge1$ and $0<\eps\le c'_{\mathrm V}L_\beta r^\beta$, every
algorithm with $\sup_{g\in\mathcal E_n(0,r)}\E[g(\hat y)-\min g]\le\eps$
uses $T\ge c_{\mathrm V}n^2\sigma^2/\eps^2$ calls. This is the operator case
of \cref{cor:convex}, whose proof applies to every $n\ge1$ and gives
$c_{\mathrm V}=(3/(256\bar K_2))^2$
and $c'_{\mathrm V}=3(2\bar K_2)^\beta/(128\bar K_2K_\beta)$, where
$\bar K_2=\max\{1,K_2\}$ and $K_2,K_\beta$ are the constants of
\cref{lem:radial-profile}.

\begin{atheorem}[Geometry transfer of one-point lower bounds]
\label{thm:geometry-lower}
Let $K\subseteq\R^d$ be nonempty, let $1\le n\le d$, and let $(P,z,r)$ be
an image ball of $K$. Consider independent $N(0,\sigma^2)$ noise and an
algorithm with $T$ calls, in any number of rounds, such that
$\E[f(\hat x)-\min f]\le\eps$ for every $f$ in a class $\mathcal F$ of
functions on $\R^d$.
\begin{enumerate}[label=(\alph*)]
\item If $\mathcal F\supseteq\{g\circ P:g\in\mathcal E^{\mathrm{cap}}_n(z,r)\}$,
for instance if $\mathcal F=\mathcal C(\beta,L_\beta;\norm{\cdot},K)$, and if
$n\ge256$ and $0<\eps\le c'_{\mathrm H}L_\beta r^\beta$, then
\[
T\ge c_{\mathrm H}\frac{\sigma^2}{\eps^2}
\min\Bigl\{n^{5/2},\,n^2\Bigl(\frac{L_\beta r^\beta}{\eps}\Bigr)^{2/\beta}\Bigr\}.
\]
\item If $\mathcal F\supseteq\{g\circ P:g\in\mathcal E_n(z,r)\}$, for
instance if $\mathcal F=\mathcal C(\beta,L_\beta;\norm{\cdot},K)$, and if
$0<\eps\le c'_{\mathrm V}L_\beta r^\beta$, then
$T\ge c_{\mathrm V}n^2\sigma^2/\eps^2$.
\end{enumerate}
Every $f=g\circ P$ with $g\in\mathcal E^{\mathrm{cap}}_n(z,r)$ satisfies,
for all $x,u\in\R^d$,
\begin{equation}
\label{eq:geo-hard-regularity}
0\le u^\top\nabla^2f(x)u\le C_{\mathrm H}L_\beta r^{\beta-2}\norm{u}^2,
\qquad
\norm{\nabla f(x)}_*\le C_{\mathrm H}L_\beta r^{\beta-2}
\bigl(\norm{Px-z}_2+r\bigr),
\end{equation}
so (a) holds also on the subclass of $\mathcal C(\beta,L_\beta;\norm{\cdot},K)$
defined by \eqref{eq:geo-hard-regularity}, even when these bounds, $P$,
$z$, $r$ and all class parameters are known. If in addition
$\norm{P}_{2\to2}\le1$, these functions satisfy the Euclidean cap
$0\preceq\nabla^2f\preceq C_{\mathrm H}L_\beta r^{\beta-2}I$, and (a) holds on the
corresponding subclass. If $K$ is compact, (a) and (b) hold with
$r=\mathcal R_n(\mathcal B,K)$; for general $K$ they hold for every
$r<\mathcal R_n(\mathcal B,K)$.
\end{atheorem}

\begin{proof}
Let $g_0$ range over $\mathcal E^{\mathrm{cap}}_n(0,r)$ for (a) and over
$\mathcal E_n(0,r)$ for (b), and put $g=g_0(\cdot-z)$. By
\cref{lem:pullback}(c), $g\circ P\in\mathcal C(\beta,L_\beta;\norm{\cdot},K)$. The
algorithm $\mathsf A_P$ of \cref{lem:pullback}(d), with its queries shifted
to $Px-z$ and its output to $P\hat x-z$, has on $g_0$ the error of the
given algorithm on $g\circ P$, with the same calls, rounds and Gaussian
noise. Bounds (H1) and (V) therefore give (a) and (b). If
$g\in\mathcal E^{\mathrm{cap}}_n(z,r)$ has the minimizer $y^\star$, then,
with $L=C_{\mathrm H}L_\beta r^{\beta-2}$,
$\norm{\nabla g(y)}_2\le L\norm{y-y^\star}_2\le L(\norm{y-z}_2+r)$; together
with \cref{lem:pullback}(b) this gives \eqref{eq:geo-hard-regularity} and
the Euclidean cap. The last assertion follows from the attainment of the
image radius for compact $K$ and from the definition of a supremum
otherwise.
\end{proof}

\Cref{thm:geometry-main} follows from parts (a) and (b), with
$c_\beta=\min\{c_{\mathrm H},c_{\mathrm V}\}$ and
$c'_\beta=\min\{c'_{\mathrm H},c'_{\mathrm V}\}$, and from the following
corollary.

\begin{acorollary}[Restricted queries and functional constraints]
\label{cor:geometry-constrained}
Let $K$ and $n$ be as in \cref{thm:geometry-lower}, let
$\phi_1,\dots,\phi_J:\R^d\to\R$ ($J\ge0$) be fixed functions, let
$X_\phi=\{x\in K:\phi_i(x)\le0,\ 1\le i\le J\}$, and let $(P,z,r)$ be an
image ball of $X_\phi$. Let $X_{\mathrm q},X_{\mathrm o}\subseteq\R^d$ be
arbitrary sets, for instance $X_{\mathrm q}=K+\rho B_2^d$ with $\rho\ge0$
and $X_{\mathrm o}=X_\phi$. Consider an algorithm that evaluates the
objective only at points of $X_{\mathrm q}$, outputs
$\hat x\in X_{\mathrm o}$, and may in addition evaluate the functions
$\phi_i$ at arbitrary points without these evaluations being counted.
Suppose that the objective values carry independent $N(0,\sigma^2)$ noise,
that the constraint values are returned exactly or with independent noise
of a fixed law, and that
\[
\E\Bigl[f(\hat x)-\min_{X_\phi}f\Bigr]\le\eps
\qquad\text{for every }f\in\mathcal C(\beta,L_\beta;\norm{\cdot},X_\phi),
\]
possibly together with further requirements on $\hat x$, such as bounds on
the constraint violation. Then the number $T$ of objective evaluations
satisfies (a) and (b) of \cref{thm:geometry-lower}.
\end{acorollary}

\begin{proof}
Every $f\in\mathcal C(\beta,L_\beta;\norm{\cdot},X_\phi)$ has a global minimizer
in $X_\phi$, so $\min_{X_\phi}f=\min f$. Let $\mathsf A'$ run the given
algorithm, forward its objective queries, answer its constraint queries by
computing the $\phi_i$ (adding, if the constraint oracle is noisy, noise
generated from additional independent random elements appended to
$\Theta_k$), and output $\hat x$. Then $\mathsf A'$ is an algorithm of the
model with $T$ calls and unrestricted queries and outputs, and
$\E[f(\hat x)-\min f]\le\eps$ for every
$f\in\mathcal C(\beta,L_\beta;\norm{\cdot},X_\phi)$. Apply
\cref{thm:geometry-lower} with $K$ replaced by $X_\phi$.
\end{proof}

The corollary applies to every class of constrained problems, for instance
with the functional constraints of \citet{Abronin2026}, that contains all
objectives in $\mathcal C(\beta,L_\beta;\norm{\cdot},X_\phi)$ for one fixed
admissible tuple $(\phi_1,\dots,\phi_J)$; the algorithm need not know the
$\phi_i$. For exact values, \citet[Theorems~3.2 and~3.3]{Dorn2026GradientFree}
treat a known regularizer, whose values are free, in a similar way.

\subsection{Lower bounds on \texorpdfstring{$\ell_p$}{lp} balls}
\label{sec:geo-lp}

\begin{alemma}[Radius bounds]
\label{lem:radius}
Let $1\le p\le2$, $R>0$ and $1\le n\le d$, and let $P_n:\R^d\to\R^n$,
$P_nx=(x_1,\dots,x_n)$, be the coordinate projection.
\begin{enumerate}[label=(\alph*)]
\item $(P_n,0,Rn^{1/2-1/p})$ is an image ball of $B_p^d(R)$, both
for the norm $\norm{\cdot}_p$ and for the Euclidean norm, and
$\norm{P_n}_{2\to2}=1$. Hence
$\mathcal R_n(B_p^d,B_p^d(R))\ge Rn^{1/2-1/p}$ and
$\mathcal R_n(B_2^d,B_p^d(R))\ge Rn^{1/2-1/p}$.
\item Let $p=1$, and let $k\le n$ be a positive integer with
$b=\lceil n/k\rceil$ satisfying $b\,5^k\le d$. Then there is a linear
$P:\R^d\to\R^n$ such that $(P,0,7R/(8\sqrt b))$ is an image ball of
$B_1^d(R)$ for the norm $\norm{\cdot}_1$. Hence
$\mathcal R_n(B_1^d,B_1^d(R))\ge7R/(8\sqrt b)\ge
\frac78R\sqrt{k/(2n)}$. The map $P$ constructed in the proof has
$\norm{P}_{2\to2}\ge(1+49/64)^{1/2}>1$.
\end{enumerate}
\end{alemma}

\begin{proof}
(a) This is the coordinate-projection case of
\citet[Proposition~I.1]{Dorn2026GradientFree}: $P_n$ is a contraction for both
norms, and $P_n(B_p^d(R))=B_p^n(R)\supseteq Rn^{1/2-1/p}B_2^n$ by
H\"older's inequality.

(b) Partition $\{1,\dots,n\}$ into $b$ blocks $I_1,\dots,I_b$ of sizes at
most $k$. A subset of the unit sphere of $\R^{I_\ell}$ with pairwise Euclidean
distances larger than $1/2$ has at most $5^k$ points, because the balls of
radius $1/4$ around its points are disjoint and lie in the ball of radius
$5/4$. Let $\mathcal N_\ell$ be a maximal such subset. By maximality,
every unit vector $u\in\R^{I_\ell}$ has some $v\in\mathcal N_\ell$ with
$\norm{u-v}_2\le1/2$, and then $\inner uv=1-\norm{u-v}_2^2/2\ge7/8$. Hence
the support function of the absolutely convex hull $A_\ell$ of $\mathcal N_\ell$
is at least $7/8$ on the unit sphere, and $A_\ell\supseteq\frac78B_2^{I_\ell}$ by
the separation theorem.

Regard the at most $b5^k\le d$ vectors of all $\mathcal N_\ell$ as vectors of
$\R^n$ supported on $I_\ell$, and let $P$ map $e_1,e_2,\dots$ to these
vectors and the remaining coordinate vectors to $0$. Since $B_1^d$ is the
absolutely convex hull of the $\pm e_i$, $\norm{P}_{B_1^d\to2}$ is the
largest column norm, which is $1$, and $P(B_1^d(R))=RA$, where $A$ is the
absolutely convex hull of $\bigcup_\ell\mathcal N_\ell$. Let
$\norm{y}_2\le7/(8\sqrt b)$, and let $y_\ell$ be the restriction of $y$ to
$I_\ell$. Then $y_\ell\in\frac87\norm{y_\ell}_2A_\ell\subseteq\frac87\norm{y_\ell}_2A$,
and since $A$ is convex and contains $0$,
$y=\sum_\ell y_\ell\in\frac87\bigl(\sum_\ell\norm{y_\ell}_2\bigr)A\subseteq A$, because
$\sum_\ell\norm{y_\ell}_2\le\sqrt b\,\norm{y}_2\le7/8$. Hence
$\frac{7R}{8\sqrt b}B_2^n\subseteq P(B_1^d(R))$, and $b\le n/k+1\le2n/k$
gives the last bound on the image radius.

Finally, let $v\in\mathcal N_1$ and choose $v'\in\mathcal N_1$ with
$\inner{-v}{v'}\ge7/8$. The vector $P^\top v$ has the entries
$\inner vv=1$ and $\inner{v}{v'}\le-7/8$, so
$\norm{P}_{2\to2}=\norm{P^\top}_{2\to2}\ge(1+49/64)^{1/2}$.
\end{proof}

Part (b) with $b=1$, that is, $k=n\le\log_5d$, gives the radius $7R/8$,
independent of $n$. \citet[Proposition~I.3]{Dorn2026GradientFree} also give
images of radius of order $R$ of every $\ell_p$ ball with $1\le p\le2$, of
dimension of order $\log d$ for $p=1$.

For $1\le p\le2$ put
\[
p^*=\frac{p}{p-1}\in[2,\infty]\ (p^*=\infty\text{ for }p=1),\qquad
\theta_p=\frac1p-\frac12\in\Bigl[0,\frac12\Bigr],
\]
and $\mathcal C_p(R)=\mathcal C(\beta,L_\beta;\norm{\cdot}_p,B_p^d(R))$. Then
$1+2/p^*=3-2/p=2-2\theta_p$.

\begin{acorollary}[Lower bounds for the $\ell_p$ geometry]
\label{cor:lp-lower}
Let $1\le p\le2$. Under independent $N(0,\sigma^2)$ noise, every algorithm
with $\sup_{f\in\mathcal C_p(R)}\E[f(\hat x)-\min f]\le\eps$ uses $T$
calls, where:
\begin{enumerate}[label=(\alph*)]
\item for every integer $n$ with $256\le n\le d$ and
$\eps\le c'_{\mathrm H}L_\beta R^\beta n^{-\beta\theta_p}$,
\[
T\ge c_{\mathrm H}\frac{\sigma^2}{\eps^2}\min\Bigl\{n^{5/2},\,
n^{1+2/p^*}\Bigl(\frac{L_\beta R^\beta}\eps\Bigr)^{2/\beta}\Bigr\};
\]
\item for every integer $n$ with $1\le n\le d$ and
$\eps\le c'_{\mathrm V}L_\beta R^\beta n^{-\beta\theta_p}$,
$T\ge c_{\mathrm V}n^2\sigma^2/\eps^2$.
\item Let $\eps\le c'_{\mathrm H}L_\beta R^\beta$ and put
\[
n_\eps=\begin{cases}(c'_{\mathrm H}L_\beta R^\beta/\eps)^{1/(\beta\theta_p)},&p<2,\\
+\infty,&p=2,\end{cases}\qquad
n_\times=(L_\beta R^\beta/\eps)^{4/(\beta(4\theta_p+1))},\qquad
n_*=\min\{d,\lfloor n_\eps\rfloor\}.
\]
If $n_*\ge256$, then
\[
T\ge c_{\mathrm H}\frac{\sigma^2}{\eps^2}\min\{n_*^{5/2},n_*^{1+2/p^*}(L_\beta R^\beta/\eps)^{2/\beta}\},
\]
and the first term is the smaller one exactly when $n_*\le n_\times$. In
particular, provided $n_*\ge256$:
\begin{enumerate}[label=(R\arabic*)]
\item if $d\le\min\{n_\eps,n_\times\}$, then $T\ge c_{\mathrm H}d^{5/2}\sigma^2/\eps^2$;
\item if $n_\times\le d\le n_\eps$, then
$T\ge c_{\mathrm H}d^{1+2/p^*}(L_\beta R^\beta/\eps)^{2/\beta}\sigma^2/\eps^2$;
\item if $p<2$, $d>n_\eps$ and $\lfloor n_\eps\rfloor\ge n_\times$, then,
independently of $d$,
\[
T\ge\frac{c_{\mathrm H}}4\Bigl(\frac{c'_{\mathrm H}L_\beta R^\beta}{\eps}\Bigr)^{\frac{2(3p-2)}{\beta(2-p)}}
\Bigl(\frac{L_\beta R^\beta}\eps\Bigr)^{2/\beta}\frac{\sigma^2}{\eps^2}.
\]
\end{enumerate}
\item For $p=1$: $n_\eps=(c'_{\mathrm H}L_\beta R^\beta/\eps)^{2/\beta}$ and
$n_\times=(L_\beta R^\beta/\eps)^{4/(3\beta)}$. If
$\min\{d,\lfloor n_\eps\rfloor\}\ge\max\{256,n_\times\}$, then
\[
T\ge\frac{c_{\mathrm H}}2\min\Bigl\{d,\Bigl(\frac{c'_{\mathrm H}L_\beta R^\beta}\eps\Bigr)^{2/\beta}\Bigr\}
\Bigl(\frac{L_\beta R^\beta}\eps\Bigr)^{2/\beta}\frac{\sigma^2}{\eps^2}.
\]
In particular, if $d\ge256$ and
$L_\beta R^\beta d^{-3\beta/4}\le\eps\le c'_{\mathrm H}L_\beta R^\beta d^{-\beta/2}$, then
$T\ge c_{\mathrm H}\,d\,(L_\beta R^\beta/\eps)^{2/\beta}\sigma^2/\eps^2$, which is
linear in $d$.
\item For $p=1$ and integers $b,k\ge1$ with $b5^k\le d$, $bk\ge256$ and
$\eps\le c'_{\mathrm H}(7/8)^\beta L_\beta R^\beta b^{-\beta/2}$,
\[
T\ge c_{\mathrm H}\frac{\sigma^2}{\eps^2}\min\Bigl\{(bk)^{5/2},\
\frac{49}{64}\,b\,k^2\Bigl(\frac{L_\beta R^\beta}\eps\Bigr)^{2/\beta}\Bigr\}.
\]
\end{enumerate}
The bounds (a), (c) and (d) hold already on the subclass of
$\mathcal C_p(R)$ defined by the Euclidean Hessian cap and the gradient
bound
\[
0\preceq\nabla^2f\preceq C_{\mathrm H}L_\beta r^{\beta-2}I,\qquad
\norm{\nabla f(x)}_{p^*}\le C_{\mathrm H}L_\beta r^{\beta-2}(\norm{P_nx}_2+r),
\]
where $r=Rn^{-\theta_p}$ and $n$ is the dimension used. The bound (e)
holds on the subclass defined by \eqref{eq:geo-hard-regularity} for the
map of \cref{lem:radius}(b).
\end{acorollary}

\begin{proof}
Parts (a), (b) and (e) follow from \cref{thm:geometry-lower} with the image
balls of \cref{lem:radius}. For the coordinate projection,
$r=Rn^{-\theta_p}$ and
$n^2(L_\beta r^\beta/\eps)^{2/\beta}=n^{1+2/p^*}(L_\beta R^\beta/\eps)^{2/\beta}$; for
(e), $n=bk$ and $r=7R/(8\sqrt b)$ give $n^2r^2=\frac{49}{64}bk^2R^2$. In
(c), the condition on $\eps$ in (a) means $n\le n_\eps$, and both terms of
the minimum increase with $n$, so the best choice is $n=n_*$; the first
term is the smaller one exactly when $n\le n_\times$. This gives (R1)--(R3),
where (R3) uses $\lfloor n_\eps\rfloor\ge n_\eps/2$ and $1+2/p^*\le2$. Part
(d) is (c) with $\theta_1=1/2$ and $1+2/p^*=1$, since
$n_*\ge\min\{d,n_\eps\}/2$, and $n_*=d$ under its last condition. The
subclass assertions follow from \eqref{eq:geo-hard-regularity} and
$\norm{P_n}_{2\to2}=1$.
\end{proof}

Part (b) also covers dimensions below $256$. If $d\ge256$ and
$\eps\le\min\{c'_{\mathrm H},c'_{\mathrm V}\}L_\beta R^\beta d^{-\beta\theta_p}$,
then part (a) with $n=d$ is at least $c_{\mathrm H}/c_{\mathrm V}$ times
part (b) with $n=d$, because $d^{5/2}\ge d^2$ and
$d^{1+2/p^*}(L_\beta R^\beta/\eps)^{2/\beta}
=d^2(d^{-\beta\theta_p}L_\beta R^\beta/\eps)^{2/\beta}\ge(c'_{\mathrm H})^{-2/\beta}d^2\ge d^2$.
Part (e) improves the regime (R3) at $p=1$: with
$b=\lfloor\frac{49}{64}n_\eps\rfloor$ and $k=\lfloor\log_5(d/b)\rfloor$, its
second term exceeds the bound of (R3) by a factor of order
$k^2\asymp\log^2(d/n_\eps)$.

\subsection{An upper bound on the \texorpdfstring{$\ell_1$}{l1} ball}
\label{sec:geo-l1}

We use the stencil of \cref{app:estimator}: $m=\lceil q/2\rceil$, nodes
$t_j=j/m$, weights $a_j$ solving \eqref{eq:stencil-moments}, and the
operator $D_{h,u}$ of \eqref{eq:directional}, here with $u=e_i$. Its
constants are $S_1$, $S_2$ and $c_0$ of \eqref{eq:stencil-constants},
$A_1$ of \eqref{eq:A1} and $C_\sigma=(S_1+(S_2/2)^{1/2})^2$ of
\cref{prop:unified-estimator} (see \cref{tab:unified-constants}); they
depend only on $\beta$. For $2<\beta\le3$ they are given in the row $m=1$ of
\cref{tab:unified-constants}, and $c_0=1/2$.

Let $\mathcal C_1(R,G_\infty)$ be the class of convex $f\in C^q(\R^d)$ with
$[f]_{\beta,1}\le L_\beta$, $\argmin f\cap B_1^d(R)\ne\emptyset$, and
$\norm{\nabla f(x)}_\infty\le G_\infty$ for all $x\in B_1^d(R)$. Let
$\Delta_{2d}=\{w\in[0,1]^{2d}:\sum_\ell w_\ell=1\}$, and for $w\in\R^{2d}$ let
$w^+,w^-\in\R^d$ be its first and last $d$ coordinates.
\Cref{alg:entropic-coordinate} runs entropic mirror descent
\citep{NemirovskiYudin1983,BeckTeboulle2003} on $\Delta_{2d}$, mapped onto
$B_1^d(R)$ by $w\mapsto R(w^+-w^-)$, with coordinate stencils whose signs
are drawn afresh in every round. Its parameters $h,N,\varsigma$ are fixed
before the run; \cref{thm:l1-upper}(b) chooses them for a target accuracy
$\eps$.

\begin{algorithm}[t]
\caption{Entropic descent with coordinate stencils on the $\ell_1$ ball}
\label{alg:entropic-coordinate}
\begin{algorithmic}[1]
\Require $\beta,L_\beta,R,G_\infty,\sigma,d,\eps$; $h>0,N\ge1,\varsigma>0$;
$t_j,a_j$ from \eqref{eq:stencil-moments}
\State $w_1\gets(2d)^{-1}\mathbf1\in\Delta_{2d}$
\For{$k=1,\ldots,N$}
  \State $x_k\gets R(w_k^+-w_k^-)$; draw independent uniform signs $s_{k,i}$
  \State In one round obtain $Y^\pm_{k,i,j}$ at $x_k\pm ht_js_{k,i}e_i$
  \State $g_{k,i}\gets\dfrac{s_{k,i}}{2h}\sum_{j=1}^ma_j
                   (Y^+_{k,i,j}-Y^-_{k,i,j})$ for $i=1,\ldots,d$
  \State $G_k\gets R(g_k,-g_k)$
  \State $w_{k+1,\ell}\gets w_{k,\ell}e^{-\varsigma G_{k,\ell}}/\sum_{r=1}^{2d}
                         w_{k,r}e^{-\varsigma G_{k,r}}$ for $\ell=1,\ldots,2d$
\EndFor
\State \Return $\bar x_N=N^{-1}\sum_{k=1}^N x_k$
\end{algorithmic}
\end{algorithm}

Since $x_k$ is a function of $g_1,\dots,g_{k-1}$, it is
$\mathcal F_{k-1}$-measurable. The queries of round $k$ are functions of
$x_k$ and of the signs $s_{k,i}$, which belong to $\Theta_k$, and their
order does not depend on $\Theta_k$. Thus the algorithm belongs to the
model, with $N$ rounds and $T=2mdN$ calls, and $g_k$ is the estimate
\eqref{eq:coord-estimate} of the
next lemma at the center $x=x_k$, with $B=1$ and $s^{(1)}_i=s_{k,i}$. We
write $b_k$ and $\Xi_k$ for the vectors $b$ and $\Xi$ of the lemma for this
estimate.

\begin{alemma}[Coordinate stencils with random signs]
\label{lem:coord-estimator}
Let $f\in C^q(\R^d)$ be such that for every $z\in\R^d$ and every $i$ the
function $u\mapsto f(z+ue_i)$ has a $\gamma$-H\"older $q$-th derivative
with constant $L_\beta$; this holds if $[f]_{\beta,p}\le L_\beta$ for some
$p\in[1,\infty]$, because $\norm{e_i}_p=1$. Let $h>0$ and an integer
$B\ge1$ be fixed, let $k$ be a round, and let $x$ be an
$\mathcal F_{k-1}$-measurable random point of $\R^d$. In round $k$, draw
independent signs $s^{(\ell)}_i$ with
$\PP(s^{(\ell)}_i=1)=\PP(s^{(\ell)}_i=-1)=1/2$, $1\le i\le d$,
$1\le\ell\le B$, as part of $\Theta_k$. Query, in an order that does not
depend on $\Theta_k$, the $2mdB$ points
$X^\pm_{i,j,\ell}=x\pm ht_js^{(\ell)}_ie_i$ ($1\le i\le d$, $1\le j\le m$,
$1\le\ell\le B$), observe the values $Y^\pm_{i,j,\ell}$, and set
\begin{equation}
\label{eq:coord-estimate}
g_i=\frac1B\sum_{\ell=1}^B\frac{s^{(\ell)}_i}{2h}\sum_{j=1}^ma_j
\bigl(Y^+_{i,j,\ell}-Y^-_{i,j,\ell}\bigr),\qquad1\le i\le d .
\end{equation}
Let the noise satisfy (C), (O) and (A), and let its Lipschitz component
$\eta$ satisfy
\begin{equation}
\label{eq:coord-lip}
\abs{\eta(y+ue_i)-\eta(y-ue_i)}\le2M\abs u\qquad
(y\in\R^d,\ u\in\R,\ 1\le i\le d);
\end{equation}
this holds if $\eta$ is $M$-Lipschitz for $\norm{\cdot}_p$ for some
$p\in[1,\infty]$, in particular under (L). Let $\xi^\pm_{i,j,\ell}$ be the
random error $\xi^{\mathrm c}_t+\xi^{\mathrm o}_t$ at the query
$X^\pm_{i,j,\ell}$. Then $g$ is $\mathcal F_k$-measurable and
$g=\nabla f(x)+b+\Xi$, where
\[
\Xi_i=\frac1B\sum_{\ell=1}^B\frac{s^{(\ell)}_i}{2h}\sum_{j=1}^ma_j
\bigl(\xi^+_{i,j,\ell}-\xi^-_{i,j,\ell}\bigr),
\]
and:
\begin{enumerate}[label=(\alph*)]
\item $\norm{b}_\infty\le\bar b_h:=c_0L_\beta h^{\beta-1}+A_1M+S_1\delta/h$
almost surely;
\item for every $i$, $\E[\Xi_i\mid\mathcal F_{k-1}]=0$ and
$\E[\Xi_i^2\mid\mathcal F_{k-1}]\le C_\sigma\sigma^2/(Bh^2)$; hence
$\E[\Xi\mid\mathcal F_{k-1}]=0$ and
\[
\E[\norm{\Xi}_\infty^2\mid\mathcal F_{k-1}]\le
\E[\norm{\Xi}_2^2\mid\mathcal F_{k-1}]\le C_\sigma d\sigma^2/(Bh^2);
\]
\item if the noise satisfies \eqref{eq:SG} and $\xi^{\mathrm o}\equiv0$, then
\[
\E[\norm{\Xi}_\infty^2\mid\mathcal F_{k-1}\vee\sigma(\Theta_k)]\le
S_2(1+\log(2d))\sigma^2/(Bh^2)\le(1+1/\log2)S_2\log(2d)\sigma^2/(Bh^2),
\]
and the same bound holds for $\E[\norm{\Xi}_\infty^2\mid\mathcal F_{k-1}]$.
\end{enumerate}
\end{alemma}

\begin{proof}
\emph{Decomposition.} The queries are functions of $x$ and $\Theta_k$, and
$g$ is a function of $\Theta_k$ and of the values observed in round $k$, so
$g$ is $\mathcal F_k$-measurable. For $s\in\{-1,1\}$ and every
$\varphi:\R^d\to\R$,
\[
\frac{s}{2h}\sum_ja_j\bigl(\varphi(x+ht_jse_i)-\varphi(x-ht_jse_i)\bigr)
=D_{h,e_i}\varphi(x),
\]
since for $s=-1$ both the factor $s$ and the difference change sign.
Inserting the values, with adversarial errors $\delta^\pm_{i,j,\ell}$, gives
$g_i=\partial_if(x)+b_i+\Xi_i$ with
\[
b_i=D_{h,e_i}f(x)-\partial_if(x)+D_{h,e_i}\eta(x)
+\frac1B\sum_{\ell=1}^B\frac{s^{(\ell)}_i}{2h}\sum_ja_j
\bigl(\delta^+_{i,j,\ell}-\delta^-_{i,j,\ell}\bigr).
\]

\emph{Part (a).} \Cref{lem:exact} in dimension one, applied to
$u\mapsto f(x+ue_i)$, gives
$\abs{D_{h,e_i}f(x)-\partial_if(x)}\le c_0L_\beta h^{\beta-1}$. By
\eqref{eq:coord-lip},
$\abs{D_{h,e_i}\eta(x)}\le\frac1{2h}\sum_j\abs{a_j}\,2M ht_j=A_1M$,
and by (A) the last term of $b_i$ is at most $S_1\delta/h$ in absolute
value. These bounds hold for every realization of the signs.

\emph{Part (b).} This is the noise part of the proof of
\cref{prop:unified-estimator}, with random signs in place of random frames.
Split $\Xi_i=\Xi^{\mathrm c}_i+\Xi^{\mathrm o}_i$ into its centered and
oblivious parts. The query $t$ at $X^\pm_{i,j,\ell}$ enters $\Xi_i$ with the
$\Theta_k$-measurable coefficient $w_t=\pm s^{(\ell)}_ia_j/(2hB)$, and the
$2mB$ queries of coordinate $i$ have $\sum_tw_t^2=S_2/(2Bh^2)$. By (C), as
there, the terms $w_t\xi^{\mathrm c}_t$ are centered and pairwise orthogonal,
so $\E[\Xi^{\mathrm c}_i\mid\mathcal F_{k-1}]=0$ and
$\mathsf N(\Xi^{\mathrm c}_i)^2\le S_2\sigma^2/(2Bh^2)$, with $\mathsf N$ as in
\cref{app:unified-setting}. For the oblivious part,
$\Xi^{\mathrm o}_i=B^{-1}\sum_\ell s^{(\ell)}_io^{(\ell)}_i$ with
$o^{(\ell)}_i=\frac1{2h}\sum_ja_j(\xi^{\mathrm o,+}_{i,j,\ell}-\xi^{\mathrm o,-}_{i,j,\ell})$.
By (O) and the argument of \cref{lem:round-facts}(a), the signs are
independent of $\mathcal F_{k-1}$ and of the oblivious errors of the round,
and they have mean zero and are independent of each other; since
$\mathsf N(o^{(\ell)}_i)\le S_1\sigma/h$, this gives
$\E[\Xi^{\mathrm o}_i\mid\mathcal F_{k-1}]=0$ and
$\mathsf N(\Xi^{\mathrm o}_i)^2\le S_1^2\sigma^2/(Bh^2)$. Minkowski's
inequality gives (b) for $\Xi_i$, and summing over $i$ and
$\norm{\cdot}_\infty\le\norm{\cdot}_2$ give the rest of (b).

\emph{Part (c).} Let $\mathcal G=\mathcal F_{k-1}\vee\sigma(\Theta_k)$. Now
$\Xi_i=\sum_tw_t\xi^{\mathrm c}_t$, summed over the $2mB$ queries $t$ of
coordinate $i$ in their order, which does not depend on $\Theta_k$. Each
$w_t$ takes two values, $w_t$ and $\sum_{t'<t}w_{t'}\xi^{\mathrm c}_{t'}$
are $\mathcal E_t$-measurable, and $\mathcal E_t\supseteq\mathcal G$.
Applying \eqref{eq:SG} to these queries from the last to the first gives
$\E[e^{\theta\Xi_i}\mid\mathcal G]\le e^{\theta^2v/2}$ for $\theta\in\R$,
where $v=\sigma^2\sum_tw_t^2=S_2\sigma^2/(2Bh^2)$. The Chernoff bound and a
union bound over $i$ give
$\PP(\norm\Xi_\infty>u\mid\mathcal G)\le2de^{-u^2/(2v)}$, so, with
$w_0=2v\log(2d)$,
\[
\E[\norm{\Xi}_\infty^2\mid\mathcal G]
\le w_0+\int_{w_0}^\infty2de^{-w/(2v)}\,dw=2v(1+\log(2d)) .
\]
This is the first bound of (c). The second follows from
$\log2\le\log(2d)$, and conditioning on $\mathcal F_{k-1}$ gives the last
claim.
\end{proof}

The random signs are needed for the oblivious part of (b): without them
($s^{(\ell)}_i\equiv1$), the oblivious errors $\pm\sigma\sgn(a_j)$ at the
points $x\pm ht_je_i$ give $\Xi_i=S_1\sigma/h$ for every $i$. Without
\eqref{eq:SG}, the factor $d$ in (b) cannot be removed. Let $B=1$, fix $j$
with $a_j\ne0$, let the centered errors at the queries $X^+_{i,j,1}$,
$1\le i\le d$, be independent of each other and of everything else, equal
to $\pm\sqrt d\,\sigma$ with probability $1/(2d)$ each and to $0$
otherwise, and let all other errors vanish. Then (C) holds, and
$\norm{\Xi}_\infty=\abs{a_j}\sqrt d\,\sigma/(2h)$ unless all these errors
vanish, so $\E\norm{\Xi}_\infty^2\ge(1-e^{-1})a_j^2d\sigma^2/(4h^2)$.

\begin{alemma}[Entropic mirror descent on the $\ell_1$ ball]
\label{lem:entropic-md}
Let $R,\varsigma>0$, let $g_1,g_2,\dots\in\R^d$ be arbitrary, and define
$w_k$, $x_k$ and $G_k$ by \cref{alg:entropic-coordinate}. Then
$x_k\in B_1^d(R)$ for all $k$, and for every $\xstar\in B_1^d(R)$ and
$N\ge1$,
\[
\sum_{k=1}^N\inner{g_k}{x_k-\xstar}\le\frac{\log(2d)}{\varsigma}
+\frac{\varsigma R^2}2\sum_{k=1}^N\norm{g_k}_\infty^2 .
\]
\end{alemma}

\begin{proof}
The update of $w_k$ is the prox-mapping, applied to $\varsigma G_k$, of the
entropy on $\Delta_{2d}$, which is $1$-strongly convex with respect to
$\norm{\cdot}_1$ \citep[Section~2.3]{NJLS2009}; the associated
prox-function is $V(w,u)=\KL(u\Vert w)$, the Kullback--Leibler divergence.
By \citet[Lemma~2.1]{NJLS2009} and $\norm{G_k}_\infty=R\norm{g_k}_\infty$,
every $u\in\Delta_{2d}$ satisfies
\[
\KL(u\Vert w_{k+1})\le\KL(u\Vert w_k)-\varsigma\inner{G_k}{w_k-u}
+\frac{\varsigma^2R^2}2\norm{g_k}_\infty^2 .
\]
Summing over $k\le N$ and using $\KL(u\Vert w_1)\le\log(2d)$ bounds
$\sum_k\inner{G_k}{w_k-u}$ by the right-hand side of the claim. Moreover
$\norm{x_k}_1\le R\sum_\ell w_{k,\ell}=R$. Given $\xstar\in B_1^d(R)$, the vector
$u$ with $u_i^\pm=(\pm\xstar_i)_+/R+(1-\norm{\xstar}_1/R)/(2d)$ lies in
$\Delta_{2d}$ and satisfies $R(u^+-u^-)=\xstar$, so that
\begin{equation}\label{eq:geo-lift}
\inner{G_k}{w_k-u}=\inner{g_k}{x_k-\xstar},
\end{equation}
and the claim follows.
\end{proof}

\begin{atheorem}[Upper bound for the $\ell_1$ geometry]
\label{thm:l1-upper}
\label{cor:l1-dimension}
Let $\beta>2$, $L_\beta,R,G_\infty>0$, let $\bar b_h$ be as in
\cref{lem:coord-estimator}, and let $\bar x_N$ be the output of
\cref{alg:entropic-coordinate}. Consider noises admissible at levels
$(M,\sigma,\delta)$, and put $V_h=C_\sigma d\sigma^2/h^2$ in general
and $V_h=S_2(1+\log(2d))\sigma^2/h^2$ for those that satisfy \eqref{eq:SG}
and $\xi^{\mathrm o}\equiv0$.
\begin{enumerate}[label=(\alph*)]
\item For all $h,\varsigma>0$, $N\ge1$, $f\in\mathcal C_1(R,G_\infty)$ and
such noises,
\[
\E[f(\bar x_N)-\fstar]\le\frac{\log(2d)}{\varsigma N}
+\varsigma R^2\bigl((G_\infty+\bar b_h)^2+V_h\bigr)+2R\bar b_h .
\]
If $\bar G^2\ge2(G_\infty+\bar b_h)^2+2V_h$ and
$\varsigma=(2\log(2d)/(R^2\bar G^2N))^{1/2}$, then
$\E[f(\bar x_N)-\fstar]\le R\sqrt{2\log(2d)\bar G^2/N}+2R\bar b_h$.
\item Let $0<\eps<G_\infty R$ and put
\begin{gather*}
h_{1,\eps}=\Bigl(\frac{\eps}{8c_0L_\beta R}\Bigr)^{\frac1{\beta-1}},\qquad
\bar G_\eps^2=2\Bigl(G_\infty+\frac{\eps}{4R}\Bigr)^2+2V_{h_{1,\eps}},\\
N=\Bigl\lceil\frac{8R^2\log(2d)\bar G_\eps^2}{\eps^2}\Bigr\rceil,\qquad
\varsigma=\Bigl(\frac{2\log(2d)}{R^2\bar G_\eps^2N}\Bigr)^{1/2}.
\end{gather*}
If $M\le\eps/(16A_1R)$ and $\delta\le\eps h_{1,\eps}/(16S_1R)$, then
$\E[f(\bar x_N)-\fstar]\le\eps$ for every $f\in\mathcal C_1(R,G_\infty)$
and every such noise. The algorithm uses $N$ rounds of $2md$
calls, all queries lie in $B_1^d(R+h_{1,\eps})$, and $T=2mdN$
satisfies, under \eqref{eq:SG} and $\xi^{\mathrm o}\equiv0$,
\[
T\le72m\,d\log(2d)\frac{G_\infty^2R^2}{\eps^2}
+79mS_2(8c_0)^{\frac2{\beta-1}}\,d\log^2(2d)\,
\frac{\sigma^2}{\eps^2}\Bigl(\frac{L_\beta R^\beta}\eps\Bigr)^{\frac2{\beta-1}},
\]
and in general
\[
T\le72m\,d\log(2d)\frac{G_\infty^2R^2}{\eps^2}
+32mC_\sigma(8c_0)^{\frac2{\beta-1}}\,d^2\log(2d)\,
\frac{\sigma^2}{\eps^2}\Bigl(\frac{L_\beta R^\beta}\eps\Bigr)^{\frac2{\beta-1}}.
\]
The parameters depend only on $\beta,L_\beta,R,G_\infty,\sigma,d,\eps$. If
$\eps\ge G_\infty R$, the output $\hat x=0$ satisfies $f(0)-\fstar\le\eps$
without any call.
\item \emph{Dimension dependence.} Let $d\ge256$, $\sigma>0$,
$G_\infty\ge G_d:=2C_{\mathrm H}L_\beta R^{\beta-1}d^{-(\beta-2)/2}$, and
\[
L_\beta R^\beta d^{-3\beta/4}\le\eps\le c'_{\mathrm H}L_\beta R^\beta d^{-\beta/2}.
\]
For independent $N(0,\sigma^2)$ noise, let $T^*(\eps)$ be the minimax
number of calls on $\mathcal C_1(R,G_\infty)$. Then
\[
T^*(\eps)\ge c_{\mathrm H}d(L_\beta R^\beta/\eps)^{2/\beta}\sigma^2/\eps^2,
\]
and the sub-Gaussian upper bound in (b) also bounds $T^*(\eps)$.
Here $\eps<G_dR\le G_\infty R$; the window is nonempty exactly when
$d\ge(c'_{\mathrm H})^{-4/\beta}$.
\end{enumerate}
\end{atheorem}

\begin{proof}
\emph{Part (a).} Let $\xstar\in\argmin f\cap B_1^d(R)$. Convexity,
\cref{lem:entropic-md} and the decomposition
$g_k=\nabla f(x_k)+b_k+\Xi_k$ of \cref{lem:coord-estimator} give
\[
N\bigl(f(\bar x_N)-\fstar\bigr)\le\sum_{k=1}^N\inner{\nabla f(x_k)}{x_k-\xstar}
\le\frac{\log(2d)}{\varsigma}+\frac{\varsigma R^2}{2}\sum_{k=1}^N\norm{g_k}_\infty^2
+2RN\bar b_h-\sum_{k=1}^N\inner{\Xi_k}{x_k-\xstar},
\]
since $\norm{x_k-\xstar}_1\le2R$. The last sum has expectation zero by
\cref{lem:coord-estimator}(b), because $x_k-\xstar$ is bounded and
$\mathcal F_{k-1}$-measurable, and
$\E\norm{g_k}_\infty^2\le2(G_\infty+\bar b_h)^2+2V_h$, because
$\norm{\nabla f(x_k)}_\infty\le G_\infty$. This proves the first bound of
(a); the second follows by inserting $\varsigma$.

\emph{Part (b).} The choices of $h_{1,\eps}$, $M$ and $\delta$ give
$\bar b_{h_{1,\eps}}\le\eps/(8R)+\eps/(16R)+\eps/(16R)=\eps/(4R)$, and the
choice of $N$ makes both terms of the tuned bound in (a) at most $\eps/2$.
Each query differs from a point of $B_1^d(R)$ in one coordinate by at most
$h_{1,\eps}$. Since $2(G_\infty+\eps/(4R))^2\le4G_\infty^2+\eps^2/(4R^2)$,
\[
N\le1+2\log(2d)+32\log(2d)\frac{G_\infty^2R^2}{\eps^2}
+16\log(2d)\frac{R^2V_{h_{1,\eps}}}{\eps^2}.
\]
As $\eps<G_\infty R$ and $1\le\log(2d)/\log2$, multiplying by $2md$ bounds
the first three terms by $72md\log(2d)G_\infty^2R^2/\eps^2$. The two call
bounds follow by inserting $V_h$, using
$R^2h_{1,\eps}^{-2}=(8c_0)^{2/(\beta-1)}(L_\beta R^\beta/\eps)^{2/(\beta-1)}$ and,
in the sub-Gaussian case, $1+\log(2d)\le(1+1/\log2)\log(2d)$ and
$32(1+1/\log2)<79$. Finally, if $\eps\ge G_\infty R$, convexity gives
$f(0)-\fstar\le\norm{\nabla f(0)}_\infty\norm{\xstar}_1\le G_\infty R$.

\emph{Part (c).} The window is the condition of the last assertion of
\cref{cor:lp-lower}(d). By the subclass assertion of \cref{cor:lp-lower}
with $p=1$ and $n=d$, that lower bound holds on the subclass of
$\mathcal C_p(R)$, $p=1$, defined by
$\norm{\nabla f(x)}_\infty\le C_{\mathrm H}L_\beta r^{\beta-2}(\norm x_2+r)$ (and
a Euclidean Hessian cap), where $r=R/\sqrt d$. For $x\in B_1^d(R)$ the
right-hand side is at most $2C_{\mathrm H}L_\beta r^{\beta-2}R=G_d$, so this
subclass is contained in $\mathcal C_1(R,G_d)\subseteq\mathcal C_1(R,G_\infty)$.
Since the lower bound is positive, no algorithm without calls is
$\eps$-accurate on $\mathcal C_1(R,G_d)$, whereas the output $0$ has error
at most $G_dR$ there (last assertion of (b) with $G_d$ in place of
$G_\infty$); hence $\eps<G_dR$. The upper bound is part (b) for Gaussian
noise, which satisfies \eqref{eq:SG} with $\xi^{\mathrm o}\equiv0$. The
window is nonempty exactly when $d^{\beta/4}\ge1/c'_{\mathrm H}$.
\end{proof}

\begin{aremark}[Local smoothness suffices]
\label{rem:l1-local}
The proof of \cref{thm:l1-upper}(a), (b) uses the H\"older condition only in
\cref{lem:coord-estimator}(a), through the functions $u\mapsto f(x_k+ue_i)$ on
$[-h,h]$ with $x_k\in B_1^d(R)$; all queries lie in $B_1^d(R+h)$. Hence (a) and
(b) hold for every $f$ that is defined, convex and continuously differentiable
on an open set containing $B_1^d(R+h)$, has a minimizer in $B_1^d(R)$ and
$\norm{\nabla f}_\infty\le G_\infty$ on $B_1^d(R)$, and for which every function
$u\mapsto f(x+ue_i)$, $x\in B_1^d(R)$, has a $\gamma$-H\"older $q$-th derivative
with constant $L_\beta$ on $[-h,h]$. Experiment~6 of \cref{app:experiments} uses
this form: for $f(x)=\sum_j\varphi(x_j-c_j)$ with $\varphi(t)=e^t-1-t$ the global
seminorm is infinite, while the local constant is the largest value of
$\varphi'''$ at the coordinates that occur.
\end{aremark}

\begin{aremark}[Comparison of the bounds]
\label{rem:lp-upper}
For fixed $\beta,L_\beta,R,\sigma,\eps$ and $G_\infty\ge G_d$, both bounds of
\cref{thm:l1-upper}(c) are linear in $d$ up to the factor $\log^2(2d)$ on
the whole range
$(L_\beta R^\beta/\eps)^{4/(3\beta)}\le d\le(c'_{\mathrm H}L_\beta R^\beta/\eps)^{2/\beta}$
of the window. Their ratio is at most
\[
C_\beta\log^2(2d)(L_\beta R^\beta/\eps)^{2/(\beta-1)-2/\beta}
+C_\beta\log(2d)G_\infty^2R^2(\eps/(L_\beta R^\beta))^{2/\beta}/\sigma^2,
\]
so the remaining gap concerns the power of $1/\eps$, which is $2+2/\beta$ in
the lower bound and $2+2/(\beta-1)$ in the upper bound. For larger $d$ the lower bound of
\cref{cor:lp-lower}(d) no longer grows with $d$, up to the logarithmic gain
of \cref{cor:lp-lower}(e), while the upper bound still grows linearly. For
noise that is only admissible at levels $(M,\sigma,\delta)$, the
upper bound grows like $d^2\log(2d)$, and by the example after
\cref{lem:coord-estimator} the second moment of $\norm{\Xi}_\infty$ can
indeed be of order $d\sigma^2/h^2$. Since
$s_{\mathrm H}\le1/2$, the threshold of \cref{thm:l1-upper}(c) satisfies
$(c'_{\mathrm H})^{-4/\beta}\ge2^{4+72/\beta}K_{\mathrm H}^{4/\beta}$ (for
instance $2^{28}K_{\mathrm H}^{4/3}$ at $\beta=3$); with the constants of
\cref{thm:convex-lower}, part (c) is therefore a statement about large
dimensions.

Upper bounds for $p>2$ are open. For $p=2$ and fixed $L_\beta,R,\sigma$, the
known convex upper bound is $O(d^2\eps^{-2\beta/(\beta-1)})$
\citep{Bach2016,Gasnikov2022Survey}, and
\cref{cor:direct-convex,thm:value-analysis} give a sharper comparison when a
Hessian bound is available.
\end{aremark}

\section{Functional constraints}
\label{app:functional}

This appendix treats convex problems with functional constraints whose
values are observed with noise. \Cref{app:fc-setting} introduces the
problem, the multipliers and the oracle. \Cref{app:fc-levels} shows that no
method controls both the objective error and the violation if the values of
the constraints carry an unknown constant error. \Cref{app:fc-mirror-prox}
analyzes stochastic Mirror--Prox with biased operator estimates, and
\cref{app:fc-stencil} builds these estimates from \cref{alg:stencil} and
calibrated values of the constraints. Lower bounds for problems whose optimal
value is not changed by the constraints are given by
\cref{cor:geometry-constrained}.

\subsection{Problem, multipliers and oracle}
\label{app:fc-setting}

\paragraph{Problem.}
Let $Q\subset\R^d$ be nonempty, closed and convex with diameter
$D\in(0,\infty)$, and let $g_0=f,g_1,\dots,g_{{n_{\mathrm c}}}$
(${n_{\mathrm c}}\ge1$) be real functions
that are differentiable on an open set containing $Q$ and convex on $Q$. We
consider
\begin{equation}
\label{eq:fc-problem}
\text{minimize } f(x)\ \text{ over } x\in Q\quad
\text{subject to } g_i(x)\le0,\ i=1,\dots,{n_{\mathrm c}} ,
\end{equation}
and write $g=(g_1,\dots,g_{{n_{\mathrm c}}})$. We assume that the feasible set
$\mathcal X=\{x\in Q:\ g(x)\le0\}$ is nonempty. It is compact, so
$\fstar=\min_{\mathcal X}f$ is attained at some $x^\star\in\mathcal X$. The
\emph{violation} of $x\in Q$ is
$v(x)=\max\{0,g_1(x),\dots,g_{{n_{\mathrm c}}}(x)\}$. A method
outputs $\hat x\in Q$; we bound $\E[f(\hat x)-\fstar]_+$ and $\E v(\hat x)$,
where $[a]_+=\max\{a,0\}$. The objective error $f(\hat x)-\fstar$ may be
negative at infeasible points.

\paragraph{Multipliers.}
We assume a known number $\bar y\ge0$ and a vector
$y^\star\in\R^{{n_{\mathrm c}}}$ with
\begin{equation}
\label{eq:fc-mult}
y^\star\ge0,\qquad\norm{y^\star}_1\le\bar y,\qquad
\inf_{x\in Q}\bigl(f(x)+\inner{y^\star}{g(x)}\bigr)=\fstar .
\end{equation}
Such a vector is a Lagrange multiplier of \eqref{eq:fc-problem};
\cref{rem:fc-multiplier} discusses its existence and its size. Fix
$\rho>\bar y$; the method below uses $\rho=\bar y+1$. Put
\[
Y_\rho=\{y\in\R^{{n_{\mathrm c}}}:\ y\ge0,\ \norm{y}_1\le\rho\},\qquad
K=Q\times Y_\rho\subset\R^{d+{n_{\mathrm c}}},\qquad
D_K=(D^2+2\rho^2)^{1/2}.
\]
We use the Euclidean norm on $\R^{d+{n_{\mathrm c}}}$ and write $w=(x,y)$ for its points;
the Euclidean projection onto $K$ is $\Proj_K(x,y)=(\Proj_Qx,\Proj_{Y_\rho}y)$.
For $y,y'\in Y_\rho$ we have $\inner{y}{y'}\ge0$, hence
$\norm{y-y'}^2\le\norm{y}_1^2+\norm{y'}_1^2\le2\rho^2$, and $D_K$ bounds the
diameter of $K$.

\paragraph{Lagrangian, operator and gap.}
For $(x,y)\in K$ let
\[
\Psi(x,y)=f(x)+\inner y{g(x)},\qquad
H(x,y)=\Bigl(\nabla f(x)+\sum_{i=1}^{{n_{\mathrm c}}}y_i\nabla g_i(x),\ -g(x)\Bigr).
\]
For each $y\in Y_\rho$ the function $\Psi(\cdot,y)$ is convex on $Q$, and
$\Psi(x,\cdot)$ is affine with gradient $g(x)$. We assume that $H$ is
$M_H$-Lipschitz on $K$ for some $M_H>0$, that is,
$\norm{H(w)-H(w')}\le M_H\norm{w-w'}$ for $w,w'\in K$;
\cref{lem:fc-lipschitz} gives a sufficient condition. The \emph{gap} of
$w=(x,y)\in K$ is
\[
\mathrm{gap}(x,y)=\max_{y'\in Y_\rho}\Psi(x,y')-\min_{x'\in Q}\Psi(x',y)
=\max_{(x',y')\in K}\bigl(\Psi(x,y')-\Psi(x',y)\bigr)\ \ge0 .
\]

\begin{alemma}[The gap controls the objective error and the violation]
\label{lem:fc-certificate}
Assume \eqref{eq:fc-mult}. For every $(x,y)\in K$,
\[
f(x)-\fstar+\rho\,v(x)\le\mathrm{gap}(x,y)\qquad\text{and}\qquad
f(x)-\fstar\ge-\bar y\,v(x).
\]
Consequently $[f(x)-\fstar]_+\le\mathrm{gap}(x,y)$ and
$v(x)\le\mathrm{gap}(x,y)/(\rho-\bar y)$.
\end{alemma}

\begin{proof}
The support function of $Y_\rho$ gives
$\max_{y'\in Y_\rho}\inner{y'}{g(x)}=\rho v(x)$, and feasibility of
$x^\star$ gives $\min_{x'\in Q}\Psi(x',y)\le\Psi(x^\star,y)\le\fstar$. This
proves the first inequality. The multiplier condition gives
$f(x)-\fstar\ge-\inner{y^\star}{g(x)}\ge-\bar yv(x)$. The two conclusions
follow from these inequalities and $\mathrm{gap}\ge0$.
\end{proof}

\begin{alemma}[A Lipschitz constant of the operator]
\label{lem:fc-lipschitz}
Let $f,g_1,\dots,g_{{n_{\mathrm c}}}$ be twice continuously differentiable on an open set
containing $Q$, with $0\preceq\nabla^2g_i(x)\preceq LI$ for $x\in Q$ and
$i=0,\dots,{n_{\mathrm c}}$, and $\norm{\nabla g_i(x)}\le G_g$ for $x\in Q$ and
$i=1,\dots,{n_{\mathrm c}}$. Then $H$ is $M_H$-Lipschitz on $K$ with
$M_H=(1+\rho)L+2\sqrt{{n_{\mathrm c}}}\,G_g$.
\end{alemma}

\begin{proof}
For $w=(x,y)$ and $w'=(x',y')$ in $K$, the Hessian bounds, integration
along segments and $\norm{y-y'}_1\le\sqrt{{n_{\mathrm c}}}\norm{y-y'}$ show that the two
blocks of $H(w)-H(w')$ have norms at most
$(1+\rho)L\norm{x-x'}+\sqrt{{n_{\mathrm c}}}G_g\norm{y-y'}$ and
$\sqrt{{n_{\mathrm c}}}G_g\norm{x-x'}$.
Add these bounds and use $\norm{x-x'},\norm{y-y'}\le\norm{w-w'}$.
\end{proof}

\begin{aremark}[The size of the multipliers]
\label{rem:fc-multiplier}
\emph{Slater points.} Let $x_\circ\in Q$ and $\theta>0$ with
$g_i(x_\circ)\le-\theta$ for all $i$. Every $y^\star$ as in \eqref{eq:fc-mult}
satisfies $\norm{y^\star}_1\le(f(x_\circ)-\fstar)/\theta$, because
$\fstar\le f(x_\circ)+\inner{y^\star}{g(x_\circ)}\le f(x_\circ)-\theta
\norm{y^\star}_1$. If $\norm{\nabla f}\le G_0$ on $Q$, then
$f(x_\circ)-\fstar\le G_0D$, so $\bar y=G_0D/\theta$ is admissible. If
$x_\circ$ lies in the relative interior of $Q$, a vector $y^\star$ as in
\eqref{eq:fc-mult} exists \citep[Theorem~28.2]{Rockafellar1970}.

\emph{Smoothness does not bound the multipliers.} On $Q=[0,1]\subset\R$, for
$c>0$, the problem of minimizing $f(x)=-cx$ subject to $x-\frac12\le0$ has
$\fstar=-c/2$. For $y\ge0$ the infimum over $[0,1]$ of
$f(x)+y(x-\frac12)=(y-c)x-y/2$ is $-y/2$ if $y\ge c$ and $y/2-c$ if $y<c$,
so the unique multiplier is $y^\star=c$. All derivatives of order at least
two of $f$ and of the constraint vanish, so the Hessian bound and the
H\"older constants are zero for every $c$. Thus $\bar y$ is a parameter of
the problem that is not controlled by the smoothness parameters.
\end{aremark}

\paragraph{Oracle.}
The method works in rounds as in \cref{sec:setup}. Here $\mathcal F_{k-1}$
is generated by all queries and all returned values of rounds
$1,\dots,k-1$ and by $\Theta_1,\dots,\Theta_{k-1}$. The queries of a round are
ordered in a way that does not depend on $\Theta_k$. For a query $t$ of round
$k$, the $\sigma$-algebra $\mathcal E_t$ is generated by $\mathcal F_{k-1}$,
$\Theta_k$, the queries of round $k$ and all error components of the queries
of round $k$ that precede $t$.

\begin{adefinition}[Value calls and level calls]
\label{def:fc-oracle}
A query is of one of two kinds.
\begin{enumerate}[label=(\roman*)]
\item A \emph{value call} at $X_t$ returns the
${n_{\mathrm c}}+1$ values
\[
Y_{i,t}=g_i(X_t)+\eta_i(X_t)+\xi^{\mathrm c}_{i,t}+\xi^{\mathrm o}_{i,t}
+\delta_{i,t},\qquad i=0,\dots,{n_{\mathrm c}} ,
\]
where, for each $i$, $\eta_i$ is fixed before the run and the four
components satisfy \eqref{eq:noise-lip}--\eqref{eq:noise-adv} at levels
$(M,\sigma,\delta)$ with the $\sigma$-algebras above; in
\eqref{eq:noise-oblivious}, the vector of the errors $\xi^{\mathrm o}_{i,t}$
of the value calls of round $k$ is independent of $\Theta_k$ given
$\mathcal F_{k-1}$.
\item A \emph{level call} at $X_t$ returns, for a level $\sigma_g\ge0$, the
${n_{\mathrm c}}$ values
\[
Z_{i,t}=g_i(X_t)+\xi^{\mathrm c}_{i,t}+\delta_{i,t},\qquad
i=1,\dots,{n_{\mathrm c}} ,
\]
with $\E[\xi^{\mathrm c}_{i,t}\mid\mathcal E_t]=0$,
$\E[(\xi^{\mathrm c}_{i,t})^2\mid\mathcal E_t]\le\sigma_g^2$ and
$\abs{\delta_{i,t}}\le\delta$.
\end{enumerate}
The components of one call may depend on each other. We count the calls of
each kind and the returned scalar values.
\end{adefinition}

Thus the values that enter differences may carry all four kinds of error of
\cref{def:oracle}, while the levels $g_i(x)$ are observed with centered noise
and a bounded adversarial error only. A level call is a value call whose
constraint components have neither a Lipschitz nor an oblivious part, with
the objective component discarded; an oracle whose constraint components
satisfy these restrictions provides both kinds of calls, with
$\sigma_g=\sigma$. \Cref{prop:fc-obstruction} shows that some restriction of
this kind on the levels is necessary.

\subsection{The levels of the constraints must be calibrated}
\label{app:fc-levels}

\begin{aproposition}[A constant error in the constraint values]
\label{prop:fc-obstruction}
Let $d={n_{\mathrm c}}=1$, $Q=[-1,1]$, $f(x)=-x$ and
$0<a<1$. For $\theta\in\{-a,a\}$ let
$g_\theta(x)=x-\theta$, and let $\fstar_\theta$ and
$v_\theta(x)=[x-\theta]_+$ be the optimal value and the violation of the
problem \eqref{eq:fc-problem} with the constraint $g_\theta$.
\begin{enumerate}[label=(\alph*)]
\item Both problems satisfy the assumptions of \cref{app:fc-setting} with
$\bar y=1$: the minimizer is $x^\star_\theta=\theta$, $\fstar_\theta=-\theta$,
the vector $y^\star=1$ satisfies \eqref{eq:fc-mult}, and $x_\circ=-1$ satisfies
$g_\theta(x_\circ)\le-(1-a)$. All derivatives of order at least two of $f$
and $g_\theta$ vanish.
\item Let the values of the objective be exact, and let the values of the
constraint carry the fixed error $\eta_\theta\equiv\theta$, which is a
constant and hence $M$-Lipschitz for every $M\ge0$; all other
error components vanish. Then every returned value of the constraint at a
query $X$ equals $X$, whatever $\theta$. Consequently, for every method,
whatever its number of calls and rounds and wherever it queries, the output
$\hat x\in Q$ has the same law under both problems, and
\[
\max_{\theta\in\{-a,a\}}\max\Bigl\{\E_\theta\bigl[f(\hat x)-\fstar_\theta\bigr]_+,\
\E_\theta v_\theta(\hat x)\Bigr\}\ge a .
\]
The output $\hat x=0$ attains equality.
\item The same conclusion holds if the constant $\theta$ is, instead of
$\eta_\theta$, an oblivious error $\xi^{\mathrm o}_t=\theta$ of every
constraint value, which is admissible at levels $(0,\sigma,0)$ for
$\sigma\ge a$, or an adversarial error $\delta_t=\theta$, which is admissible
at levels $(0,0,\delta)$ for $\delta\ge a$.
\end{enumerate}
\end{aproposition}

\begin{proof}
(a) The feasible set is $[-1,\theta]$ and $f$ is decreasing, so
$x^\star_\theta=\theta$ and $\fstar_\theta=-\theta$. For every $x$,
$f(x)+g_\theta(x)=-\theta=\fstar_\theta$, so $y^\star=1$ satisfies
\eqref{eq:fc-mult} with $\bar y=1$. Finally
$g_\theta(-1)=-1-\theta\le-1+a$, and $f$, $g_\theta$ are affine.

(b) The returned value of the constraint at $X$ is $g_\theta(X)+\theta=X$,
and that of the objective is $-X$. So under both problems every returned
value is the same function of the query. By induction over the rounds, the
queries, the returned values and the output are the same measurable functions
of the random elements $\Theta_1,\Theta_2,\dots$ of the method under both
problems, and their laws coincide. For every $x\in Q$,
\[
\bigl[f(x)-\fstar_a\bigr]_++v_{-a}(x)=[a-x]_++[x+a]_+\ge(a-x)+(x+a)=2a .
\]
Taking expectations under the common law of $\hat x$ gives
$\E_a[f(\hat x)-\fstar_a]_++\E_{-a}v_{-a}(\hat x)\ge2a$, so one of the two
terms is at least $a$. For $\hat x=0$ we have
$[f(0)-\fstar_\theta]_+=[\theta]_+$ and $v_\theta(0)=[-\theta]_+$, so the
maximum over $\theta$ and over the two errors equals $a$.

(c) In both cases the returned value of the constraint at $X$ is again $X$,
and the argument of (b) applies. A constant error is independent of the
randomness of the method and has second moment $a^2$.
\end{proof}

\begin{aremark}[Relation to \citet{Abronin2026}]
\label{rem:fc-abronin}
\citet{Abronin2026} study \eqref{eq:fc-problem} with $f(x)=\E F(x,\xi)$ and
$g(x)=\E G(x,\xi)$ for an oracle that returns $(F(x,\xi),G(x,\xi))$ at two
points with a common sample $\xi$. Their gradient estimates use differences
at a common sample, which cancel every error that does not depend on the
point, and their Lipschitz assumption on the samples (their Assumption~2.2)
concerns only such differences. The levels enter the dual updates; for them
they assume $\E\norm{G(x,\xi)-g(x)}_\infty^2\le\sigma_g^2$ (their
Assumption~2.3), which they call a genuine additional assumption (their
Remark~2.7(ii)).
\Cref{prop:fc-obstruction} shows that an assumption of this kind cannot be
dropped. Our level calls play the role of their Assumption~2.3, with an
additional adversarial error; our bound on each component implies theirs
with ${n_{\mathrm c}}\sigma_g^2$ in place of $\sigma_g^2$.

Their Theorem~9.3 shows that the term of order $\sigma_g^2/\eps^2$ in the
number of level calls of \cref{cor:fc-stencil} cannot be avoided: its
instances $G_i(x,\xi)=x-\xi_i$ are realized by our level calls, while under
the constant Lipschitz errors $\eta_i\equiv\E\xi_i$ the value calls carry no
information on the levels.
\end{aremark}

\subsection{Mirror--Prox with biased operator estimates}
\label{app:fc-mirror-prox}

The method is the stochastic Mirror--Prox method of \citet{JNT2011} in the
Euclidean setup, that is, a stochastic extragradient method with projections
onto $K$. Each iteration takes two rounds, because the second estimate is
requested at a point that depends on the first one.

\begin{atheorem}[Stochastic Mirror--Prox with bounded bias]
\label{thm:fc-mirror-prox}
Assume the setting of \cref{app:fc-setting}, including \eqref{eq:fc-mult}.
Let $(\mathcal F_k)_{k\ge0}$ be a filtration. In round $k$ the method submits
an $\mathcal F_{k-1}$-measurable point $w\in K$ and receives an
$\mathcal F_k$-measurable estimate
\[
\hat H=H(w)+\mu_k+\zeta_k,
\]
where $\zeta_k$ is $\mathcal F_k$-measurable,
$\E[\zeta_k\mid\mathcal F_{k-1}]=0$,
$\E[\norm{\zeta_k}^2\mid\mathcal F_{k-1}]\le V_K$, and
$\norm{\mu_k}\le b_K$ almost surely. Let $w_1\in K$ be deterministic and
$0<\nu\le1/(2M_H)$. For $t=1,\dots,N$:
\begin{itemize}
\item in round $2t-1$, obtain the estimate $\hat H_t$ at $w_t$ and put
$w_{t+1/2}=\Proj_K(w_t-\nu\hat H_t)$;
\item in round $2t$, obtain the estimate $\hat H_{t+1/2}$ at $w_{t+1/2}$ and
put $w_{t+1}=\Proj_K(w_t-\nu\hat H_{t+1/2})$.
\end{itemize}
Let $\hat w_N=(\hat x_N,\hat y_N)=\frac1N\sum_{t=1}^Nw_{t+1/2}$. Then
\begin{equation}
\label{eq:fc-gap}
\E\,\mathrm{gap}(\hat w_N)\le\frac{D_K^2}{2\nu N}+2\nu V_K
+D_K\sqrt{\frac{V_K}N}+3D_Kb_K .
\end{equation}
With $\nu=\min\{1/(2M_H),D_K/(2\sqrt{NV_K})\}$, and $\nu=1/(2M_H)$ if
$V_K=0$,
\begin{equation}
\label{eq:fc-gap-tuned}
\E\,\mathrm{gap}(\hat w_N)\le\mathrm{Err}_N:=\frac{M_HD_K^2}{N}
+3D_K\sqrt{\frac{V_K}N}+3D_Kb_K .
\end{equation}
Consequently $\hat x_N\in Q$, $\E[f(\hat x_N)-\fstar]_+\le\mathrm{Err}_N$ and
$\E v(\hat x_N)\le\mathrm{Err}_N/(\rho-\bar y)$. The method uses $2N$ rounds.
\end{atheorem}

\begin{proof}
A bound of this form is \citet[Theorem~1]{JNT2011}, whose oracle assumption
(Assumption~I there) also allows a bias; we give the short argument because
here the errors are martingale differences with a history-dependent bias and
because its constants are used below. Fix $u\in K$ and put
$r=\norm{w_t-w_{t+1/2}}$ and $r'=\norm{w_{t+1/2}-w_{t+1}}$. Adding the
projection inequalities for $w_{t+1/2}$, tested at $w_{t+1}$, and for
$w_{t+1}$, tested at $u$, bounds $\nu\inner{\hat H_{t+1/2}}{w_{t+1/2}-u}$ by
$(\norm{w_t-u}^2-\norm{w_{t+1}-u}^2-r^2-r'^2)/2$ plus
$\nu\inner{\hat H_{t+1/2}-\hat H_t}{w_{t+1/2}-w_{t+1}}$. In this term the
operator difference contributes at most $\nu M_Hrr'\le(r^2+r'^2)/4$, the two
biases at most $2\nu b_KD_K$, and the noise at most
$\nu^2\norm{\zeta_{t+1/2}-\zeta_t}^2+r'^2/4$. Replacing $\hat H_{t+1/2}$ by
$H(w_{t+1/2})$ on the left costs at most $\nu b_KD_K$ and a martingale term,
so
\begin{equation}\label{eq:fc-step}
\inner{H(w_{t+1/2})}{w_{t+1/2}-u}\le
\frac{\norm{w_t-u}^2-\norm{w_{t+1}-u}^2}{2\nu}
+\nu\norm{\zeta_{t+1/2}-\zeta_t}^2+3b_KD_K
-\inner{\zeta_{t+1/2}}{w_{t+1/2}-u}.
\end{equation}
For $u=(x',y')$, convexity in $x$ and affinity in $y$ give
$\inner{H(x,y)}{(x,y)-u}\ge\Psi(x,y')-\Psi(x',y)$. Sum \eqref{eq:fc-step},
use Jensen's inequality, write $w_{t+1/2}-u=(w_{t+1/2}-u_0)-(u-u_0)$ for a
fixed $u_0\in K$, and take the supremum over $u$:
\begin{align*}
N\,\mathrm{gap}(\hat w_N)\le{}&D_K^2/(2\nu)
+\nu\sum_t\norm{\zeta_{t+1/2}-\zeta_t}^2+3Nb_KD_K-\sum_t\inner{\zeta_{t+1/2}}{w_{t+1/2}-u_0}
+D_K\norm{\sum_t\zeta_{t+1/2}}.
\end{align*}
The inner products have mean zero. Martingale differences are orthogonal, so
$\E\norm{\zeta_{t+1/2}-\zeta_t}^2\le2V_K$ and
$\E\norm{\sum_t\zeta_{t+1/2}}\le\sqrt{NV_K}$, which proves \eqref{eq:fc-gap}.
For the tuned step, $1/\nu\le2M_H+2\sqrt{NV_K}/D_K$ and
$2\nu V_K\le D_K\sqrt{V_K/N}$, which proves \eqref{eq:fc-gap-tuned}.
\Cref{lem:fc-certificate} gives the last two bounds.
\end{proof}

\subsection{Operator estimates from the stencil}
\label{app:fc-stencil}

Fix a round $k$, a radius $h>0$, integers $B,B_L\ge1$ and an
$\mathcal F_{k-1}$-measurable point $w=(x,y)\in K$. The \emph{operator
estimate} $\hat H(w)=(\hat H^x,\hat H^y)$ at $w$ consists of two parts.
\begin{enumerate}[label=(\roman*)]
\item \emph{Value calls.} Run \cref{alg:stencil} at the center $x$ with radius
$h$ and $B$ frames, and apply its formula to each of the
${n_{\mathrm c}}+1$ components of the returned values;
this gives vectors
$\gbar_0,\dots,\gbar_{{n_{\mathrm c}}}\in\R^d$. Put
$\hat H^x=\gbar_0+\sum_{i=1}^{{n_{\mathrm c}}}y_i\gbar_i$.
Since the formula is linear in the
values, $\hat H^x$ is \cref{alg:stencil} applied to the values of the
Lagrangian $\Psi(\cdot,y)$ at the frozen multiplier $y$. This part uses
$2mdB$ value calls and $({n_{\mathrm c}}+1)2mdB$ scalar values.
\item \emph{Level calls.} Request $B_L$ level calls at $x$, with returned
vectors $Z_{\cdot,\ell}\in\R^{{n_{\mathrm c}}}$,
$\ell=1,\dots,B_L$, and put
$\hat H^y=-\frac1{B_L}\sum_{\ell=1}^{B_L}Z_{\cdot,\ell}$. This part uses $B_L$ level
calls and ${n_{\mathrm c}}B_L$ scalar values.
\end{enumerate}
All queries of the estimate are determined at the start of the round, so the
estimate takes one round; the value calls precede the level calls.

\begin{alemma}[Bias and variance of the operator estimate]
\label{lem:fc-estimate}
Let $\vartheta$ be as in \cref{lem:cvx-local-estimate}, and let
$g_0=f,g_1,\dots,g_{{n_{\mathrm c}}}$ be $q$ times continuously differentiable on an open set
containing $Q_{\vartheta h}$, with
$\norm{D^qg_i(x)-D^qg_i(x')}\le L_\beta\norm{x-x'}^\gamma$ for $x,x'\in
Q_{\vartheta h}$ and $i=0,\dots,{n_{\mathrm c}}$, in the chosen norm of \cref{sec:setup}
(with the corresponding exponent $\tau$). Let the calls satisfy
\cref{def:fc-oracle}. Then the value calls lie in $Q_h$, the level calls are
made at $x\in Q$, and $\hat H(w)=H(w)+\mu+\zeta$, where $\mu$ is
$\mathcal F_{k-1}$-measurable, $\zeta$ is $\mathcal F_k$-measurable,
$\E[\zeta\mid\mathcal F_{k-1}]=0$,
\begin{equation}
\label{eq:fc-bias}
\norm\mu\le b_K:=(1+\rho)\Bigl(\bfC L_\beta d^{-\tau/2}h^{\beta-1}+A_1M
+\frac{\sqrt d\,S_1\delta}h\Bigr)
+\sqrt{{n_{\mathrm c}}}\,\delta ,
\end{equation}
and, with
$\Omega_h=\Cf L_\beta^2d^{-\tau}h^{2\beta-2}+C_\eta M^2+C_\sigma d\sigma^2/h^2$,
where $\bfC$, $\Cf$, $C_\eta=4A_1^2$ and $C_\sigma$ are the constants of
\cref{prop:unified-estimator},
\begin{equation}
\label{eq:fc-var}
\E\bigl[\norm\zeta^2\bigm|\mathcal F_{k-1}\bigr]\le V_K:=
2(1+\rho)^2\Bigl(\frac{\Omega_h}B+\frac{dS_1^2\delta^2}{h^2}\Bigr)
+2{n_{\mathrm c}}\Bigl(\frac{\sigma_g^2}{B_L}+\delta^2\Bigr).
\end{equation}
\end{alemma}

\begin{proof}
Write $\mathsf N(X)=(\E[\norm X^2\mid\mathcal F_{k-1}])^{1/2}$.

\emph{Primal block.} Fix
$i\in\{0,\dots,{n_{\mathrm c}}\}$. The $i$-th components of the
value calls form an oracle for $g_i$ as in \cref{def:oracle} at levels
$(M,\sigma,\delta)$, with the $\sigma$-algebras of
\cref{def:fc-oracle}. The proofs of \cref{prop:unified-estimator} and
\cref{lem:cvx-local-estimate} use the $\sigma$-algebras only through the
$\mathcal F_{k-1}$-measurability of the center and the radius, the
independence of $\Theta_k$ and $\mathcal F_{k-1}$, the condition
\eqref{eq:noise-centered} with $\sigma$-algebras $\mathcal E_t$ that contain
$\Theta_k$ and the errors of the earlier queries of the round, the condition
\eqref{eq:noise-oblivious}, and the fact that the order of the queries does
not depend on $\Theta_k$. These properties hold here; the larger
$\sigma$-algebras and the level calls, which follow the value calls, do not
affect them. Hence \cref{lem:cvx-local-estimate}, with
its convex set (called $K$ there) equal to $Q$, and part~(iv) of
\cref{prop:unified-estimator} give
$\gbar_i=\nabla g_i(x)+\mu_i+\tilde\zeta_i$, where
$\mu_i=b_i(x)+\E[\Delta_i\mid\mathcal F_{k-1}]$ is
$\mathcal F_{k-1}$-measurable with
$\norm{\mu_i}\le\bfC L_\beta d^{-\tau/2}h^{\beta-1}+A_1M+\sqrt d\,S_1\delta/h$,
and $\tilde\zeta_i=\zeta_i+\Delta_i-\E[\Delta_i\mid\mathcal F_{k-1}]$ is
$\mathcal F_k$-measurable with $\E[\tilde\zeta_i\mid\mathcal F_{k-1}]=0$ and
$\mathsf N(\tilde\zeta_i)\le(\Omega_h/B)^{1/2}+\sqrt d\,S_1\delta/h$; here
$b_i(x)$, $\zeta_i$ and $\Delta_i$ are the terms $b(z)$, $\zeta$ and $\Delta$
of \cref{prop:unified-estimator} for $g_i$. With $y_0=1$,
\[
\hat H^x=\nabla f(x)+\sum_{i\ge1}y_i\nabla g_i(x)+\mu^x+\zeta^x,\qquad
\mu^x=\sum_{i=0}^{{n_{\mathrm c}}}y_i\mu_i,
\qquad
\zeta^x=\sum_{i=0}^{{n_{\mathrm c}}}y_i\tilde\zeta_i .
\]
Since $y$ is $\mathcal F_{k-1}$-measurable, $y\ge0$ and
$\sum_{i=0}^{{n_{\mathrm c}}}y_i\le1+\rho$, we have
$\norm{\mu^x}\le(1+\rho)\max_i\norm{\mu_i}$ and
$\E[\zeta^x\mid\mathcal F_{k-1}]=0$, and the conditional Minkowski inequality
gives $\mathsf N(\zeta^x)^2\le2(1+\rho)^2(\Omega_h/B+dS_1^2\delta^2/h^2)$.

\emph{Dual block.} Let $\bar\xi$ and $\Delta^L$ be the averages over the
$B_L$ level calls of the vectors $(\xi^{\mathrm c}_{i,\ell})_i$ and
$(\delta_{i,\ell})_i$ of their errors. Then $\hat H^y=-g(x)+\mu^y+\zeta^y$ with
$\mu^y=-\E[\Delta^L\mid\mathcal F_{k-1}]$ and
$\zeta^y=-(\bar\xi+\Delta^L-\E[\Delta^L\mid\mathcal F_{k-1}])$, so that
$\norm{\mu^y}\le\sqrt{{n_{\mathrm c}}}\,\delta$,
$\zeta^y$ is $\mathcal F_k$-measurable and
$\E[\zeta^y\mid\mathcal F_{k-1}]=0$. For $\ell<\ell'$ the error
$\xi^{\mathrm c}_{i,\ell}$ is measurable with respect to the $\sigma$-algebra
$\mathcal E_t$ of the $\ell'$-th level call, so the cross terms vanish and
$\mathsf N(\bar\xi)^2\le
{n_{\mathrm c}}\sigma_g^2/B_L$. Moreover
$\mathsf N(\Delta^L-\E[\Delta^L\mid\mathcal F_{k-1}])^2\le
{n_{\mathrm c}}\delta^2$, and
the conditional Minkowski inequality gives
$\mathsf N(\zeta^y)^2\le
2{n_{\mathrm c}}(\sigma_g^2/B_L+\delta^2)$.

\emph{Both blocks.} Put $\mu=(\mu^x,\mu^y)$ and $\zeta=(\zeta^x,\zeta^y)$.
Then $\norm\mu\le\norm{\mu^x}+\norm{\mu^y}\le b_K$ and
$\E[\norm\zeta^2\mid\mathcal F_{k-1}]=\mathsf N(\zeta^x)^2
+\mathsf N(\zeta^y)^2\le V_K$; no independence between the blocks is
needed. The value calls are the queries of \cref{alg:stencil} at $x\in Q$
and lie in $Q_h$.
\end{proof}

\begin{algorithm}[t]
\caption{Calibrated zeroth-order Mirror--Prox}
\label{alg:functional-mirror-prox}
\begin{algorithmic}[1]
\Require $x_0\in Q$, $\Proj_K$; value and level calls as in
\cref{def:fc-oracle}; $h,B,B_L,N,\nu$ as in \cref{cor:fc-stencil}
\State $w_1\gets(x_0,0)$
\For{$t=1,\ldots,N$}
  \State $a_t\gets\hat H(w_t)$, the operator estimate of
  \cref{app:fc-stencil} (one round)
  \State $w_{t+1/2}\gets\Proj_K(w_t-\nu a_t)$
  \State $b_t\gets\hat H(w_{t+1/2})$ (one round)
  \State $w_{t+1}\gets\Proj_K(w_t-\nu b_t)$
\EndFor
\State \Return the $x$-component of $N^{-1}\sum_{t=1}^N w_{t+1/2}$
\end{algorithmic}
\end{algorithm}

\begin{acorollary}[Mirror--Prox with the stencil]
\label{cor:fc-stencil}
Assume the setting of \cref{app:fc-setting} with $\rho=\bar y+1$, so that
$D_K=(D^2+2\rho^2)^{1/2}$, and let $H$ be $M_H$-Lipschitz on $K$ for some
$M_H>0$, for instance with $M_H$ of \cref{lem:fc-lipschitz} when it is
positive. The method uses the quantities below, which depend on
$\sigma,\sigma_g,\bar y,D,L_\beta,M_H,{n_{\mathrm c}},d$
and $\eps$, but not on $M$ and $\delta$; upper bounds on
$\sigma,\sigma_g,\bar y,D,L_\beta,M_H$ and
${n_{\mathrm c}}$ suffice.
Let $\eps>0$ and put
\[
h=\Bigl(\frac{\eps\,d^{\tau/2}}{24(1+\rho)\bfC L_\beta D_K}\Bigr)^{\frac1{\beta-1}},
\qquad
\bar M=\frac{\eps}{24(1+\rho)A_1D_K},\qquad
\bar\delta=\min\Bigl\{\frac{\eps h}{42(1+\rho)S_1D_K\sqrt d},\
\frac{\eps}{42\sqrt{{n_{\mathrm c}}}\,D_K}\Bigr\},
\]
\[
\bar\Omega=\Cf L_\beta^2d^{-\tau}h^{2\beta-2}+C_\eta\bar M^2
+C_\sigma\frac{d\sigma^2}{h^2},\qquad
N=\Bigl\lceil\frac{4M_HD_K^2}{\eps}\Bigr\rceil,
\]
\[
B=\max\Bigl\{1,\Bigl\lceil\frac{864(1+\rho)^2D_K^2\bar\Omega}{N\eps^2}
\Bigr\rceil\Bigr\},\qquad
B_L=\max\Bigl\{1,\Bigl\lceil
\frac{864\,{n_{\mathrm c}}\sigma_g^2D_K^2}{N\eps^2}
\Bigr\rceil\Bigr\},
\]
\[
\bar V=\frac{2(1+\rho)^2\bar\Omega}B
+\frac{2{n_{\mathrm c}}\sigma_g^2}{B_L}
+\frac{\eps^2}{432D_K^2},\qquad
\nu=\min\Bigl\{\frac1{2M_H},\frac{D_K}{2\sqrt{N\bar V}}\Bigr\}.
\]
Let $g_0=f,g_1,\dots,g_{{n_{\mathrm c}}}$ satisfy the assumptions of \cref{lem:fc-estimate}
on $Q_{\vartheta h}$, and let the calls satisfy \cref{def:fc-oracle} with
$M\le\bar M$ and $\delta\le\bar\delta$. Run
\cref{alg:functional-mirror-prox}, that is, the method of
\cref{thm:fc-mirror-prox} from $w_1=(x_0,0)$, for some $x_0\in Q$, with the
step $\nu$ and the operator estimates of \cref{lem:fc-estimate} with radius
$h$, $B$ frames and $B_L$ level calls. Then the output $\hat x=\hat x_N\in Q$
satisfies
\[
\E\bigl[f(\hat x)-\fstar\bigr]_+\le\eps\qquad\text{and}\qquad
\E v(\hat x)\le\eps
\]
after $2N$ rounds, all value calls lie in $Q_h$, and the numbers $T_V$ of
value calls and $T_L$ of level calls satisfy
\begin{align*}
T_V&\le3456\,mC_\sigma(24\bfC)^{\frac2{\beta-1}}\frac{\sigma^2}{\eps^2}\,
d^{\,2-\frac\tau{\beta-1}}\Bigl(\frac{S_K}{\eps}\Bigr)^{\frac2{\beta-1}}
+6md\Bigl(\frac{\Cf}{\bfC^2}+\frac{C_\eta}{A_1^2}\Bigr)+4mdN,
\qquad S_K=(1+\rho)^\beta L_\beta D_K^\beta,\\
T_L&\le\frac{1728\,{n_{\mathrm c}}\sigma_g^2D_K^2}{\eps^2}+2N,
\qquad N\le\frac{4M_HD_K^2}{\eps}+1 .
\end{align*}
Each value call returns ${n_{\mathrm c}}+1$ scalar values
and each level call returns ${n_{\mathrm c}}$.
\end{acorollary}

\begin{proof}
By \cref{lem:fc-estimate}, every operator estimate of the run has the form
required in \cref{thm:fc-mirror-prox}, with $b_K$ of \eqref{eq:fc-bias} and
$V_K$ of \eqref{eq:fc-var}. Since $M\le\bar M$, we have
$\Omega_h\le\bar\Omega$. Since $\delta\le\bar\delta$, the terms of $V_K$ with
$\delta$ are at most $4\eps^2/(1764D_K^2)\le\eps^2/(432D_K^2)$. Hence
$V_K\le\bar V$, and \cref{thm:fc-mirror-prox} applies with $\bar V$ in place
of $V_K$ and the step $\nu$:
\[
\E\,\mathrm{gap}(\hat w_N)\le M_HD_K^2/N+3D_K\sqrt{\bar V/N}+3D_Kb_K.
\]
The choice of $N$ gives $M_HD_K^2/N\le\eps/4$. The choices of $B$ and $B_L$
give $2(1+\rho)^2\bar\Omega/B\le N\eps^2/(432D_K^2)$ and
$2{n_{\mathrm c}}\sigma_g^2/B_L
\le N\eps^2/(432D_K^2)$, so $\bar V\le N\eps^2/(144D_K^2)$
and $3D_K\sqrt{\bar V/N}\le\eps/4$. The choices of $h$, $\bar M$ and
$\bar\delta$ bound the three parts of $b_K$ by $\eps/(24D_K)$,
$\eps/(24D_K)$ and $\eps/(21D_K)$, so $3D_Kb_K\le\eps/8+\eps/8+\eps/7<\eps/2$.
Thus $\E\,\mathrm{gap}(\hat w_N)\le\eps$, and \cref{lem:fc-certificate} with
$\rho-\bar y=1$ gives the two claims.

Each of the $2N$ rounds uses $2mdB$ value calls and $B_L$ level calls. Since
$NB\le N+864(1+\rho)^2D_K^2\bar\Omega/\eps^2$ and
$NB_L\le N+864{n_{\mathrm c}}\sigma_g^2D_K^2/\eps^2$,
we get
$T_L\le2N+1728{n_{\mathrm c}}\sigma_g^2D_K^2/\eps^2$ and
$T_V\le4mdN+3456md(1+\rho)^2D_K^2\bar\Omega/\eps^2$. By the choice of $h$ and
$\bar M$, the first two terms of $\bar\Omega$ equal
$\Cf\eps^2/(576(1+\rho)^2\bfC^2D_K^2)$ and
$C_\eta\eps^2/(576(1+\rho)^2A_1^2D_K^2)$; they contribute
$6md(\Cf/\bfC^2+C_\eta/A_1^2)$ to $T_V$. The term $C_\sigma d\sigma^2/h^2$
contributes the first term of the bound on $T_V$, since
$h^{-2}=(24\bfC(1+\rho)L_\beta D_K/\eps)^{2/(\beta-1)}d^{-\tau/(\beta-1)}$.
\end{proof}

\section{Details of the experiments}
\label{app:experiments}

\pgfplotsset{appgrid/.style={
  scale only axis, width=6.4cm, height=4.4cm, grid=major, grid style={gray!18},
  tick label style={font=\footnotesize}, label style={font=\footnotesize},
  title style={font=\small, yshift=-0.4ex},
  legend style={font=\footnotesize, draw=none, fill=none, at={(0.5,0)}, anchor=north,
    yshift=-0.9cm, legend columns=2, cells={anchor=west}, row sep=-1.5pt,
    /tikz/every even column/.append style={column sep=8pt}},
  every axis plot/.append style={mark size=1.6pt},
  error bars/error bar style={solid}}}

\paragraph{Code, randomness and reporting.}
{
The supplementary script \texttt{experiments/run\_all.py}, which uses NumPy
\citep{Harris2020NumPy}, produces Experiments~1--7 and the robustness checks
reported with Experiments~3 and~5. All random streams are derived from
\texttt{SeedSequence(20260927)}. Instances, tuning runs, reported runs, and
methods use separate streams, and tuning does not reuse evaluation data.
The reported run used Python~3.13.16 and NumPy~2.5.3 on two cores of a
2.80\,GHz Intel Xeon processor. It took about 46 minutes, including 32 minutes
for Experiment~5; \texttt{experiments/run\_log.txt} records the time of each
run.

Vertical uncertainty bars are centered on the plotted markers and represent
one standard error over independent runs. When a standard error is smaller
than its marker, the bar is covered by the marker. Paired ratios in
Experiment~3 and its robustness check use the delta
method. Calls required for a target accuracy are interpolated between adjacent
points of a doubling grid in log--log coordinates. Their standard errors use
400 fixed bootstrap resamples; the grid is extended until every resample
crosses the target. Slopes are least-squares fits in log--log coordinates.
Experiments~5 and~6 obtain their slope errors from bootstrap resamples across
independent dimensions, while the budget check in Experiment~5 uses the delta
method for the mean errors. Haar frames are generated by QR factorization of a
standard Gaussian matrix, with the correction
$Q\gets Q\,\mathrm{diag}(\sgn\mathrm{diag}\,R)$. Each oracle call evaluates the
objective at its stated query and adds the noise assigned to that call.
\Cref{fig:app-experiments-a,fig:app-experiments-b} display all seven
experiments together with the
robustness checks for Experiments~3 and~5.
Panels~(a)--(d) of \cref{fig:experiments} are Experiments~3, 4, 5, and~7.
}

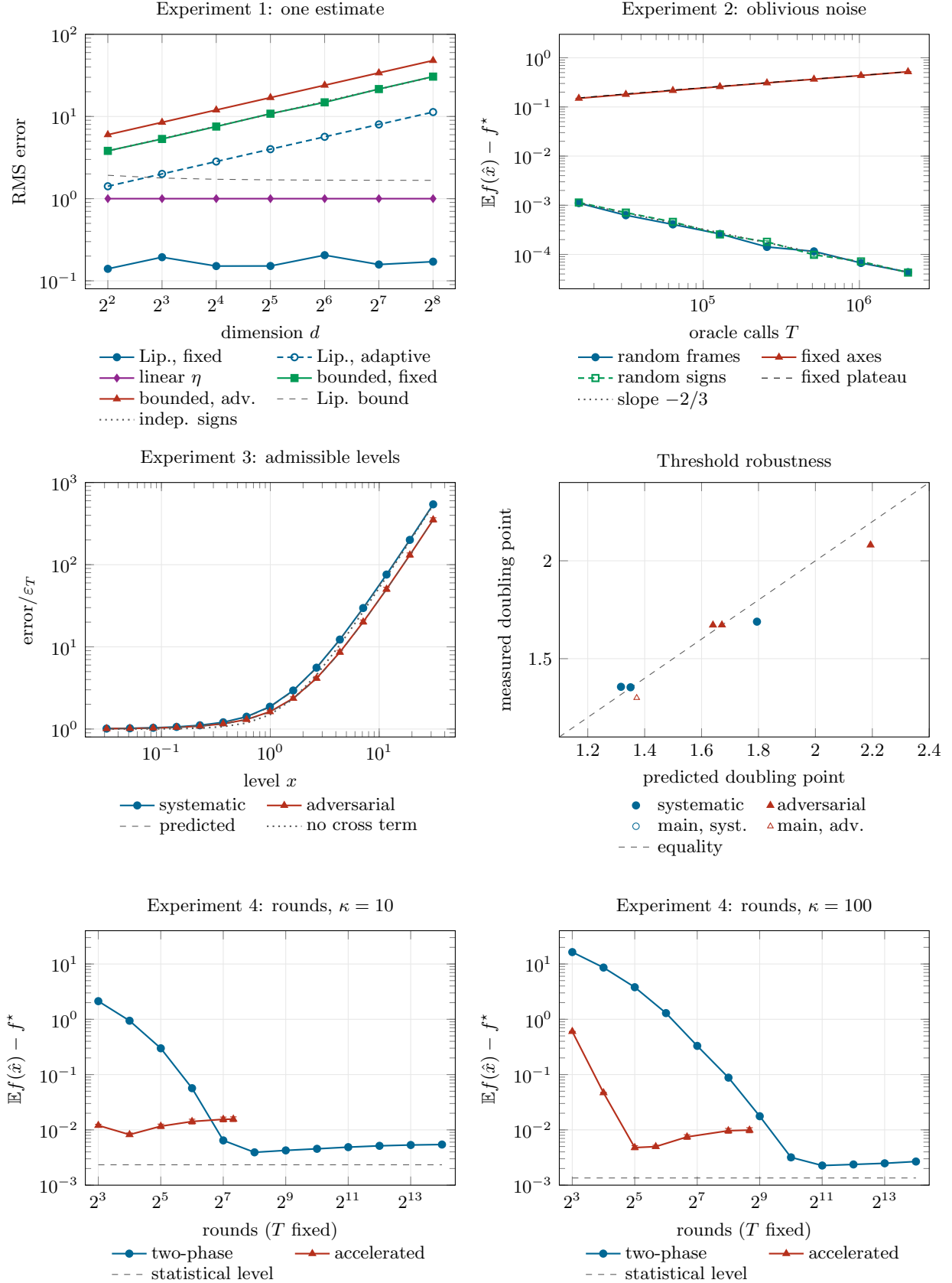
\begin{figure}[h]
\centering
\begin{tikzpicture}
\begin{groupplot}[
  appgrid,
  group style={group size=2 by 3, horizontal sep=1.8cm, vertical sep=3.35cm},
]
\nextgroupplot[title={{Experiment 1: one estimate}},
  xmode=log, ymode=log, log basis x=2,
  xlabel={dimension $d$}, ylabel={RMS error},
  xmin=3, xmax=340, ymin=0.08, ymax=100]
\addplot[thick, MidnightBlue, mark=*, error bars/.cd, y dir=both, y explicit]
  table[x=d, y=rms_lip, y error=se_rms_lip] {experiments/exp1_estimate.dat};
\addlegendentry{{Lip., fixed}}
\addplot[thick, MidnightBlue, densely dashed, mark=o, mark options={solid}]
  table[x=d, y=rms_lip_adaptive] {experiments/exp1_estimate.dat};
\addlegendentry{{Lip., adaptive}}
\addplot[thick, Plum, mark=diamond*]
  table[x=d, y=rms_lin] {experiments/exp1_estimate.dat};
\addlegendentry{linear $\eta$}
\addplot[thick, ForestGreen, mark=square*, error bars/.cd, y dir=both, y explicit]
  table[x=d, y=rms_bounded_fixed, y error=se_rms_bounded_fixed] {experiments/exp1_estimate.dat};
\addlegendentry{bounded, fixed}
\addplot[thick, BrickRed, mark=triangle*]
  table[x=d, y=rms_bounded_adaptive] {experiments/exp1_estimate.dat};
\addlegendentry{{bounded, adv.}}
\addplot[gray!70!black, dashed]
  table[x=d, y=ref_lip_bound] {experiments/exp1_estimate.dat};
\addlegendentry{{Lip. bound}}
\addplot[gray!70!black, dotted, thick]
  table[x=d, y=ref_randsign] {experiments/exp1_estimate.dat};
\addlegendentry{{indep. signs}}

\nextgroupplot[title={{Experiment 2: oblivious noise}},
  xmode=log, ymode=log,
  xlabel={{oracle calls $T$}},
  ylabel={$\E f(\hat x)-\fstar$},
  xmin=1.2e4, xmax=2.8e6, ymin=2e-5, ymax=3]
\addplot[thick, MidnightBlue, mark=*, error bars/.cd, y dir=both, y explicit]
  table[x=T, y=err_frames, y error=se_frames] {experiments/exp2_oblivious.dat};
\addlegendentry{random frames}
\addplot[thick, BrickRed, mark=triangle*, error bars/.cd, y dir=both, y explicit]
  table[x=T, y=err_fixed, y error=se_fixed] {experiments/exp2_oblivious.dat};
\addlegendentry{fixed axes}
\addplot[thick, ForestGreen, densely dashed, mark=square, mark options={solid},
  error bars/.cd, y dir=both, y explicit]
  table[x=T, y=err_randsign, y error=se_randsign] {experiments/exp2_oblivious.dat};
\addlegendentry{random signs}
\addplot[black, dashed]
  table[x=T, y=plateau_fixed] {experiments/exp2_oblivious.dat};
\addlegendentry{{fixed plateau}}
\addplot[gray!70!black, dotted, thick, domain=1.6e4:2.05e6, samples=2]
  {1.1e-3*(x/1.6e4)^(-2/3)};
\addlegendentry{slope $-2/3$}

\nextgroupplot[title={{Experiment 3: admissible levels}},
  xmode=log, ymode=log,
  xlabel={level $x$}, ylabel={error$/\eps_T$},
  xmin=0.02, xmax=50, ymin=0.8, ymax=1000]
\addplot[thick, MidnightBlue, mark=*, error bars/.cd, y dir=both, y explicit]
  table[x=x_M, y=err_lip, y error=se_lip] {experiments/exp3_levels.dat};
\addlegendentry{systematic}
\addplot[thick, BrickRed, mark=triangle*, error bars/.cd, y dir=both, y explicit]
  table[x=x_delta, y=err_adv, y error=se_adv] {experiments/exp3_levels.dat};
\addlegendentry{adversarial}
\addplot[gray!70!black, dashed]
  table[x=x_M, y=pred_lip] {experiments/exp3_levels.dat};
\addlegendentry{predicted}
\addplot[gray!70!black, dashed, forget plot]
  table[x=x_delta, y=pred_adv] {experiments/exp3_levels.dat};
\addplot[gray!70!black, dotted, thick]
  table[x=x_M, y=pred_shift_lip] {experiments/exp3_levels.dat};
\addlegendentry{{no cross term}}

\nextgroupplot[title={{Threshold robustness}},
  xlabel={predicted doubling point}, ylabel={measured doubling point},
  xmin=1.1, xmax=2.4, ymin=1.1, ymax=2.4]
\addplot[only marks, thick, MidnightBlue, mark=*]
  table[x=xM_pred, y=xM_meas] {experiments/exp3_thresholds.dat};
\addlegendentry{systematic}
\addplot[only marks, thick, BrickRed, mark=triangle*]
  table[x=xD_pred, y=xD_meas] {experiments/exp3_thresholds.dat};
\addlegendentry{adversarial}
\addplot[only marks, MidnightBlue, mark=o] coordinates {(1.109,1.07)};
\addlegendentry{{main, syst.}}
\addplot[only marks, BrickRed, mark=triangle] coordinates {(1.372,1.30)};
\addlegendentry{{main, adv.}}
\addplot[gray!70!black, dashed, domain=1.1:2.4, samples=2] {x};
\addlegendentry{equality}

\nextgroupplot[title={{Experiment 4: rounds, $\kappa=10$}},
  xmode=log, ymode=log, log basis x=2,
  xlabel={rounds ({$T$ fixed})},
  ylabel={$\E f(\hat x)-\fstar$},
  xmin=6, xmax=22000, ymin=1e-3, ymax=40]
\addplot[thick, MidnightBlue, mark=*, error bars/.cd, y dir=both, y explicit]
  table[x=N, y=err_tp, y error=se_tp] {experiments/exp4_rounds_kappa10.dat};
\addlegendentry{two-phase}
\addplot[thick, BrickRed, mark=triangle*,
  restrict expr to domain={\thisrow{new_acc}}{1:1},
  error bars/.cd, y dir=both, y explicit]
  table[x=rounds_acc, y=err_acc, y error=se_acc] {experiments/exp4_rounds_kappa10.dat};
\addlegendentry{accelerated}
\addplot[gray!70!black, dashed]
  table[x=N, y=pr_level] {experiments/exp4_rounds_kappa10.dat};
\addlegendentry{statistical level}

\nextgroupplot[title={{Experiment 4: rounds, $\kappa=100$}},
  xmode=log, ymode=log, log basis x=2,
  xlabel={rounds ({$T$ fixed})},
  ylabel={$\E f(\hat x)-\fstar$},
  xmin=6, xmax=22000, ymin=1e-3, ymax=40]
\addplot[thick, MidnightBlue, mark=*, error bars/.cd, y dir=both, y explicit]
  table[x=N, y=err_tp, y error=se_tp] {experiments/exp4_rounds_kappa100.dat};
\addlegendentry{two-phase}
\addplot[thick, BrickRed, mark=triangle*,
  restrict expr to domain={\thisrow{new_acc}}{1:1},
  error bars/.cd, y dir=both, y explicit]
  table[x=rounds_acc, y=err_acc, y error=se_acc] {experiments/exp4_rounds_kappa100.dat};
\addlegendentry{accelerated}
\addplot[gray!70!black, dashed]
  table[x=N, y=pr_level] {experiments/exp4_rounds_kappa100.dat};
\addlegendentry{statistical level}
\end{groupplot}
\end{tikzpicture}
\caption{{Experiments~1--4 and the robustness check for
Experiment~3.}}
\label{fig:app-experiments-a}
\end{figure}
\FloatBarrier

\begin{figure}[h]
\centering
\begin{tikzpicture}
\begin{groupplot}[
  appgrid,
  group style={group size=2 by 2, horizontal sep=1.8cm, vertical sep=2.75cm},
]
\nextgroupplot[title={{Experiment 5: dimension scaling}},
  xmode=log, ymode=log, log basis x=2,
  xlabel={dimension $d$},
  ylabel={{oracle calls} to reach $\eps$},
  xmin=6.5, xmax=160, ymin=4e3, ymax=4e6]
\addplot[thick, MidnightBlue, mark=*, error bars/.cd, y dir=both, y explicit]
  table[x=d, y=calls_ns, y error=se_ns] {experiments/exp5_dimension.dat};
\addlegendentry{{nonsep., op}}
\addplot[thick, MidnightBlue, densely dashed, mark=o, mark options={solid},
  error bars/.cd, y dir=both, y explicit]
  table[x=d, y=calls_ns_hs, y error=se_ns_hs] {experiments/exp5_dimension.dat};
\addlegendentry{{nonsep., HS}}
\addplot[thick, BrickRed, mark=triangle*, error bars/.cd, y dir=both, y explicit]
  table[x=d, y=calls_sep, y error=se_sep] {experiments/exp5_dimension.dat};
\addlegendentry{separable, HS}
\addplot[thick, BrickRed, densely dashed, mark=triangle, mark options={solid},
  error bars/.cd, y dir=both, y explicit]
  table[x=d, y=calls_sep_op, y error=se_sep_op] {experiments/exp5_dimension.dat};
\addlegendentry{separable, op}
\addplot[gray!70!black, thin]
  table[x=d, y=pred_ns] {experiments/exp5_dimension.dat};
\addlegendentry{finite-$d$ model}
\addplot[gray!70!black, thin, forget plot]
  table[x=d, y=pred_ns_hs] {experiments/exp5_dimension.dat};
\addplot[gray!70!black, thin, forget plot]
  table[x=d, y=pred_sep] {experiments/exp5_dimension.dat};
\addplot[gray!70!black, thin, forget plot]
  table[x=d, y=pred_sep_op] {experiments/exp5_dimension.dat};

\nextgroupplot[title={{Budget scaling at $d=32$}},
  xmode=log, ymode=log,
  xlabel={{oracle calls $T$}},
  ylabel={$\E f(\hat x)-\fstar$},
  xmin=1.8e4, xmax=1.2e6, ymin=2.5e-5, ymax=1.2e-3]
\addplot[thick, MidnightBlue, mark=*, error bars/.cd, y dir=both, y explicit]
  table[x=T, y=err_ns, y error=se_ns] {experiments/exp5_budget.dat};
\addlegendentry{{nonsep., op}}
\addplot[thick, MidnightBlue, densely dashed, mark=o, mark options={solid},
  error bars/.cd, y dir=both, y explicit]
  table[x=T, y=err_ns_hs, y error=se_ns_hs] {experiments/exp5_budget.dat};
\addlegendentry{{nonsep., HS}}
\addplot[thick, BrickRed, mark=triangle*, error bars/.cd, y dir=both, y explicit]
  table[x=T, y=err_sep, y error=se_sep] {experiments/exp5_budget.dat};
\addlegendentry{separable, HS}
\addplot[thick, BrickRed, densely dashed, mark=triangle, mark options={solid},
  error bars/.cd, y dir=both, y explicit]
  table[x=T, y=err_sep_op, y error=se_sep_op] {experiments/exp5_budget.dat};
\addlegendentry{separable, op}
\addplot[gray!70!black, dotted, thick, domain=2.56e4:8.192e5, samples=2]
  {1.0e-3*(x/2.56e4)^(-2/3)};
\addlegendentry{slope $-2/3$}

\nextgroupplot[title={{Experiment 6: the $\ell_1$ ball}},
  xmode=log, ymode=log, log basis x=2,
  xlabel={dimension $d$},
  ylabel={{oracle calls} to reach $\eps$},
  xmin=13, xmax=320, ymin=4e3, ymax=4e6]
\addplot[thick, MidnightBlue, mark=*, error bars/.cd, y dir=both, y explicit]
  table[x=d, y=calls_md, y error=se_md] {experiments/exp6_l1.dat};
\addlegendentry{mirror descent}
\addplot[thick, BrickRed, mark=triangle*, error bars/.cd, y dir=both, y explicit]
  table[x=d, y=calls_eucl, y error=se_eucl] {experiments/exp6_l1.dat};
\addlegendentry{Euclidean descent}
\addplot[gray!70!black, dashed, domain=16:256, samples=2] {6000*x/16};
\addlegendentry{slope $1$}
\addplot[gray!70!black, dotted, thick, domain=16:128, samples=2]
  {1.0e4*(x/16)^2};
\addlegendentry{slope $2$}

\nextgroupplot[title={{Experiment 7: higher-order smoothness}},
  xmode=log, ymode=log,
  xlabel={{oracle calls $T$}},
  ylabel={$\E f(\hat x)-\fstar$},
  xmin=8e3, xmax=2.6e6, ymin=5e-6, ymax=2e-3]
\addplot[thick, MidnightBlue, mark=*, error bars/.cd, y dir=both, y explicit]
  table[x=T, y=err_m2, y error=se_m2] {experiments/exp7_highorder.dat};
\addlegendentry{frames, $m=2$}
\addplot[thick, BrickRed, mark=triangle*, error bars/.cd, y dir=both, y explicit]
  table[x=T, y=err_m1, y error=se_m1] {experiments/exp7_highorder.dat};
\addlegendentry{frames, $m=1$}
\addplot[thick, ForestGreen, densely dashed, mark=square, mark options={solid},
  error bars/.cd, y dir=both, y explicit]
  table[x=T, y=err_ker, y error=se_ker] {experiments/exp7_highorder.dat};
\addlegendentry{kernel, order 4}
\addplot[gray!70!black, dashed, domain=1.28e4:1.64e6, samples=2]
  {3.0e-4*(x/1.28e4)^(-0.8)};
\addlegendentry{slope $-4/5$}
\addplot[gray!70!black, dotted, thick, domain=1.28e4:1.64e6, samples=2]
  {6.5e-4*(x/1.28e4)^(-2/3)};
\addlegendentry{slope $-2/3$}
\end{groupplot}
\end{tikzpicture}
\caption{{Experiments~5--7 and the robustness check for
Experiment~5.}}
\label{fig:app-experiments-b}
\end{figure}
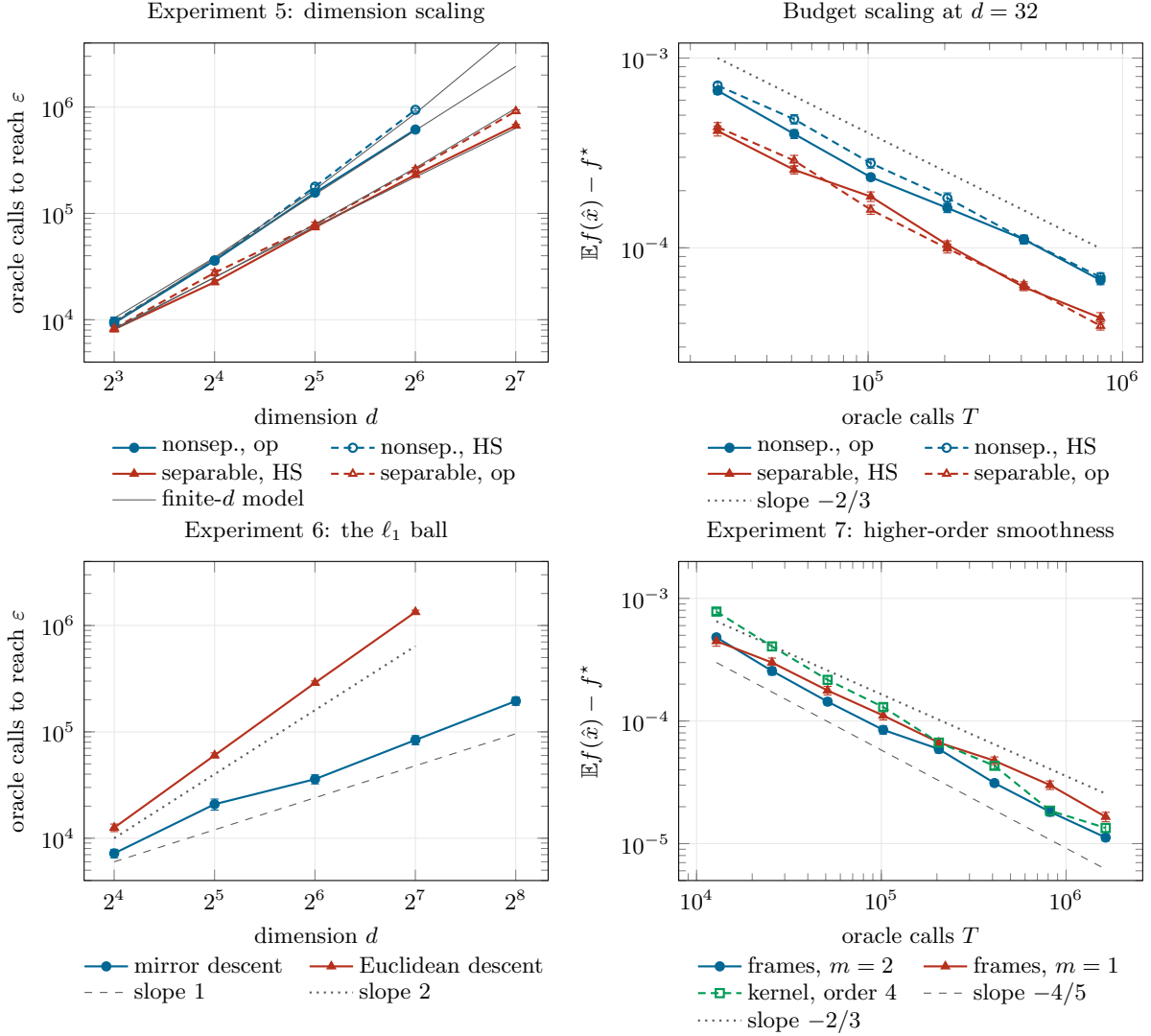

\paragraph{Experiment 1: one estimate under systematic and bounded errors.}
The objective is $f(x)=\frac12\norm x^2$, the center is $z=e_1/2$, the
stencil has $m=2$ and $h=0.1$, oracle values are exact, and
$d\in\{4,8,\dots,256\}$; the same 400 frames per $d$ serve all error
models. The Lipschitz error is
$\eta(x)=M\sum_{k=1}^{50}w_k\abs{\inner{v_k}{x}-c_k}$ with $M=1$,
independent uniform unit vectors $v_k$, weights $(w_k)$ from the flat
Dirichlet distribution and offsets
$c_k=\inner{v_k}z+(h/\sqrt d)\zeta_k$, $\zeta_k\sim U[-1,1]$, so that every
kink passes close to $z$; $\eta$ is drawn once per $d$, before the frames. This
one function does not probe the supremum over all Lipschitz errors; as a check
of the order $M$, the linear error $\eta(x)=M\inner vx$ shifts every frame by
exactly $Mv$, and the measured error has norm $1.000000$ for every $d$. The
bounded error fixed in advance is $\delta\,\sgn\sin(1000\inner ux)$ with
$\delta=Mh$ and a fresh uniform unit vector $u$ for each frame, also drawn
before the frames; its values depend on the frame only through the query
points, so it is a bounded error as in \eqref{eq:noise-adv}, not oblivious
noise. Independent random signs of size $\delta$ are not simulated; they serve
as the reference value $M\sqrt{S_2d/2}$. The adversarial error sees the frame
and adds $\pm\delta\,\sgn(a_j)\sgn\inner{U_i}w$ to the value at $z\pm ht_jU_i$
for a fixed unit vector $w$; the error of the estimate then has norm exactly
$S_1M\sqrt d=3\sqrt d$ for every frame. The root-mean-square error under the
Lipschitz error stays between $0.14$ and $0.21$ for all $d$, almost all of it a
deterministic shift (the norm of the mean error is at least $96\%$ of the
root-mean-square error), and below the bound $A_1M\sqrt{d/(d-1)}$ of
\cref{lem:lipschitz-part}(d), which decreases from $1.92$ to $1.67$. The error
under the bounded error fixed in advance agrees within $2.3\%$ (at most $1.6$
standard errors) with the reference $M\sqrt{S_2d/2}\approx1.90\sqrt d$. A
Lipschitz error chosen after the frame as in \cref{rem:adaptive-lipschitz}
gives an error of norm exactly $M\sqrt{d/2}$.

\paragraph{{Shared objective and method for Experiments 2, 3, and 5.}}
Here $d=20$, except in {Experiment~5 and the robustness
checks described with Experiments~3 and~5}, and $f(x)=\frac12\norm x^2+c\,g(x)$
with $g(x)=\inner wx\norm x^2e^{-\norm x^2}$, $c=1/4$ and
$w=d^{-1/2}(1,\dots,1)$, so that $\xstar=0$, $\fstar=0$ and
$\nabla^2f(0)=I$. The eigenvalues of $\nabla^2g$ have absolute value at most
$1.69123$ on $\R^d$, attained at $x=\pm1.1495w$; this value is computed
numerically, on a grid of step $0.005$ after reduction to the plane spanned by
$w$ and $x$ and with five local refinements, and it is not a certified bound.
The methods use the bound $1.7$, that is, $\alpha=1-1.7c=0.575$,
$L=1+1.7c=1.425$ and $\kappa=2.48$; so $\alpha=1$ in the normalizations of
{Experiment~3 and its robustness check} is the
curvature at $\xstar$, not the constant of a class.
Since $g$ is odd and $UU^\top=I$, every frame with $m=1$ at the minimizer equals
$b=ch^2e^{-h^2}w$, the gradient of the ball average of $f$ at $0$, and so does
every coordinate stencil; this bias does not depend on $d$. The noise is
$N(0,\sigma^2)$ with $\sigma=0.03$. All methods use $m=1$, the radius
$h=\sqrt2\,T^{-1/6}$ (except in {Experiment~5 and its
budget-scaling check}), and variant (a) of
\cref{alg:descent} with $N=200$ rounds ($k_0=16\kappa=39.7$) and $R=2$, from a
random start $x_0$ with $\norm{x_0}=1$ drawn once per experiment. A round
consists of $B$ frames, that is, $2dB$ calls. Without noise and with
$h=10^{-4}$, this schedule ends at $f(\hat x)-\fstar\le3.4\cdot10^{-18}$, so the
optimization error is negligible. The odd term of $f$ decays, so a very large
radius would see only the quadratic, whose minimizer is $\xstar$; in the
reported runs $h\le0.31$ for $f$, so the radii stay near the analysis.

\paragraph{Experiment 2: oblivious noise.}
The budgets are $T=1.6\cdot10^4\cdot2^k$, $k=0,\dots,7$, with 20 runs each.
In addition to the $N(0,\sigma^2)$ noise, every query $z+ht_ju$ of a round
receives $+b$ and every query $z-ht_ju$ receives $-b$, with $b=\sigma$; this
oblivious error depends only on the position of the query in the round. It is
a crafted failure mode of stencils along fixed directions, not a model of all
oblivious noise. The methods are (i) random frames, (ii) the stencil along
$e_1,\dots,e_d$ with the same query order, and (iii) coordinate stencils with
signs drawn in every round (\cref{lem:coord-estimator}). For (ii) the oblivious
error shifts every coordinate of the estimate by $b/h$, and the predicted
plateau is $f(x_b)-\fstar$, where $x_b$ minimizes $f(x)+(b/h)\sum_ix_i$
(computed by Newton's method). In (iii) the random sign of coordinate $i$
multiplies both the difference and the direction, so the oblivious shift
$\pm b/h$ enters with a random sign and has mean zero; in (i) the random frame
plays this role. Method (ii) stays at $0.97$ to $1.00$ times its plateau,
which grows from $0.15$ to $0.52$ because $b/h$ grows as $h$ decreases. For (i)
and (iii) the oblivious error adds a term of mean zero and covariance
$(b/h)^2I$ to every estimate. The ratio of their errors lies between $0.85$ and
$1.27$ (standard errors $5$ to $10\%$ of each error), and their fitted log-log
slopes are $-0.66$ and $-0.68$ over all budgets and $-0.59$ and $-0.67$ over
$T\ge2.56\cdot10^5$, against the rate $-2/3$.

\paragraph{Experiment 3: admissible levels.}
Here $T=1.024\cdot10^6$ ($B=128$), $h=\sqrt2\,T^{-1/6}=0.141$ and 20 runs.
A unit vector $v$ is drawn before the runs. The systematic error is
$\eta(x)=M\inner vx$. The adversarial error sees the frames and adds
$\pm\delta\,\sgn(a_j)\sgn\inner{U_i}v$ to the value at $z\pm ht_jU_i$, which
adds $(\delta S_1/h)\sum_i\sgn\inner{U_i}v\,U_i$ to the frame, a vector with
mean $a_\delta v$, $a_\delta=(\delta S_1/h)\,d\,\E\abs{u_1}$ for $u$ uniform on
the sphere ($d\,\E\abs{u_1}=3.61$ for $d=20$). The $15$ log-spaced values of
$M$ and of $\delta$ were set from a pilot estimate of $\eps_T$ on a separate
stream. All $31$ settings share the frames and the Gaussian noise of every
call, and the errors are reported as ratios to $\eps_T=3.03\cdot10^{-5}$,
the error at $M=\delta=0$. They are plotted against $x=M/\sqrt{\alpha\eps_T}$
and $x=\delta\sqrt d/(h\sqrt{\alpha\eps_T})$, from $0.031$ to $31$, with
$\alpha=1=\nabla^2f(0)$.

The predictions follow the mean of the estimate. Near $\xstar=0$ the
estimate has mean $b+av$ up to higher-order terms, with $a=M$ for the
systematic error, since $\nabla\eta=Mv$, and $a=a_\delta$ for the adversarial
one. The method then tends to $-(b+av)$, where $f$ exceeds its value at $-b$
by $a\inner bv+a^2/2$ up to higher-order terms. We therefore predict the
ratio $1+(F(a)+a\inner bv)/\eps_T$, where $F(a)=f(x_a)-\fstar$ and $x_a$
minimizes $f(x)+a\inner vx$ (computed by Newton's method); for the systematic
error, $F(M)$ is the error that remains with exact values of $f+\eta$. Here
$\inner bv=0.346\sqrt{\alpha\eps_T}$, and the squared bias $\norm b^2/2$ is
$0.39\,\eps_T$. The ratio of the measured to the predicted error lies between
$1.00$ and $1.02$ for both errors at all $15$ levels, at most $0.53$ standard
errors away. The errors cross $2\eps_T$ at $x\approx1.07$ (systematic) and
$x\approx1.30$ (adversarial) by linear interpolation in log-log coordinates
between the two grid points around the crossing ($1.08$ and $1.34$ by a power
law $1+cx^p$ through them), while the predictions cross $2$ at $1.11$ and
$1.37$ (computed on a fine grid); without the cross term $a\inner bv$ the first
would cross $2$ at $1.41$. For $x\ge6$ the errors grow with fitted slopes
$1.96$ and $1.94$, and at the largest level the ratios are $542.7$ and $351.8$,
against the predictions $540.9$ and $350.0$. The systematic and adversarial
errors are tested separately; their joint effect is not.

\paragraph{Experiment 4: rounds against calls.}
Here $d=20$, $m=2$, the noise is independent $N(0,\sigma^2)$ with
$\sigma=0.1$, the budget is $T=2md\cdot2^{14}=1{,}310{,}720$ calls and the
radius is $h=0.3\,(d^2\sigma^2/T)^{1/8}=0.0613$. The objective is
$f(x)=\frac12x^\top Ax-b^\top x+0.1\log\cosh\inner wx$ with $A$ diagonal,
eigenvalues log-spaced in $[1,\kappa]$, $\kappa\in\{10,100\}$, a unit vector
$w$ and a minimizer $\xstar$ of norm $2$, drawn once per $\kappa$; then
$I\preceq\nabla^2f\preceq(\kappa+0.1)I$. Its third derivative is small, so
this experiment concerns rounds, not the order of the stencil. The methods use
$R=4$, $\alpha=1$ and $L=\kappa+0.1$. Each of the $20$ runs draws a pool of
$n=2^{14}$ frames and of the noise of their calls, which both methods use in
order, for both values of $\kappa$ and all numbers of rounds (common random
numbers). Without noise, the estimate at $20$ random centers in $B(R)$
differs from $\nabla f$ by at most $1.0\cdot10^{-9}$. For
$N\in\{8,16,\dots,2^{14}\}$ rounds the batch has $B=2^{14}/N$ frames. The
two-phase method is variant (a) of \cref{alg:descent}. The accelerated method
is variant (b) with $N_0=\lceil5\sqrt{L/\alpha}\,\rceil$ ($16$ and $51$)
rounds per stage and with its iterates projected onto $B(R)$. If $N<N_0$, it
runs one stage of $N$ rounds with $B$ frames; otherwise it runs
$J=\min\{\lfloor N/N_0\rfloor,\lfloor\log_2(2^{14}/N_0+1)\rfloor\}$ stages,
stage $j$ with $\lfloor2^{13+j}/(N_0(2^J-1))\rfloor$ frames per round, so
that the batches double up to rounding. Rounds beyond the last stage are not
used: the method uses at most $160$ rounds for $\kappa=10$ and $408$ for
$\kappa=100$, and \cref{fig:experiments}(b) shows it against the rounds it
uses, once per schedule. The rounding leaves at most $1.1\%$ of the budget
unused. The statistical level is $\bar\sigma^2\tr(\nabla^2f(\xstar)^{-1})/(2n)$
with $\bar\sigma^2=\sigma^2S_2/(2h^2)$, the asymptotic error of
Polyak--Ruppert averaging \citep{PolyakJuditsky1992} of $n$ frames with exact
gradients and the same noise; it is $2.34\cdot10^{-3}$ for $\kappa=10$ and
$1.35\cdot10^{-3}$ for $\kappa=100$. The experiment compares methods at one
budget; it is not a measurement of the least number of rounds.

The two-phase method uses $N$ rounds. Its error is $24$ times the
statistical level at $N=64$, $2.8$ times it at $N=128$ and $1.7$ to $2.3$ times
it for $N\ge256$ when $\kappa=10$, and $13$ times it at $N=512$, $2.4$ times it
at $N=1024$ and $1.7$ to $2.0$ times it for $N\ge2048$ when $\kappa=100$. The
accelerated method attains its smallest error with a single stage: $3.5$
times the statistical level with $16$ rounds for $\kappa=10$ and with $32$
rounds for $\kappa=100$, where $16$ rounds still give $35$ times this level.
With more rounds it runs more stages, its last stage has fewer frames per
round, and its error rises to $5.0$ to $6.7$ times the statistical level with
$32$ to $160$ rounds for $\kappa=10$, and to $5.5$ to $7.3$ times it with
$102$ to $408$ rounds for $\kappa=100$. The number of rounds the two-phase
method needs grows by a factor $8$ from $\kappa=10$ to $\kappa=100$, that of
the accelerated method by a factor $2$.

\paragraph{Experiment 5: calls against the dimension.}
Besides the objective $f$ above we use the separable objective
$f_{\mathrm s}(x)=\sum_j\psi(x_j)$, $\psi(t)=\frac12t^2+ct^3e^{-t^2}$, with the
profile of $f$ on the line through $w$ along each axis. Its minimizer is $0$,
$\nabla^2f_{\mathrm s}(0)=I$, the eigenvalues of $\nabla^2f_{\mathrm s}$ lie in
$[1-1.6912c,1+1.6912c]$, so the methods use the same $\alpha$ and $L$, and
$\norm{D^3f_{\mathrm s}(0)}_{\op}=\norm{D^3f(0)}_{\op}=6c$. Since
$\nabla^2f_{\mathrm s}$ is diagonal, its Lipschitz constant is $\sup\abs{\psi'''}$
in the operator and in the Hilbert--Schmidt norm, while $D^3f(0)$ has
Hilbert--Schmidt norm of order $c\sqrt d$. A frame at the minimizer has bias
$ch^2e^{-h^2}w$ for $f$, and $\frac{3ch^2}{d+2}(1,\dots,1)$, of norm
$\frac{3ch^2\sqrt d}{d+2}$, for $f_{\mathrm s}$ to leading order in $h$. Both
objectives are run with both radius rules $h=C(d^{2+\tau}\sigma^2/T)^{1/6}$ of
\eqref{eq:rate}, $\tau=0$ (op) and $\tau=1$ (HS), with variant (a) of
\cref{alg:descent}, $N=200$ rounds of $B$ frames, $T=2dNB$, the target
$\eps=2\cdot10^{-4}$ and a uniform unit start drawn once per $d$. For the
matched pairs (non-separable, op) and (separable, HS) the constant
$C\in\{2^{k/4}:-4\le k\le4\}$ was tuned once, on a separate stream at $d=16$
with 10 runs, budgets $B=1,\dots,32$ and common randomness across $C$, by
minimizing the interpolated number of calls to reach $\eps$; the chosen values
are $C=1.68$ and $C=1.19$, both inside the grid. The other rule of each
objective uses the constant that gives the same radius at $d=16$, so that the
two rules differ only in how the radius scales with $d$. For each $d$ the
number of frames per round is halved or doubled until the mean error brackets
$\eps$, with $40$ runs for $d\le16$ and $20$ for $d\ge32$; the non-separable
objective is not run at $d=128$. For $f_{\mathrm s}$ the radii reach $0.49$
($d=128$), but a frame moves each coordinate by $h$ times a coordinate of a
uniform unit vector, and the next term of the bias is smaller than the leading
one by a factor of order $h^2/d$.

With the matched rules the calls grow with the fitted slopes $2.03\pm0.03$
(non-separable) and $1.62\pm0.03$ (separable) over $8\le d\le64$, and
$1.61\pm0.02$ for the separable objective up to $d=128$; with the other rules
the slopes are $2.21\pm0.05$ and $1.64\pm0.03$ over $8\le d\le64$, and
$1.68\pm0.02$ for the separable objective up to $d=128$. The ratio of the
calls of the two objectives with the matched rules grows from $1.14$ at $d=8$
to $2.66$ at $d=64$. The rule of the other seminorm needs $1.03$, $1.00$,
$1.14$ and $1.53$ times the calls of the matched rule for the non-separable
objective ($d=8,\dots,64$), and $1.02$, $1.23$, $1.06$, $1.12$ and $1.37$ times
them for the separable one ($d=8,\dots,128$). At $d=16$ the two rules give the
same radius, so the ratio $1.23$ there, $2.4$ standard errors from $1$, is a
fluctuation of the runs; it indicates their variability.

To compare the slopes with a finite-$d$ prediction we model the error after
the method as $A\,k_b(d)^2h^4+Vd^2\sigma^2/(Th^2)$, with $k_b=c$ for $f$ and
$k_b=3c\sqrt d/(d+2)$ for $f_{\mathrm s}$, the leading terms of the squared bias
and of the noise. With the radius rule and $C$ optimal at $d=16$, where the
bias term is then half the noise term, the calls to reach $\eps$ are
proportional to
\begin{equation}
\label{eq:exp5-model}
\Bigl(\tfrac12v(16)\,\frac{u(d)}{u(16)}+v(d)\Bigr)^{3/2},\qquad
u(d)=k_b(d)^2d^{2(2+\tau)/3},\quad v(d)=d^{(4-\tau)/3},
\end{equation}
and we fit one constant per curve by least squares on the log scale. The model
neglects higher-order terms of the bias, the variance of the frames without
noise and the optimization error. For $8\le d\le64$ it has the slopes $2.00$
and $2.13$ for the non-separable objective with the op and the HS rule, and
$1.59$ and $1.67$ for the separable one with the HS and the op rule ($1.57$
and $1.72$ up to $d=128$); its slopes as $d\to\infty$ are $2$, $3$, $3/2$ and
$2$. The measured slopes differ from those of the model by $+0.03$, $+0.08$,
$+0.03$ and $-0.03$ ($+0.03$ and $-0.04$ up to $d=128$). At all $18$ points the
residuals $\log(\text{measured}/\text{model})$ are at most $0.104$ in
absolute value, that is, within $11\%$, and at most $2.7$ standard errors of
the logarithm of the calls; \texttt{run\_log.txt} lists them.

\paragraph{Experiment 6: the $\ell_1$ ball.}
The objective is $f(x)=\sum_j\varphi(x_j-c_j)$ with $\varphi(t)=e^t-1-t$ and
$c=\frac12e_1$: it is convex, its minimizer is $1$-sparse and lies in
$B_1(1)$, and $f(0)-\fstar=0.107$. Since $\varphi'''$ is unbounded on $\R$, $f$
does not belong to the global class $\mathcal C_1(R,G_\infty)$ of
\cref{thm:l1-upper}, and we use its local form, \cref{rem:l1-local}. The target
is $\eps=0.01$ and the noise is $N(0,\sigma^2)$ with $\sigma=0.03$. Both methods
use $m=1$ and the radius of \cref{thm:l1-upper}(b),
$h=(\eps/(4L_\beta))^{1/2}=0.0299$, with $L_\beta=e^{1+h}$. Every query of method
(i) differs from a point of $B_1(1)$ in one coordinate by at most $h$, and every
query of method (ii) lies in $B_1(1)+hB_2(1)$; in both cases all coordinates
are at most $1+h$ in absolute value. On this set
$\nabla^2f=\mathrm{diag}(e^{x_j-c_j})$, so $e^{1+h}$ bounds the H\"older constant
of $\nabla^2f$ in the $\ell_1$ operator norm, and the local constant of
\cref{rem:l1-local}. The noise of one coordinate of either estimate has standard
deviation $\sigma/(\sqrt2h)=0.71$. Method (i) is entropic mirror descent with
coordinate stencils, \cref{alg:entropic-coordinate}, with the step
$\varsigma=\mu(2\log(2d)/(\bar G^2N))^{1/2}$; here
$\bar G^2=2(G_\infty+\eps/4)^2+2(1+\log(2d))\sigma^2/h^2$ is $\bar G_\eps^2$
of \cref{thm:l1-upper}(b) for Gaussian noise, and $G_\infty=e-1=1.72$ bounds
$\norm{\nabla f}_\infty$ on $B_1(1)$. Method (ii) is projected stochastic
gradient descent onto $B_1(1)$ with the exact Euclidean projection, one
frame per round and the step $\mu/(\bar G_2\sqrt N)$, where
$\bar G_2^2=2(G_2+\eps/4)^2+d\sigma^2/h^2$ and $G_2=1.76$ bounds
$\norm{\nabla f}$ on $B_1(1)$. Both output the average of their centers and
use $2d$ calls per round. The multiplier $\mu\in\{2^k:k=-2,\dots,4\}$ was
tuned in the same way for both methods, on a separate stream at $d=32$ with
10 runs, budgets $N=128,\dots,8192$ and common randomness across $\mu$, by
minimizing the interpolated number of calls to reach $\eps$; the chosen
values are $\mu=8$ for (i) and $\mu=2$ for (ii), both inside the grid. For
each $d$ the number of rounds is halved or doubled until the mean error
brackets $\eps$ (20 runs; 10 and 6 for (ii) at $d=64$ and $128$). Method (i)
needs $7.2\cdot10^3$, $2.1\cdot10^4$, $3.6\cdot10^4$, $8.4\cdot10^4$ and
$2.0\cdot10^5$ calls for $d=16,32,\dots,256$, that is, $225$ to $381$ rounds;
the fitted slope in $d$ is $1.14\pm0.06$ for $16\le d\le128$ and $1.15\pm0.04$
up to $256$. Method (ii) needs $1.3\cdot10^4$, $6.0\cdot10^4$, $2.9\cdot10^5$
and $1.34\cdot10^6$ calls for $d\le128$, that is, $393$ to $5{,}230$ rounds, and
a fitted slope of $2.25\pm0.04$. It was not run at $d=256$: extrapolating its
number of rounds gives about $1.2\cdot10^4$, and one literal $256\times256$
frame takes about $6$\,ms. The instance is sparse and favorable to the
$\ell_1$ geometry; it illustrates the upper bound and is not a minimax
comparison.

\paragraph{Experiment 7: higher order.}
The objective is $f_7(x)=\frac12\norm x^2+c\inner wx(1-e^{-\norm x^2})$ with
$c=1/4$, $w=d^{-1/2}(1,\dots,1)$ and $d=16$. It is $C^\infty$ with bounded
derivatives of every order, so it belongs to the classes of every $\beta$;
$\xstar=0$, $\fstar=0$, $\nabla^2f_7(0)=I$, and the eigenvalues of the Hessian
of the odd term $g_7$ have absolute value at most $1.9518$ (computed
numerically as for $g$), so the methods use $\alpha=1-1.96c=0.51$ and
$L=1.49$. Unlike the odd term of $f$, that of $f_7$ does not decay: at a very
large radius every estimate sees $x+cw$, whose root $-cw$ has error
$2.7\cdot10^{-2}$, so wide tuning grids cannot favor huge radii. At the minimizer
the bias is $ch^2w$ for central differences, to leading order, and $O(h^4)$
for $m=2$ and for the kernel estimate; noiseless means over many frames or
directions at $h=0.1,0.2,0.4$ confirm the orders ($2.49\cdot10^{-3}$,
$9.80\cdot10^{-3}$, $3.70\cdot10^{-2}$ for $m=1$ and $3.1\cdot10^{-6}$,
$4.9\cdot10^{-5}$, $7.5\cdot10^{-4}$ for $m=2$). Three estimates with $4dB$
values per round drive variant (a) of \cref{alg:descent} with $N=200$ rounds:
(i) $B$ frames with $m=2$ and the radius rule of $\beta=5$,
$h=C(d^2\sigma^2/T)^{1/10}$; (ii) $2B$ frames with $m=1$ and the rule of
$\beta=3$, $h=C(d^2\sigma^2/T)^{1/6}$; (iii) $2dB$ directions of the
$\ell_2$-randomized kernel estimate of \citet{Akhavan2020,Akhavan2024},
$\frac d{2h}(Y(x+hr\zeta)-Y(x-hr\zeta))K(r)\zeta$ with $\zeta$ uniform on the
sphere, $r$ uniform on $[-1,1]$ and the kernel $K(r)=\frac{15r}4(5-7r^2)$ of
order $4$, with the rule of $\beta=5$. The cited works use one such estimate per
iteration of stochastic gradient descent; here all three estimates drive the
same batched method, so that only the estimate differs. Each constant
$C\in\{2^{k/2}:-4\le k\le8\}$ was tuned at $T=2.05\cdot10^5$ on a separate
stream with 10 runs and common randomness across $C$, by minimizing the mean
error; the chosen values are $2.83$, $1.41$ and $2.83$, all inside the grid.
The budgets are $T=1.28\cdot10^4\cdot2^k$, $k=0,\dots,7$, with 20 runs each.
The fitted slopes in $T$ are $-0.76$, $-0.67$ and $-0.85$ over all budgets and
$-0.80$, $-0.67$ and $-0.82$ for $T\ge2.05\cdot10^5$, against the rates
$-4/5$, $-2/3$ and $-4/5$ of the analyses. The error of central differences
is $0.93$ to $1.65$ times that of the frames with $m=2$, growing with $T$, and
the error of the kernel estimate is $1.02$ to $1.62$ times it.

\paragraph{{Additional checks for Experiments 3 and 5.}}
{
\emph{Thresholds across dimensions and budgets.}
We repeat Experiment~3 at
$(d,T)=(10,1.024\cdot10^6)$, $(40,1.024\cdot10^6)$, and
$(20,2.56\cdot10^5)$, using nine log-spaced levels from $x=0.25$ to $4.5$ and
independent streams for the instance, pilot, and reported runs. The measured
systematic doubling points are $1.69$, $1.36$, and $1.35$, compared with
predictions $1.79$, $1.32$, and $1.35$. The adversarial values are $2.08$,
$1.67$, and $1.67$, compared with $2.19$, $1.64$, and $1.67$. Their variation
comes from the instance-dependent cross term $a\inner bv$:
$\inner bv/\sqrt{\alpha\eps_T}=-0.34$, $0.10$, and $0.07$.

\emph{Budget scaling at fixed dimension.}
We also rerun the four configurations of Experiment~5 at $d=32$, using its
tuned constants, $N=200$ rounds, $B=2,4,\dots,64$ frames, and 20 independent
runs per budget. Thus $T$ ranges from $2.56\cdot10^4$ to
$8.19\cdot10^5$ oracle calls and the radii range from $0.17$ to $0.39$.
Because $h$ is proportional to $T^{-1/6}$, the predicted error is
$T^{-2/3}$. The four fitted slopes are $-0.648\pm0.017$,
$-0.677\pm0.017$, $-0.668\pm0.020$, and $-0.703\pm0.018$; for
$T\ge1.02\cdot10^5$ they are $-0.60\pm0.03$, $-0.67\pm0.04$,
$-0.71\pm0.04$, and $-0.67\pm0.03$. The interpolated calls required for
$\eps=2\cdot10^{-4}$ are $1.39\cdot10^5$, $1.77\cdot10^5$,
$8.8\cdot10^4$, and $7.8\cdot10^4$, compared with $1.57\cdot10^5$,
$1.78\cdot10^5$, $7.4\cdot10^4$, and $7.8\cdot10^4$ in the independent
dimension-scaling runs of Experiment~5.
}

\section{Discussion and open questions}
\label{app:discussion}

Random noise costs calls but not rounds: the accelerated method keeps the
$O(\sqrt\kappa\log(1/\eps))$ rounds of Nesterov's method with exact gradients,
and only the number of calls per round grows with $d$ and $\sigma$. Persistent
errors are governed by their regularity rather than their size. A Lipschitz
error fixed before the run is not averaged out: for $\eta(x)=M\inner vx$ every
estimate is shifted by $Mv$. But the shift is $O(M)$, without a factor
$\sqrt d$, so the error costs $O(M^2/\alpha)$ in value and is admissible up to
a Lipschitz constant of order $\sqrt{\alpha\eps}$ in every dimension, however
large its values are. An error known only to be bounded, such as a rounding
error, is admissible only up to $\min\{\sqrt d,\sqrt{\log(4T)}\}\,\delta_\eps$,
even if it is a fixed function of the query, and only for $\eps\le\eps_0$; the
achievable level $c_\beta\delta_\eps$ does not grow with the budget, since the
radius cannot exceed the order of $h_\eps$.

\paragraph{Limitations.}
The methods need $\beta$, $L_\beta$, $\sigma$, $\alpha$ and $L$; adaptation to
unknown parameters is not studied. The counts concern oracle calls and
adaptive depth: a round assumes unlimited parallelism, a Haar frame costs
$O(d^3)$ arithmetic operations, and the methods of \cref{thm:nh-main} minimize
polynomials at a cost we do not bound. Some constants are far from practical:
the factor in $c_N$ of \cref{thm:rounds-lower}(b) is below $10^{-45}$ for
$\beta=3$ and $N=3$; the method of \cref{thm:nh-main}(b) starts with
$O(d^{\beta^2-1}\log(d+T))$ calls for integer $\beta$, that is $d^8$ for
$\beta=3$ and $d^{15}$ for $\beta=4$; and with our constants the range of
accuracies in \cref{cor:natural-curvature}(a), and that of the $\ell_1$ lower bound of
\cref{sec:other}, are nonempty only for $d$ above large thresholds, which
exceed $2^{28}$ for the latter at $\beta=3$. Under unrestricted Hessian
magnitude, our analysis of the initial polynomial fit uses centered noise. The
upper bounds on constrained sets query the $h$-neighborhood of the set. The
experiments check the predictions on particular instances, which are not
worst cases: Hilbert--Schmidt smoothness appears only through the separable
instance of \cref{fig:experiments}(c) and Experiment~5, the order $\beta$ only
through the
comparison of the stencils with $m=1$ and $m=2$ in (d), and the methods of
\cref{thm:gs,thm:nh-main} are not tested; minimax lower bounds cannot be
confirmed by finitely many instances.

\paragraph{Open questions.}
\emph{Rounds.} The lower bound $N_\beta(d)$ of \cref{thm:rounds-lower}(c)
concerns the rate in the number of calls per round for fixed $d$, while for a
fixed accuracy \cref{thm:main}(c) needs $O(\sqrt\kappa\,(1+\log(LR^2/\eps)))$
rounds in every dimension. How the minimal number of rounds depends on
$\kappa$ and $\eps$ is open, since our nested hard functions have
$\kappa\le3$. Under the condition of \cref{thm:gs}, we do not know whether
$O(\sqrt{\kappa_{\mathrm{eff}}}\log(1/\eps))$ rounds suffice.
\emph{Convex functions.} The minimax number of calls is known only in the
ranges of \cref{cor:natural-curvature}(a) and, for $L\ge L_\eps$ and
$L_\beta R^\beta d^{-\beta/4}\le\eps\le c'_\beta L_\beta R^\beta$, up to the factor
$C_\beta V/L_\eps$ between the leading statistical terms
(\cref{cor:natural-curvature}(b)); the full upper bound adds
$32d(1+LR^2/\eps)$ calls, and for smaller $\eps$ the gap grows. The admissible
adversarial level is known only up to a factor polynomial in $1/\eps$
(\cref{tab:noise}).
\emph{Unbounded Hessians.} Can the rate of \cref{thm:nh-main} be reached with
a smaller startup, under oblivious noise or with systematic and adversarial
errors, or in polynomial time?

\section{Further related work}
\label{app:related}

\paragraph{Oracle calls.}
Finite differences of noisy values go back to \citet{KieferWolfowitz1952};
higher smoothness is exploited by differences of higher order
\citep{Fabian1967,Dippon2003}, optimal in dimension one \citep{Chen1988}, and by
randomized differences weighted by a kernel \citep{Polyak1990}. Omitting
factors free of $d$ and $\eps$, the number of calls has since fallen from
$d^2\eps^{-(\beta+1)/(\beta-1)}$ \citep{Bach2016} to
$d^2\eps^{-\beta/(\beta-1)}$ \citep[Theorem~18 and Corollary~19]{Akhavan2024},
with intermediate steps by \citet{Akhavan2020} and \citet{Novitskii2022}. For
integer $\beta$ the class of \citet{Akhavan2024} is defined through derivatives
of order $\beta$, but their bias bound \citep[Lemma~6]{Akhavan2024} uses only
the Taylor remainder of order $\beta$, which is bounded on our classes. Their
noise, like that of \citet{Lobanov2024Improved}, may be oblivious: independent
of the random directions and not necessarily centered. The lower bound
$d^{\beta/(\beta-1)}\eps^{-\beta/(\beta-1)}$ of \citet[Theorem~6.1]{Akhavan2020}
and \citet[Theorem~22]{Akhavan2024} comes from separable functions, for which
it is attained \citep{AkhavanTsybakov2025}. For a Hilbert--Schmidt bound on the Lipschitz
constant of the Hessian, $\beta=3$, \citet{Yu2024} established matching bounds
of order $d^{3/2}\eps^{-3/2}$ for centered noise and unrestricted Hessian
magnitude; their Theorems~3.1 and~3.2 are asymptotic in the number of calls
with the other parameters fixed. With a bounded Hessian alone,
$d^2/\eps^2$ calls are necessary
\citep[improving \citealp{Jamieson2012}]{Shamir2013}, but the third
derivatives of the hard functions are not bounded uniformly in $\eps$.
\citet{WangBalakrishnanSingh2018} study local minimax rates for noisy
optimization of smooth functions, a neighboring problem with a different
notion of optimality. \citet{AkhavanLouatiTsybakov2026}
minimize a partial derivative $\partial_j^kg$ of an unknown function from noisy
values of $g$, with the optimal rate $N^{-(\beta-1)/(\beta+k)}$ in the number
$N$ of queries; for $k=0$ their bound is $d^{(2\beta-1)/\beta}N^{-(\beta-1)/\beta}$.
Hessians that grow at most linearly with the gradient \citep{Zhang2020Clipping}
were treated for two-point feedback by \citet{LobanovGasnikov2025}.
Zeroth-order methods are also used for nonconvex problems such as black-box
attacks on neural networks \citep{Chen2017ZOO} and fine-tuning of language
models \citep{Malladi2023MeZO}, which lie outside our assumptions.

\paragraph{Directions.}
Our estimate uses all axes of a uniformly random orthonormal basis.
Finite differences along random orthogonal directions, which include spherical
smoothing and coordinate descent as special cases, were analyzed by
\citet{Kozak2023Orthogonal}, and stochastic descent along $\ell\le d$ structured
orthogonal directions by \citet{Rando2024Structured}. Their oracles and
classes differ from ours; we use stencils of high order and the noise of
\cref{def:oracle}, in which the frames make an oblivious error average out and
a fixed Lipschitz error cost $O(M)$ (\cref{prop:estimator}).

\paragraph{Rounds.}
Parallel first-order methods need many rounds in high dimension
\citep{Nemirovski1994,DiakonikolasGuzman2020}, and zeroth-order methods for
functions with a bounded Hessian need about $d\log\log n$ rounds of $n$ calls
to keep the sequential rate \citep{DRY2018}. For highly smooth functions,
batched kernel methods with bounds on the rounds exist for convex
\citep{Lobanov2023Kernel}, Polyak--{\L}ojasiewicz
\citep{GasnikovLobanovStonyakin2024} and strongly convex
\citep[Theorem~2]{Lobanov2024Improved} objectives; the last one makes $d^2$ calls, up to
logarithmic factors, in $O(\sqrt\kappa\log(1/\eps))$ rounds for
$\kappa=L/\alpha$. Compared with it, \cref{thm:main}(a) has no logarithmic
factor in the error bound, holds under Hilbert--Schmidt smoothness with the
power {$d^s$}, and admits systematic and adversarial errors at the levels of
\cref{thm:main}(c), which are sharp up to the factors stated there.

\paragraph{Admissible errors.}
Systematic and adversarial errors have mostly been studied for less smooth
functions and for two-point feedback, where two values share one random
sample \citep{Duchi2015,Shamir2017}. There \citet{Gasnikov2022Power} bounded
the calls, the rounds and the largest admissible adversarial error, which
\citet{Pasechnyuk2026MALN} determined up to logarithms, building on
approximately convex optimization
\citep{Belloni2015,SingerVondrak2015,RisteskiLi2016}. In
\citet{Pasechnyuk2026MALN} the errors are fixed functions on a Euclidean ball,
queries and output lie in the ball, and the guarantee holds with high
probability; their lower bound
$\eps^{3/2}/(R\sqrt{dL})$ on the admissible level (\cref{tab:noise}) is proved on
the subclass with a Lipschitz constant of order $\sqrt{L\eps}$ on the ball, and
their upper bound for a range of accuracies and for $d$ at least of the order
of the logarithm of the number of queries. Lipschitz errors were treated by
\citet{GasnikovDvurechenskyNesterov2014} and by \citet{Dvinskikh2023}, whose
constant is of order $\eps/(R\sqrt d)$ for exact values and two-point feedback,
and accelerated methods with bounded errors by
\citet{GorbunovDvurechenskyGasnikov2022}. For highly smooth
Polyak--{\L}ojasiewicz functions, kernel methods tolerate bounded
deterministic errors of order $\eps^{\beta/(2(\beta-1))}d^{-\beta/(\beta-1)}$
\citep[Theorems~1 and~3]{GasnikovLobanovStonyakin2024}.

\paragraph{Convex functions and other geometries.}
For convex functions the optimal dependence on $d$ and $\eps$ is open. Finite
differences use $d^2\eps^{-2\beta/(\beta-1)}$ calls \citep{Bach2016}, as does the
method of \citet{Akhavan2024} applied to a regularized objective
(\cref{app:prior-convex}). Smoothing with a bounded Hessian
\citep{Flaxman2005,SahaTewari2011} uses $d^2L\eps^{-3}$ calls
\citep[Section~4.4]{Gasnikov2022Survey}, a bound stated with the second moment
of the values in place of $\sigma^2$; our bounds depend on $\sigma$ but not on
the size of the values, since the stencil uses differences of values. Bandit
methods use $\eps^{-2}$ calls up to logarithmic factors, removed for $d=1$ only
recently \citep{Carpentier2026OneDim}, but with a larger power of $d$
\citep{Lattimore2023Second,Fokkema2024,Lattimore2024BCO}: the methods of
\citet{Lattimore2023Second,Fokkema2024} use $d^p\eps^{-2}$ calls with
$3\le p\le7/2$ for centered sub-Gaussian noise and functions that are
Lipschitz on the domain, as ours are with constant $2LR$ on $B(R)$. The lower
bound derived from the separable functions of \citet{Akhavan2020,Akhavan2024}
(\cref{app:prior-convex}) is smaller in $d$ or in $1/\eps$. For convex Lipschitz functions,
smoothing uses $d^2\eps^{-4}$ calls \citep{Flaxman2005}, which
\citet{RajaramanHan2026} proved necessary for $\eps$ at least of order
$d^{-1/4}$ by hiding curvature near a tube defined by an unknown linear map.
For other norms, a linear map onto a Euclidean ball transfers lower bounds, as
\citet{Dorn2026GradientFree} showed for exact values. One-point methods with
nearly linear dependence on $d$ were known on the simplex for $\beta\le2$ and
Euclidean constants \citep{Gasnikov2017Online}, and two-point ones adapted to
$\ell_1$ by \citet{Gasnikov2016Simplex,Akhavan2022L1}. For first-order methods,
a mismatch between the norm of smoothness and the feasible set also changes
the complexity: for a Euclidean-Lipschitz gradient on the $\ell_1$ ball,
$LR^2/N^3$ after $N$ gradients is the minimax rate up to logarithms in high
dimension, $d\ge N+1$ for deterministic and $d\ge32N+1$ for randomized
methods \citep{Loginov2026Mismatched}.

\section{Comparison with concurrent work}
\label{app:concurrent}

\citet{Zhang2026Sharp} study noisy zeroth-order minimization of strongly
convex functions whose derivative of order $q$ is H\"older continuous in the
multilinear operator norm, from values with independent Gaussian noise and
with adaptive queries anywhere in $\R^d$. In our notation their class is
$\Fsc(\beta,1;1,3,1)$ with the operator seminorm and $\sigma=1$, and their
main theorem states that for every fixed $\beta>2$ and all integers $d,T\ge1$
the minimax error is of order $\min\{1,(d^2/T)^{(\beta-1)/\beta}\}$, so that
$d^2\eps^{-\beta/(\beta-1)}$ calls are necessary and sufficient for small $\eps$.
Their Appendix~E extends this to other fixed values of the parameters, with
constants that depend on them. The two works were carried out
independently. This appendix lists where they overlap and where they differ.

\paragraph{Overlap.} For the operator seminorm, \cref{cor:normalized} with
{$s=2$} is their main theorem, and \cref{thm:main}(b) contains their lower
bound. Both lower bounds perturb a quadratic by a linear function with
direction $\omega/\sqrt d$, $\omega\in\{-1,1\}^d$, under one radial cutoff,
and both bound the sum over the coordinates of the squared changes of a
value under a sign flip uniformly in the query point
(\cref{thm:radial-lower}, Step~3). The upper bound of this order was known
\citep{Akhavan2024}; they also prove it with a spherical estimator at the
radii $2^{-j}h$ combined with Lagrange weights and with projected
stochastic gradient descent, one estimate per iteration.

\paragraph{Differences.}
\begin{itemize}[leftmargin=*,itemsep=1pt]
\item \emph{Dependence on the parameters.} In \cref{thm:main} the constants
depend only on $\beta$: for $\delta=0$, $0\le M\le\alpha R$ and budgets above the
two thresholds stated after \cref{thm:main}, the minimax error is of order
$\Phi_T+M^2/\alpha$ jointly in all parameters, and \cref{thm:main}(b) holds for
every budget with the explicit saturation level $\Phi_0$. \citet[Section~5]{Zhang2026Sharp} do not claim a formula that is
sharp jointly in these parameters.
\item \emph{Hilbert--Schmidt smoothness.} They do not consider the
Hilbert--Schmidt seminorm, for which the dimension factor is $d^{3/2}$ for
even $q$ and $d^{2-(q-1)/(2(\beta-1))}$ for odd $q$
(\cref{thm:main,cor:normalized}); for even $q$ this needs the ridge functions
of \cref{thm:ridge-lower}. They discuss the result of \citet{Yu2024} for a
Frobenius-norm bound on the Lipschitz constant of the Hessian, $\beta=3$, as a
neighboring class.
\item \emph{Rounds.} Their method computes one estimate per iteration, so it
uses as many sequential rounds as iterations. \Cref{thm:main}(a) attains the
same number of calls in $O(\sqrt\kappa\log(1/\eps))$ rounds, and
\cref{thm:rounds-lower} bounds the number of rounds from below.
\item \emph{Noise.} They consider independent Gaussian noise. Here the
random noise may be biased if it is independent of the directions of the
method, and \cref{thm:main}(c) determines the admissible levels of
Lipschitz systematic errors and, up to $C_\beta\min\{\sqrt d,\sqrt{\log(4T)}\}$,
of adversarial errors.
\item \emph{Other classes.} Functions whose Hessian is not bounded from
above (\cref{thm:gs,thm:nh-main}), convex functions (\cref{sec:convex}) and
other norms and constraints (\cref{sec:other}) are not considered there.
\end{itemize}

\paragraph{Estimators.} Both upper bounds cancel the lower-order terms of the
Taylor expansion by combining several radii, as in \citet{Fabian1967} and
\citet{Polyak1990}. Their estimator differentiates along one random direction
at the radii $2^{-j}h$; ours uses the nodes $j/m$ along all axes of a random
orthonormal basis and averages such frames within a round
(\cref{sec:method}).

\FloatBarrier
\end{document}